\documentclass[a4paper,reqno]{amsart}

\usepackage{amsmath}
\usepackage{amsthm}
\usepackage{amssymb}
\usepackage{amscd} 

\usepackage{bbm} 
\usepackage{mathtools,mathrsfs} 

\usepackage{verbatim} 

\usepackage{enumitem}

\usepackage{xcolor}

\usepackage{xy}
\xyoption{matrix} \xyoption{arrow} \xyoption{arc} \xyoption{color}

\newtheoremstyle{newremark}
  {5pt}
  {5pt}
  {\rmfamily}
  {}
  {\rmfamily\bf}
  {.}
  {.5em}
  {}

\newtheorem{theorem}{Theorem}
\newtheorem{lemma}[theorem]{Lemma}
\newtheorem{corollary}[theorem]{Corollary}
\newtheorem{proposition}[theorem]{Proposition}

\theoremstyle{newremark}
\newtheorem{remark}[theorem]{Remark}
\newtheorem{definition}[theorem]{Definition}

\newtheorem*{definition*}{Definition} 
\newtheorem*{notations*}{Notations}

\numberwithin{theorem}{section}
\numberwithin{equation}{section}

\newcommand{\N}{\mathbb{N}} 
\newcommand{\R}{\mathbb{R}} 

\newcommand{\Z}{\mathbb{Z}} 
\newcommand{\C}{\mathbb{C}} 

\newcommand{\bbD}{\mathbb{D}}

\newcommand{\bbH}{\mathbb{H}}

\newcommand{\bbR}{\mathbb{R}}
\newcommand{\bbS}{\mathbb{S}}

\newcommand{\calL}{\mathscr{L}}
\newcommand{\calM}{\mathscr{M}}

\newcommand{\frS}{\frak S}

\DeclareMathOperator{\rmdist}{\mathrm{dist}} 

\DeclareMathOperator{\rmLip}{\mathrm{Lip}} 

\DeclareMathOperator{\rmsing}{\mathrm{Sing}} 
\DeclareMathOperator{\SO}{SO} 
\DeclareMathOperator{\U}{U} 
\DeclareMathOperator{\On}{O} 
\DeclareMathOperator{\rmKer}{Ker} 
\DeclareMathOperator{\rmRan}{Ran} 

\def\XXint#1#2#3{{%
\setbox0=\hbox{$#1{#2#3}{\int}$}
\vcenter{\hbox{$#2#3$}}\kern-.5\wd0}}

\newcommand{\veps}{\varepsilon}

\renewcommand{\leq}{\leqslant}
\renewcommand{\geq}{\geqslant}
\renewcommand{\subset}{\subseteq}

\newcommand{\LL}{\mathop{\hbox{\vrule height 6pt width .5pt depth 0pt
\vrule height .5pt width 3pt depth 0pt}}\nolimits}

\newcommand{\e}{{\rm e}}

\newcommand{\res}{\mathop{\hbox{\vrule height 7pt width .5pt depth 0pt\vrule height .5pt width 6pt depth 0pt}}\nolimits}

\newcommand{\mb}[1]{{\mathbf #1}}

\newcommand{\bb}[1]{{\mathbb #1}}

\newcommand{\eps}{\varepsilon}

\newcommand{\abs}[1]{\left| #1 \right|} 
\newcommand{\pd}[2]{\frac{\partial #1}{\partial #2}} 
\newcommand{\vol}[1]{\mbox{vol}_{#1}} 

\newcommand{\trans}{\mathsf{t}}
\newcommand{\trac}{{\rm tr}}

\newcommand{\eo}{{\bf e}_0}  
\newcommand{\euu}{{\bf e}^{(1)}_1} 
\newcommand{\eud}{{\bf e}^{(1)}_2} 
\newcommand{\edu}{{\bf e}^{(2)}_1} 
\newcommand{\edd}{{\bf e}^{(2)}_2} 

\begin{document}


\title[Torus-like solutions for the LDG model]{Torus-like  solutions for the Landau- de Gennes model. \\ Part II: Topology of $\bbS^1$-equivariant minimizers}

\author{Federico Dipasquale}
\address{Dipartimento di Informatica, Universit\`{a} di Verona, Strada Le Grazie 15, 37134 Verona, Italy}
\email{federicoluigi.dipasquale@univr.it}

\author{Vincent Millot}
\address{LAMA, Univ Paris Est Creteil, Univ Gustave Eiffel, UPEM, CNRS, F-94010, Cr\'{e}teil, France}
\email{vincent.millot@u-pec.fr}

\author{Adriano Pisante}
\address{Dipartimento di Matematica, Sapienza Universit\`{a} di Roma, P.le Aldo Moro 5, 00185 Roma, Italy}
\email{pisante@mat.uniroma1.it}

\date{\today}

\begin{abstract}
We study energy minimization of a continuum Landau-de Gennes energy functional for nematic liquid crystals, in a three-dimensional axisymmetric domain  and in a restricted class of $\mathbb{S}^1$-equivariant (i.e., axially symmetric) configurations. We assume smooth and nonvanishing $\mathbb{S}^1$-equivariant (e.g. homeotropic) Dirichlet boundary condition and a physically relevant norm constraint (the Lyuksyutov constraint) in the interior. Relying on our previous results  in the nonsymmetric setting \cite{DMP1}, we prove  partial regularity of minimizers away from a possible finite set of interior singularities located on the symmetry axis. For a suitable class of domains and boundary data, we show that for smooth minimizers (torus solutions) the level sets of the signed biaxiality are generically finite unions of tori of revolution. 
Concerning nonsmooth minimizers (split solutions), we characterize their asymptotic behavior around any singular point in terms of explicit $\mathbb{S}^1$-equivariant harmonic maps into $\mathbb{S}^4$, whence the generic level sets of the signed biaxiality contains invariant topological spheres. 
Finally, in the model case of a nematic droplet, we provide existence of torus solutions, at least when the boundary data are suitable uniaxial deformations of the radial anchoring, and existence of split solutions for boundary data which are suitable linearly full harmonic spheres.
\end{abstract}



\maketitle


\tableofcontents


\section{Introduction}
In this second part of our series \cite{DMP1}-\cite{DMP2}, we focus on the regularity and topological properties of minimizers of the Landau-de Gennes (LdG) energy under the norm constraint  considered in~\cite{DMP1} and a further symmetry constraint. We consider minimizers  
of the energy functional over a restricted class of axisymmetric (more precisely, $\bbS^1$-equivariant) configurations which are subject to a smooth axisymmetric Dirichlet boundary condition. Before entering into a detailed discussion, let us briefly review for convenience the mathematical aspects of the LdG model under investigation, notations,  and basic terminology from our first part \cite{DMP1}.  
 \vskip2pt
 
 The Landau-de Gennes theory is a continuum theory aimed to describe macroscopic configurations of nematic liquid crystals. The order parameter used to represent a given configuration is a family of second order tensors usually called $Q$-tensors. Denoting by $\mathscr{M}_{3\times3}(\R)$ the vector space made of $3\times3$-matrices with real entries, $Q$-tensors are the elements of the 5-dimensional subspace 
\[
	\mathcal{S}_0:=\Big\{Q=(Q_{ij})\in\mathscr{M}_{3\times 3}(\R) : Q=Q^\trans \, , \; \trac(Q)=0 \Big\}\,,
\]
where $Q^\trans$ denotes the transpose of $Q$, and $\trac(Q)$ the trace of $Q$. The space $\mathcal{S}_0$ is endowed with the Hilbertian structure given by the usual (Frobenius) inner product. By symmetry of the admissible matrices, the inner product and the induced norm reduce to  
$P:Q:=\sum_{i,j=1}^3P_{ij}Q_{ij}=\trac(PQ)$ and $\abs{Q}^2=\trac(Q^2)$.
Choosing an orthonormal basis, $\mathcal{S}_0$ is identified with the Euclidean $\R^5$, and then $\big\{Q\in\mathcal{S}_0: \abs{Q}=1\big\}=\bb{S}^4$ is the $4$-dimensional sphere.

Within the configuration space $\mathcal{S}_0$, one can distinguish three mutually distinct phases: {\it (i)} the isotropic phase, $Q=0$; {\it (ii)} the uniaxial phase, Q has exactly one double eigenvalue; {\it (iii)} the biaxial phase, $Q$ has three distinct eigenvalues. Following \cite{DMP1}, we shall measure biaxiality  through the {\it signed biaxiality parameter} $\widetilde{\beta}$ defined for $Q\in \mathcal{S}_0\setminus\{0\}$ by 
 \begin{equation}
\label{signedbiaxiality}
\widetilde{\beta}(Q):=\sqrt{6} \,\frac{{\rm tr}(Q^3)}{|Q|^{3}}  \in [-1,1] \, .
\end{equation}
 If $Q\in \mathcal{S}_0$ has spectrum $\sigma( Q)=\{ \lambda_1 ,\lambda_2, \lambda_3 \} \subset \mathbb{R}$ with eigenvalues in increasing order,  then $\widetilde{\beta}(Q)=\pm 1$ iff the minimal/maximal eigenvalue is double (purely positive/negative uniaxial phase), $\widetilde{\beta}(Q)=0$ iff $\lambda_2=0$ and $\lambda_1=-\lambda_3\neq 0$ (maximally biaxial phase), and $Q=0$ iff $\lambda_1=\lambda_2=\lambda_3$ (isotropic phase). Given a continuous configuration map ${\bf Q}:\overline\Omega\to \mathcal{S}_0$ defined on a closed domain $\overline\Omega\subset\R^3$, we have shown in \cite{DMP1} how the topology and the geometry of ${\bf Q}$ can be studied by means of $\widetilde{\beta}$, or more precisely through the {\it biaxiality regions}, i.e., the closed subsets of $\Omega$ of the form
 \begin{equation}\label{biaxialityregions}
	\{\beta \leq t\} := \big\{ x \in \overline{\Omega} : \widetilde{\beta} \circ {\bf Q}(x) \leq t \big\} \text{ and } \{\beta \geq t\} := \big\{ x \in \overline{\Omega} : \widetilde{\beta} \circ {\bf Q}(x) \geq t\big\}\,,\quad t \in [-1,1]\,, 
\end{equation}
and  the corresponding {\it biaxial surfaces} $\{ \beta = t \} := \big\{ x \in \overline{\Omega} : \widetilde{\beta} \circ {\bf Q}(x) = t \big\}$.

\vskip5pt

For a bounded open set  $\Omega \subset \mathbb{R}^3$ with smooth boundary, the Landau-de Gennes energy of a  configuration ${\bf Q} \in W^{1,2}(\Omega; \mathcal{S}_0)$ we consider is of the form
\begin{equation}
\label{LDGenergy}
\mathcal{F}_{\rm LG}({\bf Q})=\int_\Omega \frac{L}{2} \abs{\nabla {\bf Q}}^2+F_{\rm B}({\bf Q}) \, dx \, , 
\end{equation}
i.e., with the one-constant approximation for the elastic energy density and parameter  $L>0$. The bulk potential $F_{\rm B}$ is the quartic polynomial
\begin{equation}
\label{abc-potential}
F_{\rm B}({\bf Q}):=-\frac{a^2}{2} {\rm tr}({\bf Q}^2)-\frac{b^2}{3} {\rm tr} ({\bf Q}^3)+\frac{c^2}{4} \big({\rm tr} ({\bf Q}^2) \big)^2 \, ,
\end{equation}
where $a,b$ and $c$ are non zero material-dependent constants. As usual, we can subtract-off an additive constant and introduce the modified bulk potential 
\begin{equation}
\label{FBtilde}
\widetilde{F}_{\rm B}({\bf Q}):=F_{\rm B}({\bf Q})-\min_{\mathcal{S}_0} F_{\rm B} \, ,
\end{equation}
which is nonnegative and has $0$ as minimum value. 
The minimum is achieved when the signed biaxiality is maximal and
 $\widetilde{F}_{\rm B}({\bf Q})=0$ iff ${\bf Q} \in \mathcal{Q}_{\rm min}$, i.e., ${\bf Q}$ belongs to the vacuum-manifold of positive uniaxial matrices 
\begin{equation}
\label{vacuum}
\mathcal{Q}_{\rm min}= \left\{ {\bf Q} \in \mathcal{S}_0:  {\bf Q}=s_+ \left( n \otimes n -\frac13 I \right) \, , \quad n \in \mathbb{S}^2  \right\} \, ,
\end{equation}
where
\[ s_+:=\frac{b^2+\sqrt{b^4+24a^2 c^2}}{4c^2} \, \]
is the positive root of the equation $2c^2 t^2 -b^2 t-3 a^2=0$. 
The new energy functional corresponding to the modified  bulk potential \eqref{FBtilde} is
\begin{equation}
\label{LDGenergytilde}
\widetilde{\mathcal{F}}_{\rm LG}({\bf Q}):=\int_\Omega \frac{L}{2} |\nabla {\bf Q}|^2+\widetilde{F}_{\rm B}({\bf Q}) \, dx \, , 
\end{equation}
so that it is the sum of  two nonnegative terms penalizing both spatial variations and deviations from the vacuum manifold $\mathcal{Q}_{\rm min}$. Observe that $\mathcal{Q}_{\rm min}\simeq\mathbb{R}P^2\subset \bbS^4$, and as a consequence, $\mathcal{Q}_{\rm min}$ has nontrivial topology. Both  the homotopy groups $\pi_2(\mathcal{Q}_{\rm min})=\mathbb{Z}$ and $\pi_1(\mathcal{Q}_{\rm min})=\mathbb{Z}_2$  play a role in the presence of defects, especially when minimizing the energy in the restricted class of axisymmetric configurations  introduced below.
\vskip5pt

In continuation to \cite{DMP1}, we  still assume in this article  that the norm of  admissible configurations is prescribed to be the constant value proper of the vacuum manifold \cite{Ly} (see also \cite{DMP1} and the references therein), i.e.,
\begin{equation}
\label{Lyuk}
|{\bf Q}(x)| \equiv \sqrt{\frac23}\, s_+  \qquad (\hbox{Lyuksyutov constraint}) \, .
\end{equation}
Following \cite{DMP1}, we rescale any configuration as ${\bf Q} = s_+ \sqrt\frac{2}{3} Q$, and under the Lyuksyutov constraint  the energy functional takes the form 
$$\widetilde{\mathcal{F}}_{\rm LG}({\bf Q})=\frac23 s_+^2L\mathcal{E}_\lambda(Q)$$
where now $Q \in W^{1,2}(\Omega;\bbS^4)$, 
\begin{equation}
\label{LDGenergytilde}
\mathcal{E}_\lambda(Q):=\int_\Omega \frac{1}{2} |\nabla Q|^2+\lambda W(Q)\, dx\,,
\end{equation}
and
$\lambda:=\sqrt{\frac{2}{3}}\frac{b^2}{L}\,s_+>0$. In turn, the reduced potential $W:\mathbb{S}^4\to\R$ is given by 
\begin{equation}
\label{redpotential}
W(Q)=\frac{1}{3\sqrt{6}}\Big(1-\widetilde{\beta}(Q)\Big) \qquad \forall Q\in\mathbb{S}^4\,.
\end{equation} 
By \eqref{signedbiaxiality}, $W$ is nonnegative and it vanishes precisely on $\mathbb{R}P^2$, i.e.,  $\{W=0\}=\mathbb{R}P^2$.

A critical point $Q_\lambda\in W^{1,2}(\Omega;\mathbb{S}^4)$ of $\mathcal{E}_\lambda$ among $\mathbb{S}^4$-valued maps is a weak solution of the Euler-Lagrange equation
\begin{equation}\label{MasterEq}
\Delta Q_\lambda+|\nabla Q_\lambda|^2Q_\lambda =- \lambda \nabla_{\rm tan} W (Q_\lambda) \quad\text{in $\Omega$}\,.
\end{equation}
The tangential gradient of $W$ along $\mathbb{S}^4 \subset \mathcal{S}_0$ is given by
\[  
	\nabla_{\rm tan} W (Q_\lambda)=- \Big(Q_\lambda^2-\frac{1}{3}I - {\rm tr}(Q_\lambda^3)Q_\lambda\Big)  \,,
\]
and the left hand side in \eqref{MasterEq} is the tension field of the $\mathbb{S}^4$-valued map $Q_{\lambda}$ as in the theory of harmonic maps. 
In fact, when $\lambda = 0$, the energy $\mathcal{E}_{0}$ is  the Dirichlet energy
\[
	\mathcal{E}_0(Q) = \int_{\Omega} \frac{1}{2} \abs{\nabla Q}^2\,dx\,, 
\]
and \eqref{MasterEq} is exactly the harmonic map equation for maps from $\Omega$ into $\bbS^4$.
\vskip5pt 

In this article, we are interested in $\mathbb{S}^4$-valued maps minimizing the energy functional $\mathcal{E}_\lambda$ over a class of $\mathbb{S}^1$-equivariant configurations. Our interest arises from the fact that the symmetry of configurations yields symmetry properties of their biaxiality sets, and this allows us to describe such sets in an even more details than in \cite{DMP1}. Symmetry  ans\"atze  have been considered in several recent papers, e.g. \cite{INSZ3,INSZ4,ABL,INSZ5,Yu,ABGL}, in two and three dimensional  Landau-de Gennes models. Here  $\mathbb{S}^1$-equivariance is meant in a sense of axial symmetry. To be more precise, let us define the action of $\bbS^1$ on the space  $\mathcal{S}_0$ we are considering. 
First, we identify the group $\mathbb{S}^1$ with the subgroup of $\SO(3)$ made of all rotations around the vertical axis of $\mathbb{R}^3$. A matrix $R\in\mathscr{M}_{3\times 3}(\mathbb{R})$ represents a rotation of angle $\alpha$ around the vertical axis if it writes
\begin{equation}\label{rotmatr}
R=\begin{pmatrix} \widetilde R & 0 \\ 
  0 & 1 \end{pmatrix}
\quad\text{with}\quad \widetilde R:= \begin{pmatrix} \cos\alpha & -\sin\alpha  \\ \sin\alpha & \cos\alpha  \end{pmatrix}\,. 
\end{equation}
A rotation of angle $\alpha$ around the vertical axis is uniquely determined by $\e^{i\alpha}\in\mathbb{S}^1$ and we will often employ the notation $R_\alpha\in\bb{S}^1$ to specify its rotation angle $\alpha$. The $\mathbb{S}^1$-action by rotations on $\mathbb{R}^3$ yields an induced isometric action on $\mathcal{S}_0$ given by 
\begin{equation}\label{inducedaction}
\mathcal{S}_0 \ni A \mapsto R\cdot A:=R \, A \, R^\trans \in \mathcal{S}_0\,.
\end{equation}
Assuming that the domain $\Omega\subset\R^3$ is axially symmetric ($\mathbb{S}^1$-invariant), i.e., $R \cdot \Omega=\Omega$ for every $R \in \mathbb{S}^1$, we shall say that a map $Q \colon \Omega \to \mathcal{S}_0$ is $\mathbb{S}^1$-equivariant if
\begin{equation}
\label{S1equivariance}
Q(R x)=R\cdot Q(x) =R Q(x) R^\trans\quad \mbox{a.e. } x \in \Omega \, , \quad \forall R \in \mathbb{S}^1 \, ,
\end{equation}
with the obvious analogue definition for maps defined on the boundary. 
This notion of $\mathbb{S}^1$-equivariance is well-defined in the continuous, Lipschitz, or $W^{1,2}$-regularity class. Here we are interested in  
the space of $\mathbb{S}^1$-equivariant Sobolev maps
\begin{equation}
\label{equivsobolev}
W^{1,2}_{\rm sym}(\Omega;\mathbb{S}^4):=\Big\{ Q \in W^{1,2}(\Omega; \mathbb{S}^4) :  \text{ $Q$ is $\mathbb{S}^1$-equivariant} \Big\}\,.	
\end{equation}
Note that the action \eqref{inducedaction} being isometric, $\mathbb{S}^4 \subset \mathcal{S}_0$ is invariant and this space is well-defined.

Given a smooth boundary data  $Q_{\rm b}:\partial\Omega\to\mathbb{S}^4$ which is $\mathbb{S}^1$-equivariant, we set 
\begin{equation}
\label{S1admissibleconf}
\mathcal{A}^{{\rm sym}}_{Q_{\rm b}}(\Omega):=\Big\{ Q \in W^{1,2}_{\rm sym}(\Omega; \mathbb{S}^4) : \, Q_{\vert_{\partial \Omega}}=Q_{\rm b} \, \Big\}\,,
\end{equation}
and we aim to minimize the energy $\mathcal{E}_\lambda$ over $\mathcal{A}^{{\rm sym}}_{Q_{\rm b}}(\Omega)$. As shown in Proposition~\ref{prop:nonempty-class}, 
the set $\mathcal{A}^{{\rm sym}}_{Q_{\rm b}}(\Omega)$ is always not empty whenever $Q_{\rm b}$ is a Lipschitz map, so that the minimization problem is well-defined.
The symmetry constraint \eqref{equivsobolev} and the boundary condition in \eqref{S1admissibleconf} being  $W^{1,2}$-weakly closed, 
the direct method in the Calculus of Variations easily yields existence of minimizers. Concerning regularity, and in contrast with \cite{DMP1}, we do not expect full regularity, but only partial regularity. Indeed,  $Q_{\rm b}$ could have no $\mathbb{S}^4$-valued continuous extension to $\overline{\Omega}$ since $\mathbb{S}^1$-equivariance  
implies that $Q(0,0,x_3)=\pm \eo$ along $\Omega \cap \{x_3 \hbox{-axis}\}$, with 
\begin{equation}
\label{fixedpoint}
	\eo := \frac{1}{\sqrt 6}\begin{pmatrix} -1 & 0 & 0 \\ 0 & -1 & 0 \\ 0 & 0 & 2 \end{pmatrix} \,, 
\end{equation}
$\pm \eo$ being the unique unit norm elements of $\mathcal{S}_0$ which are fixed by the $\mathbb{S}^1$-action.  
In particular, if $\Omega$ is the unit ball and $Q_{\rm b}(0,0,\pm1)=\pm \eo$, then the class $\mathcal{A}^{{\rm sym}}_{Q_{\rm b}}(\Omega)$ contains no map continuous in $\overline{\Omega}$, and thus minimizers must have singularities. 
\vskip5pt 
 
Our first main result provides the partial regularity of minimizers of $\mathcal{E}_\lambda$ in the $\mathbb{S}^1$-equivariant class with a general Dirichlet boundary condition.

\begin{theorem}\label{thm:partial-regularity}
Let $\Omega\subset\R^3$ be a bounded and axisymmetric open set with boundary of class $C^3$, and let 
$Q_{\rm b} \in C^{1,1}(\partial \Omega;\mathbb{S}^4)$ be an $\mathbb{S}^1$-equivariant map. 
If $Q_\lambda$ is a minimizer of $\mathcal{E}_\lambda$ in the class $\mathcal{A}^{{\rm sym}}_{Q_{\rm b}}(\Omega)$,  then $Q_\lambda \in C^{\omega}(\Omega\setminus \Sigma)\cap C^{1,\delta}(\overline \Omega\setminus \Sigma)$ for every $\delta\in(0,1)$, where $\Sigma$ is a finite subset of $\Omega\cap\{x_3\text{-axis}\}$ (possibly empty). Moreover, 
\begin{enumerate}
\item[{\sl (1)}] if $Q_{\rm b} \in C^{2,\delta}(\partial \Omega)$ for some $\delta>0$, then  $Q_\lambda \in C^{2,\delta}(\overline\Omega\setminus\Sigma)$;
\vskip5pt
\item[{\sl (2)}] if $\partial\Omega$ is analytic and $Q_{\rm b}\in C^\omega(\partial\Omega)$, then $Q_\lambda\in C^\omega(\overline\Omega\setminus\Sigma)$.
\end{enumerate} 
\noindent In addition, for each singular point $\bar x\in \Sigma$, there exist a map $Q_* \in \{ \pm Q^{(\alpha)}\}_{\alpha \in \R}$ and $\nu>0$ such that 

\begin{equation}\label{asymptnearsingthm}
\|Q_\lambda^{\bar x,r}-Q_*\|_{C^2(\overline{B_2}\setminus B_1)}=O(r^\nu)\quad\text{as $r\to 0$}\,,
\end{equation}
where $Q_\lambda^{\bar x,r}(x):=Q_\lambda(\bar x+r x)$ and for a given rotation $R_\alpha\in\mathbb{S}^1$ acting on $\mathcal{S}_0$ as in \eqref{inducedaction}, \begin{equation}
\label{stableblowups}
\hbox{ } \qquad \quad Q^{(\alpha)}(x):=  R_\alpha \cdot \frac{1}{\sqrt 6} \frac{1}{|x|}\begin{pmatrix} -x_3 & 0 & \sqrt{3} x_1 \\ 0 & -x_3 & \sqrt{3} x_2 \\ \sqrt{3} x_1 & \sqrt{3} x_2 & 2x_3 \end{pmatrix} \, , \quad x=(x_1,x_2,x_3) \in \mathbb{R}^3 \setminus \{0\} \, .
\end{equation}
\end{theorem}
\vskip5pt

In establishing Theorem \ref{thm:partial-regularity}, the starting point is to realize that 
minimizers of $\mathcal{E}_\lambda$ over the symmetric class $\mathcal{A}^{{\rm sym}}_{Q_{\rm b}}(\Omega)$ essentially satisfy Palais' Symmetric Criticality Principle (although neither the functional is $C^1$-differentiable, nor the space $W^{1,2}(\Omega;\mathbb{S}^4)$ has a Banach manifold structure), see Proposition \ref{prop:symmetric-criticality}.  They are therefore  true critical points of $\mathcal{E}_\lambda$, and hence weak solutions of \eqref{MasterEq}. As a consequence,   the regularity results from our first part \cite{DMP1} apply, and we prove that the smallness of the scaled energy $\frac1r \mathcal{E}_\lambda(Q_\lambda,\Omega \cap B_r(x))$ implies the regularity of $Q_\lambda$ in a neighborhood of $x\in\overline\Omega$.
Hence, to obtain partial regularity, we employ the strategy introduced 
in the pioneering papers \cite{SU1,SU2,SU3} and already adopted in \cite{DMP1}. It is based  on three main ingredients: 1) monotonicity formulas; 2) strong compactness of blow-ups; 3) constancy of blow-up limits (Liouville property). Compared to the classical case,  energy minimality only holds in the restricted class $\mathcal{A}^{{\rm sym}}_{Q_{\rm b}}(\Omega)$ of equivariant configurations, and these three fundamental ingredients have to be reworked out carefully, as we comment in more detail at the beginning of Section~\ref{sec:6}. 
 The crucial difference  with \cite{DMP1} is that singularities can not be excluded here (as already noticed), which means that the Liouville property does not  hold. However, constancy of blow-ups holds at the boundary as in \cite{DMP1}, it still holds away from the symmetry axis, and it can only fail at finitely many interior points on the axis (the singular points) by a classical argument.
 In Section~\ref{sec:stability}, we identify all possible nonconstant blow-up limits and prove that they form  
a one-parameter family of nonconstant $0$-homogeneous $\mathbb{S}^1$-equivariant harmonic (tangent) maps $\{ \pm Q^{(\alpha)} \}_{\alpha \in \mathbb{R}}$. Those tangent maps are smooth away from the origin and minimize the Dirichlet energy among all compactly supported $\mathbb{S}^1$-equivariant  perturbations. In Theorem~\ref{thm:partial-regularity}, \eqref{asymptnearsingthm} completely characterizes the asymptotic behavior of a minimizer near a singular point, and it shows in particular uniqueness of the blow-up limit.
To prove \eqref{asymptnearsingthm}, we make use of the Simon-{\L}ojasiewicz inequality (see \cite{Sim83}) for the Dirichlet energy on $C^3(\mathbb{S}^2; \mathbb{S}^4)$, adapting the simplified argument in 
\cite{Simon} to our perturbed Dirichlet energy \eqref{LDGenergytilde}. As detailed in Section~\ref{subsec:Loj}, the set of all possible tangent maps is  contained in a smooth manifold, so that the Simon-{\L}ojasiewicz inequality holds with optimal exponent, which implies the power-type decay    in~\eqref{asymptnearsingthm}. 
\vskip5pt

The family $\{ \pm Q^{(\alpha)} \}_{\alpha \in \mathbb{R}}$ is determined through a stability/instability analysis of $0$-homogeneous equivariant harmonic maps. To this purpose, we first classify all $\mathbb{S}^1$-equivariant harmonic spheres $\boldsymbol{\omega}\in C^\infty(\mathbb{S}^2;\mathbb{S}^4)$, see Theorem~\ref{classification}.
Identifying $\mathcal{S}_0$ with $\R \oplus \C \oplus \C$,  the $\bbS^1$-action on $\mathcal{S}_0$ rewrites in terms of complex numbers as $R_\alpha \cdot (t, \zeta_1, \zeta_2)=(t,e^{i\alpha} \zeta_1, e^{i2\alpha} \zeta_2)$. In this way, and by a classical result due to E. Calabi \cite{Ca}, $\mathbb{S}^4$-valued equivariant harmonic spheres are either $\mathbb{S}^2$-valued maps ({\em linearly degenerate}) of the form
\[ 
	\boldsymbol{\omega}^{(1)} (x)=( \boldsymbol{\omega}_0 (x), \boldsymbol{\omega}_1(x), 0 ) \, , \qquad \hbox{or} \qquad \boldsymbol{\omega}^{(2)} (x)=( \boldsymbol{\omega}_0 (x), 0 , \boldsymbol{\omega}_2(x) ) \, ,
\]
 or {\em linearly full} (the image spans the whole space $\mathcal{S}_0$). 
 
 In the linearly degenerate case, we show that, up to the application of the antipodal map $\bbS^4\ni a \mapsto -a\in\bbS^4$,
 \begin{equation}
\label{degenerate-spheres}
	\boldsymbol{\omega}^{(k)} (x)=\boldsymbol{\sigma}_2^{-1} \left( \mu_k \left( \boldsymbol{\sigma}_2 (x) \right)^k  \right) \, , \quad \mu_k \in \mathbb{C}^*  \, ,\quad k=1,2\,,
\end{equation}
where $\boldsymbol{\sigma}_2 \colon \mathbb{S}^2 \to \mathbb{C} \cup \{ \infty\}$ is the stereographic projection from the south pole. 
  
To describe linearly full harmonic spheres into $\mathbb{S}^4$, we follow \cite{Ca} (see also \cite{Br,La,Verdier,BairdWood}, and Section \ref{sec:axisymm}).
Identifying $\mathbb{S}^2$ with $\mathbb{C}P^1$, we study  their {\em canonical lift} to $\mathbb{C}P^3$, the twistor space\footnote{As detailed, for instance, in \cite[Chapter~7]{BairdWood}, the twistor space of $\mathbb{S}^4$ is $\SO(5)/\U(2)$; however, it is elementary but not obvious to identify it with $\mathbb{C}P^3$ (see e.g. \cite{Fawley} for details on this identification).} of~$\mathbb{S}^4$. 
Up to the application of the antipodal map and up to the postcomposition with the twistor fibration $\boldsymbol\tau \colon \mathbb{C}P^3 \to \mathbb{S}^4$, one has a one-to-one correspondence between harmonic spheres into  $\mathbb{S}^4$ and {\em horizontal} algebraic curves $\widetilde{\boldsymbol\omega} \colon \mathbb{C}P^1 \to \mathbb{C}P^3$ corresponding to their twistor lift\footnote{For a precise definition of the twistor lift, we refer to Proposition \ref{liftings}, Remark \ref{twistorlift}, and Theorem \ref{classification}.}. The commutative diagram
  \begin{equation} \label{sobolevcomm}
 \xymatrix{
 & \mathbb{C}P^3 \ar[d]^{\boldsymbol\tau} \\
\mathbb{S}^2=\mathbb{C}P^1 \ar[r]^{{\boldsymbol\omega}} \ar[ur]^{{\widetilde{\boldsymbol\omega}}} & \mathbb{S}^4 } \end{equation}
reflects the fact that $\boldsymbol\omega= \boldsymbol\tau \circ \widetilde{\boldsymbol\omega}$. Lifting the $\mathbb{S}^1$-action to $\mathbb{C}P^3$ and specializing to the equivariant maps allows to classify all the possible canonical lifts that  can be written in homogeneous coordinates $[z_0, z_1] \in \mathbb{C}P^1$ as
\begin{equation}
\label{twistorlifts}
\widetilde{\boldsymbol\omega} ([z_0, z_1] )=\left[ z_0^3, \mu_1 z_0^2 z_1 , \mu_2 z_0 z_1^2 , -\frac{\mu_1 \mu_2}{3} z_1^3 \, \right] \in \mathbb{C}P^3 \, , \qquad (\mu_1,\mu_2)  \in \mathbb{C}^* \times \mathbb{C}^* \, . 
\end{equation}

Any $0$-homogeneous harmonic map into $\mathbb{S}^4$ is of the form $Q(x)=\boldsymbol\omega\left(x/|x|\right)$ for some harmonic sphere $\boldsymbol\omega$.  A tricky but elementary argument shows that such a map is unstable within the equivariant class as soon as the last component  $\pmb{\omega}_2$ does not vanish identically. Therefore, both $\boldsymbol{\omega}^{(2)}\left(x/|x|\right)$ and all $0$-homogeneous extensions of linearly full harmonic spheres $\boldsymbol\omega= \boldsymbol\tau \circ \widetilde{\boldsymbol\omega}$ corresponding to \eqref{twistorlifts} are unstable. On the other hand, stability holds for $\boldsymbol{\omega}^{(1)}\left(x/|x|\right)$ if and only if $|\mu_1|=1$ in~\eqref{degenerate-spheres}. Setting $\mu_1=e^{i \,\alpha}$ with  $\alpha \in \mathbb{R}$, we obtain  the family $\{Q^{(\alpha)}\}_{\alpha\in\R}$ defined in \eqref{stableblowups}, and we prove that each $Q^{(\alpha)}$ is in fact  locally minimizing the Dirichlet energy among all equivariant configurations. 
\vskip5pt

Given a minimizer $Q_\lambda$ of the energy $\mathcal{E}_\lambda$ in the equivariant class, existence or nonexistence of singularities  
turns out to be a subtle issue, depending on the nature of the boundary data $Q_{\rm b}$, as well as on the topology and the geometry of the domain $\Omega$.  
Our third part \cite{DMP2} is dedicated to a detailed analysis on this problem. In this article, we want to emphasize that both cases can occur, providing 
some natural and topologically nontrivial data $Q_{\rm b}$ leading to either smoothness or singularities in the case where $\Omega$ is a nematic droplet (i.e., $\Omega=B_1$ the unit ball). Beyond the question of existence of singularities, we are also interested in the topological properties  of the biaxial surfaces $\{ \beta=t\}$, $t\in (-1,1)$, as they encode the topology and the geometry of a configuration.  Natural boundary data to consider 
are smooth maps with values in $\mathbb{R}P^2$  exploiting  the nontrivial topology of $\mathbb{R}P^2$.
 Boundary data with  low regularity, namely maps in $W^{1/2,2}(\partial \Omega;\mathbb{R}P^2)$ with point singularities representing the nontrivial element in $\pi_1(\mathbb{R}P^2)=\mathbb{Z}_2$, are also of interest as they lead, at least in nonsymmetric settings, to topological defects touching the boundary, see \cite{Can,CanOrl}. They are not considered here, and $Q_{\rm b}$ is always assumed to be smooth.

As recalled in \cite{DMP1}, if the domain $\Omega$ is simply connected, then the same holds for $\partial \Omega$, and any map $Q_{\rm b}\in C^1(\partial \Omega; \mathbb{R}P^2)$ can be written in the form
\begin{equation}
  \label{Qblift}
  Q_{\rm b}(x)=\sqrt{\frac32} \left( v(x) \otimes v(x) -\frac13 I\right) \qquad \hbox{for all } x\in \partial \Omega \, , \quad v\in C^1(\partial \Omega;\mathbb{S}^2) \, .
\end{equation} 
If $\partial \Omega$ is of class $C^2$ and $v(x)$ in \eqref{Qblift} is the outer unit normal $\overset{\to}{n}(x)$, 
we obtain the so-called {\sl homeotropic} boundary condition (or {\sl radial anchoring}). When $\Omega$ is axially symmetric, 
such $Q_{\rm b}$ is $\mathbb{S}^1$-equivariant if and only if $v$ is $\mathbb{S}^1$-equivariant (with respect to the obvious action of $\mathbb{S}^1$ on $\mathbb{S}^2\subset \mathbb{R}^3$ by rotations around the vertical axis). In particular, if $\Omega$ is axially symmetric and $v(x)=\overset{\to}{n}(x)$, then  $Q_{\rm b}$ as in \eqref{Qblift} is $\mathbb{S}^1$-equivariant.

To motivate and illustrate our discussion, we now consider the important case where $\Omega$ is a nematic droplet, that is $\Omega=B_1$  the unit ball, and  
 $v(x)=\overset{\to}{n}(x)=\frac{x}{|x|}$ in \eqref{Qblift}. Then $Q_{\rm b}(x)=\overline{H}(x)$ for $x\in\partial B_1$, where $\overline{H}$ is the {\sl constant-norm hedgehog}

\begin{equation}
  \label{Hbar}
  \overline{H}(x)=\sqrt{\frac32} \left( \frac{x}{|x|}\otimes \frac{x}{|x|} -\frac13 I\right) \,.
\end{equation}
Notice that $\overline{H}$ is  equivariant with respect to the full orthogonal group $\On(3)$, and it turns out that $\overline{H}$ is the unique $\On(3)$-equivariant critical point of $\mathcal{E}_\lambda$ with homeotropic boundary condition.
We have shown in \cite{DMP1} that  $\overline{H}$ is unstable with respect to $\mathbb{S}^1$-equivariant perturbations, so that it is {\it not} a minimizer of  $\mathcal{E}_\lambda$, neither globally nor among the $\mathbb{S}^1$-equivariant class.
Therefore, $\On(3)$-symmetry breaking occurs. Concerning minimizers  $Q_\lambda$ of $\mathcal{E}_\lambda$ in the class  $ \mathcal{A}^{\rm sym}_{\overline{H}}(B_1)$, we expect them 
to be smooth, although this remains a major open problem. In addition to smoothness, and as already discussed in \cite{DMP1} in the nonsymmetric context, the corresponding biaxiality regions $\{ \beta < t\}$  at  regular values $t\in (-1,1)$  (see \eqref{biaxialityregions}) should form an increasing family of axially symmetric solid tori. In turn, their complements $\{ \beta \geq t \}$ should be kind of distance neighborhoods from the boundary $\partial B_1$ with cylindrical neighborhoods of the vertical axis added. When $t=- 1$, the set $\{\beta=-1\}$ should be a {\sl disclination line}, i.e., a horizontal circle $\Gamma$ where exchange of the two smallest eigenvalues of $Q_\lambda$ occurs. Finally, the set $\{ \beta=1\}$ should be the union of $\partial B_1$ with the vertical diameter $I$. In this picture, sub- and superlevel sets of the biaxiality function are {\em mutually linked} in the sense of \cite{DMP1} (i.e., each set is not contractible in the complement of the other)
because the subsets $\Gamma$ and $\partial \Omega \cup I$ are. There is a wide numerical evidence for these symmetry properties to hold. Indeed, this conjectural description has been already investigated, first in \cite{ScSl,PeTr,Kleman}, and then in \cite{SKH,KVZ,GaMk,KV}, where authors refer to such an equilibrium configuration as the {\bf ``torus solution''} of the Landau-de Gennes model (see also \cite{DR,LPZZ,HQZ} for further numerical results in this direction). 
\vskip5pt

In the attempt to validate partially this  aforementioned picture, we provide in the next theorem the first existence result of torus solutions to \eqref{MasterEq} in the case of a nematic droplet and suitable deformations of the radial anchoring $\overline{H}$ as boundary data.

\begin{theorem}\label{thm:examples-tori}
Assume that $\Omega = B_1$, and let 
$\overset{\to}{n}$ be the outer unit normal field on $\partial B_1$. There exists a sequence of $\mathbb{S}^1$-equivariant maps $\{v_j \} \subset C^\infty(\partial B_1;\mathbb{S}^2)$ which are equivariantly homotopic to $\overset{\to}{n}$ and satisfying $\overset{\to}{n} \cdot v_j>0$ on $\partial B_1$ for all $j$,  such that for the following holds. For each $j$, let $Q_{\rm b}^j\in C^\infty_{\rm sym}(\partial B_1;\R P^2)$ be as in \eqref{Qblift} with $v\equiv v_j$, and $Q^{j}$  any minimizer of  $\mathcal{E}_\lambda$ over  $\mathcal{A}^{\rm sym}_{Q_{\rm b}^{j}}(B_1)$. Then, 
\begin{enumerate}
\item[(1)] the sequence $\{Q_{\rm b}^j\}$ is bounded in $W^{1/2,2}(\partial B_1;\mathbb{S}^4)$;
\vskip5pt
\item[(2)] $Q^j \rightharpoonup \mathbf{e}_0$ weakly in $W^{1,2}(\Omega;\mathbb{S}^4)$ as $j\to \infty$;
\vskip5pt
\item[(3)] up to a subsequence, $|\nabla Q^j|^2 dx \mathop{\rightharpoonup}\limits^{*} c\, \mathcal{H}^1 \res \mathcal{C}$ weakly-$*$ as measures on $\overline B_1$ as $j\to \infty$, where $\mathcal{C}=\partial \Omega \cap \{ x_3=0\}$ and $c>0$ is a constant.
\end{enumerate}
As a consequence,   for $j$ large enough and setting $\beta_j:=\beta \circ Q^j$, 
\begin{enumerate}
\item[(4)] $Q^{j}$ is smooth in $\overline{\Omega}$; 
\vskip5pt
\item[(5)] the negative uniaxial set $\{ \beta_j=-1\}$ is not empty and contains an $\mathbb{S}^1$-invariant circle $\Gamma_j\subset B_1 \setminus I$ with $I:=B_1\cap\{x_3 \hbox{-axis}\}$, while the positive uniaxial set  $\{\beta_j=1\}$contains  $\partial \Omega \cup I$, and in particular $\{\beta_j=1\}$ and $\{ \beta_j=-1\}$ are mutually linked; 
\vskip5pt
\item[(6)] for every $0<\rho<1$ and $t\in[-1,1)$, we have $ \{\beta_j \leq t\} \subset \big\{ x \in \Omega : \rmdist(\mathcal{C},x) < \rho \big\}$ for $j$ large enough (depending on $\rho$ and $t$); 
\vskip5pt
\item[(7)] any regular biaxial surface $\{\beta_j=t\}$, $t\in (-1,1)$, is a finite union of axially symmetric tori.
\end{enumerate}
\end{theorem} 
The proof of the theorem will be presented in Section~\ref{sec:topology}, and here we outline its main ideas in order to illustrate how suitable topologically nontrivial boundary data yields the emergence of disclination lines and in turn of biaxial tori. In view of \eqref{Qblift}, the boundary data $Q_{\rm b}^j$ is uniaxial with simple eigenvalue $\lambda_3\equiv \frac2{\sqrt{6}}$. In addition, the map $v_j$ orients the corresponding eigenspace map $V^j_{\rm max}$, which can  actually be identified with $Q_{\rm b}^j$ because of \eqref{Qblift}. Setting $\mathcal{O}_\rho:=\big\{ x \in \Omega : \rmdist(\mathcal{C},x) < \rho \big\}$, we construct  $v_j$  as a deformation of $\overset{\to}{n}$ in such a way that  $Q_{\rm b}^j \equiv \mathbf{e}_0$ in $\partial \Omega \setminus \overline{\mathcal{O}_\rho}$ for $j$ large enough. 
Extending  $V^j_{\rm max}$ to the segment $I$ as $\mathbf{e}_0$, for each vertical open half-disc $\mathcal{D}^+$ with $\partial \mathcal{D}^+ \subset \partial \Omega \cup I$, we have a well-defined continuous map 
\begin{equation}
\label{nontrivialloop-j}
\gamma^j \colon \partial \mathcal{D}^+ \to \mathbb{R}P^2 \, , \qquad \gamma^j(x):=V^j_{\rm max}(x) \, .
\end{equation} 
 Then $[\gamma^j] \neq 0$ in $\pi_1(\mathbb{R}P^2)$ since $v_j$ and $\overset{\to}{n}$ are (equivariantly) homotopic. Hence, the boundary data $Q_{\rm b}^j$ is topologically nontrivial. By construction, the sequence $\{Q^j\}$ converges weakly to the constant map $\mathbf{e}_0$  (i.e, claim {\it (1)} holds),  exhibiting on each meridian a $W^{1/2,2}$-bubbling of the nontrivial element of  $\pi_1(\mathbb{R}P^2)$ (see Remark \ref{H12bubbling}). By the compactness property of minimizers, we show that the weak convergence of $\{Q^j\}$ improves to strong convergence 
away from the set $\mathcal{C}$ toward a limiting minimizing map. This limiting map is constant because of its constant trace, which proves claim {\it (2)}. Applying the $\varepsilon$-regularity theorem from \cite{DMP1} near the vertical axis for $j$ large enough, we infer that each map $Q^j$ must be smooth up to the boundary, and  $W^{1,2}$-boundedness easily yields claim~{\it (3)}.  Since $[\gamma^j] \neq 0$ in $\pi_1(\mathbb{R}P^2)$ and $Q^j$ is smooth, the set $\{ \beta_j=-1\}\cap \mathcal{D}^+$ contains at least one point  
(otherwise the loop $\gamma^j$ would be contractible). Thus $\{ \beta_j=-1\}$ contains an invariant circle $\Gamma_j\subset \Omega \setminus I$, and $\partial \Omega \cup I \subset \{\beta_j=1\}$ by regularity. As a consequence,  any pair of biaxial sets $\{ \beta_j \leq t_1\}$ and $\{\beta_j \geq t_2\}$ with $-1\leq t_1< t_2\leq 1$ are mutually linked (for $j$ large enough), and any regular surface $\{\beta_j=t\} \subset \Omega$, $t\in(-1,1)$, is a finite union of axially symmetric tori, in agreement with the discussion above. At this stage, we do not know whether or not  the ``disclination line'' $\Gamma_j$ is unique, or if the biaxial surfaces $\{ \beta_j=t\}$ are connected and provide a regular foliation of $\Omega \setminus (I\cup\Gamma_j)$ by tori as $t$ runs from $-1$ to $1$. These questions seem to be quite difficult and remain open problems. 
\vskip5pt

Still in the case of a nematic droplet, we provide in our next result  examples of boundary data  leading to singular minimizers with an even number of singularities ({\em split minimizers}, according to Definition \ref{def:split-minimizer}). 


\begin{theorem}\label{thm:examples-split}
	Assume that $\Omega = B_1$ and set $Q^*_{\rm b}:= Q^{(0)}$ where $Q^{(0)}$ is given by \eqref{stableblowups} (with $\alpha=0$). There exists a sequence of $\bbS^1$-equivariant boundary conditions $\{Q_{\rm b}^j\} \subset C^{1,1}(\partial B_1; \bbS^4)$ such that,  for any minimizer $Q^j$    of $\mathcal{E}_\lambda$ over $\mathcal{A}^{\rm sym}_{Q_{\rm b}^j}(B_1)$,  the following properties hold:  
	\begin{enumerate}
		\item for each $j$, the maximal eigenvalue $\lambda^j_{\rm max}(x)$ of $Q_{\rm b}^j$ is simple for every $x \in \partial B_1$, and the corresponding eigenspace map $V^j_{\rm max}: \partial B_1 \to \R P^2$ is equivariantly homotopic to the radial anchoring  $\overline{H}$ in \eqref{Hbar}; 
		\vskip5pt
		\item 
		$Q^j_{\rm b} \rightharpoonup Q_{\rm b}^*$ weakly in $W^{1,2}(\partial B_1;\mathbb{S}^4)$ as $j\to\infty$; 
		\vskip5pt
		\item up to a subsequence, $Q^j \to Q^*$ strongly  in $W^{1,2}(B_1;\bbS^4)$ as $j \to \infty$, where $Q^*$ is a singular energy minimizer of $\mathcal{E}_\lambda$ over $\mathcal{A}^{\rm sym}_{Q_{\rm b}^*}(B_1)$.
	\end{enumerate}	 
	As a consequence, for  $j$ large enough, $Q^j$ is a (singular) split minimizer of $\mathcal{E}_\lambda$ in the sense of Definition \ref{def:split-minimizer}. \end{theorem}

The proof is based on an argument similar to the one used for torus-like minimizers in Theorem~\ref{thm:examples-tori}. It still relies on the compactness and regularity properties of minimizers combined with a suitable choice of linearly full harmonic spheres as boundary data, i.e., $Q_{\rm b}= \boldsymbol\tau \circ \widetilde{\boldsymbol\omega}$ for well-chosen values of the parameters $(\mu_1, \mu_2)\in \C^*\times \C^*$ in \eqref{twistorlifts}. An important point is that, in connection with Remark \ref{oddsingularities}, such boundary data could in principle allow for smooth minimizers but they actually {\it do not}. In  particular, this proves that singularities may show up for energetic reasons, in analogy with the gap phenomenon for $\mathbb{S}^2$-valued harmonic maps (see \cite{AlLi,HL}, and  also \cite{HLP,HKL} for similar results in a related axially symmetric context). Finally, as an interesting consequence of the existence of singular minimizers  
in the class $\mathcal{A}^{\rm sym}_{Q_{\rm b}}(\Omega)$ and the regularity of global minimizers 
in \cite{DMP1}, we conclude symmetry breaking results for global minimizers under equivariant boundary conditions with or without norm constraint (see Corollaries \ref{cor:symmetry-breaking-I} and \ref{cor:symmetry-breaking-II} for precise statements). As already announced in \cite{DMP1}, such phenomena are already known from \cite{AlLi} for minimizers of the Frank-Oseen energy. Our results are the natural counterpart for  minimizers of the Landau-de Gennes energy, in agreement with the numerical simulations in \cite{DR} in the case of cylindrical domains and radial anchoring.

\vskip5pt
The final part of the paper is dedicated to topological properties of equivariant minimizers under more general assumptions on the domains and  the boundary data, in analogy with the results obtained in \cite[Theorem 1.6]{DMP1} where no symmetry constraint is considered. As it will be apparent in Sec.~\ref{smoothtopology}, the conclusions actually do not depend on energy minimality but just on the few properties below.  
Besides axial symmetry, we follow \cite{DMP1} and we assume that $\Omega$ and  a configuration $Q$ satisfy 
\vskip5pt
\begin{itemize}
    \item[($HP_0$)] $Q \in C^{1}(\overline{\Omega};\mathbb{S}^4) \cap C^\omega(\Omega;\mathbb{S}^4)$;
    \vskip5pt
	\item[($HP_1$)] $\bar{\beta}:=\min_{x \in \partial \Omega} \widetilde{\beta}\circ Q(x)>-1$; 
	\vskip5pt
	\item[($HP_2$)] $\Omega$ is connected and simply connected; 
	\vskip5pt
	\item[($HP_3$)] $\deg(v, \partial \Omega)= \sum_{i=1}^M \deg(v,S_i)$  is odd;
\end{itemize}
where the  $S_i$'s denote  the connected components of $\partial\Omega$.

 In view of ($HP_0$) and ($HP_1$), the maximal eigenvalue $\lambda_{\rm max}(x)$ of $Q(x)$ is simple and smooth on the boundary $\partial \Omega$,  so there is a well-defined and smooth eigenspace map $V_{\rm max}: \partial \Omega \to \mathbb{R}P^2$. Since the boundary $\partial \Omega$ is a finite union of topological spheres due to $(HP_2)$, the map $V_{\rm max}$ has a (nonunique) smooth lifting $v: \partial \Omega \to \mathbb{S}^2$ which is required to satisfy ($HP_3$).  Because of assumption $(HP_2)$, any axisymmetric (smooth) domain is topologically an axially symmetric ball with finitely many disjoint closed balls removed from its interior and having centers on the symmetry axis. If the trace of $Q$ at the boundary is the radial anchoring, then ($HP_3$) implies that an even number of balls are removed (possibly none). 
\vskip5pt

For smooth minimizers $Q_\lambda \in {A}^{{\rm sym}}_{Q_{\rm b}}(\Omega) $ like those constructed in Theorem \ref{thm:examples-tori}, we have the following topological result. We point out that, as in \cite{DMP1}, the unit norm constraint in assumption $(HP_0)$ could be relaxed to $Q(x)\neq 0$ in $\overline{\Omega}$ without affecting the conclusions below.

\begin{theorem}
\label{topology-sym-torus}
Let $\Omega\subset\R^3$ be a bounded and axisymmetric open set with boundary of class $C^3$, and let 
$Q_{\rm b} \in C^{1,1}(\partial \Omega;\mathbb{S}^4)$ be an $\mathbb{S}^1$-equivariant map.  
Assume that $Q:=Q_\lambda$ is a smooth minimizer of $\mathcal{E}_\lambda$ in the class $ \mathcal{A}^{{\rm sym}}_{Q_{\rm b}}(\Omega)$ and that ($HP_1$)-($HP_3$) hold. Then the biaxiality regions associated with $Q$  are nonempty $\mathbb{S}^1$-invariant closed subset of $\overline{\Omega}$, and setting $\beta:=\widetilde\beta\circ Q$, the following holds.
\begin{itemize}
\item[(1)] The set of singular values of $\beta$ in $[-1,\bar{\beta}]$ is at most countable, possibly accumulating only at~$\bar{\beta}$. Moreover, for any regular value $\displaystyle{t\in(-1,\bar{\beta})}$, the set $\{\beta = t\}$ is the disjoint union of finitely many (at least one) revolution tori  contained in $\Omega$. For any regular value $t\in[\bar{\beta},1)$, the set $\{\beta = t\}$ is the disjoint union of finitely many connected sets which are either 
revolution tori, $\mathbb{S}^1$-invariant strips touching the boundary, or circles lying on the boundary.
\vskip5pt

\item[(2)] The set $\{\beta =-1 \}$ contains an invariant  circle $\Gamma\subseteq \Omega \setminus I$ and the set $\{ \beta \geq \bar{\beta} \}$ contains $\partial \Omega \cup I$, where $I=\Omega \cap \{ x_3 \hbox{-axis }\}$ is a finite union of open segments. As a consequence,  $\Gamma$ and $\partial \Omega \cup I$ 
are nonempty, compact, non simply connected and mutually linked. Given  $-1\leq t_1 < t_2\leq \bar\beta$, the same hold for the sets $\{\beta \leq t_1 \}$ and $\{\beta \geq t_2 \}$. 
\end{itemize}
\end{theorem}
\vskip3pt

The proof of this result is somehow similar to the proof of \cite[Theorem 1.6]{DMP1}. However, thanks to the symmetry constraint, we can use here a more direct argument leading to more precise conclusions. Concerning claim (1), all the smooth biaxial surfaces contained in $\Omega$ must have genus one and must be tori of revolution by axial symmetry. In this simplified setting, we can even discuss biaxial surfaces for regular values $t \in [\bar{\beta},+1)$. Since $\overline{I}\subset \{\beta=1\}$, such values of the biaxiality are all attained on the boundary by continuity. The corresponding biaxial surfaces, which are of course $\mathbb{S}^1$-invariant, thus have connected components with boundary on $\partial \Omega$. 

Another very  interesting feature appears in connection with claim 2), as a far-reaching extension of what we already observed in Theorem \ref{thm:examples-tori}.
As discussed in Section~\ref{sec:remarks-domains}, we can consider the half slice $\mathcal{D}^+_\Omega:=\Omega\cap\{x_2=0\, , \, x_1>0\}$, and reconstruct $\Omega$ from it by axial symmetry. More precisely, inside the plane $\{x_2=0\}$ the set $\mathcal{D}^+_\Omega $ is open, connected, simply connected and with piecewise smooth boundary. Regarding the boundary and the closure of $\mathcal{D}^+_\Omega $ relative to $\{ x_2=0\}$, we have 
\begin{equation}
\label{reconstruction}	
 \Omega \setminus I=\mathbb{S}^1 \cdot \mathcal{D}^+_\Omega \,  , \qquad \partial \Omega \cup I=  \mathbb{S}^1 \cdot \partial \mathcal{D}^+_\Omega \, , \qquad \overline{\Omega} =\mathbb{S}^1 \cdot \overline{\mathcal{D}^+_\Omega} \, . 
\end{equation}
In view of ($HP_1$), the eigenspace map $V_{\rm max} \colon \partial \Omega \to \mathbb{R}P^2$ is well-defined and smooth. Extending $V_{\rm max}$  by continuity and invariance to be $\mathbf{e}_0 \in  \mathbb{R}P^2 \subset \mathbb{S}^4$ on $I$, and then restricting it to $\partial \mathcal{D}^+_\Omega$, we obtain a well-defined continuous map $\gamma \colon \partial \mathcal{D}^+_\Omega \to \mathbb{R}P^2$ as defined  in \eqref{nontrivialloop-j}.  
A simple argument based essentially on $(HP_3)$ shows that $[\gamma] \neq 0$ in $\pi_1(\mathbb{R}P^2)$, which leads to the existence of an invariant circle $\Gamma\subset \{\beta=-1\}\subset\Omega \setminus I$ (see Proposition \ref{prop:semidisk} for more details). Then the linking properties claimed in (2) are straightforward consequences of the linking properties between the ``disclination line'' $\Gamma$ and $\partial \mathcal{D}^+_\Omega$ inside $\overline{\Omega}$. Even in the more general context of Theorem \ref{topology-sym-torus}, we refer to the solutions to \eqref{MasterEq} coming from smooth minimizers $Q_\lambda$ of $\mathcal{E}_\lambda$ over $\mathcal{A}^{{\rm sym}}_{Q_{\rm b}}(\Omega)$ as {\em torus solutions} of the Landau-de Gennes model (with Lyuksyutov constraint).
\vskip5pt

Finally we discuss the topology of biaxial regions corresponding to singular configurations which are assumed to satisfy conditions $(HP_1)$-$(HP_3)$ and $(HP_0)$, the latter except on a finite set $\rmsing \, Q \subset \Omega \cap \{ x_3 \mbox{-axis}\}$. For simplicity we consider only energy minimizers and the model examples are those constructed in Theorem~\ref{thm:examples-split}, for which $(HP_1)$-$(HP_3)$ hold because of Theorem~\ref{thm:partial-regularity} and properties iv) and v) in the proof of Theorem~\ref{thm:examples-split}. Due to $(HP_1)$ and axial symmetry, $Q_{\rm b}(x)=\mathbf{e}_0$ for any $x \in \partial \Omega \cap \{ x_3 \mbox{-axis}\}$ and $Q_{\lambda}(x)=\pm \mathbf{e}_0$ for $x \in \Omega \cap \{x_3\mbox{-axis}\} \setminus \rmsing \, Q_\lambda$, therefore singularities come in finitely many pairs which are the endpoints of the vertical segments in $\Omega \cap \{x_3 \mbox{-axis}\}$ where $Q_\lambda(x)=-\mathbf{e}_0$. In addition each singularity carries a sign in the obvious way. We will refer to the solutions to \eqref{MasterEq} coming from these axially symmetric singular minimizers as the {\em split solutions} of the Landau-de Gennes model (with Lyuksyutov constraint).
 
 \begin{theorem}
 \label{topology-sym-split}
 
 Let $\Omega\subset\R^3$ be a bounded and axisymmetric open set with boundary of class $C^3$, and let 
$Q_{\rm b} \in C^{1,1}(\partial \Omega;\mathbb{S}^4)$ be an $\mathbb{S}^1$-equivariant map. Assume 
 that $Q:=Q_\lambda$ is a singular minimizer of $\mathcal{E}_\lambda$ in the class $ \mathcal{A}^{{\rm sym}}_{Q_{\rm b}}(\Omega)$ and that ($HP_1$)-($HP_3$) hold. Then the biaxiality regions associated with~$Q$, as subsets of $\overline{\Omega}\setminus \rmsing Q$, are nonempty and $\mathbb{S}^1$-invariant. Setting $|\rmsing Q|=:2N>0$ and $\beta:=\widetilde\beta\circ Q$, the following holds. 
\begin{itemize}
	\item[(1)] The set of singular values of ${\beta}$ in $[-1, \bar{\beta}]$ is at most countable, and it can accumulate only at $\bar{\beta}$ or $-1$. For any regular value $t\in (-1, \bar{\beta})$, the closure of $\{\beta = t\}\subset \Omega$ in $\overline{\Omega}$ contains $N$ mutually disjoint topological  spheres $S^t_j$ (smooth if $t = 0$, and smooth with corners on the $x_3$-axis otherwise), each of them  
	obtained by adding to a component of $\{\beta = t\}$ the corresponding pair of singular points. The set $\{\beta=t\}\setminus \cup_{j=1}^NS^t_j$ is either  a finite union of disjoint revolution tori or empty.
	\vskip5pt
	\item[(2)] For any regular value $t\in(\bar\beta,1)$, besides possible disjoint tori and at most $N$ axisymmetric spheres as above,  the set $\{\beta = t \}\subset \overline{\Omega}\setminus \rmsing Q$ may contain finitely many strips touching the boundary, finitely many topological discs touching the boundary with puncture at the singularities, and finitely many circles lying on the boundary. 
	\vskip5pt
	\item[(3)] For any pair of regular values $t_1$, $t_2$ in $(-1, \bar{\beta})$ with $t_1 < t_2 $, the closure of $\{\beta \geq t_2\}\subset \Omega$ in $\overline{\Omega}$ is not contractible in the complement of $\{\beta \leq t_1\}$ in $\overline{\Omega}$.
	\end{itemize}
\end{theorem}

From this theorem, the topological structure of the biaxial regions in the smooth and the singular case appears to be quite different.
For singular minimizers, the emergence of topological spheres inside $\{ \beta=t\}$ (at least for any regular value $t$ below $\bar{\beta}$) can be understood by looking at their intersection with the vertical slice $\mathcal{D}^+_\Omega$, first very close to the symmetry axis and then far away from it. Far away from the singularities of $Q_\lambda$, such surfaces cannot touch the axis, where indeed $Q_{\lambda}(x)\in\{\pm \mathbf{e}_0\}$, i.e., it is uniaxial. 
In view of the asymptotic expansion in Theorem \ref{thm:partial-regularity}, the biaxial surfaces of $Q_\lambda$  touch the $x_3$-axis precisely at each singular point with a cone-like behavior. Indeed, these surfaces are exactly cones  for the tangent maps $Q^{(\alpha)}$ with opening angle depending on $t$ (see Proposition \ref{prop:conic-structure}). Moreover, they can be extended away from $\rmsing Q_\lambda$, but they are trapped inside the domain, at least for $-1<t<\bar{\beta}$. On the other hand, two leaves of  $\{ \beta=t\}$ corresponding to different singularities cannot intersect transversally if $t$ is a regular value. Hence each leaf has to end into another singularity, giving a topological sphere. Of course such spheres are compatible with the presence of extra tori (but also with  other subsets, as in claim (2)). The first appearance of split solutions in numerical studies seems to be in \cite{GaMk}. They were lately found in other numerical papers, such as \cite{HQZ} (in particular, Fig.~8 in \cite{HQZ} contains a schematic picture of split solutions that can be helpful to visualize the first conclusion of Theorem~\ref{topology-sym-split}). 
\vskip5pt

In the last article of our series \cite{DMP2}, we further analyse existence and even coexistence of torus and split $\mathbb{S}^1$-equivariant  minimizers under radial anchoring at the boundary. The results obtained in \cite{DMP2} are perturbative in nature and depend in a subtle way on the geometry of the domain $\Omega$. Unfortunately, they do not cover the case of a nematic droplet with radial anchoring. In the recent paper \cite{Yu}, results somehow related to ours, both here and in \cite{DMP2}, are presented in  the case of a nematic droplet with homeotropic boundary data. In \cite{Yu}, the coexistence property is shown by a clever minimization argument for the energy functional in a class of $\On(2)\times \mathbb{Z}_2$-equivariant constant-norm configurations (the extra $\mathbb{Z}_2$-action being induced by reflection across  the plane $\{x_3=0\}$). Since the class considered in \cite{Yu}  is strictly smaller than the class $\mathcal{A}^{{\rm sym}}_{Q_{\rm b}}{(B_1)}$ (see the discussion in  Section~\ref{sec:comparison}), 
  it is not known (at present) whether the minimizers of $\mathcal{E}_\lambda$ in the class $\mathcal{A}^{{\rm sym}}_{Q_{\rm b}}(B_1)$ with homeotropic boundary values are smooth, singular, or if smooth and singular minimizers may coexist. Numerical simulations from \cite{GaMk,HQZ} in some range of parameters for the LdG theory  without norm constraint suggest that the torus solution should be energetically more convenient. However, at present no rigorous result in this direction  is available.
\vskip10pt

\noindent \textbf{Acknowledgements.} The research in this paper grew out of material in the Ph.D. thesis of F.D.\,. F.D. would like to express his deepest gratitude to his supervisor A.P. and to V.M. for continuous support during these years and for having involved him into this project.

\section{Auxiliary results for axially symmetric configurations}\label{sec:s1-equivariance}

\subsection{Decomposition of $\mathcal{S}_0$ into invariant subspaces}
In this subsection, we provide a decomposition of $\mathcal{S}_0$ into a direct sum of  (linear) subspaces which are invariant under the action of $\bbS^1$. This decomposition will allow identifications with the complex plane  (see  Lemma \ref{prop:Li-action}) and in turn the use of methods from complex geometry in the classification of the harmonic spheres contained in the next section.


\begin{lemma}[Decomposition of $\mathcal{S}_0$ into invariant subspaces]\label{thm:decomposition-S0}
There is a distinguished orthonormal basis $\big\{\eo, \euu, \eud, \edu, \edd\big\}$ of $\mathcal S_0$ given by 
\begin{multline}\label{eq:basis-S0}
\eo := \frac{1}{\sqrt 6}\begin{pmatrix} -1 & 0 & 0 \\ 0 & -1 & 0 \\ 0 & 0 & 2 \end{pmatrix}\,,\;\euu := \frac{1}{\sqrt 2}\begin{pmatrix} 0 & 0 & 1 \\ 0 & 0 & 0 \\ 1 & 0 & 0 \end{pmatrix}\,, \;\eud := \frac{1}{\sqrt 2}\begin{pmatrix} 0 & 0 & 0 \\ 0 & 0 & 1 \\ 0 & 1 & 0 \end{pmatrix}\,,\\[5pt]
 \edu:= \frac{1}{\sqrt 2}\begin{pmatrix} 1 & 0 & 0 \\ 0 & -1 & 0 \\ 0 & 0 & 0 \end{pmatrix}\,, \; \edd: = \frac{1}{\sqrt 2}\begin{pmatrix} 0 & 1 & 0 \\ 1 & 0 & 0 \\ 0 & 0 & 0 \end{pmatrix}\,,\qquad\qquad\qquad\qquad
\end{multline}
such that the subspaces
$$L_0 := \bb{R}\eo\,,\quad L_1:= \bb{R} \euu \oplus \bb{R} \eud\,,\quad L_2 := \bb{R} \edu \oplus \bb{R} \edd\,,$$
are invariant under the action of $\bb{S}^1$ defined in \eqref{inducedaction}, and 
\begin{equation}
		\label{eq:decomposition-S0}
		\mathcal{S}_0 =L_0 \oplus L_1 \oplus L_2\,.
\end{equation}
\end{lemma}

\begin{proof}
For elements in $\mathcal S_0$, let us use the  notation 
\[
  A =: \begin{pmatrix} \widetilde A - \frac{a_0}{2}I & \mb{a} \\ 
  \mb{a}^\mathsf{t} & a_0 \end{pmatrix},
\] 
where $a_0\in\R$, $\mb{a}\in \mathscr{M}_{2\times 1}(\R)\simeq\bb{R}^2$, and $\widetilde A\in \mathscr{M}_{2\times 2}(\R)$ has zero trace. In this way, for a rotation around the $x_3$-axis $R \in \bbS^1$, and $\widetilde{R}\in \SO(2)$  the corresponding rotation in the $(x_1,x_2)$-plane, we have 
\begin{equation}
\label{eq:S1-action}
  R A  R^\trans = \begin{pmatrix} \widetilde R \widetilde A \widetilde R^\trans  - \frac{a_0}{2}I& \widetilde R \mb{a} \\ (\widetilde R \mb{a})^\mathsf{t} & a_0\end{pmatrix}.
\end{equation} 
The key observation is that each block is invariant under the $\bb{S}^1$-action. Therefore, to determine the desired basis, it is enough to determine an orthonormal basis of symmetric traceless matrices for each block. Clearly, $\{\eo\}$, $\{\euu,\eud\}$, and $\{\edu,\edd\}$ provide such basis. 
\end{proof}

In the next result we explain in which sense the $\bb{S}^1$-action is diagonalized by our decomposition of the space $\mathcal{S}_0$. 
\begin{lemma}\label{prop:Li-action}
We have the following isometric isomorphisms:
\begin{enumerate}
		\item[\rm (0)] $L_0 \simeq \bb{R}$ \emph{via} $\boldsymbol{\iota}_0: A\in L_0 \mapsto (A:\eo)\in \bb{R}$;
		\vskip3pt
		\item[\rm (1)] $L_1 \simeq \bb{C}$ \emph{via} $\boldsymbol{\iota}_1:A\in L_1\mapsto (A:\euu)+i (A:\eud)\in\C$;
		\vskip3pt
		\item[\rm (2)] $L_2 \simeq \bb{C}$ \emph{via} $\boldsymbol{\iota}_2: A\in L_2\mapsto (A:\edu)+i(A:\edd)\in \C$.
\end{enumerate}
	Moreover, the $\bb{S}^1$-action on $\mathcal{S}_0$ corresponds on each $L_k$ to an $\bb{S}^1$-action by rotations of degree $k$. In other words, we have for $k=0,1,2$, 
	\begin{equation}\label{degreeaction}
	\boldsymbol{\iota}_k(R_\alpha A R_\alpha^\trans)=e^{ik\alpha}\boldsymbol{\iota}_k(A) \quad\forall A\in L_k\,,\;\forall R_\alpha\in\mathbb{S}^1\,.
	\end{equation}
\end{lemma}

\begin{proof}
The statements (0), (1), and (2) are elementary. Identity \eqref{degreeaction} for $k=0$ is obvious, while it follows from \eqref{eq:S1-action} for $k=1$ in view of the simple identities
\[ R_\alpha \euu R_\alpha^\trans =\cos \alpha \,\, \euu + \sin \alpha \,\, \eud \, , \qquad R_\alpha \eud R_\alpha^\trans =-\sin \alpha \,\, \euu + \cos \alpha \,\, \eud \, .  \] 

It then remains to check \eqref{degreeaction} for $k=2$. For this purpose, let us consider the $2\times2$ diagonal matrix $J:={\rm diag}(1,-1)$.
Using the notation in \eqref{eq:S1-action}, we observe that for $A\in L_2$ and $R_\alpha\in\mathbb{S}^1$, 
$$
\widetilde R_\alpha \widetilde A\widetilde{R}_\alpha^\trans=\widetilde R_\alpha  J^2 \widetilde A \widetilde R_{-\alpha}= \widetilde R_\alpha J \widetilde R_{-\alpha} J \widetilde A = \widetilde R_\alpha  \widetilde R_\alpha  \widetilde A=  \widetilde R_{2\alpha} \widetilde A\,, 
$$
so that 
\[ R_\alpha \edu R_\alpha^\trans =\cos 2\alpha \,\, \edu + \sin 2\alpha \,\, \edd \, , \qquad R_\alpha \edd R_\alpha^\trans =-\sin 2\alpha \,\, \edu + \cos 2\alpha \,\, \edd \, .  \] 
 Thus, writing $A= c\,\edu+ d \,\edd$ we have $z:=c+id=\boldsymbol{\iota}_2(A)$ and $\boldsymbol{\iota}_2(R_\alpha A R_\alpha^\trans)=e^{2i\alpha}z$, hence the conclusion follows.
\end{proof}

\begin{remark}\label{remisom&S1action}
By Lemma \ref{prop:Li-action}, $\mathcal{S}_0=L_0 \oplus L_1 \oplus L_2$ is isometrically isomorphic to $\R \oplus \C \oplus \C$ through the mapping $\boldsymbol{\iota}_*:\mathcal{S}_0\to \R\oplus\C\oplus\C$ defined by 
\begin{equation}\label{isomS_0RCC}
\boldsymbol{\iota}_*(A):=\big(\boldsymbol{\iota}_0(A),\boldsymbol{\iota}_1(A),\boldsymbol{\iota}_2(A)\big)\,. 
\end{equation}
In addition, when considering on $\R \oplus \C \oplus \C$ the $\mathbb{S}^1$-action
\begin{equation}\label{S1actionRCC}
R_\alpha\cdot(t,\zeta_1,\zeta_2):=(t,e^{i\alpha}\zeta_1,e^{2i\alpha}\zeta_2)\qquad\forall R_\alpha\in\mathbb{S}^1 \,,
\end{equation} 
the map $\boldsymbol{\iota}_*$ is $\mathbb{S}^1$-equivariant.
 In the next two sections we will rely on this identification of $\mathcal{S}_0$ with $\R \oplus \C \oplus \C$ and we will always consider on the latter the diagonal action given by \eqref{S1actionRCC}.
 Clearly, since this action is isometric, $\mathbb{S}^1$ also acts on the unit sphere of $\R \oplus \C \oplus \C$. 
\end{remark}

\begin{remark}\label{rmk:invariant-subspaces}
	In view of the previous remark it is obvious that the only 1d vector subspace fixed by the action is $L_0$.
	In addition, the only 3-dimensional invariant linear subspaces $V<\mathcal{S}_0$ are $V=L_0 \oplus L_1$ and $V=L_0 \oplus L_2$. To see this, first note that any invariant odd-dimensional subspace must contain a vector $v$ s.t. $R v = v$ for all $R \in \bbS^1$. This implies $L_0$ is a linear subspace of any such invariant subspace of $\mathcal S_0$. Thus, if $V < \mathcal{S}_0$ is invariant, then $V = L_0 \oplus (L_0^\perp \cap V)$. Let $W = L_0^\perp \cap V$. Therefore, $W < L_1 \oplus L_2$ and it is invariant. Moreover, if $\dim V = 3$, then $\dim W = 2$. Notice that vectors in $W$ cannot have components both along $L_1$ and along $L_2$, otherwise $W$ would be equivariantly isomorphic both to $L_1$ and to $L_2$ under the projection maps, which is impossible because $\bbS^1$ acts with different degrees on $L_1$ and $L_2$. Hence, $W$ must be either $L_1$ or $L_2$.
\end{remark}


\subsection{Structure of axisymmetric domains}\label{sec:remarks-domains}
The purpose of this subsection is to collect some geometric properties of axisymmetric domains of $\R^3$ that we shall use in the next sections. We start recalling the following auxiliary result characterizing the  simple connectivity of a bounded domain $\Omega \subset \mathbb{R}^3$ with smooth boundary. We will always suppose $\Omega$ to be $C^1$-smooth in this section, although to be able to prove boundary regularity as in \cite{DMP1}, we will require $\partial\Omega$ of class $C^3$ when necessary in the paper.

 \begin{lemma}\cite[Theorem 3.2 and Corollary 3.5]{BeFr}
 Let $\Omega \subset \mathbb{R}^3$ be a bounded and connected open set with $C^1$-boundary. Then $\Omega$ is simply connected if and only if $\partial \Omega=\cup_{i=0}^M S_i$ and each surface $S_i$ is diffeomorphic to the standard sphere $\mathbb{S}^2 \subset \mathbb{R}^3$.
 \label{simpledomains}
 \end{lemma} 
 
 Let us now recall that we identify $\mathbb{S}^1$ with the subgroup of $\SO(3)$ made of all rotations around the vertical $x_3$-axis (see \eqref{rotmatr}), and that we define axisymmetry accordingly.
 
 \begin{definition}
 \label{domain-section}
 A set $\Omega\subset \R^3$ is said to be \emph{axisymmetric} (or \emph{$\bbS^1$-invariant}, or \emph{rotationally symmetric}) if it is invariant under the action of $\mathbb{S}^1$, i.e., $R\cdot\Omega=\Omega$ for every $R\in \mathbb{S}^1$. Equivalently, $\Omega$ is axisymmetric if 
 $$\Omega=\bigcup_{R\in\mathbb{S}^1}R\cdot\mathcal{D}_\Omega \quad\text{where}\quad\mathcal{D}_\Omega:=\Omega\cap\{x_2=0\}\,.$$
 \end{definition}
 
In case of an axisymmetric domain, the geometric description in Lemma \ref{simpledomains} can be made more precise. Starting from this lemma and the structure of bounded multiply connected smooth domain in the plane we have the following Corollary \ref{domains-structure}, the proof of which is elementary and left to the reader. 
 
\begin{corollary}
\label{domains-structure}
Let $\Omega \subset \mathbb{R}^3$ be a bounded and connected open set with boundary of class $C^1$. If $\Omega$ is axisymmetric and simply connected, then the following holds. 
\begin{itemize}
\item[(i)] $\mathcal{D}_\Omega$ is a bounded and connected (relatively) open subset with $C^1$-boundary of the vertical plane $\{x_2=0\}$, and $R_\pi\mathcal{D}_\Omega=\mathcal{D}_\Omega$ (i.e., $\mathcal{D}_\Omega$ is symmetric with respect to the $x_3$-axis). 
\vskip5pt
\item[(ii)] There is a connected and simply connected (relatively) open subset $\mathcal{T}$ of the vertical plane $\{x_2=0\}$ satisfying $R_\pi\mathcal{T}=\mathcal{T}$ such that either $\mathcal{D}_\Omega=\mathcal{T}$, or $\mathcal{D}_\Omega=\mathcal{T}\setminus \bigcup_{i=1}^M H_i$  where the $H_i$'s are connected and simply connected (relatively) closed subset of $\{x_2=0\}$. In addition, the ``holes'' $H_i\subset \mathcal{T}$ are mutually disjoint,  $R_\pi H_i=H_i$ and $\partial \mathcal{D}_\Omega=\partial \mathcal{T}\cup \left(\cup_{i=1}^M \partial H_i \right)$. 

\vskip5pt
\item[(iii)] The (relatively) open subsets $\mathcal{D}^+_\Omega:=\mathcal{D}_\Omega\cap\{x_1>0\}$ and $\mathcal{D}^-_\Omega:=\mathcal{D}_\Omega\cap\{x_1<0\}$ of the vertical plane $\{x_2=0\}$ are connected and simply connected, and $R_\pi\mathcal{D}^{\pm}_\Omega=\mathcal{D}^{\mp}_\Omega$. Moreover, if $I=\Omega \cap \{ x_3\hbox{-axis}\}$ then identities \eqref{reconstruction} hold.
\end{itemize}
\end{corollary} 
 
 When discussing topological properties of minimizers we will suppose $\Omega$ satisfies the hypotheses of Lemma~\ref{simpledomains}. Thus, unless otherwise stated, from now on $\Omega$ will be a bounded open connected simply connected axisymmetric set in $\R^3$ with $C^1$-smooth boundary that can be reconstructed from its vertical slices $\mathcal{D}_\Omega$ and $\mathcal{D}^+_\Omega$ according to Corollary \ref{domains-structure}.
   
  Observe that any such $\Omega$ is obtained as follows: fix arbitrarily a vertical plane $\Pi$ through the $x_3$-axis and let $\mathcal{T}$ be a bounded region in $\Pi$, symmetric with respect to the $x_3$-axis, whose boundary $\partial\mathcal{T}$ is a simple closed curve, so that $\mathcal{T}$ is simply connected. Create in $\mathcal{T}$ a finite number $M \geq 0$ of symmetric disjoint hollows $H_i\subset \mathcal{T}$ (we set $H_0 = \emptyset$ for convenience), each one diffeomorphic to the closed unit disc, whose boundaries $\partial H_i$ are simple closed curves. Let
\begin{equation}
	\mathcal{D} = \mathcal{T} \setminus \cup_{i=0}^M H_i.
\end{equation} 
Note that, unless $M = 0$, $\mathcal{D}$ is \emph{not} simply connected. Rotating $\mathcal{D}$ of an angle $\pi$ around the $x_3$-axis, we obtain a domain $\Omega$ with the desired properties. Conversely, given $\Omega$ as in the above, $\Omega \cap \Pi$ is a planar domain $\mathcal{D}$ as in the above. Notice that $I:=\Omega \cap \{x_3 \hbox{-axis}\}=\mathcal{D} \cap  \{x_3 \hbox{-axis}\} $, hence we can write
\begin{equation}\label{eq:decomp-D}
	\mathcal{D} = \mathcal{D}^+ \cup I \cup \mathcal{D}^-,
\end{equation}
where $\mathcal{D}^+ = \mathcal{D} \cap \{ x_1 > 0\}$ and $\mathcal{D}^- = \mathcal{D}\cap \{ x_1 < 0 \}$ (of course, $\mathcal{D}^+$ and $\mathcal{D}^-$ are congruent by symmetry). In contrast to $\mathcal{D}$, $\mathcal{D}^\pm$ are simply connected. Clearly, one can also re-obtain $\Omega$ by rotating, say, $\mathcal{D}^+ \cup{I}$ around the $x_3$-axis of an angle $2\pi$. Notice that $\partial\mathcal{D}^+$ is given by a unique simple piecewise smooth closed curve that can be thought of as a parametrized curve embedding $\bbS^1$ into $\overline{\Omega}$. For instance, if $\Omega = B_1$, then $\mathcal{D}$ is a disc, $\mathcal{D}^+$ a semidisc and $\partial\mathcal{D}^+$ the boundary of such semidisc in the vertical plane $\Pi$ with flat part on the $x_3$-axis.  

From the above it follows that the $x_3$-axis intersects $\partial \Omega$ exactly $(2M + 2)$-times and that $I=\Omega \cap \{x_3\mbox{-axis}\}$ is the union of $M +1$ segments $\ell_k$:
\begin{equation}\label{eq:def-ellk}
	I=\Omega \cap \{x_3\mbox{-axis}\} = \cup_{k=1}^{M+1} \ell_k.
\end{equation}
Of course, $\cup_{k=1}^{M+1} \ell^k$ is also the flat part of $\partial\mathcal{D}^+$ on the $x_3$-axis in the plane $\Pi$. We also denote
\begin{equation}\label{eq:def-B}
	\mathcal{B} := \partial I=\partial \Omega \cap \{ x_3\mbox{-axis} \} = \{ b_1, b_2, \dots, b_{2M+2}\},
\end{equation}
where the boundary is taken in the $x_3$-axis. We label such points increasingly with their $x_3$-coordinate. Thus, $b_1$ is the lowest point of $\overline{\Omega}$ along the $x_3$-axis and $b_{2M+2}$ the highest. Finally, notice that the validity of \eqref{eq:def-ellk} and \eqref{eq:def-B} actually relies only on the smoothness $\partial\Omega$, hence they hold true for every $\bbS^1$-invariant bounded open set $\Omega$ with Lipschitz boundary, independently of its connectedness properties.

\begin{remark}
We observe that if $\Omega$ is $\mathbb{S}^1$-invariant then the same holds for the function $\tilde{d}$ giving the signed distance from its boundary, hence its gradient is $\mathbb{S}^1$-equivariant and in particular the outer normal field $\overset{\to}n(x)$ along $\partial \Omega$ is equivariant. As a consequence we see that the corresponding radial anchoring $Q_{\rm b}(x)$ given by \eqref{Qblift} is equivariant. In addition, as $\overset{\to}n(b_j)=(0,0, (-1)^j)$ for $j=1, \ldots , 2M+2$, for such data we obtain $Q_{\rm b}(b_j)=\mathbf{e}_0$ for each $x \in \mathcal{B}$.  
\end{remark}
\begin{remark}
\label{rmk:invariantconf}
More generally, for any boundary map $Q_{\rm b} \in \rmLip_{\rm sym}(\partial \Omega; \mathbb{S}^4)$ 	the equivariance property \eqref{S1equivariance} together with Remark \ref{rmk:invariant-subspaces} yield $Q_{\rm b}(b_j)=\pm \mathbf{e}_0$ for any $x \in \mathcal{B}=\partial I$. Analogous property holds for an admissible configuration  $Q \in \mathcal{A}^{\rm sym}_{Q_{\rm b}}(\Omega)$ at each point $x \in I$ whenever the restriction to the $x_3$-axis makes sense. 
\end{remark}

\subsection{Stereographic projections, projective spaces, and the twistor fibration}\label{subsec:stereo}
In the following sections, we shall consider $\bbS^1$-equivariant maps from the standard $2$-sphere $\mathbb{S}^2\subset\R^3$ into~$\bbS^4$, the unit sphere of $\mathcal{S}_0$. To describe those maps, it will be useful to make use of stereographic projections. Although these notions are elementary, to fix the notations and for the reader's convenience we collect them in this subsection.
 
\subsubsection*{The stereographic projection of $\mathbb{S}^2$}  
Let $S^{(2)} := (0,0,-1)$ be the south pole of $\bbS^2 \subset \R^3$. We write $x = (x_1,x_2,x_3)$ for a point in $\bbS^2$, and $y = (y_1, y_2)$ a point in $\R^2$. The stereographic projection of $\mathbb{S}^2$ from the south pole is the map $\boldsymbol{\sigma}_2:\mathbb{S}^2\setminus\{S^{(2)}\}\to \R^2$ given by 
$$\boldsymbol{\sigma}_2(x) := \left(\frac{x_1}{1+x_3},\frac{x_2}{1+x_3}\right)\, ,$$
which is a diffeomorphism whose inverse map is given by 
$$\boldsymbol{\sigma}_2^{-1}(y) = \left( \frac{2y_1}{1+\abs{y}^2}, \frac{2y_2}{1+\abs{y}^2}, \frac{1-\abs{y}^2}{1+\abs{y}^2} \right)\,.$$
Identifying $\R^2$ with $\C$, $\boldsymbol{\sigma}_2$ can be seen as a map from $\mathbb{S}^2\setminus\{S^{(2)}\}$ into $\C$, and 
$$ \boldsymbol{\sigma}_2(x) =\frac{x_1+ i x_2}{1+x_3}\,.$$
It then follows that 
\begin{equation}\label{equivsigma2intoC}
\boldsymbol{\sigma}_2(R_\alpha x)=e^{i\alpha}  \boldsymbol{\sigma}_2(x)\qquad\forall R_\alpha\in\mathbb{S}^1\,.
\end{equation}
In terms of {\sl spherical coordinates} on $\mathbb{S}^2$, i.e., 
\begin{equation}\label{sphercoord}
x=\big(\cos\phi \sin\theta, \sin\phi \sin\theta,\cos\theta\big)
\end{equation}
with $\theta\in[0,\pi]$ (the {\sl latitude}) and $\phi\in[0,2\pi)$ (the {\sl colatitude}), we have 
\begin{equation}\label{sphercoordsterproj}
 \boldsymbol{\sigma}_2(x) =\tan\big(\theta/2\big) e^{i\phi}\quad\text{and}\quad | \boldsymbol{\sigma}_2(x)|=\tan(\theta/2)\,.
 \end{equation}
Writing $z=y_1+i y_2$, the complex version of the formula for the inverse map $\boldsymbol{\sigma}^{-1}_2$ reads
\begin{equation}\label{forminvsterproj}
\boldsymbol{\sigma}_2^{-1}(z) =\left( \frac{2z}{1+\abs{z}^2}, \frac{1-\abs{z}^2}{1+\abs{z}^2}\right)\,. 
\end{equation}

Let us now recall that the projective complex line $\C P^1$ is the smooth manifold made of all complex lines through the origin in $\C^2$, i.e., $\C P^1=(\C\times\C \setminus\{(0,0)\} )/\C^*$. We denote by $[z_0,z_1]$, $(z_0,z_1 )\neq(0,0)$, the {\sl homogeneous coordinates} for a point in $\C P^1$. In other words, if $z_0\in \C^*$, then $[z_0,z_1]$ is the  complex line $\big\{(\zeta_0,\zeta_1)\in\C^2:\zeta_1-(z_0^{-1}z_1)\zeta_0=0\big\}$, while $[0,1]$ is the line (at infinity) $\{(0,\zeta_1)\in \C^2, \, \zeta_1\in \C \}$. We refer to as an {\sl inhomogeneous coordinate} on $\C P^1$ the complex number $z=z_0^{-1}z_1$, and the mapping $[z_0,z_1]\mapsto z$ allows to identify $\C P^1$ with $\C\cup\{\infty\}$, agreeing that $[0,1]\in \C P^1$ is mapped to $\infty$ (the point at infinity). With the convention that $\boldsymbol{\sigma}_2$ maps the south pole $S^{(2)}$ to $\infty$, the stereographic projection $\boldsymbol{\sigma}_2$ can be then seen as a bijective map from $\mathbb{S}^2$ into $\C\cup \{\infty\} $ and in turn to $\C P^1$.  This map turns out to be a diffeomorphism. In terms of the inverse map $\boldsymbol{\sigma}_2^{-1}:\C P^1\to\mathbb{S}^2$, we have 
$$\boldsymbol{\sigma}_2^{-1}\big([z_0,z_1]\big) =\boldsymbol{\sigma}_2^{-1}(z_0^{-1}z_1)\,.$$ 
In view of \eqref{equivsigma2intoC},  considering the following $\mathbb{S}^1$-action on $\C P^1$: 
\begin{equation}
\label{S1actionCP1}
	e^{i\alpha}\cdot [z_0,z_1]:=[z_0,e^{i\alpha}z_1] \,,
\end{equation} 
the map $\boldsymbol{\sigma}_2:\mathbb{S}^2\to\C P^1$ is equivariant with respect to the $\mathbb{S}^1$-action, i.e., 
\begin{equation}\label{equivsigma2intoCP1}
 \boldsymbol{\sigma}_2(R_\alpha x)=e^{i\alpha}\cdot  \boldsymbol{\sigma}_2(x)\qquad\forall R_\alpha\in\mathbb{S}^1\,.
 \end{equation}

\subsubsection*{The stereographic projection of $\mathbb{S}^4$}

By means of the isometric isomorphism $\boldsymbol{\iota}_*$ between $\mathcal{S}_0$ and $\R \oplus \C \oplus \C$ (see \eqref{isomS_0RCC}), we identify $\mathbb{S}^4$ with unit sphere of $\R \oplus \C \oplus \C$. Setting $S^{(4)}:=(-1,0,0)\in\mathbb{S}^4$, the stereographic projection $\boldsymbol{\sigma}_4 : \bbS^4 \setminus \{S^{(4)}\}\to \C^2$ from the south pole $S^{(4)}$ and its inverse are given by 
$$\boldsymbol{\sigma}_4(p) := \left(\frac{\zeta_1}{1+t},\frac{\zeta_2}{1+t} \right)\,, \qquad \boldsymbol{\sigma}_4^{-1}(\eta_1,\eta_2) = \frac{1}{1+|\eta_1|^2+|\eta_2|^2}\left(1-|\eta_1|^2-|\eta_2|^2,2\eta_1 , 2 \eta_2 \right) $$
where $p=(t,\zeta_1,\zeta_2)\in\bbS^4 \setminus \{S^{(4)}\}\subset \R\oplus \C \oplus \C$ and $(\eta_1, \eta_2) \in \C^2$. 
Adding a point at infinity to $\C^2$ and sending $\{S^{(4)}\}$ to it, $\boldsymbol{\sigma}_4$ induces a diffeomorphism between $\bbS^4$ and  $\C^2 \cup \{\infty\}$. Note that, according to \eqref{S1actionRCC}, we have for the  extended map the equivariance property (fixing $S^{(4)}$ and $\infty$), namely  
\begin{equation}\label{S1actionC2}
\boldsymbol{\sigma}_4(R_\alpha\cdot p)=\left(\frac{e^{i\alpha}\zeta_1}{1+t},\frac{e^{2i\alpha}\zeta_2}{1+t} \right)\qquad\forall R_\alpha\in\mathbb{S}^1\,.
\end{equation}
\vskip2pt

Let us now denote by $\mathbb{S}^2_{(1)}\subset L_0 \oplus L_1$ the unit sphere of $\R \oplus \C \oplus \{0\}$ and by $\mathbb{S}^2_{(2)} \subset L_0 \oplus L_2$ the unit sphere of $\R \oplus \{0\} \oplus \C$ (which are equatorial $2$-spheres of $\mathbb{S}^4$).  We notice that $\boldsymbol{\sigma}_4$ maps $\mathbb{S}^2_{(1)}\setminus  \{S^{(4)}\}$ and $\mathbb{S}^2_{(2)} \setminus \{S^{(4)}\}$ into $\C\times\{0\}$ and $\{0\}\times\C$ respectively. Moreover, its restrictions give the mappings $\boldsymbol{\sigma}^{(1)}_2: \mathbb{S}^2_{(1)}\setminus  \{S^{(4)}\}\to\C$ and $\boldsymbol{\sigma}^{(2)}_2: \mathbb{S}^2_{(2)}\setminus  \{S^{(4)}\}\to\C$ defined by
 
 \begin{equation}
 \label{partialstereo}	
\boldsymbol{\sigma}^{(1)}_2(t,\zeta_1,0):=\frac{\zeta_1}{1+t}\quad\text{and}\quad \boldsymbol{\sigma}^{(2)}_2(t,0,\zeta_2):=\frac{\zeta_2}{1+t} \, ,
\end{equation}
which are stereographic projections, and in view of \eqref{S1actionC2} they satisfy
\begin{equation}
\label{partialstereoequiv} \boldsymbol{\sigma}^{(1)}_2(R_\alpha\cdot\big(t,\zeta_1,0)\big)=e^{i\alpha}\boldsymbol{\sigma}^{(1)}_2(t,\zeta_1,0)\quad\text{and}\quad \boldsymbol{\sigma}^{(2)}_2(R_\alpha\cdot\big(t,0,\zeta_2)\big)=e^{2i\alpha}\boldsymbol{\sigma}^{(2)}_2(t,0,\zeta_2)
\end{equation}
for every $R_\alpha\in\mathbb{S}^1$.

\subsubsection*{The twistor fibration $\C P^3\to\mathbb{S}^4$}

The complex projective $3$-space $\C P^3$ is the smooth (complex) manifold made of all complex lines through the origin in $\C^4$, i.e., $\C P^3=(\C^4 \setminus \{0 \})/\C^*$. If we denote by $[w_0,w_1,w_2,w_3]$, with $(w_0,w_1,w_2,w_3)\in\C^4\setminus\{0\}$, the homogeneous coordinates of a point in $\C P^3$, then $[w_0,w_1,w_2,w_3]$ represents the complex line $\big\{(\zeta_0,\zeta_1,\zeta_2,\zeta_3)\in \C^4: w_\ell \zeta_m-w_m \zeta_\ell =0 \, , \, \, \, 1\leq i<j \leq 4 \big\}$ (i.e., the $w_\ell$'s and the $\zeta_\ell$'s differ by a common nonzero factor).

Considering $\mathbb{S}^4$ as the unit sphere of $\R\oplus\C\oplus\C$, the twistor fibration $\boldsymbol{\tau}: \bb{C}P^3 \to \bb{S}^4$ is the map given by 
\begin{equation}\label{eq:tau}
\boldsymbol{\tau}\big([w_0, w_1, w_2, w_3]\big):= \frac{\left(\abs{w_0}^2  +\abs{w_3}^2- \abs{w_1}^2 - \abs{w_2}^2 , 2(\overline{w_0} w_1 + \overline{w_2} w_3),2(\overline{w_0} w_2 - \overline{w_1} w_3)\right)}{\abs{w_0}^2 + \abs{w_1}^2 + \abs{w_2}^2 + \abs{w_3}^2}\,,
\end{equation}
see \cite{BoWo1,BoWo2,BoWo3,LemaireWood}. Considering the $\mathbb{S}^1$-action on $\C P^3$ defined by 
\begin{equation}\label{S1actionCP3}
R_\alpha\cdot[w_0,w_1,w_2,w_3]:= [w_0,e^{i\alpha}w_1,e^{2i\alpha}w_2,e^{3i\alpha}w_3]\qquad\forall R_\alpha\in\mathbb{S}^1\,,
\end{equation}
and the (induced) $\mathbb{S}^1$-action on $\mathbb{S}^4\subset\R\oplus\C\oplus\C$ given by \eqref{S1actionRCC} (see Remark \ref{remisom&S1action}), the twistor map $\boldsymbol{\tau}$ turns out to be equivariant. We state this property in the following lemma, whose proof is a straightforward consequence of formulas \eqref{S1actionRCC}, \eqref{eq:tau}, and \eqref{S1actionCP3}, hence it is left to the reader.

\begin{lemma}\label{lemma:tau-s1-eq}
The twistor fibration $\boldsymbol{\tau} : \C P^3 \to \bbS^4$ is equivariant with respect to the $\bb{S}^1$-actions on $\C P^3$ and $\mathbb{S}^4$ given in \eqref{S1actionCP3} and \eqref{S1actionRCC}. In other words, 
\begin{equation}
\label{tauequivariance} 
\boldsymbol{\tau}\big( R_\alpha \cdot [w_0, w_1, w_2, w_3] \big) = R_\alpha \cdot \boldsymbol{\tau}\big([w_0, w_1, w_2, w_3]\big)
\end{equation}
for every $R_\alpha\in\mathbb{S}^1$ and every $[w_0,w_1,w_2,w_3]\in \C P^3$. 
\end{lemma}

\begin{remark}
A simple way to get some insight in the formula \eqref{eq:tau} for the twistor fibration is to interpret it in terms of quaternions. We recall that quaternions may be thought of as a set $\bbH$ of ordered pairs of complex numbers endowed with a noncommutative multiplication.
We identify $\C^2$ with $\bbH$ \emph{via} $(\zeta_1,\zeta_2) \mapsto \zeta_1 + \zeta_2j$ (here $j$ is the second imaginary unit of quaternions anticommuting with $i$, whence noncommutativity of the multiplication), and we  also identify $\C^4$ with $\bbH^2$ by the map\footnote{Under this rather unconventional identification the twistor fibration $\boldsymbol\tau$ takes the form \eqref{eq:tau} and it is equivariant in the sense of Lemma \ref{lemma:tau-s1-eq}.}
 $(\zeta_0,\zeta_1,\zeta_2,\zeta_3) \mapsto (\zeta_0 + \zeta_3j, \zeta_1+\zeta_2 j)$. The quaternionic projective space $\bb{H} P^1$ is the quotient of $\bb{H}^2$ by the left action by $\bb{H} \setminus \{0\}$. As recalled above for $\C P^1$, we can identify $\bbH P^1$ with $\bbH \cup \{ \infty\}$ \emph{via} $[q_1,q_2] \mapsto q_1^{-1}q_2$ using the inhomogeneous quaternionic coordinate $q_1^{-1}q_2$ for $q_1 \neq 0$ (extended sending $[0,1] \mapsto \infty$), and in turn with $\C^2 \cup \{\infty\}$ writing $q_1^{-1}q_2 =\eta_1 + \eta_2 j$, $\eta_1 , \, \eta_2 \in \C$. This way (i.e., considering the composite map) we can see the stereographic projection $\boldsymbol{\sigma}_4$ as a map identifying $\bbH P^1$ with $\mathbb{S}^4$ through $\C^2 \cup \{\infty\}$, namely 
 \[
\bb{H}P^1 \ni [ q_1,q_2] \quad \overset{\boldsymbol\sigma_4^{-1}}\longrightarrow \quad \left( \frac{|q_1|^2-|q_2|^2} {|q_1|^2+|q_2|^2} , 2 q_1^{-1} q_2 \right) \in \mathbb{S}^4 \subset \R \oplus \bb{H} \, . 
\]
As a consequence one can check that the map $\boldsymbol\tau$ in \eqref{eq:tau} is exactly the composition of the Hopf map $\rho : \bb{C} P^3 \to \bb{H} P^1$, taking complex lines in $\C^4\simeq \bb{H}^2$ to their quaternionic envelope in $\bb{H}^2$ (i.e., $\rho([w_0,w_1,w_2,w_3]) = [w_0 +w_3j, w_1 + w_2j]$), with the inverse of $\boldsymbol{\sigma}_4$. 
\end{remark}

\section{Equivariant harmonic spheres into $\bbS^4$}\label{sec:axisymm}

The auxiliary results contained in this section will be used later in the present paper, when discussing the asymptotic behavior of LdG minimizers at isolated singularities, and they will be also of use in our companion paper \cite{DMP2}. Since such profiles turn out to be homogeneous extensions of $\bbS^1$-equivariant harmonic maps  from the two-sphere, their classification  we present here is of independent interest and of possible use in the analysis of minimizing harmonic maps under symmetry constraint (see, e.g., \cite{Gastel,HorMos} and references therein).

Recall that any weakly harmonic map in $\omega \in W^{1,2}(\bbS^2;\bbS^4)$ is by definition a critical point of the energy functional
\begin{equation}
\label{tanenergy}
 E(\omega)=\int_{\bb{S}^2} \frac12 |\nabla_T \omega|^2 \,d\vol{\bb{S}^2} \, , 
\end{equation} 
hence it is a weak solution to
 \begin{equation}
 \label{eqharmspheres}	
 \Delta_T \omega+ \abs{\nabla_T \omega}^2\omega=0 \, ,
 \end{equation}
and therefore $C^\infty$-smooth due to H\'{e}lein's theorem (see, e.g., \cite[Section 10.4.1]{GiaqMart}) and in turn real analytic by the analyticity results in \cite{Morrey}. Smooth harmonic maps between Euclidean spheres are usually called \emph{harmonic spheres} and we will often use such terminology here. 

%

We start with the following important result due to E.~Calabi \cite{Ca}.
\begin{lemma}
\label{lemma1calabi}
{\rm(\cite{Ca})}	Every nonconstant harmonic sphere $\omega$ is a weakly conformal branched minimal immersion for which the energy (or, equivalently by conformality, the area) and the dimension of the image satisfy
\begin{equation}
\label{Calabi} 
E(\omega)=\int_{\bb{S}^2} \frac12 |\nabla_T \omega|^2 \,d\vol{\bb{S}^2} =4 \pi \, |d| \,, \quad |d| \in \N \, , \quad {\rm dim} \, \, {\rm span}_\R\,\omega(\mathbb{S}^2) =k\in \{ 1, 3, 5\} \, .
\end{equation}
\end{lemma}

Besides constant maps (for which $d=0$ and the range has dimension $k=1$), we will distinguish between {\em linearly degenerate} and {\em linearly full} harmonic spheres (for which $k=3$ and $k=5$ respectively) and we will classify them in the next two subsections under the $\mathbb{S}^1$-equivariance assumption.
 

In the next lemma we recall two other well known facts concerning harmonic spheres which will be used in the next subsections (we refer to \cite[Chapter 6]{HeleinBook} for a proof).

\begin{lemma}
\label{lemma2calabi}
Any harmonic map $\omega : \bbS^2 \to \bbS^4$ is {\em real isotropic}. In particular, any harmonic map $\omega : \bbS^2 \to \bbS^4$ is (weakly) conformal, hence in terms of the spherical coordinates $(\theta,\phi)$ on $\mathbb{S}^2$ 
\begin{equation}
\label{conformalityI}
\abs{\partial_\theta \omega}^2\equiv \frac1{\sin^2 \theta}\abs{\partial_\phi \omega}^2 \, \qquad \hbox{and} \qquad \partial_\theta \omega \,\cdot \, \frac1{\sin \theta} \partial_\phi\omega \equiv 0 \, .	\end{equation}
 \end{lemma}

In view of Remark \ref{remisom&S1action}, we can identify $\mathbb{S}^4$ with the unit sphere of $\R\oplus\C\oplus\C$, which allows to write an $\mathbb{S}^1$-equivariant harmonic map $\omega:\mathbb{S}^2\to\mathbb{S}^4$ as a map $\boldsymbol{\omega}$ given by
$$ \boldsymbol{\omega}=(\boldsymbol{\omega}_0,\boldsymbol{\omega}_1,\boldsymbol{\omega}_2)\,,$$
where $\boldsymbol{\omega}_0:\mathbb{S}^2\to[-1,1]$, $\boldsymbol{\omega}_1:\mathbb{S}^2\to \C$, and $\boldsymbol{\omega}_2:\mathbb{S}^2\to \C$. By \eqref{S1actionRCC}, the $\mathbb{S}^1$-equivariance of $\boldsymbol{\omega}$ translates into 
\begin{equation}
\label{omegaequiv}
\boldsymbol{\omega}_0(R_\alpha x)=\boldsymbol{\omega}_0(x)\,,\quad\boldsymbol{\omega}_1(R_\alpha x)=e^{i\alpha}\boldsymbol{\omega}_1(x)\,,\quad\boldsymbol{\omega}_2(R_\alpha x)=e^{2i\alpha}\boldsymbol{\omega}_2(x)\qquad\forall R_\alpha\in\mathbb{S}^1\,.
\end{equation}
In terms of the spherical coordinates $(\theta,\phi)$ of $x\in\mathbb{S}^2$ (see \eqref{sphercoord}), the identities above imply that 
\begin{equation}\label{eq:omega}
\boldsymbol{\omega}_0(x)=\omega_0(\theta)\,,\quad  \boldsymbol{\omega}_1(x)=\omega_1(\theta)e^{i\phi}\,,\quad  \boldsymbol{\omega}_2(x)=\omega_2(\theta)e^{2 i\phi}\,,
\end{equation}
where $\omega_0:[0,\pi]\to[-1,1]$ and $\omega_k:[0,\pi]\to\C$ for $k=1,2$. By smoothness of $ \boldsymbol{\omega}$, we have 
\begin{equation}\label{eq:constraint-components}
\omega_0(0)\in\{\pm1\}\,,\quad\omega_0(\pi)\in\{\pm1\}\,,\quad\omega_1(0)=\omega_1(\pi)=0\,,\quad\omega_2(0)=\omega_2(\pi)=0 \,.
\end{equation}
In particular, $\boldsymbol{\omega}$ sends the south pole $S^{(2)}\in\mathbb{S}^2$ either to the south pole $S^{(4)}\in\mathbb{S}^4$, or to the north pole $N^{(4)}:=-S^{(4)}$.   
Finally, in view of \eqref{sphercoordsterproj}, precomposing $\boldsymbol{\omega}$ with the inverse  stereographic projection $\boldsymbol{\sigma}_2^{-1}:\C\to\mathbb{S}^2$ (given by \eqref{forminvsterproj}) leads to 
\begin{equation}
\label{omegacoefficients}
\boldsymbol{\omega}\circ\boldsymbol{\sigma}_2^{-1}(z) =\left(\widetilde \omega_0(|z|),\widetilde\omega_1(|z|)\frac{z}{|z|},\widetilde\omega_2(|z|)\frac{z^2}{|z|^2} \right)\,,
\end{equation}
where $\widetilde\omega_0:[0,+\infty)\to[-1,1]$ and $\widetilde\omega_k:[0,+\infty)\to\C$  for $k=1,2$ are given by $\widetilde\omega_k(r):=\omega_k\big(2\arctan(r)\big) $.

\begin{remark}
When restricting to equivariant harmonic spheres, the first equality in \eqref{conformalityI}, together with \eqref{eq:omega}, yields the useful identities
\begin{equation}	
\label{conformalityII}
	\abs{\partial_\theta \boldsymbol{\omega}}^2= \frac{\abs{\boldsymbol{\omega}_1}^2+4 \abs{\boldsymbol{\omega}_2}^2 }{\sin^2 \theta}= \frac12 \abs{\nabla_T \boldsymbol\omega}^2 \, . 
\end{equation}

\end{remark}

An interesting consequence of the previous identity is the following fact.

\begin{remark}\label{rmk:lemaire} ({\em Branch points}). 
If a smooth nonconstant map $\boldsymbol\omega : \bbS^2 \to \bbS^4$ is harmonic and $\bbS^1$-equivariant, then $| \boldsymbol{\omega_0}(p)|<1$ and $(\boldsymbol{\omega}_1(p),\boldsymbol{\omega}_2(p))\neq (0,0)$ whenever $p\neq \pm S^{(2)}$. Otherwise we would have $\boldsymbol{\omega}_0 \equiv \pm 1$ on a circle, hence  $\boldsymbol\omega$ would be a constant map in a disc in view of Lemaire's theorem \cite{Le} and then everywhere by unique continuation. As a consequence, by \eqref{conformalityII} the only possible \emph{branch points} (i.e., the points where $\boldsymbol\omega$ has zero differential) are the poles $S=S^{(2)}$ and $N=-S^{(2)}$.
\end{remark}

\subsection{Classification of linearly degenerate harmonic spheres} 

The key preliminary fact is given in the following simple lemma.

\begin{lemma}\label{dichotnonfulllemma}
Let $\boldsymbol{\omega}:\mathbb{S}^2\to\mathbb{S}^4$ be a nonconstant $\bbS^1$-equivariant  harmonic map. If $\boldsymbol{\omega}$ is not linearly full, then either ${\rm span}_\R\,\boldsymbol{\omega}(\mathbb{S}^2)=L_0\oplus L_1$ or  ${\rm span}_\R\,\boldsymbol{\omega}(\mathbb{S}^2)=L_0\oplus L_2$.
\end{lemma}

\begin{proof}
Set $V:={\rm span}_\R\,\boldsymbol{\omega}(\mathbb{S}^2)$. Since $\boldsymbol{\omega}$ is assumed to be neither constant nor linearly full, then $V$ is linear subspace of $\mathcal{S}_0$ of dimension $k=3$ because of Lemma~\ref{lemma1calabi}. On the other hand, the equivariance of $\boldsymbol{\omega}$ implies that $V$ is invariant under the action of $\mathbb{S}^1$. By Remark~\ref{rmk:invariant-subspaces}, it follows that either $V=L_0\oplus L_1$ or $V=L_0\oplus L_2$, as claimed. 
\end{proof}

The classification of the linearly degenerate harmonic spheres is now a simple consequence of the classification results for $\mathbb{S}^2$-valued harmonic maps from \cite{BCL} specialized to the equivariant setting.

\begin{proposition}
\label{degeneratespheres}
If $\boldsymbol{\omega}=(\boldsymbol{\omega}_0,\boldsymbol{\omega}_1,\boldsymbol{\omega}_2):\mathbb{S}^2\to \mathbb{S}^4$ is a nonconstant $\mathbb{S}^1$-equivariant and nonfull harmonic map and $d \in \Z$ is as in \eqref{Calabi}, then the following holds: 
\begin{itemize}
\item[(i)] either  $\boldsymbol{\omega}_2\equiv 0$, $deg \, \boldsymbol{\omega}=d=\pm 1$ and for some $\mu_1 \in \C\setminus \{0\}$ we have
\begin{equation}
\label{degeneratesp1}
\boldsymbol{\omega}\big(\boldsymbol{\sigma}_2^{-1}(z)\big)= \pm \boldsymbol{\omega}^{(1)}_{\rm eq}\big(\boldsymbol{\sigma}_2^{-1}(\mu_1 z)\big) \, ,
\end{equation}
with the $+$ sign if $\boldsymbol{\omega}(S^{(2)})=S^{(4)}$, and the $-$ sign if $\boldsymbol{\omega}(S^{(2)})=-S^{(4)}$; 
\item[(ii)]  or $\boldsymbol{\omega}_1\equiv 0$, $deg \, \boldsymbol{\omega}=d=\pm 2$ and for some $\mu_2 \in \C\setminus \{0\}$ we have
\begin{equation}
\label{degeneratesp2}
\boldsymbol{\omega}\big(\boldsymbol{\sigma}_2^{-1}(z)\big)= \pm \boldsymbol{\omega}^{(2)}_{\rm eq}\big(\boldsymbol{\sigma}_2^{-1}(\mu_2 z^2)\big) \, ,
\end{equation}
with the $+$ sign if $\boldsymbol{\omega}(S^{(2)})=S^{(4)}$, and the $-$ sign if $\boldsymbol{\omega}(S^{(2)})=-S^{(4)}$; 
\end{itemize}
here $\boldsymbol{\omega}^{(1)}_{\rm eq}:\mathbb{S}^2\to \mathbb{S}^2_{(1)}\subset\mathbb{S}^4$ and $\boldsymbol{\omega}^{(2)}_{\rm eq}:\mathbb{S}^2\to \mathbb{S}^2_{(2)}\subset\mathbb{S}^4$ are the equatorial embeddings 
\begin{equation}
\label{equatorial}	
\boldsymbol{\omega}^{(1)}_{\rm eq}(x):=(x_3,x_1+ix_2,0)\quad\text{and}\quad \boldsymbol{\omega}^{(2)}_{\rm eq}(x):=(x_3,0,x_1+ix_2)\,.
\end{equation}
\end{proposition}

\begin{proof}
By Lemma~\ref{dichotnonfulllemma}, if $\boldsymbol{\omega}=(\boldsymbol{\omega}_0,\boldsymbol{\omega}_1,\boldsymbol{\omega}_2):\mathbb{S}^2\to\mathbb{S}^4$ is a nonconstant $\mathbb{S}^1$-equivariant and  nonfull harmonic map, then either $\boldsymbol{\omega}_2\equiv 0$, or $\boldsymbol{\omega}_1\equiv 0$. We shall consider the two cases separately. 
\vskip3pt

\noindent{\it Case 1: $\boldsymbol{\omega}_2\equiv 0$.}  Recalling that $\mathbb{S}^2_{(1)}$ denotes the unit sphere of $\R\oplus\C\oplus\{0\}\simeq L_0\oplus L_1$, the mapping $\boldsymbol{\omega}:\mathbb{S}^2\to \mathbb{S}^2_{(1)}$ is harmonic. Since it is not constant, it has a nonzero topological degree $d\in \Z\setminus \{0\}$ such that \eqref{Calabi} holds and, considering $-\boldsymbol{\omega}$ instead of  $\boldsymbol{\omega}$ if necessary, we may assume that $d>0$, so that $\boldsymbol{\omega}$ is holomorphic. Note that $\boldsymbol{\omega}^{-1}(\{\pm S^{(4)}\})\subset \mathbb{S}^2$ are finite sets and $\boldsymbol{\omega}^{-1}(\{\pm S^{(4)}\}) \subset \{ \pm S^{(2)}\}$ because of equivariance.  In view of \cite{Le} (see also \cite[Section 7]{BCL}) and \eqref{partialstereo}, the  map $f:=\boldsymbol{\sigma}_2^{(1)}\circ\boldsymbol{\omega}\circ \boldsymbol{\sigma}^{-1}_2:\C\to\C$ is a well-defined rational function of the form $f(z)=P(z)/Q(z)$ for some coprime polynomials $P$ and $Q$ such that $d=\max \{ \deg P , \deg Q\}$. Since $\{Q=0\} \subset \boldsymbol{\omega}^{-1}(S^{(4)})$ and $\{P=0\} \subset \boldsymbol{\omega}^{-1}(-S^{(4)})$, we must have $f(z)=\mu_1 z^l$ for some $l\in \Z\setminus \{0\}$. Since $\boldsymbol{\sigma}_2$ and $\boldsymbol{\sigma}^{(1)}_2$
are equivariant in view of \eqref{equivsigma2intoC} and \eqref{partialstereoequiv}, then the map $f$ is also equivariant, namely $f(e^{i\alpha}z)=e^{i\alpha}f(z)$ for every $e^{i\alpha}\in\mathbb{S}^1$ and for every $z\neq 0$. Comparing with \eqref{omegacoefficients} yields $l=1$ and the conclusion follows from the definition of $f$ and the identity $\left( \boldsymbol{\sigma}^{(1)}_2 \right)^{-1} = \boldsymbol{\omega}^{(1)}_{\rm eq} \circ \boldsymbol{\sigma}_2^{-1}$.

\noindent{\it Case 2: $\boldsymbol{\omega}_1\equiv 0$.} This second case being entirely similar to the first one, we briefly sketch it just for the reader's convenience. Recalling that $\mathbb{S}^2_{(2)}$ denotes the unit sphere of $\R\oplus\{0\}\oplus\C \simeq L_0 \oplus L_2$,  $\boldsymbol{\omega}:\mathbb{S}^2\to \mathbb{S}^2_{(2)}$ is a nonconstant harmonic map, hence it has a nonzero topological degree $d\in \Z\setminus \{0\}$ such that \eqref{Calabi} holds. As above, up to a sign we may assume $d>0$, hence $\boldsymbol{\omega}$ is holomorphic. Defining the  map $f:=\boldsymbol{\sigma}_2^{(2)}\circ\boldsymbol{\omega}\circ \boldsymbol{\sigma}^{-1}_2:\C\to\C$ we still have a rational function of the form $f(z)=P(z)/Q(z)$ for some coprime polynomials $P$ and $Q$ such that $d=\max \{ \deg P , \deg Q\}$ and the same argument above gives $f(z)=\mu_2 z^l$ for some $l\in \Z\setminus \{0\}$ because of equivariance. The equivariance properties \eqref{equivsigma2intoC} and \eqref{partialstereoequiv} of $\boldsymbol{\sigma}_2$ and $\boldsymbol{\sigma}^{(2)}_2$ now imply $f(e^{i\alpha}z)=e^{i2\alpha}f(z)$ for every $e^{i\alpha}\in\mathbb{S}^1$ and for every $z\neq 0$. Comparing with \eqref{omegacoefficients} in this case yields $l=2$ and the conclusion follows from the definition of $f$ and the identity $\left( \boldsymbol{\sigma}^{(2)}_2 \right)^{-1} = \boldsymbol{\omega}^{(2)}_{\rm eq} \circ \boldsymbol{\sigma}_2^{-1}$.
\end{proof}

\subsection{Linearly full harmonic spheres and horizontal algebraic curves.}
In this subsection, we construct $\bbS^1$-equivariant linearly full harmonic spheres $\boldsymbol{\omega} : \mathbb{S}^2 \to \mathbb{S}^4$ by composing $\mathbb{S}^1$-equivariant {\em horizontal} algebraic curves $\Phi: \C P^1 \to \C P^3$ with the twistor fibration  $\boldsymbol\tau :\bb{C} P^3 \to \bb{S}^4$ defined in \eqref{eq:tau}. We will limit as much as possible the geometric terminology referring the unfamiliar reader to \cite{BairdWood} and \cite{Ura} for all the relevant definitions. 

Recall that $\C P^3$ is a compact manifold with a natural Riemannian metric called the Fubini-Study metric. A possible way to define it is to consider the embedding $\C P^3 \hookrightarrow \calM_{4\times 4} (\C)$ induced by the map $\C^4 \setminus \{ 0\} \ni w \longrightarrow I_4-2 \frac{w \otimes \bar{w}}{|w|^2} \in \calM_{4\times 4} (\C)$ sending each complex line in $\C^4$ into the reflection across the complex 3-plane $w^\perp$, and to consider the pull-back metric on $\C P^3$ of the Riemannian metric $(A, B)= \Re \, \rm{tr} (A^*B)$ \, on  $\calM_{4\times 4} (\C)$. Actually the Fubini-Study metric extends to an Hermitian metric on the complexified tangent bundle which is the pull-back of the standard Hermitian metric $\langle A, B \rangle=  \rm{tr} (A^*B)$ on $\calM_{4\times 4} (\C) $ and endowed with this metric $\C P^3$ is a complex K\"ahler manifold. 

Once $\bb{C} P^3$ is endowed with the Fubini-Study metric the map $\boldsymbol\tau$ becomes a Riemannian submersion \cite{BoWo1}. The \emph{horizontal distribution} $\mathcal{H} = \rmKer {\rm d}\boldsymbol\tau^\perp \subset T \bb{C} P^3$ is the family of subspaces $ \mathcal{H}_{[w]} \subset T_{[w] }\C P^3$  parametrized by $[w] \in \C P^3$ consisting of those tangent vectors to $\bb{C} P^3$ at $[w]$ that are orthogonal to the fibers of $\boldsymbol\tau$ passing through $[w]$, i.e., of those vectors belonging to the $4$-dimensional plane $(\rmKer {\rm d}\boldsymbol\tau_{[w]})^\perp $. Being $\boldsymbol\tau$ by definition a Riemannian submersion, ${\rm d}{\boldsymbol \tau}_{[w]}: \mathcal{H}_{[w]} \to T_{\boldsymbol\tau([w])} \mathbb{S}^4$ is an isometry for any $[w] \in \C P^3$.

We say that a holomorphic map $\tilde \Phi : \C P^1 \to \C P^3$ is \emph{horizontal} if at each point of its image it intersects the fibers of $\boldsymbol\tau$ orthogonally, i.e., $\rmRan {\rm d} \Phi_{[z]} \subset \mathcal{H}_{\Phi(z)}$ for any $z=[z_0,z_1] \in \C P^1$.

In order to check both holomorphicity and horizontality, but also to take advantage of $\mathbb{S}^1$-equivariance, it is convenient to compose the maps with the stereographic projection $\boldsymbol{\sigma}_2^{-1}$ and to consider the induced maps $\Phi:\C \to\C P^3$, $\Phi(z)=\Phi([1,z])$. 
The map $\Phi$ is usually defined (at least locally) as $\Phi(z)=[\Psi(z)]$, where $\Psi :\C \to \C^4 \setminus \{0\}$ is a smooth map, $\Psi(z)=(\Psi_0(z), \ldots, \Psi_3(z))$.

For a smooth function $f:\C \to \C$ of the complex variable $z=x_1+i x_2 \in \C$  we will consider complex derivatives $f_{z}:=\partial_z f$ and $f_{\bar{z}}:=\partial_{\bar{z}} f$ with respect to the usual Wirtinger's operators
\[
	\partial_z = \frac{1}{2}\left( \pd{}{x_1} - i\pd{}{x_2} \right), \quad \partial_{\bar{z}} = \frac{1}{2}\left( \pd{}{x_1} + i\pd{}{x_2} \right) \, .
\] 
Clearly $f$ is holomorphic (resp., antiholomorphic) in $\C$ if and only if $f_{\bar{z}} \equiv 0$ (resp., $f_z \equiv 0$). In addition, the following intertwining relation between Wirtinger's operators and complex conjugation, namely $\bar{f}_{\bar{z}} = \overline{f_z}$,  will be tacitly used during the computations. 

Note that if $f$ is {\em $\mathbb{S}^1$-equivariant of degree k} in the sense that $f(e^{i \alpha} z)= e^{i k \alpha} f(z)$ for all $z\in\C$ and some $k\in \mathbb{Z}$, then we have the identities on $\C$ for any $\alpha \in \R$,
\begin{equation}
\label{derfequivariance}	
 f_{\bar{z}} (e^{i \alpha} z)= e^{i (k+1) \alpha} f_{\bar{z}} (z) \, , \qquad  f_{z} (e^{i \alpha} z)= e^{i (k-1) \alpha} f_{z} (z) \, ,
\end{equation}
showing how the ``degree of equivariance'' changes under complex differentiation. Clearly the conjugate function $\bar{f}$ inherits equivariance of degree $-k$ and similar relations for the degree of equivariance hold for its complex derivatives $\bar{f}_z$ and $\bar{f}_{\bar{z}}$.

The following simple result gives a full description of equivariant holomorphic maps $\Psi : \C P^1 \to \C P^3$ together with their horizontality property. 

\begin{lemma} 
\label{horizontallifts}
For each map $\Psi : \C \setminus \{0\}\simeq \C P^1 \setminus \{\pm S^{(2)}\} \to \C P^3 $,  $\Psi([z_0,z_1])=\Psi(z_1/z_0)$, the following are equivalent:

\begin{enumerate}
\item $\Psi$ is a nonconstant holomorphic map on  $\C P^1\setminus\{\pm S^{(2)}\} $ and $\mathbb{S}^1$-equivariant with respect to the actions \eqref{S1actionCP1} and \eqref{S1actionCP3};

\item there exist $\mu=(\mu_0,\mu_1, \mu_2, \mu_3) \in \C^4 $	with at least two nonzero entries such that on  $\C P^1\setminus\{\pm S^{(2)}\} $ we have
\begin{equation}
\label{algebraiccurves}
 \Psi ([z_0, z_1]) =	\left[ \mu_0 z_0^3, \mu_1 z_0^2 z_1, \mu_2 z_0 z_1^2, \mu_3 z_1^3 \right] \, .
	\end{equation}	 
\end{enumerate}
As a consequence, $\Psi$ extends to a holomorphic map $\Psi : \C P^1 \to \C P^3$ still given by \eqref{algebraiccurves}.
In addition, the map $\Psi$ in \eqref{algebraiccurves} is horizontal if and only if the parameters $\mu_0,\mu_1, \mu_2$ and $\mu_3$ satisfy 
\begin{equation}
\label{mu-horizontal}
\mu_0 \mu_3=-\frac{\mu_1 \mu_2}3 \, .
\end{equation}

\end{lemma}

\begin{proof}
We only discuss the implication $(1) \Rightarrow (2)$, as the converse implication is trivial.

Clearly it is enough to show that for each $z_* \neq 0$ there exists an $\mathbb{S}^1$-invariant open annulus $A=A_{z_*}$ on which the representation \eqref{algebraiccurves} holds. Indeed, taking a collection $\mathcal{A}$ of such annuli, when two annuli $A, A' \in \mathcal{A}$ overlap the corresponding representations with parameters $\mu$ and $\mu'$ must be pointwise proportional, hence the parameters $\mu$ and $\mu'$ must be proportional, i.e. each representation is valid on both the annuli. As a consequence, passing to a subfamily $\mathcal{A}' \subset \mathcal{A}$ which is a locally finite cover of $\C \setminus \{0\}$ allows to pick a representation \eqref{algebraiccurves} with fixed $\mu$ which is valid in the whole $\C \setminus \{0\}$.

Since $\Psi$ is holomorphic and $\mathbb{S}^1$-equivariant, for each $z_*\neq 0$ there exists an invariant annulus $A$ containing $z_*$ and a holomorphic map $\Phi :A \to \C^4 \setminus \{0\}$ such that $\Psi=[\Phi]$ on $A$ and $\Phi(z)=(\Phi_0(z), \ldots , \Phi_3(z))$ has each entry $\Phi_\ell$ which is equivariant of degree $\ell$ on $A$ and at least two of them do not vanish identically because $\Psi$ is nonconstant. Note that because of equivariance each entry $\Phi_\ell$ is either identically zero or nowhere vanishing on $A$ (for if $\Phi_\ell$ vanishes on a circle then it is zero everywhere by the identity principle for holomorphic functions). Choosing $\mu_\ell=0$ whenever $\Phi_\ell \equiv 0$, we can focus on the nonzero components.

We first fix an invariant circle $\mathcal{C}\subset A$ passing through $z_*$, we take $0 \leq m <\ell\leq 3$ corresponding to nonzero components of $\Phi$, with $m$ being the minimum of such indices, and we set $\mu_m=1$. Then the ratio $\Phi_\ell / \Phi_m$ is well-defined and holomorphic on $A$ and it is equivariant of degree $\ell-m$. As a consequence $\Phi_\ell(z) / \Phi_m(z)$ is a nonzero constant multiple $\mu_\ell\neq 0$ of $z^{\ell -m}$ on the circle $\mathcal{C}$, hence $\Phi_\ell \equiv \mu_\ell z^{\ell-m}\Phi_m$ on $A$ again because of the identity principle. Varying $\ell$ with the restriction above we have shown in the annulus $A$ the identity $\Phi(z)=\Phi_m(z) \left( \mu_0 z^{-m}, \ldots , \mu_3 z^{3-m}\right)$, hence \eqref{algebraiccurves} holds on $A$ because $\Phi_m\neq 0$ on $A$. Thus $(1) \Rightarrow (2)$ is completely proved.

Once the equivalence is proved, it is obvious that $\Psi$ admits the obvious holomorphic extension to $\C P^1$  still given by \eqref{algebraiccurves}. Concerning horizontality, in view of the homogeneity in \eqref{algebraiccurves} we can regard $\Phi$ as a globally defined map $\Phi: \C^2 \to \C^4$ and considering on $\C^4$ the 1-form 
\[
  \Theta= w_0 \,\mbox{d}w_3 -  w_3 \,\mbox{d}w_0 +  w_1 \,\mbox{d}w_2 -  w_2 \,\mbox{d}w_1,
\]
by the explicit form of $\boldsymbol\tau$ the horizontality condition can be rewritten (see \cite{LemaireWood}, see also \cite{BoWo1}, \cite{BoWo2})
 as
\begin{equation}
\label{eq:horizontality1}
  \Phi^* \Theta = 0.
\end{equation} 
Then, a simple calculation combining \eqref{algebraiccurves} and \eqref{eq:horizontality1} yields \eqref{mu-horizontal}.
\end{proof}

The next result provides an explicit family of linearly full equivariant harmonic spheres.

\begin{proposition}\label{prop:explfullspheres}
Let $\Psi : \bbS^2=\C P^1 \to \C P^3$ be an $\mathbb{S}^1$-equivariant horizontal holomorphic map as in \eqref{algebraiccurves}-\eqref{mu-horizontal} and assume $\mu_0=1$. Then the composition $\boldsymbol{\omega}=\boldsymbol\tau \circ \Psi$ is given by the formulas

\begin{equation}\label{eq:classification}
	\begin{split}
	\boldsymbol{\omega}([z_0, z_1]) = &\frac{1}{D(z_0, z_1)} \left( \abs{z_0 }^6 - \abs{\mu_1}^2 \abs{z_0}^4\abs{z_1}^2-  \abs{\mu_2}^2 \abs{z_0}^2 \abs{z_1}^4  +\frac{\abs{\mu_1}^2\abs{\mu_2}^2}{9}\abs{z_1}^6 , \right. \\
	& \left. 2 \mu_1 \overline{z_0} z_1 \left( \abs{z_0}^4 - \frac{\abs{\mu_2}^2}{3}\abs{z_1}^4 \right), 2 \mu_2 \overline{z_0}^2 z_1^2 \left( \abs{z_0}^2 + \frac{\abs{\mu_1}^2}{3} \abs{z_1}^2 \right)  \right),
	\end{split}
	\end{equation}		
where

	\begin{equation}
		D(z_0,z_1) =  \abs{z_0 }^6 + \abs{\mu_1}^2 \abs{z_0}^4\abs{z_1}^2+  \abs{\mu_2}^2 \abs{z_0}^2 \abs{z_1}^4  +\frac{\abs{\mu_1}^2\abs{\mu_2}^2}{9}\abs{z_1}^6 \, .
	\end{equation}
In addition, for each $(\mu_1, \mu_2) \in \C^2$ the map $\boldsymbol{\omega}:\mathbb{S}^2 \to \mathbb{S}^4$ is an  $\mathbb{S}^1$-equivariant harmonic map with $\boldsymbol{\omega}([1, 0]) = (1,0,0)$. Moreover, $\boldsymbol{\omega}$ is linearly full with energy $E(\pmb{\omega})=4\pi d=12 \pi$   if and only if $\mu_1\neq 0$ and $\mu_2 \neq 0$.
\end{proposition}

\begin{proof}
The validity of the formula \eqref{eq:classification} for the composition $\boldsymbol\tau \circ \Psi$ together with the normalization $\boldsymbol{\omega}([1, 0]) = (1,0,0)$ are just straightforward computations. It is easy to check that the corresponding maps are $\mathbb{S}^1$-equivariant with respect to the actions given in \eqref{S1actionCP1}-\eqref{S1actionRCC}, hence \eqref{omegaequiv} holds. 
Harmonicity of $\boldsymbol{\omega}$ as in \eqref{eq:classification} can be verified by a direct checking of \eqref{eqharmspheres}. Alternatively, since the projective spaces $\C P^N$ are K\"ahler manifolds, by \cite[Chapter 4, Prop. 3.14]{Ura} each holomorphic map in \eqref{algebraiccurves} is harmonic. In addition, since $\boldsymbol\tau$ is a Riemannian submersion and under \eqref{mu-horizontal} each $\Psi$ is a horizontal harmonic map, then  the composite maps $\boldsymbol{\omega}=\boldsymbol\tau \circ \Psi $ in \eqref{eq:classification} are harmonic in view of \cite[Chapter 6, Prop. 2.36]{Ura}. In view of the explicit form of its components, we see that $\boldsymbol{\omega}_1 \equiv 0$ iff $\mu_1=0$ and  $\boldsymbol{\omega}_2 \equiv 0$ iff $\mu_2=0$, therefore the map $\boldsymbol{\omega}$ is full if and only if $\mu_1\neq 0$ and $\mu_2\neq 0$.
 Finally, according to \eqref{Calabi}, the value of the energy being discrete, it has to be locally constant under continuous changes of the parameters, hence it must be independent of their effective values, whenever the condition $\mu_1 \mu_2 \neq 0$ holds (because the dependence of the map on the parameters is easily seen to be continuous in $C^1(\mathbb{S}^2;\mathbb{S}^4)$). Choosing $\mu_1=\mu_2=\sqrt{3}$ a direct calculation gives $|\nabla_T \boldsymbol{\omega}|^2 \equiv 6$, hence $E(\boldsymbol{\omega})=12 \pi$, $d=3$ and the proof is complete. 
\end{proof} 		

\begin{remark}\label{rmk:hedgehog}
As already observed in the previous proof and as we will comment more in the next section, the case $\mu_1 = \mu_2 = \sqrt{3}$ is special. The corresponding harmonic sphere is by a direct computation 
\[
	\boldsymbol{\omega}^{(H)}([z_0,z_1]) = \frac{1}{\left(\abs{z_0}^2+\abs{z_1}^2\right)^2} \left( \abs{z_0}^4 - 4\abs{z_0}^2\abs{z_1}^2 + \abs{z_1}^4, 2 \sqrt{3} \, \overline{z_0} z_1 \left( \abs{z_0}^2 - \abs{z_1}^2\right), 2 \sqrt{3} \, \overline{z_0}^2 z_1^2  \right),
\]
it is invariant under the antipodal map $\C P^1 \ni [z_0,z_1] \to [\overline{z_1}, -\overline{z_0}] \in \C P^1$  and it corresponds to the {\em Veronese embedding} $\mathbb{S}^2 / \{\pm 1\}= \R P^2 \hookrightarrow \mathbb{S}^4$.
\end{remark}

\begin{remark}
  Notice that the function $\omega_0$ is always real-valued, while the functions $\omega_1$, $\omega_2$ in \eqref{eq:omega} are complex-valued whenever the parameters $\mu_1$, $\mu_2$ are such and real-valued otherwise.
\end{remark}

\begin{remark}
\label{degeneratecases}
It is worth noticing that letting $\mu_2=0$ or $\mu_1=0$ in \eqref{eq:classification} we obtain precisely, up to the double sign, the linearly degenerate harmonic spheres of energy $4\pi$ and $8\pi$ respectively described in Proposition~\ref{degeneratespheres}.
\end{remark}

\subsection{Classification of linearly full harmonic spheres.}
In this final subsection we are going to show that  every equivariant linearly full harmonic sphere $\boldsymbol\omega$ is actually one of those constructed in Proposition~\ref{prop:explfullspheres}, possibly up to composition with the antipodal involution on $\mathbb{S}^4$. As already announced in the Introduction, the key step to achieve such classification  is to obtain each $\bbS^1$-equivariant linearly full harmonic sphere ${\boldsymbol\omega}: \bb{S}^2=\C P^1 \to \bb{S}^4$  as a composition of the twistor fibration  $\boldsymbol\tau :\bb{C} P^3 \to \bb{S}^4$ defined in \eqref{eq:tau} with a horizontal $\bbS^1$-equivariant algebraic curve $\widetilde{\boldsymbol \omega} : \bb{C} P^1 \to \bb{C} P^3$, called the \emph{twistor lift}\footnote{The \emph{twistor space} of $\bbS^4$ is $\SO(5)/\U(2)$, see \cite[Chapter 7]{BairdWood}. However, for the purpose of presenting the lift in terms of explicit formulas, we find more convenient to work with its equivalent presentation as $\C P^3$. For details on this identification and the deduction of these formulas from those in \cite{Ca} we refer the interested reader to \cite{Fawley}.} of $\boldsymbol\omega$. The geometric meaning of such construction in terms of orthogonal almost complex structures on the tangent spaces $\{T_{\boldsymbol\omega(p)} \mathbb{S}^4\}_{p \in \mathbb{S}^2}$ together with its far-reaching higher dimensional generalizations in the framework of twistor theory are described in details in the references given in the Introduction but they will not be discussed at all here. Indeed, in order to keep to the minimum the needed background from complex geometry, we will follow quite closely the concrete description of the lift from \cite{BoWo1,BoWo2,BoWo3} and \cite{Fawley}, but avoiding any use of quaternions, so that the argument here will be more elementary and it will be presented in an almost self-contained form. Thus, we will take the general formulas for the lift $\widetilde{\boldsymbol\omega}$ from those papers as an ansatz, and prove that they actually satisfy all the desired properties in order to reconstruct ${\boldsymbol\omega}$. In particular, we will check that the construction of $\widetilde{\boldsymbol\omega}$ is compatible with the $\mathbb{S}^1$-equivariance constraint,
which, apparently, has not yet been considered in the literature. As a final consequence, it will be quite easy to obtain explicit formulas for all the linearly full equivariant harmonic spheres ${\boldsymbol\omega} : \mathbb{S}^2 \to \mathbb{S}^4$ in terms of those in \eqref{eq:classification}.

We now explain how to construct an algebraic $\bbS^1$-equivariant horizontal lift starting from ${\boldsymbol\omega}$. Since ${\boldsymbol\omega}(-S^{(2)})\in \{\pm S^{(4)}\}$ then, without loss of generality, up to composing with the antipodal map $a :\mathbb{S}^4 \to \mathbb{S}^4$ given by $a(x)=-x$ we may assume $\boldsymbol{\omega}(-S^{(2)})= \boldsymbol{\omega}(-S^{(4)})$. We follow \cite{BoWo1} and for each linearly full harmonic sphere we consider the complex valued smooth functions $(\xi, \eta)= \boldsymbol{\sigma}_4 \circ \boldsymbol{\omega}$, so that
\begin{equation}
\label{eq:xi-eta}
	\xi := \frac{\pmb{\omega}_1}{1 + \pmb{\omega}_0}, \quad \eta := \frac{\pmb{\omega}_2}{1+\pmb{\omega}_0}.
\end{equation}
By Remark \ref{rmk:lemaire}, both $\xi$ and $\eta$ are well-defined everywhere on $\bbS^2$ except, possibly, at the south pole because of our normalization above. Further, they are $\bbS^1$-equivariant in the sense of \eqref{S1actionCP1}-\eqref{S1actionC2} and real-analytic, since both $\boldsymbol{\sigma}_4$ and $\boldsymbol{\omega}$ enjoy these properties. Composing with $\boldsymbol{\sigma}^{-1}_2$ and identifying $\mathbb{S}^2 \setminus \{S^{(2)}\}$ with $\C$, we will regard \eqref{eq:xi-eta} as smooth equivariant complex-valued functions defined in the whole complex plane $\C$. In addition, by \eqref{eqharmspheres} and \eqref{eq:xi-eta}  simple computations yield  
\begin{equation}
\label{harmxi}
\xi_{z\bar{z}} + \frac{\xi (\abs{\eta_{\bar{z}}}^2+\abs{\bar{\eta}_{\bar{z}}}^2 )- 2\bar{\xi} \xi_{\bar{z}} \xi_z - \xi_{\bar{z}} \partial_z \abs{\eta}^2 - \xi_{z} \partial_{\bar{z}}\abs{\eta}^2}{1 + \abs{\xi}^2 + \abs{\eta}^2} = 0
\end{equation}
and
\begin{equation}
\label{harmeta}	
\eta_{z\bar{z}} + \frac{\eta (\abs{\xi_{\bar{z}}}^2+\abs{\bar{\xi}_{\bar{z}}}^2 )  - 2\bar{\eta} \eta_{\bar{z}} \eta_z - \eta_{\bar{z}} \partial_z \abs{\xi}^2 - \eta_{z} \partial_{\bar{z}}\abs{\xi}^2}{1 + \abs{\xi}^2 + \abs{\eta}^2} = 0,
\end{equation}
where the previous equations hold in the whole complex plane $\C$.

Applying \eqref{derfequivariance} to $\xi$ and $\eta$ with $k=1$ and $k=2$ respectively, we have the identities on $\C$ for any $\alpha \in \R$, namely,
\begin{equation}
\label{equivxiderivatives}
	\xi_{\bar{z}} (e^{i \alpha} z)= e^{i 2 \alpha} \xi_{\bar{z}} (z) \, , \qquad  \xi_{z} (e^{i \alpha} z)=  \xi_{z} (z) \, ,
\end{equation}
and
\begin{equation}
\label{equivetaderivatives}
	\eta_{\bar{z}} (e^{i \alpha} z)= e^{i 3 \alpha} \eta_{\bar{z}} (z) \, , \qquad  \eta_{z} (e^{i \alpha} z)=  e^{i \alpha} \eta_{z} (z) \, ,
\end{equation}
with similar properties for the complex derivatives of $\bar{\xi}$ and $\bar{\eta}$.

The next result is a consequence of conformality and real isotropy of harmonic spheres stated in Lemma \ref{lemma2calabi}, once rewritten in complex coordinates and in terms of $\xi$, $\eta$.
\begin{lemma}\label{lemma:conf+iso}
	Let $\boldsymbol\omega : \bbS^2 \to \bbS^4$ be an $\bbS^1$-equivariant harmonic map and $\xi$, $\eta$ as in \eqref{eq:xi-eta}. Then:
	\begin{equation}
	\label{eq:conf}
\xi_{\bar{z}} \bar{\xi}_{\bar{z}} + \eta_{\bar{z}} \bar{\eta}_{\bar{z}} = 0
	\end{equation}
	and
	\begin{equation}
	\label{eq:iso}
		\xi_{\bar{z}\bar{z}} \bar{\xi}_{\bar{z}\bar{z}} + \eta_{\bar{z}\bar{z}} \bar{\eta}_{\bar{z}{\bar{z}}} = 0.
	\end{equation}
\end{lemma}
\begin{proof}
Using the definition of $\xi$, $\eta$ one may check directly that Eq.~\eqref{eq:conf} and Eq.~\eqref{eq:iso} follow taking, respectively, $\alpha =1$, $\beta=1$ and $\alpha =2$, $\beta = 2$ in \cite[Proposition 6.1]{HeleinBook}. 
\end{proof}

Another key consequence is the following lemma. 

\begin{lemma}\label{lemma:nonzero}
	If $\boldsymbol\omega : \bbS^2 \to \bbS^4$ is a linearly full $\bbS^1$-equivariant harmonic map and $\xi$, $\eta$ are as in \eqref{eq:xi-eta}, then at any point $z \in \C \setminus \{ 0\}$ the complex derivatives $\bar{\xi}_{\bar z}$, $\eta_{\bar z}$, $\xi_{\bar z}$ and $\bar{\eta}_{\bar{z}}$ cannot vanish simultaneously.  Moreover, neither $\bar{\xi}_{\bar{z}}$ nor $\bar{\eta}_{\bar z}$ can vanish identically. Finally, $\xi_{\bar z}$ and $\eta_{\bar z}$ cannot vanish identically at the same time. 
\end{lemma}

\begin{proof}
 The first statement follows from conformality and Remark \ref{rmk:lemaire}. Indeed, if we take $z\neq 0$ then  $p= \boldsymbol{\sigma}_2^{-1}(z)  \in \mathbb{S}^2 \setminus \{ \pm S^{(2)}\}$ and $\boldsymbol\omega (p) \neq \pm S^{(4)}$; hence we have $(\xi(z), \eta(z)) =\boldsymbol{\sigma}_4(\omega(p)) \neq (0,0)$ and, by the conformality of $\boldsymbol{\sigma}_4$, we can rewrite the energy density of $\boldsymbol\omega$ at $p$ as
\[ 
	\abs{\nabla_T \boldsymbol\omega (p)}^2=4\frac{\abs{\bar{\xi}_{\bar z}}^2+ \abs{\eta_{\bar z}}^2+\abs{\xi_{\bar z}}^2+\abs{\bar{\eta}_{\bar{z}}}^2}{\left(1+\abs{\xi(z)}^2+\abs{\eta(z)}^2\right)^2} \neq 0 \, ,
\]
since the only branch points of $\boldsymbol\omega$ can be at the poles.
 
Now we claim that $\bar{\xi}_{\bar{z}}$ and $\bar{\eta}_{\bar{z}}$ cannot vanish identically. Suppose by contradiction that $\bar{\xi}_{\bar{z}} \equiv 0$ in $\C$. Then $\bar{\xi}$ is an entire holomorphic function, hence $\xi$ is antiholomorphic in the whole $\C$. Note that $\xi$ is equivariant of degree one on $\C \setminus \{0\}$, as the function $g(z)=1/\bar{z}$, hence $\xi(z)/g(z) \equiv \xi(1)  $ on the circle $|z|=1$ again by equivariance. By the identity principle for antiholomorphic functions we have $\xi(z) =\xi(1)/\bar{z}$ on $\C \setminus \{0\}$, hence $\xi(1)=0$ and in turn $\xi \equiv 0$ because the function $\xi$ is bounded near the origin. As a consequence, \eqref{eq:xi-eta} yields $\boldsymbol{\omega}_1\equiv 0$, which is impossible, because the map $\boldsymbol{\omega}$ is linearly full. A similar argument applies if we assume $\bar{\eta}_{\bar{z}}\equiv 0$. Indeed, then $\eta(z)/h(z)\equiv \eta(1)$ in $\C \setminus \{0\}$ with $h(z)=1/\bar{z}^2$, hence $\eta \equiv 0$ because it is bounded near the origin. Thus $\boldsymbol{\omega}_2\equiv 0$ because of \eqref{eq:xi-eta}, which is again impossible because the map $\boldsymbol{\omega}$ is linearly full. 

Finally, if $\xi_{\bar z}$ and $\eta_{\bar z}$ both vanish identically then $\xi$ and $\eta$ are both holomorphic and equivariant of degree one and two respectively. Arguing as above we have $\xi \equiv c_1 z$ and $\eta=c_2 z^2$, where $c_1 c_2 \neq 0$ because the map $\boldsymbol\omega$ is linearly full.
Then simple calculations give a contradiction because neither $\eqref{harmxi}$ nor $\eqref{harmeta}$ are satisfied.
\end{proof}

The last preliminary fact we need before defining the twistor lift is the following.
\begin{lemma}
\label{holomorphicfactor} 
Let ${\boldsymbol\omega}: \bbS^2 \to \bbS^4$ be an $\bbS^1$-equivariant harmonic map and $\xi$, $\eta$ as in \eqref{eq:xi-eta}. Then in the whole complex plane we have
\begin{equation}
\label{zeroproduct}
\left( \bar{\xi}_{\bar{z}\bar{z}}\bar{\eta}_{\bar{z}} - \bar{\eta}_{\bar{z}\bar{z}} \bar{\xi}_{\bar{z}} \right) \left( \xi_{\bar{z}\bar{z}}\bar{\eta}_{\bar{z}} - \bar{\eta}_{\bar{z}\bar{z}} \xi_{\bar{z}} \right) = 0 \, .
\end{equation}
As a consequence, at least one of the two factor in \eqref{zeroproduct} vanishes identically on $\C$.

If, in addition, the map ${\boldsymbol\omega}$ is linearly full, then only one factor in \eqref{zeroproduct} vanishes identically.
	\end{lemma}

\begin{proof}
The conclusion of the first part will follow from Lemma \ref{lemma:conf+iso}  by simple manipulations in the whole complex plane.
	Differentiating \eqref{eq:conf} with respect to $\bar{z}$ gives
\begin{equation}\label{eq:conf-diff}
	\xi_{\bar{z}\bar{z}} \bar{\xi}_{\bar{z}} + \xi_{\bar{z}} \bar{\xi}_{\bar{z}\bar{z}} + \eta_{\bar{z}\bar{z}} \bar{\eta}_{\bar{z}} + \eta_{\bar{z}} \bar{\eta}_{\bar{z}\bar{z}} = 0 \, . 
\end{equation}
Multiplying \eqref{eq:iso} by $\bar{\eta}_{\bar{z}}$ we obtain
\[
	\bar{\eta}_{\bar{z}} \xi_{\bar{z}\bar{z}} \bar{\xi}_{\bar{z}\bar{z}} + \bar{\eta}_{\bar{z}} \eta_{\bar{z}\bar{z}} \bar{\eta}_{\bar{z}\bar{z}} = 0 \, . 
\]
Solving \eqref{eq:conf-diff} with respect to $\bar{\eta}_{\bar{z}} \eta_{\bar{z}\bar{z}}$ and substituting into the last identity above, we have, after some rearrangements,
\[
	\bar{\xi}_{\bar{z}\bar{z}}\left( \xi_{\bar{z}\bar{z}}\bar{\eta}_{\bar{z}} - \bar{\eta}_{\bar{z}\bar{z}} \xi_{\bar{z}} \right) - \bar{\eta}_{\bar{z}\bar{z}}\left( \xi_{\bar{z}\bar{z}} \bar{\xi}_{\bar{z}} + \bar{\eta}_{\bar{z}\bar{z}} \eta_{\bar{z}} \right) = 0 \, .
\]
 Multiplying also the last identity by $\bar{\eta}_{\bar{z}}$ we have 
 
 \[ \bar{\eta}_{\bar{z}}\bar{\xi}_{\bar{z}\bar{z}}\left( \xi_{\bar{z}\bar{z}}\bar{\eta}_{\bar{z}} - \bar{\eta}_{\bar{z}\bar{z}} \xi_{\bar{z}} \right) - \bar{\eta}_{\bar{z}\bar{z}}\left( \bar{\eta}_{\bar{z}\bar{z}} \eta_{\bar{z}} \bar{\eta}_{\bar{z}} + \xi_{\bar{z}\bar{z}} \bar{\eta}_{\bar{z}} \bar{\xi}_{\bar{z}} \right) = 0 \, .\]
 By the conformality relation \eqref{eq:conf}, we have $\eta_{\bar{z}}  \bar{\eta}_{\bar{z}} = -\xi_{\bar{z}} \bar{\xi}_{\bar{z}}$ and, after substituting in the last identity and rearranging the resulting terms, Equation \eqref{zeroproduct} follows. Finally, since both factors in \eqref{zeroproduct} are real-analytic functions in the whole $\C$, then at least one of the two must vanish identically by unique continuation.
  
  Concerning the second part, we argue by contradiction, supposing that both factors \eqref{zeroproduct} vanish identically. We start observing that in view of Lemma \ref{lemma:nonzero} we have $\{ \bar{\xi}_{\bar{z}} \neq 0\}\neq \emptyset$ and this open set is $\mathbb{S}^1$-invariant because of \eqref{equivxiderivatives}. Hence, this open set contains a circle $\mathcal{C}=\{|z|=a>0\}$ and an invariant annulus $\mathcal{A}$ around it. In the annulus $\mathcal{A}$, as the first factor in \eqref{zeroproduct} vanishes, the  function $\bar{\eta}_{\bar{z}}/  \bar{\xi}_{\bar{z}}$ is holomorphic with the same degree of equivariance of $1/z$. Since by equivariance $\bar{\eta}_{\bar{z}}/  \bar{\xi}_{\bar{z}}$ and $1/z$ differ only by a constant factor on $\mathcal{C}$, by the identity principle for holomorphic functions the same holds on $\mathcal{A}$. Hence, there is  a complex number $c_1$ such that $z \bar{\eta}_{\bar{z}} -c_1   \bar{\xi}_{\bar{z}} \equiv 0$ first on $\mathcal{A}$ and then on $\C$ by unique continuation for real-analytic functions. Clearly $c_1 \neq 0$, otherwise we would get $ \bar{\eta}_{\bar{z}} \equiv 0$ which is impossible in view of Lemma \ref{lemma:nonzero}. 
  We argue in a similar way using the second factor in \eqref{zeroproduct}. By Lemma~\ref{lemma:nonzero}, the set $\{ \bar{\eta}_{\bar{z}} \neq 0\}$ is not empty and  invariant, hence it contains a second annulus where we can consider the function $\xi_{\bar{z}}/\bar{\eta}_{\bar{z}}$, which is holomorphic and equivariant of degree $3$, hence arguing as above we may write $\xi_{\bar{z}}=c_2 z^{3} \bar{\eta}_{\bar{z}}$ for some $c_2 \in \C$, where the identity on the annulus also extends to $\C$ by unique continuation as above. Note that $c_2 \neq 0$, otherwise $\xi$ would be an entire holomorphic function equivariant of degree one, hence $\xi=c_3 z$, with $c_3\neq 0$ because ${\boldsymbol\omega}$ is linearly full. On the other hand, in view of the conformality condition \eqref{eq:conf} the function $\eta$ would satisfy $\eta_{\bar{z}}=0$ since $\bar{\eta}_{\bar{z}}\not \equiv 0$. Being holomorphic, nontrivial and equivariant of degree two, we would have $\eta=c_4 z^2$ for some $c_4\neq 0$, but then $\xi= c_3z$ and $\eta=c_4 z^2$ with $c_3c_4 \neq 0$ should satisfy \eqref{harmxi}-\eqref{harmeta} which is impossible by direct computation.
  Hence, we showed that if both factors in \eqref{zeroproduct} vanish then for some $c_1, c_2 \in \C \setminus \{0\}$ the functions $\xi$ and $\eta$ solve the first order system
 \begin{equation}
 \label{firstordersystem}	
\begin{cases}
 	z \bar{\eta}_{\bar{z}} =c_1   \bar{\xi}_{\bar{z}} \, ,\\

 	\xi_{\bar{z}}=c_2 z^{3} \bar{\eta}_{\bar{z}} \, ,
\end{cases} 
\end{equation}
whence $(\xi-c_1 c_2 z^2  \bar{\xi})_{\bar{z}}=0 $ on $\C$. 
 Again, by holomorphicity and equivariance of degree one we would have $\xi-c_1 c_2 z^2  \bar{\xi}=c_5 z$, with $c_5\neq 0$ (otherwise we would easily get $|\xi|^2(1- |c_1c_2z^2|^2|)=0$ identically, which contradicts $\xi \not \equiv 0$). Combining the last equation with its conjugate $\bar{\xi}=\overline{c_1c_2}\bar{z}^2 \xi+\overline{c_5}\bar{z}$     we finally get
 \[ \xi=\frac{c_5z(1+c_1c_2 \overline{c_5}/c_5|z|^2)}{1-|c_1c_2 z^2|^2} =\frac{c_5z}{1+|c_1c_2 z^2|} \, , \]
    because $\xi$ must be locally bounded, whence we must have $c_1c_2 \overline{c_5}/{c_5}=-|c_1c_2|$ and the last identity follows. In view of the previous formula for $\xi$, the function  $\bar{\xi}_{\bar{z}}$ does not vanish as $z \to 0$, hence letting $z \to 0$ in the first equation in \eqref{firstordersystem} we derive a contradiction and the proof is complete.
\end{proof}

We are finally in the position to introduce the {\em twistor lift} of a harmonic sphere. Starting from $\boldsymbol{\omega}$, we can define two lifts
\begin{equation}\label{eq:lift-pos}
	\widetilde{\boldsymbol{\omega}}^+(z) = \begin{cases}
		[\bar{\xi}_{\bar{z}}, \bar{\xi}_{\bar{z}} \xi + \eta_{\bar{z}}\bar{\eta}, \bar{\xi}_{\bar{z}}\eta - \eta_{\bar{z}} \bar{\xi},  -\eta_{\bar{z}}], & \mbox{if} \,\,\, {(\bar{\xi}_{\bar z}, \eta_{\bar z}) \neq (0,0)} \, , \\
	[\bar{\eta}_{\bar{z}}, \bar{\eta}_{\bar{z}} \xi - \xi_{\bar{z}} \bar{\eta}, \bar{\eta}_{\bar{z}} \eta + \xi_{{\bar z}}\bar{\xi}, \xi_{\bar{z}}], & \mbox{if} \, \, \,  (\xi_{\bar z}, {\bar \eta}_{\bar z}) \neq (0,0)\,,
	\end{cases}
\end{equation}
and 

\begin{equation}\label{eq:lift-neg}
	\widetilde{\boldsymbol{\omega}}^-(z) = \begin{cases}
		[\xi_{\bar{z}} \bar{\xi} + \eta_{\bar{z}} \bar{\eta}, - \xi_{\bar{z}}, -\eta_{\bar{z}},  \eta_{\bar{z}} \xi - \xi_{\bar{z}}\eta ], & \mbox{if } (\xi_{\bar z}, \eta_{\bar z}) \neq (0,0)\, ,\\
		[\bar{\xi}_{\bar{z}} \bar{\eta} -\bar{\eta}_{\bar{z}} \bar{\xi}, \bar{\eta}_{\bar{z}}, -\bar{\xi}_{\bar{z}}, \bar{\xi}_{\bar{z}} \xi + \bar{\eta}_{\bar{z}} \eta], & \mbox{if } ({\bar \xi}_{\bar z}, {\bar \eta}_{\bar z}) \neq (0,0)\,. 
	\end{cases}
\end{equation}
In accordance with the standard terminology from twistor theory (see, e.g., \cite{BairdWood}, \cite{LemaireWood} and \cite{Fawley}), we call $\widetilde{\boldsymbol{\omega}}^+$ the \emph{positive lift} of $\boldsymbol{\omega}$ to $\C P^3$ and $\widetilde{\boldsymbol{\omega}}^-$ its \emph{negative lift}. The reason for this terminology will not be explained and it depends on the geometric meaning of the above formulas (more precisely on the orientation of the associated almost complex structures on $T_{\boldsymbol{\omega}(z)} \mathbb{S}^4$). Here we do not justify the expressions for the lifts in \eqref{eq:lift-pos}-\eqref{eq:lift-neg} but we just show that these are suitable for our purposes. 

The following result is an adaptation to our symmetric context of \cite[Theorem 1]{Verdier} (see also \cite{Br} and \cite{Ca}).

\begin{proposition}
\label{liftings}
	Let $\boldsymbol{\omega}: \bbS^2 \to \bbS^4$ be a linearly full $\bbS^1$-equivariant harmonic map such that $\boldsymbol{\omega}(-S^{(2)}) = -S^{(4)}$. Let $\widetilde{\boldsymbol{\omega}}^+, \widetilde{\boldsymbol{\omega}}^- : \bbS^2 \simeq \C P^1 \to \C P^3$ be as in \eqref{eq:lift-pos}, \eqref{eq:lift-neg} and $a : \bbS^4 \to \bbS^4$ be the antipodal map on $\bbS^4$. Then:
\begin{enumerate}
\item\label{item:welldef} $\widetilde{\boldsymbol{\omega}}^+$, $\widetilde{\boldsymbol{\omega}}^-$ are well-defined at any point of $\bbS^2 \setminus \{ \pm S^{(2)}\}$ and real-analytic.
\item\label{item:s1-eq} $\widetilde{\boldsymbol{\omega}}^+$, $\widetilde{\boldsymbol{\omega}}^-$ are $\bbS^1$-equivariant with respect to to the action \eqref{S1actionCP1} on $\bbS^2 \simeq \C P^1$ and the action \eqref{S1actionCP3} on $\C P^3$.
\item\label{item:lift} $\widetilde{\boldsymbol{\omega}}^+$ is a lift of $\boldsymbol{\omega}$ and $\widetilde{\boldsymbol{\omega}}^-$ is a lift of $a \circ \boldsymbol{\omega}$; i.e., $\boldsymbol{\tau} \widetilde{\boldsymbol{\omega}}^+ = \boldsymbol{\omega}$ and $\boldsymbol{\tau} \circ \widetilde{\boldsymbol{\omega}}^- = a \circ \boldsymbol{\omega}$. 
\item\label{item:hol} At least one (and actually exactly one) among $\widetilde{\boldsymbol{\omega}}^+$ and $\widetilde{\boldsymbol{\omega}}^-$ is holomorphic in $\bbS^2 \setminus \{ \pm S^{(2)}\}$. Thus, the corresponding formula for the lift extends holomorphically to the whole $\bbS^2$.
\item\label{item:hor} If $\widetilde{\boldsymbol{\omega}}^+$ ($\widetilde{\boldsymbol{\omega}}^-$) is holomorphic, then it is also horizontal.
	\end{enumerate}	 
\end{proposition} 

\begin{proof}
	The claimed statements follow combining the preliminary results presented above, therefore the argument are more direct and much more elementary than those in the existing literature covering the nonsymmetric case.
	
	For (\ref{item:welldef}), we discuss only the case of $\widetilde{\boldsymbol{\omega}}^+$. The case of $\widetilde{\boldsymbol{\omega}}^-$ is entirely similar and it is left to the reader. First observe that the domain of definition of the two formulas for $\widetilde{\boldsymbol{\omega}}^+$, namely $\{(\bar{\xi}_{\bar z}, \eta_{\bar z}) \neq (0,0) \}$ and $ \{(\xi_{\bar z}, {\bar \eta}_{\bar z}) \neq (0,0) \}$, are both nonempty and open and their union is $\C \setminus \{ 0\}$ because of Lemma \ref{lemma:nonzero}, hence they have nonempty open intersection $A \subset \C \setminus \{ 0\}$. Thus, equations \eqref{eq:lift-pos} give on each open set a well-defined real-analytic map with values into $\C^4 \setminus \{ 0\}$ because $\xi$ and $\eta$ are real-analytic. In addition, on the intersection $A$ the two quadruples in $\C^4$ are proportional because of the conformality property of $\boldsymbol{\omega}$ (considering them as rows of a matrix in $M_{2\times4}(\C)$, the rank is one at every point of $A$ because of \eqref{eq:conf}). As a consequence \eqref{eq:lift-pos} gives a well-defined real-analytic function from $\C\setminus \{0\}$ to $\C P^3$.
		
	Statement (\ref{item:s1-eq}) is a consequence of the equivariance properties of $\xi$ and $\eta$ and their derivatives given in \eqref{equivxiderivatives}-\eqref{equivetaderivatives}. Thus the complex entries in the first quadruple in \eqref{eq:lift-pos} are equivariant of degrees $(0,1,2,3)$, those in the second of degrees $(-1,0,1,2)$, hence equivariance holds with respect to the action \eqref{S1actionCP3} on $\C P^3$. Similar considerations apply to $\widetilde{\boldsymbol{\omega}}^-$.
	
 Statement (\ref{item:lift}) is nothing more than a straightforward computation combining the explicit formulae for the lifts with \eqref{eq:tau} and using again the fact that the derivatives of $\xi$ and $\eta$ cannot vanish simultaneously because of Lemma \ref{lemma:nonzero}.

For what concerns (\ref{item:hol}), we observe that, as detailed for instance in \cite{Fawley}, the holomorphicity condition for the first representation of $\widetilde{\boldsymbol\omega}^+$ in terms of $\xi$, $\eta$ clearly implies
\begin{equation}\label{eq:hol+1}
	\bar{\xi}_{\bar{z}\bar{z}} \eta_{\bar{z}} - \eta_{\bar{z}\bar{z}} \bar{\xi}_{\bar{z}} = 0  \quad \mbox{ in } (\bar{\xi}_{\bar z}, \eta_{\bar z}) \neq (0,0)
\end{equation}
or for the second representation,
\begin{equation}\label{eq:hol+2}
	\bar{\eta}_{\bar{z}\bar{z}} \xi_{\bar{z}} - \xi_{\bar{z}\bar{z}} \bar{\eta}_{\bar{z}} = 0 \quad \mbox{ in } (\xi_{\bar z}, {\bar \eta}_{\bar z}) \neq (0,0) \, .
\end{equation}
Similarly, for $\widetilde{\boldsymbol\omega}^-$ we get
\begin{equation}\label{eq:hol-1}
	\xi_{\bar{z}\bar{z}} \eta_{\bar{z}} - \eta_{\bar{z}\bar{z}} \xi_{\bar{z}} = 0 \quad \mbox{ in } (\xi_{\bar z}, \eta_{\bar z}) \neq (0,0) \, ,
\end{equation}
or for the second definition
\begin{equation}\label{eq:hol-2}
	\bar{\xi}_{\bar{z}\bar{z}} \bar{\eta}_{\bar{z}} - \bar{\eta}_{\bar{z}\bar{z}} \bar{\xi}_{\bar{z}} = 0 \quad \mbox{ in } ({\bar \xi}_{\bar z}, {\bar \eta}_{\bar z}) \neq (0,0) \, .
\end{equation}
These is easily seen by differentiating the ratio of the functions defining each domain. Simple calculations also show that the same holds for each ratio which is well-defined, i.e., it is routine to check that all the conditions \eqref{eq:hol+1}-\eqref{eq:hol-2} are also sufficient to characterize holomorphicity of the lifts in the respective domain of definition. Observe that when they hold these identity extend to the whole $\C \setminus \{ 0\}$ by unique continuation (although the corresponding ratio may not be well-defined) because $\xi$ and $\eta$ are real-analytic.
Finally, observe that \eqref{eq:hol+1}-\eqref{eq:hol+2} and \eqref{eq:hol-1}-\eqref{eq:hol-2} are equivalent when two domains of definition overlap because of the conformality condition \eqref{eq:conf}. 

In order to finish the proof of the claim we apply Lemma \ref{holomorphicfactor}. It follows from \eqref{zeroproduct} that at least one lift is holomorphic because one (actually two) among the conditions \eqref{eq:hol+1}-\eqref{eq:hol-2} is (are) satisfied. Moreover still Lemma \ref{holomorphicfactor} implies than only one lift is holomorphic whenever the harmonic map $\boldsymbol{\omega}$ is full.

Now we come to (\ref{item:hor}). Suppose $\widetilde{\boldsymbol\omega}^+$ is holomorphic. Then, in the open subset $\{ \bar{\xi}_{\bar{z}} \neq 0\}$ where the first representation of the lift $\widetilde{\boldsymbol{\omega}}^+$ is defined (analogous argument applies in the subset $\{ \eta_{\bar{z}}\neq 0\}$) we can write
\[\notag
	\widetilde{\boldsymbol{\omega}}^+ = \left[1, \frac{\bar{\xi}_{\bar{z}} \xi + \eta_{\bar{z}}\bar{\eta}}{\bar{\xi}_{\bar{z}}}, \frac{\bar{\xi}_{\bar{z}}\eta - \eta_{\bar{z}} \bar{\xi}}{\bar{\xi}_{\bar{z}}}, -\frac{\eta_{\bar{z}}}{\bar{\xi}_{\bar{z}}} \right]
\]
where the functions
\[\notag
	w_0=1 \, , \quad w_1 = \frac{\bar{\xi}_{\bar{z}} \xi + \eta_{\bar{z}}\bar{\eta}}{\bar{\xi}_{\bar{z}}} \, , \quad w_2 = \frac{\bar{\xi}_{\bar{z}}\eta - \eta_{\bar{z}} \bar{\xi}}{\bar{\xi}_{\bar{z}}} \, , \quad w_3 = -\frac{\eta_{\bar{z}}}{\bar{\xi}_{\bar{z}}} 
\]
are holomorphic. The horizontality property, previously encoded in the condition \eqref{eq:horizontality1}, reduces to
\[\notag
	(w_3)_z + w_1 (w_2)_z - w_2 (w_1)_z = 0 \, ,
\]
which is easily verified taking advantage of conformality condition \eqref{eq:conf} and the fact that, since $\boldsymbol\omega$ is harmonic, $\xi$ and $\eta$ respectively satisfy \eqref{harmxi} and \eqref{harmeta} in $\C$.

The same holds for the second representation of $\widetilde{\boldsymbol\omega}^+$ in its domain of definition. In case $\widetilde{\boldsymbol\omega}^-$ is holomorphic its horizontality property is treated in a similar way. The details are left to the reader.
\end{proof}

\begin{remark}
Although sufficient to our purposes in the present form, Proposition~\ref{liftings} actually holds without the assumption $\boldsymbol\omega(-S^{(2)}) = -S^{(4)}$. Indeed, in case $\boldsymbol\omega(-S^{(2)}) = S^{(4)}$ instead of $\xi$ and $\eta$ one can define complex functions $\zeta$ and $\chi$ using the stereographic projection from the north pole $-S^{(4)}\in \bbS^4$ (i.e., as in formulas \eqref{eq:xi-eta} but with a minus sign in the denominator) and obtain equations similar to \eqref{harmxi}-\eqref{harmeta}. Then one can construct  positive and negative lifts with formulas similar to \eqref{eq:lift-pos}-\eqref{eq:lift-neg} in terms of $\zeta$, $\chi$ so that the claims in Proposition \ref{liftings} still hold. In addition, such new lifts can be proven to be equivalent to those in terms of $\xi$, $\eta$ (hence the lifts do not depend on which stereographic projection has been chosen). For further details we refer the interested reader to \cite[pp. 77-78 and 98-100]{Fawley}.

\end{remark}

\begin{remark}
\label{notholomorphic}
It follows from the previous proposition, combined with the rigidity result in Lemma \ref{horizontallifts}, that for a linearly full map $\boldsymbol{\omega}$ the lift $\widetilde{\boldsymbol{\omega}}^-$ cannot be holomorphic when $\boldsymbol{\omega}(-S^{(2)})=-S^{(4)}$. Otherwise we would have $\mu_0 \neq 0$ from \eqref{mu-horizontal} and hence  \eqref{algebraiccurves} would imply
\[ \boldsymbol{\omega}(-S^{(2)})=-S^{(4)}= \boldsymbol\tau([1,0,0,0])=\boldsymbol\tau \circ \widetilde{\boldsymbol{\omega}}^-( -S^{(2)})=a \circ \boldsymbol{\omega}(-S^{(2)}) \, , \]
which is a contradiction.
Similarly, $\widetilde{\boldsymbol{\omega}}^+$ cannot be holomorphic when $\boldsymbol{\omega}(-S^{(2)})=S^{(4)}$. Thus, from claim (4) of the previous proposition we conclude that $\widetilde{\boldsymbol\omega}^+$ (resp.  $\widetilde{\boldsymbol\omega}^-$) is holomorphic iff $\boldsymbol\omega(-S^{(2)})=-S^{(4)}$ (resp. $\boldsymbol\omega(-S^{(2)})=S^{(4)}$). 
\end{remark}

\begin{remark}
\label{twistorlift}
As a consequence Proposition~\ref{liftings} and the previous remark we see that if $\boldsymbol{\omega} : \mathbb{S}^2 \to \mathbb{S}^4$ is a linearly full equivariant harmonic sphere such that $\boldsymbol{\omega}(-S^{(2)})=\mp S^{(4)}$ then $\Psi=\widetilde{\pm \boldsymbol{\omega}}^+: \C P^1 \to \C P^3$ is always an equivariant holomorphic horizontal curve.  In addition, $\pm \boldsymbol{\omega}=\boldsymbol\tau \circ \Psi$. We will refer to it as {\em the twistor lift} of $\boldsymbol{\omega}$ and we will denote it simply by $\widetilde{\boldsymbol\omega}$.	
\end{remark}

Combining Proposition \ref{liftings} with Lemma~\ref{horizontallifts} and Proposition~\ref{prop:explfullspheres} we finally have the classification result for linearly full equivariant harmonic spheres.

\begin{theorem}
\label{classification}	
Let $\boldsymbol\omega :\mathbb{S}^2 \to \mathbb{S}^4$ be an $\mathbb{S}^1$-equivariant harmonic sphere. If $\boldsymbol\omega$ is linearly full then it can be recovered and explicitly described in terms of the composition of the twistor fibration $\boldsymbol\tau$ in \eqref{eq:tau} with its $\mathbb{S}^1$-equivariant holomorphic and horizontal twistor lift $\widetilde{\boldsymbol\omega}: \C P^1 \to \C P^3$ defined in Remark \ref{twistorlift}.

 More explicitly, the following holds:
\begin{enumerate}
\item We have $\mp \boldsymbol\omega=\boldsymbol\tau \circ \widetilde{\boldsymbol\omega}$, where $\widetilde{\boldsymbol\omega}=\widetilde{\mp \boldsymbol\omega}^+$ is defined through \eqref{eq:lift-pos} and \eqref{eq:xi-eta}, with the minus sign if $\boldsymbol\omega(-S^{(2)})=S^{(4)}$ and the plus sign otherwise.
\item The lift $\widetilde{\boldsymbol\omega}$ is given by \eqref{algebraiccurves} with $\mu_0=1$, $\mu_1, \mu_2 \in \C \setminus \{0\}$ and $\mu_3$ as in \eqref{mu-horizontal}.
\item Up to reversing the sign as in (1), the map $\boldsymbol\omega$ is given by \eqref{eq:classification} with restriction on the parameters $\mu_1$, $\mu_2$ and $\mu_3$ as in (2).
\end{enumerate}

\end{theorem}

\section{Minimizing $\bbS^1$-equivariant tangent maps}
\label{sec:stability}
In this section, we rely on the classification results for $\bbS^1$-equivariant harmonic spheres from Sec.~\ref{sec:axisymm} and we discuss the stability and energy minimality properties of the corresponding degree-zero homogeneous extensions, the so-called {\em tangent maps}, with respect to the Dirichlet energy $\mathcal{E}_0$ in the class of $\bbS^1$-equivariant maps.  
 Since the fundamental papers by R.~Schoen and K.~Uhlenbeck \cite{SU1,SU2,SU3}, tangent maps are used to study the local behavior of minimizing harmonic maps around possible singularities. As detailed in Sec.~\ref{sec:6}, our main interest here relies in the fact that they allow as well to investigate the local behavior around any point $x \in \Omega$ of $\bbS^1$-equivariant $Q$-tensor fields minimizing $\mathcal{E}_\lambda$ in the class $\mathcal{A}^{\rm sym}_{Q_{\rm b}}(\Omega)$. Because of the extra symmetry constraint, this specific stability/minimality analysis is really necessary in the present case, since we are not allowed to apply the general results in \cite{SU3} and \cite{LiWa1} valid in the nonsymmetric setting.

As in the previous sections, we identify $\mathcal{S}_0$ with $\R\oplus \C \oplus \C$ and consider $\mathbb{S}^4 $-valued maps.  
\begin{definition}[tangent map]\label{deftangmap}
We say that a degree-zero homogeneous map $\widehat{\boldsymbol{\omega}}:\R^3\setminus\{0\}\to\mathbb{S}^4$ is an $\bbS^1$-equivariant {\em tangent map} if 
\begin{equation}
\label{harmhomextension}
\widehat{\boldsymbol{\omega}}(x) = \boldsymbol{\omega}\left( \frac{x}{\abs{x}}\right) \quad \forall x \in \bbR^3 \setminus \{0\}\,,
\end{equation}
and  $\boldsymbol{\omega} : \bbS^2 \to \bbS^4$ is an $\bbS^1$-equivariant harmonic map (see Sec.~\ref{sec:axisymm}). Correspondingly, we say that  $\widehat{\boldsymbol{\omega}}$ is the tangent map induced by the ($\mathbb{S}^1$-equivariant) harmonic sphere $\boldsymbol{\omega}:\mathbb{S}^2\to\mathbb{S}^4$. 
\end{definition}

\begin{remark} A tangent map $\widehat{\boldsymbol{\omega}}$ is smooth in $\bbR^3\setminus \{0\}$ (see Sec.~\ref{sec:axisymm}) and belongs to $W^{1,2}_{\rm loc}(\bbR^3, \bbS^4)$. It is also $\bbS^1$-equivariant and weakly harmonic in the whole space $\R^3$. Conversely, any weakly harmonic map from $\R^3$ into $\mathbb{S}^4$ which is $0$-homogeneous and $\bbS^1$-equivariant is a tangent map (i.e., the $0$-homogeneous extension of an $\bbS^1$-equivariant harmonic sphere into $\bbS^4$). 
\end{remark}

To investigate the stability of tangent maps in the class of $\mathbb{S}^1$-equivariant maps into $\mathbb{S}^4$, we need to recall the definition of the second variation of energy along admissible deformations. We proceed as follows. 
Consider  a tangent map $\widehat{\boldsymbol{\omega}} : \bbR^3\setminus\{0\} \to \bb{S}^4$.  Given a compactly supported vector field $X\in C^\infty_c(\bb{R}^3, \R \oplus \C \oplus \C)$ which is $\bb{S}^1$-equivariant (i.e., satisfying $X(Rx)=R \cdot X(x)$ with respect to \eqref{S1actionRCC} for every $R\in\mathbb{S}^1$ and every $x \in \R^3$), we can find $\varepsilon> 0$ small enough such that for every $t\in (-\varepsilon,\varepsilon)$, 
\[
	\widehat{\boldsymbol{\omega}}_t:= \frac{\widehat{\boldsymbol{\omega}}+t X}{\abs{\widehat{\boldsymbol{\omega}}+tX}} \in W^{1,2}_{\rm loc}(\mathbb{R}^3;\mathbb{S}^4)\,.
\]
Clearly, $\{\widehat{\boldsymbol{\omega}}_t\}_{t\in(-\eps,\eps)}$ is a one parameter family of $\mathbb{S}^1$-equivariant maps into $\mathbb{S}^4$, $\widehat{\boldsymbol{\omega}}_t-\widehat{\boldsymbol{\omega}}$ is compactly supported in ${\rm spt} \, X$, and $\widehat{\boldsymbol{\omega}}_0=\widehat{\boldsymbol{\omega}}$. The second variation of the Dirichlet energy $\mathcal{E}_0$ of $\widehat{\boldsymbol{\omega}}$ evaluated at $X$ is defined as 
\begin{equation}
\label{defsecondvariation}
\delta^2\mathcal{E}_0(\widehat{\boldsymbol{\omega}})[X]:=\left. \frac{d^2}{dt^2}\mathcal{E}_0(\widehat{\boldsymbol{\omega}}_t,B_\rho)\right\vert_{t=0}\,, 
\end{equation}
where the radius $\rho$ is chosen in a way that ${\rm spt}\,X\subset B_\rho$. 
 
The explicit representation of the second variation $\delta^2\mathcal{E}_0(\widehat{\boldsymbol{\omega}})$ follows from  classical computations as in \cite{SU3} (see also \cite[Chapter 1]{LiWa2}), and it is provided  in the following lemma.

\begin{lemma}
\label{secondvariation}
Let	 $\widehat{\boldsymbol{\omega}}:\R^3\setminus\{0\}\to\mathbb{S}^4$ be a tangent map. For every $\mathbb{S}^1$-equivariant vector field $X\in C^\infty_c(\R^3;\R \oplus \C \oplus \C)$, we have
  \begin{equation}\label{eq:2nd-var}
    \begin{split}    
   \delta^2\mathcal{E}_0(\widehat{\boldsymbol{\omega}})[X] &= \int_{\bb{R}^3} \left\{ \abs{\nabla X}^2+ (4 (\widehat{\boldsymbol{\omega}}\cdot X)^2 - \abs{X}^2 ) \abs{\nabla \widehat{\boldsymbol{\omega}}}^2  - \abs{\nabla(\widehat{\boldsymbol{\omega}} \cdot X)}^2 - 4 (\widehat{\boldsymbol{\omega}}\cdot X) \nabla X:\nabla \widehat{\boldsymbol{\omega}}   \right\}\,\mbox{d}x \\
    &= \int_{\bb{R}^3} \left\{ \abs{\nabla X_T}^2-    \abs{\nabla \widehat{\boldsymbol{\omega}} }^2 \abs{X_T}^2 \right\} \,\mbox{d}x\,,  
    \end{split}      
  \end{equation}
where $X_T := X - (\widehat{\boldsymbol{\omega}}\cdot X) \widehat{\boldsymbol{\omega}}$ is the tangential component of $X$ along $\widehat{\boldsymbol{\omega}}$ and $\widehat{\boldsymbol{\omega}}\cdot X$ denotes the pointwise scalar product of the two functions as elements in the underlying real vector space.
\end{lemma}

Classically, we shall say that an $\bbS^1$-equivariant tangent map $\widehat{\boldsymbol{\omega}}$ is \emph{stable} if the quadratic form $\delta^2\mathcal{E}_0(\widehat{\boldsymbol{\omega}})$ is nonnegative, i.e., 
\begin{equation}\label{eq:stability-condition}
	\delta^2\mathcal{E}_0(\widehat{\boldsymbol{\omega}})[X] \geq 0 \quad\text{for every $\mathbb{S}^1$-equivariant $X\in C^\infty_c(\R^3;\R \oplus \C \oplus \C)$.}
\end{equation}
If a tangent map $\widehat{\boldsymbol{\omega}}$ is not stable, we shall say that it is {\it unstable}. 
\vskip5pt

A stronger property than stability for a tangent map  is to be minimizing. We say that an $\bbS^1$-equivariant tangent map $\widehat{\boldsymbol{\omega}}$ is {\it locally minimizing}  if it is energy minimizing in every ball $B_\rho$, i.e., for every $\rho>0$ and every $\mathbb{S}^1$-equivariant competitor $w\in W^{1,2}(B_\rho;\mathbb{S}^4)$ such that 
${\rm spt}(\widehat{\boldsymbol{\omega}}-w)\subset B_\rho$,
\begin{equation}
\label{tm-minimality}
	\mathcal{E}_0(\widehat{\boldsymbol{\omega}},B_\rho)\leq  \mathcal{E}_0(w,B_\rho)\,.
\end{equation}
(Note that by $0$-homogeneity it is enough to consider minimality in the unit ball $B_1$.) 
\begin{remark}
\label{obviousstability}
	According to this definition, locally minimizing $\bbS^1$-equivariant tangent maps are indeed stable in the sense of \eqref{eq:stability-condition} by the second order condition for minimality.
\end{remark}
\vskip5pt

We are now ready to state the main result of this section, whose proof is the object of the following two subsections. 

\begin{theorem}\label{thm:stability}
	Let $\widehat{\boldsymbol{\omega}} : \R^3 \setminus \{0\} \to \bbS^4$ be a nonconstant $\bbS^1$-equivariant tangent map. 
	\begin{itemize}
		\item[1)] If the range of $\widehat{\boldsymbol{\omega}} $ is not contained in $\R \oplus \C \oplus \{0\} \simeq L_0\oplus L_1$, then $\widehat{\boldsymbol{\omega}}$  is unstable.
		\vskip5pt
		\item[2)] The map $\widehat{\boldsymbol{\omega}} $ is locally minimizing iff there exists $\alpha \in \R$ such that $\widehat{\boldsymbol{\omega}}(x)= \pm R_\alpha \cdot  \boldsymbol{\omega}^{(1)}_{\rm eq}\left(\frac{x}{|x|} \right)$ with $\boldsymbol{\omega}^{(1)}_{\rm eq}$ the equatorial embedding in \eqref{equatorial} and $R_\alpha \in \mathbb{S}^1$ acts as in \eqref{S1actionRCC}.
	\end{itemize}
\end{theorem}

Applying explicitly the identification $\mathcal{S}_0\simeq\R \oplus \C \oplus \C$ in Lemma~\ref{prop:Li-action}, straightforward calculations yield the following two useful corollaries of Theorem~\ref{thm:stability} for the corresponding maps $\widehat{\omega}$ into $\mathbb{S}^4 \subset \mathcal{S}_0$. 

Concerning unstable maps the following Corollary recovers the instability result from \cite[Proposition 4.7]{DMP1} for the constant norm hedgehog (see also \cite[Theorem~1.2 and Remark~1.3]{INSZ1} for a related instability result for the non-constant norm hedgehog in the entire space under symmetric perturbations).
\begin{corollary}
\label{hedgehog}
	Let $\widehat{\omega} : \mathbb{R}^3 \setminus \{0\} \to \mathbb{S}^4 \subset \mathcal{S}_0 $ be the tangent map induced by the Veronese embedding $\boldsymbol{\omega}^{(H)}$ in Remark \ref{rmk:hedgehog}, i.e., the harmonic sphere in \eqref{eq:classification} with $\mu_1=\mu_2=\sqrt{3}$. Then $\widehat{\omega}=\overline{H}$, where $\overline{H}$ is the constant norm hedgehog given in \eqref{Hbar}. As a consequence, the map $\overline{H}$ is unstable with respect to $\mathbb{S}^1$-equivariant perturbations. 
\end{corollary}

For locally minimizing tangent maps we have the following result.
\begin{corollary}
\label{formulablowups} Let $\widehat{\omega} : \mathbb{R}^3 \setminus \{0\} \to \mathbb{S}^4 \subset \mathcal{S}_0 $ be a locally minimizing tangent map. Then there exists $\alpha \in \R$ such that $\widehat{\omega}=Q^{(\alpha)}$ or $\widehat{\omega}=-Q^{(\alpha)}$,
where $Q^{(\alpha)}$ is the matrix-valued map given in \eqref{stableblowups}. 
\end{corollary}

\subsection{Instability of degree-two and linearly full tangent maps}
For the reader's convenience we restate the first part of Theorem \ref{thm:stability} in the following result.

\begin{proposition}
\label{thm:instability-full}
 A tangent map $\widehat{\boldsymbol{\omega}}$ whose range is not contained in $\R \oplus \C \oplus \{0\} $ is unstable.
\end{proposition}

The proof of this proposition relies on the classification of all $\mathbb{S}^1$-equivariant harmonic spheres into $\mathbb{S}^4$ provided by Theorem \ref{classification} and Proposition \ref{degeneratespheres}. We start with a reduction to real parameters in the representation of harmonic spheres. 

\begin{lemma}\label{instability-reduced}
	Let $\boldsymbol{\omega} :\bb{S}^2 \to \bb{S}^4$ be a nonconstant $\bb{S}^1$-equivariant harmonic map as in \eqref{eq:classification} with complex parameters $\mu_1, \mu_2 \in \C $. 
If $\boldsymbol{\omega} = (\boldsymbol{\omega}_0, \pmb{\omega}_1, \pmb{\omega}_2)$ and for $s=(s_1,s_2) \in \bb{R}^2$ we set 
	\[
		\boldsymbol{\omega}_s := (\boldsymbol{\omega} _0, e^{i s_1}\pmb{\omega}_1, e^{i s_2}\pmb{\omega}_2)\,, 
	\]
then $\boldsymbol{\omega}_{s}$ is harmonic and $\bbS^1$-equivariant, moreover the induced tangent map  $\widehat{\boldsymbol{\omega}_{s}}$ is stable if and only if $\widehat{\boldsymbol{\omega}}$ is stable. In particular, $\widehat{\boldsymbol{\omega}}$ is stable if and only if the tangent map $\widehat{\boldsymbol{\omega}_+}$ is stable, where $\boldsymbol{\omega}_+$ denotes the harmonic map with parameters  $\abs{\mu_1}$, $\abs{\mu_2}$.
\end{lemma}  

\begin{proof}
	The first statement follows immediately from \eqref{eqharmspheres} and \eqref{S1actionRCC}, noticing that $\abs{\nabla_T \boldsymbol{\omega}}^2 = \abs{\nabla_T \boldsymbol{\omega}_s}^2$. The second is a direct consequence of the first and of the second variation formula \eqref{eq:2nd-var}. Indeed, for every $\mathbb{S}^1$-equivariant test vector field $X=(X_0,X_1, X_2)$ with tangential part along $\boldsymbol{\omega}$ given by $X_T=X-(\widehat{\boldsymbol{\omega}} \cdot X) \widehat{\boldsymbol{\omega}}$, we can consider the rotated vector field $X_s=(X_0, e^{i s_1} X_1, e^{i s_2} X_2)$ and its tangential component $(X_s)_T$ along $\widehat{\boldsymbol{\omega}_s}$. It's easy to check that $|(X_s)_T|^2=|X_T|^2$, $|\nabla (X_s)_T|^2=|\nabla X_T|^2$ and 
	$$
	\delta^2\mathcal{E}_0(\widehat{\boldsymbol{\omega}})[X] = \delta^2\mathcal{E}_0(\widehat{\boldsymbol{\omega}_s})[X_s] \, .
	$$
	As a consequence, $X$ destabilizes $\widehat{\boldsymbol{\omega}}$ if and only if $X_s$ destabilizes $\widehat{\boldsymbol{\omega}_s}$, therefore the first equivalence is proved. 
	
	The last claim is now obvious, choosing $s=(s_1,s_2)$ such that $e^{i s_1}\mu_1=|\mu_1|$ and $e^{i s_2} \mu_2= |\mu_2|$ and checking that $\widehat{\boldsymbol{\omega}_s}=\widehat{\boldsymbol{\omega}_+}$ in view of  \eqref{eq:classification}.
\end{proof}

\begin{remark}
As a consequence of the previous lemma, in what follows the functions $\omega_1$, $\omega_2$ in \eqref{eq:omega} corresponding to \eqref{eq:classification} will be always considered as real-valued. 
\end{remark}

As a second step, we observe an easy consequence of stability inequality \eqref{eq:stability-condition}. 

\begin{lemma}\label{lemma:stability}
  Let $\widehat{\boldsymbol{\omega}}: \bb{R}^3 \setminus \{0\} \to \bb{S}^4 \subset \R \oplus \C \oplus \C$ be an $\bb{S}^1$-equivariant tangent map induced by a harmonic sphere $\boldsymbol{\omega}$ as in \eqref{eq:classification} with $\mu_1, \mu_2 \geq 0$. If $\widehat{\boldsymbol{\omega}}$ is stable then $\boldsymbol{\omega}$ satisfies the following inequality:
  
  \begin{equation}\label{eq:stability}
    \int_{\bb{S}^2} g^2 \abs{\nabla_T \boldsymbol{\omega}}^2\,d\vol{\bb{S}^2} \leq \int_{\bb{S}^2} \left\{ \frac{1}{4}g^2 + \abs{\partial_\theta g}^2 + \frac{g^2}{\sin^2 \theta}\right\}\,d\vol{\bb{S}^2},
  \end{equation}
for all $g \in C^1(\mathbb{S}^2)$ which depend only on the latitude $\theta$ and vanish at the poles.
\end{lemma}

\begin{proof}
  First we fix $X\in C^\infty_c(\R^3\setminus \{ 0\};\R \oplus \C \oplus \C)$ an $\mathbb{S}^1$-equivariant deformation vector field which in polar coordinates has the form
\begin{equation}
 X = \left( 0, \psi(r,\theta) i e^{i\phi}, 0 \right),
\end{equation}
with $\psi \in C^\infty_0((0, +\infty) \times (0, \pi);\R)$.

Since $\mu_1 \in \R$ then \eqref{eq:classification} yields  $\widehat{\boldsymbol{\omega}} \cdot X \equiv 0$ and from \eqref{eq:2nd-var} we have (recall that $\widehat{\boldsymbol{\omega}}$ is degree-zero homogeneous)
  
  \[
   \delta^2\mathcal{E}_0(\widehat{\boldsymbol{\omega}})[X]= \int_{\bb{R}^3}\left\{ -\frac{\psi^2}{r^2} \abs{\nabla_T \widehat{\boldsymbol{\omega}}}^2 + \abs{\partial_r \psi}^2 + \frac{1}{r^2}\left( \abs{\partial_\theta \psi}^2 + \frac{\psi^2}{\sin^2 \theta} \right) \right\}\,dx. 
  \]
We now decompose $\psi(r, \theta) = \varphi(r) g(\theta)$, with $\varphi \in C^\infty_c((0,+\infty))$ and $g \in C^\infty_0((0,\pi))$, hence the stability property \eqref{eq:stability-condition} yields for any $\varphi$ and $g$ the inequality
 \[ 
  \delta^2\mathcal{E}_0(\widehat{\boldsymbol{\omega}})[X] = \int_{\bb{R}^3}\left\{ -\frac{\varphi^2 g^2}{r^2} \abs{\nabla_T \widehat{\boldsymbol{\omega}}}^2 + g^2\abs{\partial_r \varphi}^2 + \frac{\varphi^2}{r^2}\left( \abs{\partial_\theta g}^2 + \frac{g^2}{\sin^2 \theta} \right) \right\}\,dx
  \geq 0 \, .\]
Now optimize with respect to $\varphi$ using the sharp Hardy inequality. Integrating with respect to $r\in (0,\infty)$ and on $\mathbb{S}^2$ separately, we conclude that for any $g \in C^\infty_0((0,\pi))$ we have
  
  \[
   \int_{\bb{S}^2} \left\{ -g^2 \abs{\nabla_T \boldsymbol{\omega}}^2 + \frac{1}{4}g^2 + \abs{\partial_\theta g}^2 + \frac{g^2}{\sin^2 \theta} \right\}\,d\vol{\bb{S}^2} \geq 0 \, .
  \]
Finally, by a standard approximation argument in the $C^1$-norm on the function $g$ the previous inequality yields \eqref{eq:stability}.
 \end{proof}

The last lemma we need is a straightforward consequence of a general fact about harmonic maps from $\bb{S}^2$ into spheres together with $\bbS^1$-equivariance.

\begin{lemma}
\label{extraidentity}
	Let $\boldsymbol{\omega}: \bb{S}^2 \to \bb{S}^4 \subset \R \oplus \C \oplus \C$ be an $\bb{S}^1$-equivariant harmonic map as in \eqref{eq:classification}. Let $\mu_1, \mu_2 \geq 0$ and write $\boldsymbol{\omega}$ in the form $\boldsymbol{\omega} = (\omega_0(\theta), \omega_1(\theta) e^{i\phi}, \omega_2(\theta) e^{2 i \phi})$ as in \eqref{eq:omega}, with $\omega_1, \omega_2$ real-valued. Then,
	
	\begin{equation}\label{eq:omega2-identity}
		\int_{\bb{S}^2} \omega_2^2\abs{\nabla_T \boldsymbol{\omega}}^2 \,d\vol{\bb{S}^2} = \int_{\bb{S}^2} \left\{ \abs{\partial_\theta \omega_2}^2 + \frac{4 \omega_2^2}{\sin^2\theta}\right\}\,d\vol{\bb{S}^2}.
	\end{equation}
\end{lemma}

\begin{proof}
	Suppose $u : \bb{S}^2 \to \bb{S}^d\subset \R^{d+1}$ is harmonic. Let $(e_i)_{i=0}^d$ be an orthonormal basis for $\bb{R}^{d+1}$ and $u_i$ the components of $u$ with respect to this basis, so that $u = \sum_{i}u_i e_i$ and $u$ solves $\Delta_{T} u+ |\nabla_{T} u|^2 u=0$, i.e.,  $\Delta_{T} u_i = -\abs{\nabla_{T} u}^2 u_i$ for each $i=0, \ldots , d$.  
	
	Since each $u_i$ is a smooth function, evaluating $\Delta_{T} u_i^2$ we have
	\[
		\frac{1}{2}\Delta_{T} u_i^2 = u_i \Delta_{T} u_i + \abs{\nabla_{T} u_i}^2 \, ,
	\] 	
	whence the harmonic map equation together with the divergence theorem yield
	\[
		\int_{\bb{S}^2} u_i^2 \abs{\nabla_{T} u}^2\,\mbox{d}\vol{\bb{S}^2} = \int_{\bb{S}^2} \abs{\nabla_{T} u_i}^2 \,\mbox{d}\vol{\bb{S}^2}.
	\] 
Taking $d=4$ and $u=\boldsymbol{\omega}$ as in equation \eqref{eq:omega},  identity \eqref{eq:omega2-identity} follows by summing the last equality over  $i \in \{ 3,4 \}$.
\end{proof}

We can finally prove the main result of this subsection.

\begin{proof}[Proof of Proposition  \ref{thm:instability-full}]
As a consequence of Lemma \ref{instability-reduced} it is enough to prove the claim when the harmonic sphere $\boldsymbol{\omega}$ has parameters $\mu_1,\mu_2 \geq 0$.
 In view of its degree-zero homogeneity we can write $\widehat{\boldsymbol{\omega}}$ in the form $\widehat{\boldsymbol{\omega}}(\theta,\phi) = \left( \omega_0(\theta), \omega_1(\theta) e^{i\phi}, \omega_2(\theta) e^{2 i \phi} \right) = \boldsymbol{\omega}(\theta,\phi) $, i.e., with $\boldsymbol{\omega}$ as in \eqref{eq:omega} and with ${\omega}_2 \not\equiv 0$. Suppose, for a contradiction, that $\widehat{\boldsymbol{\omega}}$ is stable. Observe that, because of the explicit formulae \eqref{eq:classification}, the function $\omega_2$ satisfies the same hypotheses as $g$ in the statement of Lemma~\ref{lemma:stability}, so that it can be plugged into \eqref{eq:stability} to obtain
  
  \begin{equation}\label{eq:stability-1}
    \int_{\bb{S}^2} \omega_2^2 \abs{\nabla_T \boldsymbol{\omega}}^2\,d\vol{\bb{S}^2} \leq \int_{\bb{S}^2}\left\{ \frac{1}{4} \omega_2^2 + \abs{\partial_\theta \omega_2}^2 + \frac{\omega_2^2}{\sin^2\theta} \right\}\,d\vol{\bb{S}^2}.
  \end{equation} 
On the other hand, using \eqref{eq:omega2-identity} for the left hand side and comparing to \eqref{eq:stability-1}, we see that stability of $\widehat{\boldsymbol{\omega}}$ implies
  
  \[
   \int_{\bb{S}^2} \frac{3}{\sin^2\theta} \omega_2^2 \,d\vol{\bb{S}^2}\leq \int_{\bb{S}^2} \frac{1}{4} \omega_2^2 \,d\vol{\bb{S}^2}\,,
  \]
which is clearly impossible because $\omega_2 \not \equiv 0$. Thus $\widehat{\boldsymbol{\omega}}$ cannot be stable and the proof is complete.
\end{proof}

\subsection{Minimality of degree-one tangent maps}
In this subsection we complete the proof of Theorem~\ref{thm:stability} by establishing claim 2).
We first prove that the equator map $\widehat{\boldsymbol{\omega}}_\text{eq} (x)= \left( \frac{x}{\abs{x}}, 0\right)$ is locally minimizing with respect to compactly supported perturbations, whence the same clearly holds also for the rotated maps $\pm R_\alpha \cdot \widehat{\boldsymbol{\omega}}_\text{eq} $ because of the invariance of $\mathcal{E}_0$ under isometries on $\bbS^4$.

\begin{proposition}[Equivariant minimality of the equator map]\label{thm:min-x/x}
   Let $\widehat{\boldsymbol{\omega}}  : \bb{R}^3 \setminus \{0\} \to \bb{S}^4 $ be the tangent map defined by $\widehat{\boldsymbol{\omega}} (x)=\boldsymbol{\omega}^{(1)}_{\rm eq}\left(\frac{x}{|x|} \right)$, with $\boldsymbol{\omega}^{(1)}_{\rm eq}$ the equatorial embedding in \eqref{equatorial}. Then  $\widehat{\boldsymbol{\omega}}$ is locally minimizing with respect to compactly supported $\bbS^1$-equivariant perturbations. 
\end{proposition}

\begin{proof}
By the degree-zero homogeneity of $\widehat{\boldsymbol{\omega}}$, it suffices to prove the minimizing property in $B_1$. Let $u \in W^{1,2}(B_1; \bb{S}^4)$ be an $\bb{S}^1$-equivariant map such that $u \vert_{\bb{S}^2} = \widehat{\boldsymbol{\omega}} \vert_{\bb{S}^2}= \boldsymbol{\omega}^{(1)}_{\rm eq} $. 
Write 
  
  \[
   u = ( u_0, u_1 e^{i\phi}, u_2 e^{2i\phi}) \in \R \oplus  \C \oplus \C\,, 
  \] 
where in polar coordinates $u_i = u_i(r,\theta)$, $i=0,1,2$, because of equivariance; here $u_0$ is real-valued while $u_1$, $u_2$ are generally complex-valued. Moreover functions $u_i$ are in $W^{1,2}_{\rm{loc}}$ away from the symmetry axis with 
\begin{equation}
\label{polarintegrability}
\int_{B_1} \abs{\nabla u}^2 \, dx = \int_{B_1} \abs{\partial_r u}^2 + \frac{1}{r^2}\abs{\partial_\theta u}^2 + \frac{\abs{u_1}^2 + \abs{4u_2}^2}{r^2 \sin^2\theta} \, dx < \infty \, .
  \end{equation}
It follows easily from the previous relation together with the 1-d Sobolev embedding in the variable $\theta$ that for a.e. radius $r \in (0,1)$ the functions $u_1$, $u_2$ are continuous on the sphere $\{|x|=r\}$ and vanish at poles, i.e, at $\theta=0$ and $\theta=\pi$. In addition, the boundary condition on $\partial B_1=\mathbb{S}^2$ implies that $u_2=0$, while $u_1=\sin\theta$ and $u_0=\cos\theta$ respectively, in the sense of traces on $\bb{S}^2$.
  
  Now we use a trick similar to the one exploited in the proof of \cite[Theorem 1.3]{INSZ2}, that is, from $u$ we construct the auxiliary comparison map
  
  \[
   \tilde u = (\tilde{u}_0, \tilde{u}_1,\tilde{u}_2) =\left(u_0, \sqrt{\abs{u_1}^2 + \abs{u_2}^2}e^{i\phi}, 0\right).
  \]
In view of \eqref{polarintegrability} it is routine to check that $\tilde u \in W^{1,2}(B_1; \bb{S}^4)$, $\tilde u$ is $\bb{S}^1$-equivariant and $\tilde u \vert_{\bb{S}^2} = \boldsymbol{\omega}_\text{eq}^{(1)} $. 

We claim that $ \mathcal{E}_0(\tilde u) \leq \mathcal{E}_0(u).$
  Indeed, a simple calculation gives
  
  \[
  \begin{split}
   \abs{\nabla \tilde{u}}^2 &= \abs{\pd{\tilde{u}_0}{r}}^2 + \frac{1}{r^2}\abs{\pd{\tilde{u}_0}{\theta}}^2 + \abs{\pd{\tilde u_1}{r}}^2  + \frac{1}{r^2}\abs{\pd{\tilde u_1}{\theta}}^2 + \frac{1}{r^2 \sin^2\theta}\abs{\pd{\tilde u_1}{\phi}}^2 \\
   &= \abs{\nabla_{r,\theta} {u_0}}^2 + \abs{\nabla_{r,\theta} \tilde{u}_1}^2 + \frac{\abs{u_1}^2 + \abs{u_2}^2}{r^2 \sin^2\theta},
   \end{split}
  \]
where $\abs{\nabla_{r,\theta} u_0}^2$, $\abs{\nabla_{r,\theta} \tilde u_1}^2$ are defined in an obvious way from the first line above. Now, note that
  
  \[
   \abs{\nabla_{r,\theta} \tilde u_1}^2 = \frac{1}{\abs{u_1}^2 + \abs{u_2}^2}\abs{u_1\nabla_{r,\theta} u_1 + u_2 \nabla_{r,\theta} u_2 }^2\leq \abs{\nabla_{r,\theta} u_1}^2 + \abs{\nabla_{r,\theta} u_2}^2.
  \]
To conclude that $\mathcal{E}_0(\tilde u) \leq \mathcal{E}_0(u)$, it now suffices to note that the previous relations yield 
  
  \[
   \abs{\nabla \tilde u}^2 \leq \abs{\nabla_{r,\theta}u_0}^2 + \abs{\nabla_{r,\theta}u_1}^2 + \abs{\nabla_{r,\theta}u_2}^2 + \frac{\abs{u_1}^2 + 4\abs{u_2}^2}{r^2 \sin^2\theta} = \abs{\nabla{u}}^2.
  \]
Thus, $\tilde u$ can be regarded as a map in $W^{1,2}(B_1; \bb{S}^2)$ that coincides with $\boldsymbol{\omega}_\text{eq}^{(1)}$ on $\partial B_1 = \bb{S}^2$ and having lower energy than $u$, as claimed. Since $\widehat{\boldsymbol{\omega}}_\text{eq}^{(1)}$ is minimizing among the maps in $W^{1,2}(B_1; \bb{S}^2)$ subject to its own boundary condition (see \cite[Theorem 7.3]{BCL}), we have $
   \mathcal{E}_0(\widehat{\boldsymbol\omega}_\text{eq}) \leq \mathcal{E}_0(\tilde{u}) \leq \mathcal{E}_0(u), 
  $ 
and this concludes the proof.    
\end{proof}
In order to conclude the proof of Theorem \ref{thm:stability} we need the converse of Proposition~\ref{thm:min-x/x}, i.e., we need to show that any locally minimizing tangent map $\widehat{\boldsymbol{\omega}}$ has actually the form $\pm R_\alpha \cdot  \boldsymbol{\omega}^{(1)}_{\rm eq}\left(\frac{x}{|x|} \right) $ for some $\alpha \in \R$.
 To see this, we first recall that in view of Remark \ref{obviousstability} any locally minimizing tangent map is stable, therefore the first part of Theorem \ref{thm:stability} yields $\widehat{\boldsymbol{\omega}}_2 \equiv 0$, hence the corresponding harmonic sphere $\boldsymbol{\omega}$ is linearly degenerate. It follows from Proposition \ref{degeneratespheres} that 
 \[
 \widehat{\boldsymbol{\omega}}(x)= \pm    \boldsymbol{\omega}^{(1)}_{\rm eq} \circ \boldsymbol{\sigma}_2^{-1}\left( \mu_1 \boldsymbol{\sigma}_2 \left(\frac{x}{|x|} \right) \right) \, ,
 \]
 for some $\mu_1 \in \C \setminus \{0\}$ and by assumption it is locally minimizing among equivariant $\mathbb{S}^4$-valued perturbations. Writing $\mu_1=  \delta e^{i \alpha}$, with $\alpha \in \R$ and $\delta>0$, and in view of the $\mathbb{S}^1$-equivariance of both $\boldsymbol{\sigma}_2$ and $\boldsymbol{\omega}^{(1)}_{\rm eq}$ with respect to the rotation $R_\alpha=e^{i\alpha}$, it remains to prove that $\delta=1$. 
 
  It follows from the previous argument that the map $v(x) =\boldsymbol{\sigma}_2^{-1}\left( \delta \, \boldsymbol{\sigma}_2 \left(\frac{x}{|x|} \right) \right) $
 is locally minimizing among compactly supported symmetric perturbations in $W^{1,2}_{\rm{loc}}(\mathbb{R}^3;\mathbb{S}^2)$ and it remains to infer that, because of local minimality, $\delta=1$. This is a classical argument from \cite[Theorem~7.3]{BCL} which requires minor modifications because of symmetry. Following \cite[Chapter~2, p.~21]{Simon}, $v$ is stationary with respect to equivariant {\em inner variations}, i.e., for any $\mathbb{S}^1$-equivariant vector field $\Phi \in C^\infty_0(B_1;\mathbb{R}^3)$ we have
 \[
 0=\left.\frac{d}{dt} \mathcal{E}_0\left( v\circ({\rm id} + t \Phi) \right) \right\vert_{t=0}  = \int_{B_1} |\nabla v|^2 {\rm div} \, \Phi - 2 \left((\nabla v)^\trans \nabla v \right) : \nabla \Phi \, dx  \, . \]
 Choosing admissible vector fields $\Phi = \varphi e_3$, with $e_3=(0,0,1)$ and $\varphi \in C^\infty_0(B_1)$ a radial function, we have $\nabla v  (\nabla \Phi)^\trans \equiv 0$, because $v$ is degree-zero homogeneous and, if $\varphi=\varphi_n$ is a  sequence of radial functions increasing to $\chi_{B_1}$ and bounded in $W^{1,1}$, we have ${\rm div} \, \Phi_n \overset{*}{\rightharpoonup} x_3 \mathcal{H}^2 \LL \mathbb{S}^2$ as measures in $\R^3$. Passing to the limit in the previous equality we obtain
 \begin{equation}\label{eq:CoM}
 	\int_{\mathbb{S}^2} x_3 |\nabla v|^2 \,d\vol{\bb{S}^2} =0 \, ,
 \end{equation}
and the conclusion $\delta =1$ follows exactly from \cite[p. 678]{BCL} (see also \cite[p. 125]{HKL}) because the energy measure $|\nabla v|^2 \, d\vol{\bb{S}^2}$ has barycenter at the origin if and only if $\delta=1$ (note that the analogues of \eqref{eq:CoM} in which $x_3$ is replaced by $x_1$, $x_2$ are obvious because of the invariance of $\abs{\nabla v}^2$ under the $\bbS^1$-action).


\section{Compactness of minimizing  $\mathbb{S}^1$-equivariant maps}
\label{sec:comp}


In this section, we discuss compactness properties for equivariant minimimizers of the Landau-de Gennes energy \eqref{LDGenergytilde} both in the interior and near the boundary. Such results will be used both in the next section for the proof of the partial regularity Theorem \ref{thm:partial-regularity} and in the final section of the paper to obtain existence of minimizing torus and split solutions to equations \eqref{MasterEq} as described in Theorem \ref{thm:examples-tori} and Theorem \ref{thm:examples-split} respectively. The results presented here are the natural counterpart in the LdG case of the Luckhaus compactness Theorem for harmonic and $p$-harmonic maps established in the in the influential paper \cite{Lu}. As in the harmonic map case, the key technical step is the construction of  comparison maps by a gluing argument, in the spirit of the Luckhaus interpolation Lemma from the reference above, to exploit the local minimality property and turn it into a compactness one. Here, inspired by a similar construction in \cite[Proof of Theorem 4.2]{HKL} for axially symmetric maps into $\mathbb{S}^2$,  we give a simple self-contained construction of equivariant competitors both in the interior and near the boundary which is well-suited for our case. We refer the interested readers to \cite[Lemma 4.4]{Gastel} for a similar but much more complicated construction of equivariant competitors in the interior in the context of minimizing harmonic maps between Riemannian manifolds equivariant with respect to fairly general group actions.

\subsection{Local compactness}

\begin{theorem}\label{compintthm}
Let $\{\lambda_n\}\subset[0,\infty)$ be such that $\lambda_n\to\lambda_*\in[0,\infty)$, and $\{Q_n\}\subset W^{1,2}_{\rm sym}(B_1;\mathbb{S}^4)$. Assume that each $Q_n$ is minimizing $\mathcal{E}_{\lambda_n}(\cdot, B_1)$ among all $Q\in  W^{1,2}_{\rm sym}(B_1;\mathbb{S}^4)$ such that $Q=Q_n$ on  $\partial B_1$. 
If $\sup_n \mathcal{E}_{\lambda_n}(Q_n, B_1)<\infty$, then there exist a (not relabeled) subsequence and a map $Q_*\in  W^{1,2}_{\rm sym}(B_1;\mathbb{S}^4)$ such that 
$Q_n\rightharpoonup Q_*$ weakly in $W^{1,2}(B_1)$ and $Q_n\to Q_*$ strongly in  $W_{\rm loc}^{1,2}(B_1)$. In addition, $Q_*$ is minimizing $\mathcal{E}_{\lambda_*}(\cdot, B_1)$ among all $Q\in  W^{1,2}_{\rm sym}(B_1;\mathbb{S}^4)$ such that $Q=Q_*$ on  $\partial B_1$.
\end{theorem}

\begin{proof}
By assumption, the sequence  $\{Q_n\}$ is bounded in $W^{1,2}(B_1)$. Hence we can find a subsequence and a map $Q_*$ such that $Q_n\rightharpoonup Q_*$ weakly in $W^{1,2}(B_1)$. By the compact embedding of  $W^{1,2}(B_1)$ into $L^2(B_1)$, we have $Q_n\to Q_*$ strongly in $L^2(B_1)$, and extracting a further subsequence, $Q_n\to Q_*$ a.e. in $B_1$. Moreover, by dominated convergence, we then have 
\begin{equation}\label{convpotcomphm}
W(Q_n)\to W(Q_*)\quad\text{in $L^1(B_1)$}\,.
\end{equation} 
As a consequence of the pointwise convergence we also deduce that the $\mathbb{S}^4$-constraint and the $\mathbb{S}^1$-equivariance property are weakly closed, hence $Q_*\in  W^{1,2}_{\rm sym}(B_1;\mathbb{S}^4)$.

In what follows, we shall first prove the local strong convergence of the sequence $\{Q_n\}$ in $W^{1,2}$, and then the minimality of $Q_*$. We start fixing an arbitrary parameter $\delta\in(0,1/2)$ and a competitor $\bar Q\in W_{\rm sym}^{1,2}(B_1;\mathbb{S}^4)$ such that ${\rm spt}(\bar Q-Q_*)\subset  B_{1-\delta}$. Extracting another subsequence if necessary, by Fatou's lemma and Fubini's theorem, we can find a radius $\rho\in(1-\delta,1)$ such that 
\begin{equation}\label{esti1compint}
\lim_{n\to\infty}\int_{\partial B_\rho}|Q_n-Q_*|^2\,d\mathcal{H}^2=0 \, \, \text{ and } \, \int_{\partial B_\rho}\big(|\nabla_{\rm tan} Q_n|^2+|\nabla_{\rm tan} Q_*|^2\big)\,d\mathcal{H}^2\leq C\,,
\end{equation}
for a constant $C$ independent of $n$. Set $D_\rho:=B_r\cap\{x_2=0\}$ to be the disc of radius $\rho$ centered at the origin and lying in the vertical plane $\{x_2=0\}$.  Using spherical coordinates $(r,\phi,\theta)$ on $B_1$,  we infer from \eqref{esti1compint} and the $\mathbb{S}^1$-equivariance of $Q_n$ and $Q_*$ that those maps belongs to the weighted Sobolev space $W^{1,2}$ over $\partial D_\rho\setminus\{x_1=0\}$ with respect to the weight $|x_1|$. Moreover, since by equivariance the integrands in \eqref{esti1compint} do not depend on $\phi$, we have a uniform bound on the sequence
$$\int_{\partial D_\rho}|x_1|\big(|\nabla_{\rm \mathbb{S}^1} Q_n|^2+|\nabla_{\rm \mathbb{S}^1} Q_*|^2\big)\,d\mathcal{H}^1\leq C\,,$$
where $\nabla_{\rm \mathbb{S}^1}=\partial_\theta$ stands for the tangential derivative along the circle $\partial D_\rho$. In particular, $Q_n$ and $Q_*$ are absolutely continuous on $\partial D_\rho\setminus\{x_1=0\}$. 

We now consider the sequence $\sigma_n:=\|Q_n-Q_*\|^{1/3}_{L^2(\partial B_\rho)}+2^{-n}\to 0$.  
Using the absolute continuity  of $Q_n$ and $Q_*$ on $\partial D_\rho\setminus\{x_1=0\}$, we estimate by 1d-calculus and Cauchy-Schwarz inequality, 
\begin{multline*}
\sup_{\partial D_\rho\cap\big\{|x_1|\geq \frac{\rho\sigma_n}{2}\big\}} |x_1|^2|Q_n-Q_*|^2\leq \\
C\left(\int_{\partial D_\rho}|x_1||\nabla_{\rm \mathbb{S}^1}(Q_n-Q_*)|^2\,d\mathcal{H}^1\right)^{\frac{1}{2}}\left(\int_{\partial D_\rho}|x_1||Q_n-Q_*|^2\,d\mathcal{H}^1\right)^{\frac{1}{2}}\\
+C\int_{\partial D_\rho}|x_1||Q_n-Q_*|^2\,d\mathcal{H}^1\,.
\end{multline*}
Still by $\mathbb{S}^1$-equivariance, we deduce from \eqref{esti1compint} that 
\begin{multline}\label{reestiunifcompint}
\sup_{\partial B_\rho\cap\big\{(x^2_1+x_2^2)^{1/2}\geq \frac{\rho\sigma_n}{2}\big\}} |Q_n-Q_*|^2\leq \\
 C\sigma_n^{-2}\left(\int_{\partial B_\rho}|\nabla_{\rm tan}(Q_n-Q_*)|^2\,d\mathcal{H}^2\right)^{\frac{1}{2}}\left(\int_{\partial B_\rho}|Q_n-Q_*|^2\,d\mathcal{H}^2\right)^{\frac{1}{2}}\\
+C\sigma_n^{-2}\int_{\partial B_\rho}|Q_n-Q_*|^2\,d\mathcal{H}^2\leq C\sigma_n\,.
\end{multline}
We now introduce the subsets
$$T^\pm_n:=\Big\{x\in B_\rho\setminus \overline{B_{\rho(1-\sigma_n)}}: (x_1^2+x_2^2)^{1/2}<\frac{\sigma_n}{2} |x|\,,\;\pm x_3>0\Big\}\,, $$
$$L^\pm_n:=\Big\{x\in B_\rho\setminus \overline{B_{\rho(1-\sigma_n)}}: (x_1^2+x_2^2)^{1/2}=\frac{\sigma_n}{2} |x|\,,\;\pm x_3>0\Big\}\,, $$
and
$$A_n:= \big(B_\rho\setminus \overline{B_{\rho(1-\sigma_n)}}\big)\setminus\big(T^+_n\cup T^-_n \big)\,.$$
We define for $x\in A_n$, 
 $$v_n(x):= Q_*\Big(\rho \frac{x}{|x|}\Big)+ \frac{|x|-\rho(1-\sigma_n)}{\rho\sigma_n}\Big(Q_n\Big(\rho \frac{x}{|x|}\Big)-Q_*\Big(\rho \frac{x}{|x|}\Big)\Big)\,.$$
 Then we have by \eqref{esti1compint}, 
 \begin{equation}\label{aslusid1307}
 \int_{A_n}|\nabla v_n|^2\,dx\leq  C\sigma_n \int_{\partial B_\rho}\big(|\nabla_{\rm tan} Q_n|^2+|\nabla_{\rm tan} Q_*|^2\big)\,d\mathcal{H}^2+C\sigma_n^{-1}\int_{\partial B_\rho}|Q_n-Q_*|^2\,d\mathcal{H}^2\leq C\sigma_n\,,
 \end{equation}
 and by \eqref{reestiunifcompint}, 
 \begin{equation}\label{supesticompthmint}
 \sup_{x\in A_n} {\rm dist}^2(v_n(x),\mathbb{S}^4)\leq C\sigma_n\,. 
 \end{equation}
Using the equivariance of $v_n$ and the fact that $|Q_n|=|Q_*|=1$, we have $|\nabla_{\rm tan}v_n|^2\leq C \sigma_n^{-2}$, $\mathcal{H}^2(L_n^\pm) \leq C \sigma_n^2$ and finally  
 \begin{equation}\label{controllateralcompthmint}
 \int_{L_n^\pm}|\nabla_{\rm tan}v_n|^2\,d\mathcal{H}^2\leq  C\sigma_n^{-2}\mathcal{H}^2(L_n^\pm)\leq C\,.
 \end{equation}
 Then we set for $x\in\partial T_n^\pm$, 
 $$w_n(x):=\begin{cases} 
 Q_n(x) & \text{if $x\in \partial T_n^\pm\cap\partial B_\rho$}\,,\\[5pt]
\displaystyle Q_*\left(\frac{x}{1-\sigma_n}\right) & \text{if $x\in \partial T_n^\pm\cap\partial B_{\rho(1-\sigma_n)}$}\,,\\[5pt]
 v_n(x) & \text{if $x\in L_n^{\pm}$}\,,
 \end{cases}
 $$
 and we extend $w_n$ to $T_n^\pm$ by $0$-homogeneity with respect to the center point $a_n^\pm:=(0,0,\pm\rho(1-\frac{\sigma_n}{2}))$, i.e., 
 $$w_n(x)=w_n\left(\frac{x-a_n^\pm}{|x-a_n^\pm|}\right) \quad\text{for $x\in T_n^\pm\setminus\{a_n^\pm\}$}\,.$$
 Combining \eqref{esti1compint} and \eqref{controllateralcompthmint}, we derive that
\begin{multline}\label{samediasid1305}
\int_{T_n^\pm}|\nabla w_n|^2\,dx\leq C\sigma_n\int_{\partial T_n^\pm}|\nabla_{\rm tan} w_n|^2\,dx\\ \leq
 C\sigma_n\int_{\partial B_\rho}\big(|\nabla_{\rm tan} Q_n|^2+|\nabla_{\rm tan} Q_*|^2\big)\,d\mathcal{H}^2+C\sigma_n\leq C\sigma_n\,.
 \end{multline}
 Finally, we extend $w_n$ to whole annulus $B_\rho\setminus \overline{B_{\rho(1-\sigma_n)}}$ by setting 
 $$w_n(x):=v_n(x) \quad\text{for $x\in A_n$}\,.$$  
 By construction, 
 $w_n\in W^{1,2}_{\rm sym}(B_\rho\setminus \overline{B_{\rho(1-\sigma_n)}};\mathcal{S}_0)$, and we infer from \eqref{aslusid1307}, \eqref{supesticompthmint}, and  \eqref{samediasid1305} that 
 \begin{equation}\label{contrwncomthmint}
 \int_{B_\rho\setminus \overline{B_{\rho(1-\sigma_n)}}}|\nabla w_n|^2\,dx\leq 
 C\sigma_n \,,
 \end{equation}
 and
 $$ \sup_{x\in B_\rho\setminus \overline{B_{\rho(1-\sigma_n)}}} {\rm dist}^2(w_n(x),\mathbb{S}^4)\leq C\sigma_n \,.$$
 As a consequence, $|w_n|>1/2$ for $n$ large enough, and we can define a competitor $\bar Q_n\in W^{1,2}_{\rm sym}(B_1;\mathbb{S}^4)$ by setting 
 \begin{equation}
 \label{intcompetitors}
 \bar Q_n(x):=\begin{cases}
 \displaystyle \bar Q\Big(\frac{x}{1-\sigma_n}\Big) & \text{if $|x|\leq \rho(1-\sigma_n)$}\,,\\[8pt]
\displaystyle \frac{w_n(x)}{|w_n(x)|} & \text{if $\rho(1-\sigma_n)< |x|< \rho$}\,,\\[8pt]
Q_n(x) & \text{if $|x|\geq \rho$}\,. 
 \end{cases}
 \end{equation}
 Note that each map $\bar{Q}_n$ in \eqref{intcompetitors} is equivariant and $W^{1,2}$ because we are gluig together equivariant maps which are $W^{1,2}$ on each subdomain and on the spheres $\{ |x|=\rho\}$ and $\{|x|= \rho(1-\sigma_n)\}  $ the traces on both sides agree by construction of $w_n$.
 
 By minimality of $Q_n$, we have $\mathcal{E}_{\lambda_n}(Q_n,B_1)\leq \mathcal{E}_{\lambda_n}(\bar Q_n,B_1)$, which reduces to 
 \begin{equation}\label{compthmcmpint}
 \mathcal{E}_{\lambda_n}(Q_n,B_\rho)\leq \mathcal{E}_{\lambda_n}(\bar Q_n,B_\rho) \,,
 \end{equation}
 since $\bar Q_n=Q_n$ in $B_1 \setminus B_\rho$.
 By \eqref{convpotcomphm} and dominated convergence, we have 
 $$\lim_{n\to\infty}\int_{B_\rho} \lambda_nW(Q_n)\,dx=\int_{B_\rho} \lambda_*W(Q_*)\,dx\; \, \text{ and } \, \;\lim_{n\to\infty}\int_{B_\rho} \lambda_nW(\bar Q_n)\,dx=\int_{B_\rho} \lambda_*W(\bar Q)\,dx\,.$$
 On the other hand, in view of \eqref{contrwncomthmint} we have 
 \begin{multline*}
 \int_{B_\rho}|\nabla \bar Q_n|^2\,dx=(1-\sigma_n) \int_{B_\rho}|\nabla \bar Q|^2\,dx+\int_{B_\rho\setminus B_{\rho(1-\sigma_n)}}|\nabla \bar Q_n|^2\,dx\\
 \leq (1-\sigma_n) \int_{B_\rho}|\nabla \bar Q|^2\,dx+C\int_{B_\rho\setminus B_{\rho(1-\sigma_n)}}|\nabla w_n|^2\,dx \leq (1-\sigma_n) \int_{B_\rho}|\nabla \bar Q|^2\,dx+C\sigma_n\,,
 \end{multline*}
 where we have used that $|w_n|>1/2$ in the first inequality. Therefore, $ \mathcal{E}_{\lambda_n}(\bar Q_n,B_\rho)\to \mathcal{E}_{\lambda_*}(\bar Q,B_\rho) $.  
 By lower semicontinuity of the Dirichlet energy $\mathcal{E}_0(\cdot,B_\rho)$, we infer from \eqref{compthmcmpint} that
 \begin{multline}\label{chainineqcompthmint}
 \mathcal{E}_{\lambda_*}(Q_*,B_\rho)\leq \liminf_{n\to+\infty}\mathcal{E}_0(Q_n,B_\rho)+\int_{B_\rho}\lambda_*W(Q_*)\,dx\\
 \leq\limsup_{n\to+\infty}\mathcal{E}_0(Q_n,B_\rho)+\int_{B_\rho}\lambda_*W(Q_*)\,dx\leq \mathcal{E}_{\lambda_*}(\bar Q,B_\rho) \,.
 \end{multline}
 Since $\bar Q=Q_*$ in $B_1\setminus B_\rho$, we  deduce that 
\begin{equation}\label{almradminimcompthmint}
\mathcal{E}_{\lambda_*}(Q_*,B_1)\leq \mathcal{E}_{\lambda_*}(\bar Q,B_1)\,.
\end{equation}
 In view of the arbitrariness of $\bar Q$, we may have chosen $\bar Q=Q_*$, in which case \eqref{chainineqcompthmint} shows that  $\mathcal{E}_0(Q_n,B_\rho) \to \mathcal{E}_0(Q_*,B_\rho)$. Combined with the weak convergence in $W^{1,2}$, it leads to the strong $W^{1,2}$-convergence of $Q_n$  toward $Q_*$ in $B_\rho$, and hence in $B_{1-\delta}$. 
 
 By arbitrariness of $\delta$, we have thus shown that $Q_n\to Q_*$ strongly in $W^{1,2}_{\rm loc}(B_1)$, and  \eqref{almradminimcompthmint} holds for every $\bar Q\in W_{\rm sym}^{1,2}(B_1;\mathbb{S}^4)$ such that ${\rm spt}(\bar Q-Q_*)\subset B_1$. 
 
 Finally, in order to prove the minimality of $Q_*$ with respect to its own boundary condition, we now consider an arbitrary $\bar{Q}\in W_{\rm sym}^{1,2}(B_1;\mathbb{S}^4)$ such that $\bar{Q}=Q_*$ on $\partial B_1$. For $\eps\in(0,1/4)$, we set 
 $$\bar{Q}_\eps(x):=\begin{cases}
 Q_*(x) & \text{if $1-\eps\leq |x|\leq 1$}\,,\\[8pt]
\displaystyle Q_*\Big((2-2\eps-|x|)\frac{x}{|x|}\Big) & \text{if $1-2\eps\leq |x|\leq 1-\eps$}\,,\\
\displaystyle \bar{Q}\Big(\frac{x}{1-2\eps}\Big) & \text{if $|x|\leq 1-2\eps$}\,. 
 \end{cases}$$
 Straightforward computations yield
 $$(1-2\eps)\mathcal{E}_{(1-2\eps)^2\lambda_*}(\bar{Q},B_1)\leq \mathcal{E}_{\lambda_*}(\bar{Q}_\eps,B_1)\leq  \mathcal{E}_{\lambda_*}(\bar{Q},B_1)+C \mathcal{E}_{\lambda_*}(Q_*,B_1\setminus B_{1-\eps})\,,$$
 so that $\mathcal{E}_{\lambda_*}(\bar{Q}_\eps,B_1)\to \mathcal{E}_{\lambda_*}(\bar{Q},B_1)$ as $\eps\downarrow 0$. Since $\bar{Q}_\eps\in W_{\rm sym}^{1,2}(B_1;\mathbb{S}^4)$ satisfies ${\rm spt}(\bar{Q}_\eps-Q_*)\subset \overline{B_{1-\eps}}$, from the previous part of the proof we have $\mathcal{E}_{\lambda_*}(Q_*,B_1)\leq \mathcal{E}_{\lambda_*}(\bar{Q}_\eps,B_1)$. Letting $\eps\to 0$, we conclude that  $\mathcal{E}_{\lambda_*}(Q_*,B_1)\leq \mathcal{E}_{\lambda_*}(\bar{Q},B_1)$, and the proof is complete. 
\end{proof}

\begin{remark}
In the case $\lambda_*=0$, Theorem \ref{compintthm} tells us that the limiting map $Q_*$ is minimizing the Dirichlet energy among all $\mathbb{S}^1$-equivariant maps with values in $\mathbb{S}^4$ agreeing with $Q_*$ on $\partial B_1$. In particular, $Q_*$ is a weakly harmonic map in $B_1$, see e.g. Proposition \ref{prop:symmetric-criticality}.
\end{remark}
\begin{remark}
\label{harmoniccompactness}
In the particular case $\lambda_n\equiv \lambda_*=0$ the previous result reduces to a local compactness property for equivariant harmonic maps into $\mathbb{S}^4$ and it is just a particular case of the much more general statement established in \cite[Proposition 4.6]{Gastel}. 	
\end{remark}

\subsection{Compactness up to the boundary}

For an axisymmetric open neighborhood $U$ of a point $x_0\in \partial\Omega\cap\{\text{$x_3$-axis}\}$, we shall assume on $U$ that 
\begin{equation}\label{assumptioncompthmbdrypre}
\text{$U\cap\Omega$ has a Lipschitz boundary,}
\end{equation}
and that there exists a $C^1$-diffeomorphism $\Phi : B_2\subset \R^3 \to \Phi(B_2)\subset \R^3$ satisfying 
\begin{equation}\label{assumptioncompthmbdry}
\begin{cases}
\text{$\Phi(0)=x_0$;}\\ 
\text{$\Phi$ is symmetric, i.e., $\Phi(R x)=R\Phi(x)$ for all $R\in\mathbb{S}^1$ and for all $x\in B_2$;}\\
\text{$\Phi(B_1)=U$, $\Phi(B_1^+)=U\cap\Omega$, $\Phi(B_1\cap \{x_3=0\})=U\cap\partial\Omega$, $\Phi(\partial B_1\cap\{x_3>0\})\subset\Omega$;}
\end{cases} 
\end{equation}
 where we have set $B_1^+:=B_1\cap\{x_3>0\}$. 
 \begin{remark}
If $\partial\Omega$ is of class $C^1$ and rotationally symmetric then a sufficiently small neighborhood $U$ of a given point $x_0\in \partial\Omega\cap\{\text{$x_3$-axis}\}$ satisfying properties \eqref{assumptioncompthmbdrypre}-\eqref{assumptioncompthmbdry} above clearly exists and indeed it is enough to choose $U=B_r(x_0)$ for a radius $r>0$ small enough.
 \end{remark}
 
\begin{theorem}\label{compintthmbdry}
Let $x_0\in\partial\Omega\cap\{\text{$x_3$-axis}\}$ and $U$ an axisymmetric open neighborhood  of $x_0$ satisfying  \eqref{assumptioncompthmbdrypre}-\eqref{assumptioncompthmbdry} above. Let  $\{Q^n_{\rm b}\}\subset C^1(\partial\Omega;\mathbb{S}^4)$ be a sequence of boundary conditions satisfying 
\begin{equation}\label{assumpbdrycompthm}
\sup_{n}\int_{U\cap\partial\Omega}|\nabla_{\rm tan}Q^n_{\rm b}|^2\,d\mathcal{H}^2<+\infty\,. 
\end{equation}
Let $\{\lambda_n\}\subset[0,\infty)$ be such that $\lambda_n\to\lambda_*\in[0,\infty)$, and $\{Q_n\}\subset \mathcal{A}^{\rm sym}_{Q^n_{\rm b}}(\Omega)$. Assume that $Q_n$ is minimizing $\mathcal{E}_{\lambda_n}$ over $\mathcal{A}^{\rm sym}_{Q^n_{\rm b}}(\Omega)$. 
If $\sup_n \mathcal{E}_{\lambda_n}(Q_n, U\cap\Omega)<\infty$, then there exist a (not relabeled) subsequence and $Q_*\in  W^{1,2}_{\rm sym}(U\cap\Omega;\mathbb{S}^4)$ such that $Q_n\to Q_*$ strongly in  $W_{\rm loc}^{1,2}(U\cap\overline\Omega)$. 
\end{theorem}

\begin{proof}
We proceed as in the proof of Theorem \ref{compintthm}, and we partially sketch the argument, focusing on the main differences. First, we infer from the uniform energy bound that there exist a subsequence and $Q_*\in W^{1,2}_{\rm sym}(U\cap\Omega;\mathbb{S}^4)$ such that $Q_n\rightharpoonup Q_*$ weakly in $W^{1,2}(U\cap\Omega)$, $Q_n\to Q_*$ strongly in $L^2(U\cap\Omega)$ and a.e. on $U\cap \Omega$. Then, $W(Q_n)\to W(Q_*)$ strongly in $L^1(U\cap \Omega)$. 

We now consider the maps $\widetilde Q_n:=Q_n\circ\Phi \in W^{1,2}_{\rm sym}(B_1^+;\mathbb{S}^4)$ and $\widetilde Q_*:=Q_*\circ\Phi \in W^{1,2}_{\rm sym}(B_1^+;\mathbb{S}^4)$ (note that the $\mathbb{S}^1$-equivariance of $\widetilde Q_n$ and $\widetilde Q_*$ follows from the equivariance assumption on $\Phi$). Then,  $\widetilde Q_n\rightharpoonup \widetilde Q_*$ weakly in $W^{1,2}(B_1^+)$, and $\widetilde Q_n\to \widetilde Q_*$ strongly in $L^2(B_1^+)$ because the corresponding properties for $\{Q_n \}$ are preserved under composition with the diffeomorphism $\Phi$. By weak continuity of the trace operator, we also have $\widetilde Q_n\rightharpoonup \widetilde Q_*$ weakly in $W^{1/2,2}(B_1\cap\{x_3=0\})$, and hence $\widetilde Q_n\to \widetilde Q_*$ strongly in $L^2(B_1\cap \{x_3=0\})$ by the compact embedding $W^{1/2,2}\hookrightarrow L^2$. On the other hand, assumption \eqref{assumpbdrycompthm} implies that (the traces of) $\widetilde Q_n$ are bounded in $W^{1,2}(B_1\cap\{x_3=0\})$, and thus $\widetilde Q_n\rightharpoonup \widetilde Q_*$ weakly in $W^{1,2}(B_1\cap\{x_3=0\})$. 

We fix an arbitrary parameter $\delta\in(0,1/2)$, and we aim to show that $Q_n\to Q_*$ strongly in $W^{1,2}(\Phi(B^+_{1-\delta}))$. Since $\widetilde Q_n\rightharpoonup \widetilde Q_*$ weakly in $W^{1,2}(B_1\cap\{x_3=0\})$, arguing as in the proof of Theorem \ref{compintthm}, up to a subsequence, we can find $\rho\in(1-\delta,1)$ such that 
\begin{equation}\label{esti1compbdry}
\lim_{n\to\infty}\int_{\partial B^+_\rho}|\widetilde Q_n-\widetilde Q_*|^2\,d\mathcal{H}^2=0 \, \text{ and } \,  \int_{\partial B^+_\rho}\big(|\nabla_{\rm tan} \widetilde Q_n|^2+|\nabla_{\rm tan} \widetilde Q_*|^2\big)\,d\mathcal{H}^2\leq C\,.
\end{equation}
It is now convenient to consider a biLipschitz map $\Psi:\overline B_\rho\to \overline B_\rho^+$ satisfying the properties
\begin{itemize}
\item $\Psi(Rx)=R\Psi(x)$ for every $R\in\mathbb{S}^1$;  
\vskip3pt
\item $\Psi(\partial B_\rho\cap\{x_3>0\})= \partial B_\rho^+\cap\{x_3=0\}$,
\end{itemize}
and to define $\widehat Q_n(x):=\widetilde Q_n(\Psi(x))$ and $\widehat Q_*(x):=\widetilde Q_*(\Psi(x))$ for $x\in B_\rho$. Then, $\widehat Q_n,\widehat Q_*\in W^{1,2}_{\rm sym}( B_\rho;\mathbb{S}^4)$, and for the corresponding traces on $\partial B_\rho$ the estimate \eqref{esti1compbdry} yields
\begin{equation}\label{esti1compbdrybis}
\lim_{n\to\infty}\int_{\partial B_\rho}|\widehat Q_n-\widehat Q_*|^2\,d\mathcal{H}^2=0\text{ and } \int_{\partial B_\rho}\big(|\nabla_{\rm tan} \widehat Q_n|^2+|\nabla_{\rm tan} \widehat Q_*|^2\big)\,d\mathcal{H}^2\leq C\, .
\end{equation}

Setting $\sigma_n:=\|\widehat Q_n-\widehat Q_*\|^{1/3}_{L^2(\partial B_\rho)}+2^{-n}\to 0$, we proceed as in the proof of Theorem  \ref{compintthm} to construct a sequence $w_n\in W^{1,2}_{\rm sym}(B_\rho\setminus \overline{B_{\rho(1-\sigma_n)}};\mathcal{S}_0)$ satisfying $w_n(x)=\widehat Q_n(x)$ on $\partial B_\rho$ and $w_n(x)=\widehat Q_*(\frac{x}{1-\sigma_n})$ on $\partial B_{\rho(1-\sigma_n)}$, together with the bounds
\begin{equation}\label{mercr1223juill}
\int_{B_\rho\setminus B_{\rho(1-\sigma_n)}} |\nabla w_n|^2\,dx\leq C\sigma_n\,
\end{equation}
and 
$$ \sup_{x\in B_\rho\setminus B_{\rho(1-\sigma_n)}} {\rm dist}^2(w_n(x),\mathbb{S}^4)\leq C\sigma_n \,,$$
for a constant $C>0$ independent of $n$.

When $n$ is large enough, we have $|w_n|>1/2$, and we can then define a map $\bar Q_n$ on $\Omega$ by setting
$$\bar Q_n(x):=
\begin{cases}
Q_n(x) & \text{for $x\in \Omega\setminus \Phi\circ\Psi(B_\rho)$}\\[8pt]
\displaystyle \frac{w_n(\Psi^{-1}\circ\Phi^{-1}(x))}{|w_n(\Psi^{-1}\circ\Phi^{-1}(x))|} & \text{for $x\in \Phi\circ\Psi(B_\rho\setminus B_{\rho(1-\sigma_n)})$}\,,\\[10pt]
\displaystyle Q_*\Big(\Phi\circ\Psi\Big(\frac{\Psi^{-1}\circ\Phi^{-1}(x)}{1-\sigma_n}\Big)\Big) & \text{for $x\in \Phi\circ\Psi(B_{\rho(1-\sigma_n)})$}\,.
\end{cases}
$$
By construction, $\bar Q_n\in\mathcal{A}^{\rm sym}_{Q^n_{\rm b}}(\Omega)$, so that $\mathcal{E}_{\lambda_n}(Q_n)\leq \mathcal{E}_{\lambda_n}(\bar Q_n)$. Since $\bar Q_n=Q_n$ outside $\Phi\circ\Psi(B_\rho)$, the last inequality reduces to $\mathcal{E}_{\lambda_n}(Q_n,\Phi\circ\Psi(B_\rho))\leq \mathcal{E}_{\lambda_n}(\bar Q_n,\Phi\circ\Psi(B_\rho))$. On the other hand, it follows from \eqref{mercr1223juill} that 
$$\mathcal{E}_{\lambda_n}(\bar Q_n,\Phi\circ\Psi(B_\rho\setminus B_{\rho(1-\sigma_n)}))\to 0\,,$$
and by a change of variables, 
$$\mathcal{E}_{\lambda_n}(\bar Q_n,\Phi\circ\Psi( B_{\rho(1-\sigma_n)}))\to \mathcal{E}_{\lambda_*}(Q_*,\Phi\circ\Psi( B_{\rho}))\,.$$
Hence, $\mathcal{E}_{\lambda_n}(\bar Q_n,\Phi\circ\Psi(B_\rho))\to \mathcal{E}_{\lambda_*}(Q_*,\Phi\circ\Psi(B_\rho))$. By lower semicontinuity of the Dirichlet energy and the established convergence of the potential term, we  have 
$$\mathcal{E}_{\lambda_*}(Q_*,\Phi\circ\Psi( B_{\rho}))\leq \liminf_{n\to\infty}\mathcal{E}_{\lambda_n}(Q_n,\Phi\circ\Psi( B_{\rho}))\leq\limsup_{n\to\infty}\mathcal{E}_{\lambda_n}(Q_n,\Phi\circ\Psi( B_{\rho}))\leq  \mathcal{E}_{\lambda_*}(Q_*,\Phi\circ\Psi(B_\rho))\,.$$
We have thus proved that $\mathcal{E}_{\lambda_n}(Q_n,\Phi\circ\Psi( B_{\rho}))\to  \mathcal{E}_{\lambda_*}(Q_*,\Phi\circ\Psi(B_\rho))$. As in the proof of Theorem~\ref{compintthm}, it implies the $W^{1,2}$-strong convergence of $Q_n$ in the open set $\Phi\circ\Psi(B_\rho)=\Phi(B^+_\rho)$, whence the strong $W^{1,2}$-convergence in the smaller open set $\Phi(B^+_{1-\delta})$. Finally, as $\delta \downarrow 0$ we have $\Phi(B^+_{1-\delta}) \uparrow U\cap \overline{\Omega}$ and the conclusion follows.
\end{proof}

\begin{remark}\label{remmovingboundf}
Theorem \ref{compintthmbdry} can be easily extended to the case of varying domains $\Omega_n$ depending on the sequence index $n$. This case is of specific interest when analyzing blow-up sequences at the boundary as we will do in the next section when discussing boundary regularity properties for energy minimizers. To this latter purpose, we consider a sequence $\{\Omega_n\}_{n\in\mathbb{N}}$ of axisymmetric bounded open sets such that $x_0\in\partial \Omega_n\cap \{\text{$x_3$-axis}\}$.   
We assume that for a fixed axisymmetric neighborhood $U$ of $x_0$ there exists a sequence of $C^1$-diffeomorphisms $\Phi_n:B_2\to\Phi_n(B_2)$ satisfying \eqref{assumptioncompthmbdrypre}-\eqref{assumptioncompthmbdry} with $U=\Phi_n(B_1)$, and such that $\Phi_n\to \tau_{x_0}$ as $n\to\infty$ in $C^1(B_2)$, where $\tau_{x_0}(x):=x+x_0$. This latter condition implies that $U\cap\Omega_n\to B_1^+(x_0)$ as $n\to\infty$ in the Hausdorff metric. Under these assumptions and \eqref{assumpbdrycompthm}, the conclusion (and proof) of Theorem~\ref{compintthmbdry} holds in the following form: there exist a (not relabeled) subsequence and $Q_*\in W^{1,2}_{\rm sym}(B_1^+(x_0);\mathbb{S}^4)$ such that  $Q_n\circ\Phi_n\to Q_*\circ\tau_{x_0}$ strongly in $W^{1,2}(B_r^+)$ for every radius $r\in(0,1)$ (and in particular, $Q_n\to Q_*$ strongly in $W^{1,2}_{\rm loc}(B_1^+(x_0))$). 
\end{remark}



\section{Partial regularity of LdG minimizers under axial symmetry}
\label{sec:6}

In this section we provide the proof of Theorem \ref{thm:partial-regularity} as outlined in the Introduction, following the well-known strategy introduced in \cite{SU1}-\cite{SU3} for minimizing harmonic maps as already adapted to the LdG context in \cite{DMP1} without the symmetry constraint. However, the arguments here are similar, but sometimes deviate substantially from \cite{DMP1} as we illustrate now.

 Our proof of the partial regularity is based on the results in Sec.~\ref{sec:comp} and \ref{sec:6} for axially symmetric minimizers of \eqref{LDGenergytilde} combined with the regularity results for arbitrary weak solutions to \eqref{MasterEq} under smallness of the scaled energy from our previous paper \cite{DMP1}. Here, the smallness property both in the interior and at the boundary automatically holds out of the symmetry axis in view of a classical capacity argument for $W^{1,2}$ functions, therefore the singular set is confined to the symmetry axis. When dealing with points on the symmetry axis, we apply the general strategy by suitably modifying the arguments used there because the minimality property is now available only in the restricted class of symmetric competitors. Here monotonicity formulas are obtained through the same penalization trick from \cite{DMP1} adapted to the $\mathbb{S}^1$-equivariant case. Compactness of blow-ups centered on the symmetry axis are obtained as in the previous section, through a Luckhaus' type interpolation argument but constructing the comparison maps by suitable $\mathbb{S}^1$-equivariant extensions into spherical shells. The novelty when discussing the Liouville property in the interior has been already pointed out in the Introduction whereas the analogous Liouville property at the boundary is obtained as in the nonsymmetric case, since only criticality and no energy minimality was used in \cite{DMP1}. As a consequence, complete boundary regularity also follows in the present case. 

Finally, the asymptotic analysis at isolated singularities relies instead on the classification and (in)stability results for tangent maps from Sec.~\ref{sec:axisymm} and \ref{sec:stability} together with the celebrated Simon-{\L}ojasiewicz inequality, adapting to our context the simplified proof from \cite{Simon} for the case of harmonic maps.  
We refer the interested reader to the introduction to Sec.~\ref{subsec:Loj} for further remarks on the proof.

\subsection{Symmetric criticality \& the Euler-Lagrange equations}

In this subsection, we establish the Euler-Lagrange equation satisfied by critical points of  $\mathcal{E}_\lambda$ in the constrained class $W^{1,2}_{\rm sym}(\Omega;\mathbb{S}^4)$. As a preliminary step, we show that the admissible class of symmetric configurations $\mathcal{A}_{Q_{\rm b}}^{\rm sym}(\Omega)$ defined in \eqref{S1admissibleconf} is not empty for any equivariant boundary condition of interest in this work, a fact which immediately implies also existence of minimizers of $\mathcal{E}_\lambda$ over $\mathcal{A}_{Q_{\rm b}}^{\rm sym}(\Omega)$. 
\begin{proposition}\label{prop:nonempty-class}
	Let $\Omega \subset \R^3$ be an $\bbS^1$-invariant bounded open set with Lipschitz boundary and let $Q_{\rm b} \in \rmLip_{\rm sym}(\partial\Omega;\bbS^4)$. Then $\mathcal{A}^{\rm sym}_{Q_{\rm b}}(\Omega)$ is not empty. As a consequence, there exists at least one minimizer of $\mathcal{E}_\lambda$ over $\mathcal{A}^{\rm sym}_{Q_{\rm b}}(\Omega)$.
\end{proposition}

\begin{proof}
	Once we have proved that $\mathcal{A}^{\rm sym}_{Q_{\rm b}}(\Omega)$ is not empty, existence of minimizers is standard from the direct method in the Calculus of Variations. Indeed, whenever not empty $\mathcal{A}^{\rm sym}_{Q_{\rm b}}(\Omega)$ is sequentially closed under weak convergence in $W^{1,2}_{\rm sym}(\Omega;\bbS^4)$ (which is in turn sequentially weakly closed in $W^{1,2}(\Omega;\bbS^4)$) and $\mathcal{E}_\lambda$ is bounded below, coercive and lower semicontinuous w.r.to the $W^{1,2}$-weak convergence.
	 
	To show $\mathcal{A}^{\rm sym}_{Q_{\rm b}}(\Omega)$ is not empty, there are two cases to deal with: writing $\ell_k = (x_k^-, x_k^+)$ for each segment $\ell_k \subset \Omega \cap \{x_3\mbox{-axis}\}$ as in Sec.~\ref{sec:s1-equivariance}, either we have $Q_{\rm b}(x_k^-) = Q_{\rm b}(x_k^+)$ for every $k$ (Case 1) or there is a segment $\ell_k$ at the endpoints of which the map $Q_{\rm b}$ attains different values (Case 2). Recall that, according to Remark~\ref{rmk:invariantconf}, $Q_{\rm b}=\pm\eo$ on $\mathcal{B}$.
	
	\emph{Case 1.} As in Sec.~\ref{sec:s1-equivariance}, we consider $\mathcal{D}_\Omega^+ = \Omega \cap \{x_2 = 0, x_1 > 0 \}$. Let $\psi$ denote the restriction of $Q_{\rm b}$ to $\partial \mathcal{D}_\Omega^+$ extended to each segment $\ell_k$ as the constant corresponding to the common value of $Q_{\rm b}$ at the endpoints of the segment. Then $\psi \in \rmLip(\partial\mathcal{D}_\Omega^+;\bbS^4)$. Being Lipschitz, $\psi$ omits at least a point $P \in \bbS^4$ (indeed, $\mathcal{H}^1(\psi(\partial\mathcal{D}_\Omega^+)) \leq  \rmLip(\psi) \cdot \mathcal{H}^1(\partial\mathcal{D}_\Omega^+)<\infty$, where $\rmLip(\psi)$ is the Lipschitz constant of $\psi$, therefore $\psi(\partial\mathcal{D}_\Omega^+)$ can't be the whole $\bbS^4$), therefore the composition $\widehat{\psi} := \sigma_4 \circ \psi$, where $\sigma_4 : \bbS^4 \setminus \{P\} \to \R^4$ is the stereographic projection from $P$, belongs to $\rmLip(\partial\mathcal{D}_\Omega^+; \R^4)$. By McShane-Whitney extension theorem, $\widehat{\psi}$ has an extension $\widehat{\Psi} \in \rmLip(\overline{\mathcal{D}_\Omega^+}; \R^4)$. Letting $\Psi:=\sigma_4^{-1} \circ \widehat{\Psi}$, we then have $\Psi \in \rmLip(\overline{\mathcal{D}_\Omega^+};\bbS^4)$. From $\Psi$, we define a map $\Phi \in \rmLip_{\rm sym}(\overline{\Omega};\bbS^4)$ by letting $\Phi(R x) := R\Psi(x)R^\trans$ for every $R \in \bbS^1$ and every $x \in \overline{\mathcal{D}_\Omega^+}$. Noticing that, by construction, $\Phi \vert_{\partial\Omega} = Q_{\rm b}$, we finally see that $\Phi \in \mathcal{A}_{Q_{\rm b}}^{\rm sym}(\Omega)$ and, by the boundedness of the potential, $\mathcal{E}_\lambda(\Phi)$ is finite. 

	\emph{Case 2.} In this case any $\bbS^1$-equivariant extension of $Q_{\rm b}$ necessarily has singularities and the argument below slightly deviates from the one used in Case~1 precisely to deal with this fact. Let $p \leq M+1$ (where $M$ is as in Sec.~\ref{sec:s1-equivariance}) denote the number of segments $\ell_k\subset \Omega \cap \{x_3 \mbox{-axis}\}$ so that at the endpoints of these $Q_{\rm b}$ attains different values and denote such segments $\tilde{\ell}_1, \tilde{\ell}_2\,\dots,\tilde{\ell}_p$, ordering them increasingly with the $x_3$-coordinate of their lower endpoint. For each $j = 1, 2, \dots, p$, we pick the mid-point $x_j \in \tilde{\ell}_j=(x_j^-,x_j^+)$ and we fix a small $\delta > 0$ such that $V:=\cup_{j=1}^p \overline{B_\delta(x_j)}\subset \Omega$ and the union is disjoint. Let us set $U := \Omega \setminus V$ and notice $U \cap \tilde{\ell}_j = (x_j^-, y_j^-) \cup (y_j^+,x_j^+)$, where $y_j^-$, $y_j^+$ are the intersections of $\partial B_\delta(x_j)$ with the $x_3$-axis, ordered in the obvious way. 
	
	To construct the desired extension $\Phi$ of $Q_{\rm b}$, we first define $\Phi := Q_{\rm b}$ on $\partial \Omega$ (in the pointwise sense) so that, in particular, $\Phi(x_j^\pm) = Q_{\rm b}(x_j^\pm)$ for every $j=1,2,\dots,p$. Then we define $\Phi$ on $V \setminus \{x_1,x_2,\dots,x_p\}$. To this purpose, on each set $\overline{B_\delta(x_j)}\setminus\{x_j\}$ we let $\Phi(x) := \pm Q^{(0)}(x-x_j)$, where $Q^{(0)}$ is given by the formula \eqref{stableblowups} with $\alpha = 0$ and we take the positive sign if $Q_{\rm b}(x_j^-) = -\eo$ and the negative sign otherwise, i.e., if $Q_{\rm b}(x_j^-) = \eo$. Thanks to this sign convention, we thus have $\Phi(x_j^\pm) = \Phi(y_j^\pm)$ for each $j=1,2,\dots,p$. Therefore we can extend $\Phi$ to each segment $(x_j^-,y_j^-)$, resp. $(y_j^+,x_j^+)$, as the corresponding constant at the endpoints. Noticing $\Phi \in C^\infty_{\rm sym}(V \setminus \{x_1,x_2,\dots,x_p\};\bbS^4)$, we see $\Phi$ is a well-defined $\bbS^1$-equivariant Lipschitz map on $\partial U$ and therefore, arguing as in Case~1, we can extend $\Phi$ to $\overline{U}$ in an $\bbS^1$-equivariant Lipschitz way. Because of this, and since we also have $\Phi \in W^{1,2}_{\rm sym}(B_\delta(x_j);\bbS^4)$ for every $j = 1,2,\dots,p$, we finally deduce $\Phi \in \mathcal{A}_{Q_{\rm b}}^{\rm sym}(\Omega)$. Moreover, by the boundedness of the potential, $\mathcal{E}_\lambda(\Phi)$ is finite. This concludes the proof.
\end{proof}

A map $Q_\lambda\in W^{1,2}_{\rm sym}(\Omega;\mathbb{S}^4)$ is said to be a critical point of  $\mathcal{E}_\lambda$ in $W^{1,2}_{\rm sym}(\Omega;\mathbb{S}^4)$ if 
\begin{equation}\label{defcritpt}
\left. \frac{d}{dt}\mathcal{E}_\lambda\left(\frac{Q_\lambda+t\Phi}{|Q_\lambda+t\Phi|}\right)
 \right\vert_{t=0} =0\quad\forall \Phi\in C_c^{1,{\rm sym}}(\Omega;\mathcal{S}_0)\,, 
\end{equation}
while $Q_\lambda$ is said to be a critical point in the unconstrained class $W^{1,2}(\Omega;\mathbb{S}^4)$ if \eqref{defcritpt} actually holds for every $\Phi\in C_c^{1}(\Omega;\mathcal{S}_0)$, see \cite[Definition 2.1]{DMP1}. We shall prove in Proposition \ref{prop:symmetric-criticality} that a critical point in the symmetric class $W^{1,2}_{\rm sym}(\Omega;\mathbb{S}^4)$ is always a critical point in the global class $W^{1,2}(\Omega;\mathbb{S}^4)$ in the spirit of the general Palais symmetric criticality principle \cite{Palais}. Note that this principle does not directly apply here since $W^{1,2}(\Omega;\mathbb{S}^4)$ and $W^{1,2}_{\rm sym}(\Omega;\mathbb{S}^4)$ are not Banach manifolds, and we need to prove it by hands (see also \cite{Gastel} and \cite{HorMos} for a similar results in the context of harmonic and biharmonic maps respectively).

\begin{proposition}\label{prop:symmetric-criticality}
If $Q_\lambda \in W^{1,2}_{\rm sym}(\Omega;\mathbb{S}^4)$ is a critical point of $\mathcal{E}_\lambda$ among maps in   $W^{1,2}_{\rm sym}(\Omega;\mathbb{S}^4)$, then $Q_\lambda$ is a critical point of $\mathcal{E}_\lambda$ among all maps $W^{1,2}(\Omega;\mathbb{S}^4)$.
\end{proposition}

\begin{proof}
Arguing as in \cite[proof of Proposition 2.2, Step 1]{DMP1}, we derive from \eqref{defcritpt} that 
\begin{equation}\label{weakformeqsym}
\int_\Omega\nabla Q_\lambda :\nabla\Phi -\Big(|\nabla Q_\lambda|^2Q_\lambda+\lambda f(Q_\lambda)\Big):\Phi\,dx=0 \qquad\forall \Phi\in C_{c,{\rm sym}}^{1}(\Omega;\mathcal{S}_0)\,,
\end{equation}
where we have set $f(Q):=Q^2-{\rm tr}(Q^3)Q$. Still by \cite[Proposition 2.2]{DMP1}, it is enough to show that \eqref{weakformeqsym} actually holds for every $\Phi\in C_c^{1}(\Omega;\mathcal{S}_0)$. To this purpose, let us fix an arbitrary $\Phi\in C_c^{1}(\Omega;\mathcal{S}_0)$. Given $R\in\mathbb{S}^1$, we define the ``twisted action'' of $R$ on $\Phi$ by setting
$$R\ast\Phi(x):=R\Phi(R^\trans x)R^\trans\,, $$
and we set 
$$\Phi^{\rm s}:=\int_{\mathbb{S}^1} R\ast\Phi\,d\mathfrak{h}\,, $$
where  $\mathfrak{h}$ denotes the normalized Haar measure on $\mathbb{S}^1$. Since $R'*(R*\Phi)=(R'R)*\Phi$, using the invariance under translations of~$\mathfrak{h}$, we obtain $R'*\Phi^s=\Phi^s$ and in turn $\Phi^{\rm s}\in  C_{c,{\rm sym}}^{1}(\Omega;\mathcal{S}_0)$, which is indeed the subclass of deformations fixed by the twisted action of $\mathbb{S}^1$ on $C_c^{1}(\Omega;\mathcal{S}_0)$.

By equivariance of $Q_\lambda$, we have $Q_\lambda(x)=R^\trans Q_\lambda(Rx)R$ a.e. in $\Omega$ for every $R\in \mathbb{S}^1$. Using this identity, straightforward computations yield for every $R\in\mathbb{S}^1$, 
$$
\nabla Q_\lambda(x) :\nabla\Phi(x)
= \nabla Q_\lambda(Rx) :\nabla (R\ast\Phi)( Rx)\quad \text{a.e. in $\Omega$}\,,
$$
and 
\begin{multline*}
\Big(|\nabla Q_\lambda(x)|^2Q_\lambda(x) +\lambda f(Q_\lambda(x))\Big):\Phi(x)\\
=\Big(|\nabla Q_\lambda(Rx)|^2Q_\lambda(Rx) +\lambda f(Q_\lambda(Rx))\Big):(R\ast\Phi)(Rx)\quad\text{a.e. in $\Omega$}\,. 
\end{multline*}
Integrating the previous identities over $\Omega$ and averaging over $\mathbb{S}^1$, using a change of variables and Fubini's theorem, we are then led to 
\begin{align*}
 \int_\Omega\nabla Q_\lambda :\nabla\Phi -&\Big(|\nabla Q_\lambda|^2Q_\lambda+\lambda f(Q_\lambda)\Big):\Phi\,dx\\
 =&\int_{\mathbb{S}^1}\bigg\{\int_\Omega\nabla Q_\lambda :\nabla(R\ast\Phi) -\Big(|\nabla Q_\lambda|^2Q_\lambda+\lambda f(Q_\lambda)\Big):(R\ast\Phi) \,dx \bigg\}\,d\mathfrak{h}\\
 =&\int_\Omega\bigg\{\int_{\mathbb{S}^1}\nabla Q_\lambda:\nabla(R\ast\Phi) -\Big(|\nabla Q_\lambda|^2Q_\lambda+\lambda f(Q_\lambda)\Big):(R\ast\Phi)\,d\mathfrak{h}\bigg\}\,dx\\
 =& \int_\Omega\nabla Q_\lambda :\nabla\Phi^{\rm s} -\Big(|\nabla Q_\lambda|^2Q_\lambda+\lambda f(Q_\lambda)\Big):\Phi^{\rm s}\,dx=0\,,
 \end{align*}
thanks to  \eqref{weakformeqsym}, and the proof is complete.
\end{proof}

As a consequence of Proposition \ref{prop:symmetric-criticality} and \cite[Proposition 2.2]{DMP1}, we thus have 

\begin{corollary}\label{corELeq}
A map $Q_\lambda \in W^{1,2}_{\rm sym}(\Omega;\mathbb{S}^4)$ is a critical point of $\mathcal{E}_\lambda$ over  $W^{1,2}_{\rm sym}(\Omega;\mathbb{S}^4)$ if and only if it satisfies
\begin{equation}\label{ELeqsym}
-\Delta Q_\lambda=|\nabla Q_\lambda|^2Q_\lambda+\lambda\Big(Q^2_\lambda-\frac{1}{3}I-{\rm tr}(Q_\lambda^3)Q_\lambda\Big)\quad\text{in $\mathscr{D}^\prime(\Omega)$}\,.
\end{equation} 
\end{corollary}

\begin{remark}
If a map $Q_\lambda \in W^{1,2}_{\rm sym}(\Omega;\mathbb{S}^4)$ is a minimizer of  $\mathcal{E}_\lambda$ among all   $Q \in W^{1,2}_{\rm sym}(\Omega;\mathbb{S}^4)$ such that $Q-Q_\lambda$ is compactly supported in $\Omega$, then $Q_\lambda$ is a critical point of $\mathcal{E}_\lambda$. In particular, if $Q_\lambda$ is minimizing $\mathcal{E}_\lambda$ over $\mathcal{A}^{\rm sym}_{Q_{\rm b}}(\Omega)$, then $Q_\lambda$ solves \eqref{ELeqsym}. 
\end{remark}

\begin{remark}\label{remcomplambda0harmmap}
The discussion above applies also in the particular case $\lambda=0$. In other words, $Q_0 \in W^{1,2}_{\rm sym}(\Omega;\mathbb{S}^4)$ is a critical point of the Dirichlet energy $\mathcal{E}_0$ over  $W^{1,2}_{\rm sym}(\Omega;\mathbb{S}^4)$ if and only if $Q_0$ is a weakly harmonic into $\mathbb{S}^4$ map in $\Omega$. In particular, if $Q_0$ is a minimizer of $\mathcal{E}_0$ among all $Q \in W^{1,2}_{\rm sym}(\Omega;\mathbb{S}^4)$ such that $Q-Q_0$ is compactly supported in $\Omega$, then $Q_0$ is a weakly harmonic map in $\Omega$. 
\end{remark}


\subsection{Monotonicity formulas}

The partial regularity for minimizers of $\mathcal{E}_\lambda$ in the symmetric class $\mathcal{A}^{\rm sym}_{Q_{\rm b}}(\Omega)$ is based in a fundamental way on (standard) energy monotonicity formulas for the scaled energy on balls. Due to the symmetry constraint, such formulas cannot be directly deduced from inner variations of the energy, unless the center of the balls is on the symmetry axis. Here we rely on the results in \cite[Section 2.1]{DMP1} which were developed precisely for this purpose.

\begin{proposition}\label{corbdmonotform}
Assume that $\partial \Omega$ is of class $C^3$ and $Q_{\rm b} \in C^{1,1}(\partial\Omega,\bbS^4)$. If $Q_\lambda \in \mathcal{A}^{\rm sym}_{Q_{\rm b}}(\Omega)$ is a minimizer of $\mathcal{E}_\lambda$ over $ \mathcal{A}^{\rm sym}_{Q_{\rm b}}(\Omega)$, then $Q_\lambda$ satisfies
\begin{enumerate}
\item[\rm 1)]  the {\sl Interior Monotonicity Formula}:
\begin{multline}\label{IntMonForm}
\frac{1}{r}\mathcal{E}_\lambda(Q_{\lambda},B_r(x_0)) -\frac{1}{\rho}\mathcal{E}_\lambda(Q_{\lambda},B_\rho(x_0))=\\
\int_{B_r(x_0)\setminus B_\rho(x_0)} \frac{1}{|x-x_0|}\bigg|\frac{\partial Q_{\lambda}}{\partial |x-x_0|}\bigg|^2\,dx
+ 2\lambda\int_\rho^r\bigg(\frac{1}{t^2}\int_{B_t(x_0)}W(Q_{\lambda})\,dx\bigg)\,dt
\end{multline}
for every $x_0\in\Omega$ and every $0<\rho<r\leq{\rm dist}(x_0,\partial\Omega)$;
\vskip5pt

\item[\rm 2)]  the {\sl Boundary Monotonicity Inequality}: there exist two constants $C_\Omega>0$ and ${\bf r}_\Omega>0$ (depending only on $\Omega$) such that 
\begin{multline}\label{BdMonForm}
\frac{1}{r}\mathcal{E}_\lambda(Q_{\lambda},B_r(x_0)\cap\Omega) -\frac{1}{\rho}\mathcal{E}_\lambda(Q_{\lambda},B_\rho(x_0)\cap\Omega)\geq -(r-\rho)K_\lambda(Q_{\rm b},Q_{\lambda})\\
+\int_{\big(B_r(x_0)\setminus B_\rho(x_0)\big)\cap\Omega} \frac{1}{|x-x_0|}\bigg|\frac{\partial Q_\lambda}{\partial |x-x_0|}\bigg|^2\,dx
+ 2\lambda\int_\rho^r\bigg(\frac{1}{t^2}\int_{B_t(x_0)\cap\Omega}W(Q_\lambda)\,dx\bigg)\,dt
\end{multline}
for every $x_0\in\partial\Omega$ and every $0<\rho<r<{\bf r}_\Omega$, where 
$$K_\lambda(Q_{\rm b},Q_{\lambda}):= C_\Omega\bigg(\|\nabla_{\rm tan} Q_{\rm b}\|^2_{L^\infty(\partial\Omega)}+\lambda\|W(Q_{\rm b})\|_{L^1(\partial\Omega)}+\|\nabla Q_{\lambda}\|^2_{L^2(\Omega)}\bigg)  \,.$$
\end{enumerate}
Moreover the quantity $K_\lambda(Q_{\rm b},Q_{\lambda})$ in \eqref{BdMonForm} satisfies 
\begin{equation}\label{contrKOm}
K_\lambda(Q_{\rm b},Q_{\lambda})\leq C_\Omega\bigg(\|\nabla_{\rm tan} Q_{\rm b}\|^2_{L^\infty(\partial\Omega)}+\lambda\|W(Q_{\rm b})\|_{L^1(\partial\Omega)}+\mathcal{E}_\lambda(\bar{Q}_{\rm b})\bigg)\,,
\end{equation}
where $\bar{Q}_{\rm b}\in  \mathcal{A}^{\rm sym}_{Q_{\rm b}}(\Omega)$ is any given extension of $Q_{\rm b}$ to $\Omega$. 
\end{proposition}

\begin{proof}
We are going to prove that  $Q_\lambda$ satisfies the assumptions in  \cite[Proposition 2.4]{DMP1} (with $Q_{\rm ref}=Q_\lambda$). This will lead to \eqref{IntMonForm} and \eqref{BdMonForm}. Hence, according to \cite[Proposition 2.4]{DMP1}, we consider for $\eps>0$ the energy functional $\mathcal{GL}_\eps(\,\cdot \,;Q_\lambda)$ defined over $W^{1,2}(\Omega;\mathcal{S}_0)$ by  
\begin{equation}\label{defGLepsQref}
\mathcal{GL}_{\varepsilon}(Q; Q_\lambda):= \mathcal{E}_\lambda(Q)+\frac{1}{4\varepsilon^2}\int_\Omega(1-|Q|^2)^2\,dx+\frac{1}{2}\int_\Omega|Q-Q_\lambda|^2\,dx\,.
\end{equation}
Next we set for convenience 
\[
	\mathcal{W}^{\rm sym}_{Q_{\rm b}}(\Omega) := \left\{ Q \in W^{1,2}_{\rm sym}(\Omega;\mathcal{S}_0) : Q = Q_{\rm b} \mbox{ on }\partial\Omega \right\}\,.
\]
Since the potential  $W$ is nonnegative (see \eqref{redpotential} and \eqref{signedbiaxiality}), for each $\eps>0$ the functional $\mathcal{GL}_\eps(\, \cdot \, ;Q_\lambda)$ is coercive on $W^{1,2}(\Omega;\mathcal{S}_0)$. Moreover, $\mathcal{GL}_\eps(\, \cdot \, ;Q_\lambda)$ is lower semi-continuous with respect to the weak $W^{1,2}$-convergence (see \cite[Proposition 3.1]{DMP1}). The class $\mathcal{W}^{\rm sym}_{Q_{\rm b}}(\Omega)$ being closed under weak $W^{1,2}$-convergence, $\mathcal{GL}_\eps(\, \cdot \, ;Q_\lambda)$ admits a minimizer $Q_\eps$ over 
$\mathcal{W}^{\rm sym}_{Q_{\rm b}}(\Omega)$ by the direct method of calculus of variations. Such minimizer $Q_\eps$ is then a critical point of $\mathcal{GL}_\eps(\, \cdot \, ;Q_\lambda)$ over $\mathcal{W}^{\rm sym}_{Q_{\rm b}}(\Omega)$. Arguing exactly as in the proof of Proposition \ref{prop:symmetric-criticality} (with minor modifications), the symmetric criticality principle holds and $Q_\eps$ is actually a critical point of $\mathcal{GL}_\eps(\, \cdot \, ;Q_\lambda)$ over $W^{1,2}_{Q_{\rm b}}(\Omega;\mathcal{S}_0)$.

Now we consider an arbitrary sequence  $\eps_n\to 0$. By minimality of $Q_{\eps_n}$ in $\mathcal{W}^{\rm sym}_{Q_{\rm b}}(\Omega)$ (which contains $ \mathcal{A}^{\rm sym}_{Q_{\rm b}}(\Omega)$), we have 
\begin{equation}\label{firstestglapprox}
\mathcal{GL}_{\varepsilon_n}( Q_{\eps_n}; Q_{\lambda}) \leq \mathcal{GL}_{\varepsilon_n}(Q_{\lambda}; Q_\lambda)=\mathcal{E}_\lambda(Q_\lambda)\,.
\end{equation}
From this estimate, we can argue as in \cite[Proof of Proposition 3.1]{DMP1} to find a (not relabeled) subsequence and $Q_*\in  \mathcal{A}^{\rm sym}_{Q_{\rm b}}(\Omega)$ 
such that $Q_{\eps_n}\rightharpoonup Q_*$ weakly in $W^{1,2}(\Omega)$ and strongly in $L^2(\Omega)$. By lower semi-continuity of $\mathcal{E}_\lambda$ and \eqref{firstestglapprox}, we have 
$$\mathcal{E}_\lambda(Q_\lambda)\leq \mathcal{E}_\lambda(Q_*)+\frac{1}{2}\int_\Omega|Q_*-Q_\lambda|^2\,dx\leq \liminf_{n\to\infty}\mathcal{GL}_{\varepsilon_n}(Q_{\eps_n}; Q_{\lambda} )\leq \mathcal{E}_\lambda(Q_\lambda)\,,$$
where we have used the minimality of $Q_\lambda$ in the class $ \mathcal{A}^{\rm sym}_{Q_{\rm b}}(\Omega)$ in the first inequality. Therefore, $Q_*=Q_\lambda$ and 
$\lim_n\mathcal{GL}_{\varepsilon_n}( Q_{\eps_n}; Q_{\lambda} )=\mathcal{E}_\lambda(Q_\lambda)$, which proves that the assumptions of \cite[Proposition~2.4]{DMP1} are satisfied. 

To complete the proof, it only remains to prove \eqref{contrKOm}. It is a direct consequence of the minimality of $Q_\lambda$. Indeed, if $\bar{Q}_{\rm b}\in  \mathcal{A}^{\rm sym}_{Q_{\rm b}}(\Omega)$, then $\|\nabla Q_{\lambda}\|^2_{L^2(\Omega)}\leq 2\mathcal{E}_\lambda(Q_\lambda)\leq 2\mathcal{E}_\lambda(\bar{Q}_{\rm b})$, which clearly implies~\eqref{contrKOm}. 
\end{proof}

\subsection{Compactness of blow-ups and smallness of the scaled energy}\label{sec:cpt-sym}

\begin{proposition}
\label{compblowup}
Let $Q_\lambda$ be a minimizer of $\mathcal{E}_\lambda$ over $\mathcal{A}^{\rm sym}_{Q_{\rm b}}(\Omega)$. Given $x_0\in\Omega\cap\{\text{$x_3$-axis}\}$ and $r_0>0$ such that 
$\overline{B_{r_0}(x_0)}\subset\Omega$, consider the rescaled map $Q_{\lambda}^{x_0,r}\in W^{1,2}_{\rm sym}(B_{r_0/r};\mathbb{S}^4)$ defined by 
$$Q_{\lambda}^{x_0,r}(x):=Q_\lambda(x_0+rx)\,. $$
For every sequence $r_n\to 0$, there exist a (not relabeled) subsequence and $Q_*\in W^{1,2}_{{\rm sym},{\rm loc}}(\R^3;\mathbb{S}^4)$ such that $Q_{\lambda}^{x_0,r_n}\to Q_*$ strongly in $W^{1,2}_{\rm loc}(\R^3)$. In addition, $Q_*$ is homogeneous of degree zero, and $Q_*$ is a weakly harmonic map which is energy minimizing with respect to $\mathbb{S}^1$-equivariant compactly supported perturbations. 
\end{proposition}

\begin{proof}
Rescaling variables, we have 
$$\mathcal{E}_{\lambda r_n^2}(Q_{\lambda}^{x_0,r_n},B_\rho) =\frac{1}{r_n}\mathcal{E}_\lambda(Q_\lambda, B_{\rho r_n}(x_0))\,,$$
for an arbitrarily fixed radius $\rho\in(0,r_0/r_n)$. Using the monotonicity formula \eqref{IntMonForm} we see that the sequence $\{Q_n\}$ of rescaled maps, $Q_n:=Q_{\lambda}^{x_0, r_n}$, is eventually bounded in $W^{1,2}(B_\rho)$ for any $\rho>0$.
Then the proof follows the argument in \cite[Proposition~3.2]{DMP1},  using again the monotonicity formula \eqref{IntMonForm}, the compactness property established in Theorem~\ref{compintthm} (with $Q_n:=Q_{\lambda}^{x_0, r_n}$ and $\lambda_n:=\lambda r_n^2$), and taking also Remark~\ref{remcomplambda0harmmap} into account. 
\end{proof}

\begin{proposition}\label{vanishdenspoles}
Assume that $\partial\Omega$ is of class $C^3$ and $Q_{\rm b}\in C^{1,1}_{\rm sym}(\partial\Omega;\mathbb{S}^4)$. If $Q_\lambda$ is a minimizer of $\mathcal{E}_\lambda$ over $\mathcal{A}^{\rm sym}_{Q_{\rm b}}(\Omega)$, then 
\begin{equation}\label{limenergdensbdry}
\lim_{r\to 0}\frac{1}{r}\mathcal{E}_{\lambda}(Q_\lambda,B_r(x_0)\cap\Omega)=0
\end{equation}
for every $x_0\in\partial\Omega\cap\{\text{$x_3$-axis}\}$. 
\end{proposition}

\begin{proof}
Again, we essentially argue as in the proof of \cite[Propositions 3.5 and 3.7]{DMP1} with the help of the boundary monotonicity formula \eqref{BdMonForm} and Remark  \ref{remmovingboundf} (which is based on the proof of Theorem~\ref{compintthmbdry}). Hence we only sketch the proof and refer to \cite{DMP1} for details. 
First notice  that the limit in \eqref{limenergdensbdry} exists thanks to \eqref{BdMonForm}. Given a sequence $r_n\to 0$, we consider for $n$ large the domain $\Omega_n:=r_n^{-1}(\Omega-x_0)$. By smoothness of $\partial\Omega$, for $n$ large enough, $U=B_1$ is an axisymmetric open neighborhood of the origin satisfying the assumptions in Remark \ref{remmovingboundf} (since there is no loss of generality to assume that $B_1\cap\Omega_n \to B^+_1$). Considering as above the rescaled maps $Q_n(x):=Q_{\lambda}(x_0+r_n x)$, we infer from \eqref{BdMonForm} and a rescaling of variables that $\mathcal{E}_{\lambda r_n^2}(Q_n,B_{1}\cap\Omega_n)$ remains bounded independently of $n$. Moreover, since $Q_n(x)=Q_{\rm b}(x_0+r_n x)$ for $x\in B_1\cap\partial\Omega_n$, we have $\|Q_n - Q_{\rm b}(x_0)\|_{L^\infty(B_1\cap\partial\Omega_n)}+\|\nabla Q_n\|_{L^\infty(B_1\cap\partial\Omega_n)}\leq C r_n\to 0$. We can thus apply Remark  \ref{remmovingboundf} to find a (not relabeled) subsequence and $Q_*\in W^{1,2}_{\rm sym}(B_1^+;\mathbb{S}^4)$ satisfying $Q_*=Q_{\rm b}(x_0)$ on $B_1\cap\{x_3=0\}$ such that $Q_n\circ\Phi_n \to Q_*$ strongly in $W^{1,2}(B_\rho^+)$ and $Q_n \to Q_*$ strongly in $W_{\rm loc}^{1,2}(B_\rho^+)$ for every $\rho\in(0,1)$, where the diffeomorphisms $\Phi_n$ satisfy $\Phi_n(B_1^+)=B_1\cap\Omega_n$ and $\|\Phi_n -{\rm id}\|_{C^1(B_2)}\to 0$.  

Rescaling variables, we deduce from \eqref{corELeq} that $Q_n$ satisfies equation \eqref{ELeqsym} in $B_1\cap\Omega_n$ with $\lambda r_n^2$ in place of $\lambda$. In view of the strong $W^{1,2}$-convergence of $Q_n$, we deduce that $Q_*$ is a weakly harmonic map in $B_1^+$. On the other hand, letting $n\to\infty$ in the monotonicity formula satisfied by $Q_n$ as in \cite[Proof of Proposition 3.5]{DMP1}, we infer that $Q_*$ is homogeneous of degree zero. 
As a consequence, $Q_*(x)=\omega(\frac{x}{|x|})$ where $\omega:\mathbb{S}^2_+\to\mathbb{S}^4$ is weakly harmonic and satisfies $\omega(x)=Q_{\rm b}(x_0)$ on $\partial \mathbb{S}^2_+$ (here we have set $\mathbb{S}^2_+:=\mathbb{S}^2\cap\{x_3>0\}$). Exactly as in \cite[Proof of Proposition 3.7]{DMP1}, it follows that $\omega$ is constant, and hence $Q_*\equiv Q_{\rm b}(x_0)$. From the strong convergence  $Q_n\circ\Phi_n \to Q_*$ in $W^{1,2}(B^+_{1/2})$, we easily deduce that $\mathcal{E}_{\lambda r_n^2}(Q_n,B_{1/2}\cap\Omega_n)\to \mathcal{E}_{0}(Q_*,B^+_{1/2})=0$, so that 
$$\lim_{r\to 0}\frac{1}{r}\mathcal{E}_{\lambda}(Q_\lambda,B_r(x_0)\cap\Omega)= \lim_{n\to \infty}\frac{2}{r_n}\mathcal{E}_{\lambda}(Q_\lambda,B_{r_n/2}(x_0)\cap\Omega)
=\lim_{n\to \infty}2\mathcal{E}_{\lambda r_n^2}(Q_n,B_{1/2}\cap\Omega_n)=0\,, $$
which completes the proof. 
\end{proof}

\begin{proposition}\label{vanishdensbdry}
If $Q_\lambda$ is a minimizer of $\mathcal{E}_\lambda$ over $\mathcal{A}^{\rm sym}_{Q_{\rm b}}(\Omega)$, then 
$$\lim_{r\to 0}\frac{1}{r}\mathcal{E}_{\lambda}(Q_\lambda,B_r(x_0)\cap\Omega)=0$$
for every $x_0\in\overline\Omega\setminus\{\text{$x_3$-axis}\}$. 
\end{proposition}

\begin{proof}
By $\mathbb{S}^1$-equivariance of $Q_\lambda$ and the invariance under translations, we can assume without loss of generality that $x_0$ belongs to the $x_1$-axis. We set $r_0:=|x_0|$. Using the cylindrical coordinates $x=(\rho\cos(\phi),\rho\sin(\phi),x_3)$ with $\rho>0$ and $\phi\in[0,2\pi)$, we observe that  for every $r\in(0,r_0)$, 
$$B_r(x_0)\subset G_r(x_0):=\bigcup_{\phi\in(-\phi_r,\phi_r)} R_\phi\cdot D_r(x_0)\,, $$ 
where $D_r(x_0):=\big\{x=(\rho,0,x_3): (\rho-r_0)^2+x_3^2<r^2\big\}$ and $\phi_r:=\arcsin(r/r_0)$, 
therefore  $\mathcal{E}_{\lambda}(Q_\lambda,B_r(x_0)\cap\Omega)\leq \mathcal{E}_{\lambda}(Q_\lambda,G_r(x_0)\cap\Omega)$. Combining the $\mathbb{S}^1$-invariance of the energy density, the equivariance of $Q_\lambda$ and Fubini's Theorem, we infer that for every $r\in(0,r_0/2)$, 
$$\frac{1}{r}\mathcal{E}_{\lambda}(Q_\lambda,G_r(x_0)\cap\Omega)\leq C r_0^{-1}\int_{D_r(x_0)\cap\mathcal{D}_\Omega}\Bigg(|\partial_\rho Q_\lambda|^2+\frac{1}{\rho^2}|Q_\lambda^2|+|\partial_{x_3}Q_\lambda|^2+\lambda W(Q_\lambda))\Bigg) \rho\,d\rho dx_3$$
(where $\mathcal{D}_\Omega$ is the section of $\Omega$ with the plane $x_2=0$, see Definition \ref{domain-section}). Since the measure of $D_r(x_0)\cap\mathcal{D}_\Omega$ goes to zero as $r\to0$, the conclusion follows. 
\end{proof}

\subsection{Partial regularity}

\begin{proof}[Proof of Theorem \ref{thm:partial-regularity}, part 1.]
By the monotonicity formulas established in Proposition \ref{corbdmonotform} (and the fact that $W(Q_\lambda)$ is bounded), the limit 

\begin{equation}
\label{defdensity}
\boldsymbol{\Theta}(Q_\lambda,x_0):=\lim_{r\to 0}\frac{1}{2r}\int_{B_r(x_0)\cap\Omega}|\nabla Q_\lambda|^2\,dx= \lim_{r\to 0}\frac{1}{r}\mathcal{E}_\lambda(Q_\lambda,B_r(x_0)\cap\Omega)
\end{equation}
exists at every $x_0\in \overline\Omega$.

Combining Proposition \ref{corbdmonotform} and \cite[Lemma 2.6]{DMP1} with \cite[Corollary 2.19]{DMP1} we obtain the existence of a universal constant $\boldsymbol{\eps}_{\rm in}>0$ such that for every $x_0\in \Omega$, the condition $\boldsymbol{\Theta}(Q_\lambda,x_0)< \boldsymbol{\eps}_{\rm in}$ implies $Q_\lambda\in C^\omega(B_\rho(x_0))$ for some radius $\rho>0$  depending only on $x_0$ and $\lambda$. In particular, for every $x_0\in \Omega$ the assumption $\boldsymbol{\Theta}(Q_\lambda,x_0)< \boldsymbol{\eps}_{\rm in}$ implies $\boldsymbol{\Theta}(Q_\lambda,x)=0$ for every $x \in B_\rho(x_0)$. 

For points on the boundary, we invoke \cite[Lemma 2.10]{DMP1} and \cite[Corollary 2.20]{DMP1} in place  of \cite[Lemma 2.6]{DMP1} and \cite[Corollary 2.19]{DMP1}, respectively. This yields the existence of a constant $\boldsymbol{\eps}_{\rm bd}>0$ depending only on $\Omega$ and $Q_{\rm b}$ such that for every $x_0\in\partial\Omega$, the  condition 
$\boldsymbol{\Theta}(Q_\lambda,x_0)< \boldsymbol{\eps}_{\rm bd}$ implies $Q_\lambda\in C^{1,\delta}(B_\rho(x_0)\cap\overline\Omega)$. In view of Proposition~\ref{vanishdenspoles} and Proposition~\ref{vanishdensbdry}, we have $\boldsymbol{\Theta}(Q_\lambda,x_0)=0$ for every $x_0\in\partial\Omega$. Consequently, $Q_\lambda$ is of class $C^{1,\delta}$ for every $\delta\in(0,1)$ in a neighborhood of $\partial\Omega$ up to~$\partial\Omega$. In particular $\boldsymbol{\Theta}(Q_\lambda,\cdot)=0$ in a neighborhood of $\partial\Omega$ up to $\partial\Omega$.  Moreover, \cite[Corollary~2.20]{DMP1} tells us that {\sl (i)} if $Q_{\rm b}\in C^{2,\delta}(\partial\Omega)$ for some $\delta>0$, then $Q_\lambda$ is of class $C^{2,\delta}$ in a neighborhood of $\partial\Omega$ up to $\partial\Omega$; {\sl (ii)} if $\partial\Omega$ is real-analytic and $Q_{\rm b}\in C^\omega(\partial\Omega)$, then $Q_\lambda$ is of class $C^\omega$ in a neighborhood of $\partial\Omega$ up to $\partial\Omega$. 
 
As a consequence of the discussion above,  the set $\Sigma:=\big\{x\in\overline \Omega : \boldsymbol{\Theta}(Q_\lambda,x)>0\big\}$ is a closed set which is contained in $\Omega$, and 
$\Sigma=\big\{x\in\overline \Omega : \boldsymbol{\Theta}(Q_\lambda,x)\geq \boldsymbol{\eps}_{\rm in}\big\}$. In view of Proposition  \ref{vanishdensbdry}, we also have 
$\Sigma\subset \Omega\cap\{x_3\text{-axis}\}$. Since we have proved the announced regularity in $\overline\Omega\setminus\Sigma$, it now remains to show that $\Sigma$ is a finite set. Since $\Sigma$ is a compact set, it is enough to prove that all the points of $\Sigma$ are isolated. We argue by contradiction following a somehow classical argument (see e.g. \cite[Section 3.4]{Simon}). Assume that there exist  $\bar x\in\Sigma$ and a sequence $\{x_n\}\subset \Sigma\setminus\{\bar x\}$ such that $x_n\to \bar x$. Set $r_n:=2|x_n-\bar x|$ and define (for $n$ large enough) $Q_n\in W^{1,2}_{\rm sym}(B_1;\mathbb{S}^4)$ by setting $Q_n(x):=Q_\lambda(\bar x+r_n x)$. 
According to Proposition \ref{compblowup}, there exist a (not relabeled) subsequence and $Q_*\in W^{1,2}_{\rm sym, loc}(\mathbb{R}^3;\mathbb{S}^4)$ such that $Q_n\to Q_*$ strongly in $W^{1,2}_{\rm loc}(\R^3)$ and $Q_*$ is degree-zero homogeneous and weakly harmonic. Extracting a further subsequence if necessary, we  can assume that $r_n^{-1}(x_n-\bar x)=(0,0,1/2)=:a$ for every $n$ (or, alternatively, $r_n^{-1}(x_n-\bar x)=-a$ for every $n$, a case for which the argument below is the same, up to obvious modifications).  

As recalled in Sec.~\ref{sec:stability}, if $Q_*\in W^{1,2}_{\rm sym, loc}(\mathbb{R}^3;\mathbb{S}^4)$ is a degree-zero homogeneous weakly harmonic map then $Q_*(x)=  \omega\left( \frac{x}{\abs{x}}\right)$, for some harmonic sphere $\omega \in C^\infty_{\rm sym}(\mathbb{S}^2;\mathbb{S}^4)$. Thus $Q_*\in C^\infty(\R^3 \setminus \{0\};\mathbb{S}^4)$ and in turn 
$\boldsymbol{\Theta}(Q_*,a)=0$.
In view of this property, we can find a radius $\rho_*\in(0,1/2)$ such that $\frac{1}{\rho_*}\mathcal{E}_0(Q_*,B_{\rho_*}(a))\leq \boldsymbol{\eps}_{\rm in}/2$. Once again, by strong $W^{1,2}$-convergence of $Q_n$, we deduce that $\frac{1}{\rho_*}\mathcal{E}_{\lambda r_n^2}(Q_n,B_{\rho_*}(a))< \boldsymbol{\eps}_{\rm in}$ for $n$ large enough. Scaling back, it implies that $\frac{1}{\rho_*r_n}\mathcal{E}_{\lambda}(Q_\lambda,B_{\rho_*r_n}(x_n))< \boldsymbol{\eps}_{\rm in}$ for $n$ large enough, and thus $\boldsymbol{\Theta}(Q_\lambda,x_n)=0$. In other words, $x_n\not\in\Sigma$ for $n$ large enough, a contradiction. 
\end{proof}

\vskip10pt

\subsection{Uniqueness of tangent maps at isolated singularities}\label{subsec:Loj}
In this subsection we prove the second part of Theorem~\ref{thm:partial-regularity}, concerning the asymptotic decay of a singular minimizer to a unique tangent map at any of its isolated singular points. This property
can be regarded as a consequence of the fundamental result  from the paper \cite{Sim83} and its further improvements and simplifications by the same author in \cite[Chapter 3]{Simon} (see also \cite[Chapter 2.5]{LiWa2} for an account on the subject).
 
 Since \cite{Sim83}, the key tool to obtain this property is the celebrated Simon-{\L}ojasiewicz inequality recalled in Proposition~\ref{prop:Loj-Sim} and, as in \cite[Theorem~2.6.3]{LiWa2} for the harmonic map case, the power-type decay will depend on its validity with the optimal exponent $s=1$. Note that in our setting this validity is not obvious, since the space of harmonic spheres $\mathcal{H}\mbox{arm} (\mathbb{S}^2;\mathbb{S}^4)$ is not a smooth submanifold of $C^3 (\mathbb{S}^2;\mathbb{S}^4)$ (indeed, according to \cite{Verdier} it is just a singular complex variety in the sense of Algebraic Geometry) and even the integrability property of the Jacobi fields (see \cite[Chapter 3.14]{Simon} and \cite[Chapter 2.6]{LiWa2} for explanations) along any of its element may fail because of the results in \cite{LemaireWood}. However, as we detail below, in the present case the classification of minimizing equivariant tangent maps from Sec.~\ref{sec:stability} allows to restrict the attention to the stratum  $\mathcal{H}\mbox{arm}_1 (\mathbb{S}^2;\mathbb{S}^4)$ of harmonic spheres with energy $4\pi$ which is a nice analytic manifold, whence the integrability condition is obviously satisfied and the  Simon-{\L}ojasiewicz inequality \eqref{eq:Loj-Sim} holds with the optimal exponent. 
 
 Our proof of the asymptotic decay is a modification of the simplified argument from \cite[Chapter~3.15]{Simon} for harmonic maps, taking into account the optimal exponent in \eqref{eq:Loj-Sim} but without using the integrability property at linearized level as in \cite[Lemma 2.6.5]{LiWa2} or \cite[Part II, proof of Theorem 6.6]{Sim84}. Here, instead, an elementary iterative argument gives at once the power-type decay of the radial derivative keeping the rescaled map at bounded small distance from any of its asymptotic limit. Then the $L^2$-decay to a unique limit, and in turn the $C^2$-decay, follow, still with a power-type decay rate which is however not optimal.

The following preliminary result allows to classify the possible blow-up limit and to identify a first good approximation at some sufficiently small scale.

\begin{lemma}\label{lemma:min-tang-maps}
	Let $Q_\lambda$, $\bar{x}$, $\{r_n\}$, $Q_{\lambda}^{\bar{x},r_n}$ and $Q_*$ be as in Proposition~\ref{compblowup} and suppose in addition $\bar{x} \in \Sigma = \rmsing Q_\lambda$. Then $Q_* = Q^{(\alpha)}$, where $Q^{(\alpha)}$ is one of the maps described by \eqref{stableblowups}. In particular, up to subsequences 
	\[ \lim_{n \to \infty} \left\{\int_{B_1} \left| \frac{\partial Q_{\lambda}^{\bar{x},r_n} }{\partial |x-\bar{x}|}\right|^2 \frac{dx}{|x-\bar{x}|} + \int_{B_1 \setminus B_{1/2}} |Q_{\lambda}^{\bar{x},r_n}-Q_* |^2 dx\right\}=0 \, . \]
\end{lemma}

\begin{proof}
As $Q_\lambda$ is a minimizer, we can apply the monotonicity formula \eqref{IntMonForm} with $\rho=r_n \to 0$ to conclude that  $\left| \frac{\partial Q_\lambda }{\partial |x-\bar{x}|}\right|^2 \frac{1}{|x|}$ is integrable near $\bar{x}$, hence

\[ \lim_{n \to \infty} \int_{B_1} \left| \frac{\partial Q_{\lambda}^{\bar{x},r_n} }{\partial |x-\bar{x}|}\right|^2 \frac{dx}{|x-\bar{x}|} 
=
 \lim_{n \to \infty} \int_{B_{r_n}} \left| \frac{\partial Q_\lambda }{\partial |x-\bar{x}|}\right|^2 \frac{dx}{|x-\bar{x}|}=0 \, . 
\]
 According to Proposition~\ref{compblowup}, for every sequence $r_n \to 0$, the sequence $Q_{\lambda}^{\bar{x},r_n}$ has a subsequence converging in $W^{1,2}_{\rm loc}(\R^3)$ to $Q_*$, where $Q_*$ belongs to $W^{1,2}_{\rm sym,loc}(\R^3;\bbS^4)$ and it is a 0-homogeneous weakly harmonic map minimizing the Dirichlet energy with respect to $\bbS^1$-equivariant compactly supported perturbations. 
 Applying Corollary \ref{formulablowups}, the conclusion follows.
\end{proof}

In view of the previous lemma and according to Theorem~\ref{thm:stability} and Corollary~\ref{formulablowups},
we see that the possible minimizing tangent maps correspond to  the set of equivariant harmonic spheres given by
\begin{equation}
\label{harmstar}
\mathcal{H}\mbox{arm}_*(\mathbb{S}^2;\mathbb{S}^4):= \{ \pm R_\alpha \cdot \boldsymbol{\omega}^{(1)}_{\rm eq} \, , \, \,   R_\alpha \in \mathbb{S}^1 \} \subset \mathcal{H}\mbox{arm} (\mathbb{S}^2;\mathbb{S}^4) \, . 
\end{equation}
It follows from Lemma \ref{lemma1calabi} that the space of harmonic spheres can be 
decomposed according to the values of the energy \eqref{tanenergy}, i.e, 
\[ \mathcal{H}\mbox{arm} (\mathbb{S}^2;\mathbb{S}^4)= \bigcup_{d=0}^\infty \mathcal{H}\mbox{arm}_d (\mathbb{S}^2;\mathbb{S}^4) \, , \quad \mathcal{H}\mbox{arm}_d (\mathbb{S}^2;\mathbb{S}^4)
=\{ \omega \in \mathcal{H}\mbox{arm} (\mathbb{S}^2;\mathbb{S}^4) \,\,\,\mbox{s.t.}\,\, E(\omega) =4\pi d \,   \} \, .\]
Note that  $\mathcal{H}\mbox{arm}_*(\mathbb{S}^2;\mathbb{S}^4) \simeq \mathbb{S}^1 \cup \mathbb{S}^1$, in addition
\begin{equation}
\label{harmstarinharmone}
 \mathcal{H}\mbox{arm}_*(\mathbb{S}^2;\mathbb{S}^4) \subset \mathcal{H}\mbox{arm}_1(\mathbb{S}^2;\mathbb{S}^4) \subset C^3(\mathbb{S}^2;\mathbb{S}^4) \, ,
\end{equation}
and  $\mathcal{H}\mbox{arm}_1$ is a $C^1$-closed subset since the energy is continuous in the $C^1$-topology. 

The following fact is well known from \cite{Verdier}. 
\begin{lemma}
\label{analyticmanifold}
 $\mathcal{H}\mbox{arm}_1 \subset C^3(\mathbb{S}^2;\mathbb{S}^4) $ is a finite dimensional real-analytic submanifold.
\end{lemma}
\begin{proof}
We sketch an elementary proof for the reader's convenience. First notice that by Lemma~\ref{lemma1calabi} every  $\omega \in \mathcal{H}\mbox{arm}_1(\mathbb{S}^2;\mathbb{S}^4)$ is not linearly full, it has three dimensional image and its energy is $4\pi$. Thus, it is a harmonic sphere into $\mathbb{S}^2$ with energy $4\pi$ embedded isometrically along a 3-plane in $\mathcal{S}_0$. In view of \cite{Le} we have $\omega=A \circ \Phi$, where $\Phi \in {\rm Conf^+}(\mathbb{S}^2)$ is an orientation preserving conformal diffeomorphism and $A \in {\rm Isom} (\mathbb{R}^3;\mathcal{S}_0)$.          
The map 
\[ {\rm Isom} (\mathbb{R}^3;\mathcal{S}_0) \times {\rm Conf^+}(\mathbb{S}^2)  \ni (A, \Phi) \longrightarrow A \circ \Phi \in \mathcal{H}\mbox{arm}_1(\mathbb{S}^2;\mathbb{S}^4) \]
is clearly smooth and surjective, moreover is constant along the $\SO(3)$-orbits of the diagonal action on ${\rm Isom} (\mathbb{R}^3;\mathcal{S}_0)  \times {\rm Conf^+}(\mathbb{S}^2) $ given by $(A, \Phi) \to (A R^\trans ,R \Phi)$, $R \in \SO(3)$. Since the previous representation of $\omega$ in terms of $(A,\Phi)$ is clearly unique up to the choice of an orthonormal basis in ${\rm Ran}\,\omega$ we see that 
\[ \mathcal{H}\mbox{arm}_1(\mathbb{S}^2;\mathbb{S}^4)\simeq \left(  {\rm Isom} (\mathbb{R}^3;\mathcal{S}_0)  \times {\rm Conf^+}(\mathbb{S}^2) \right)/ \SO(3) \, ,\] 
and the quotient has a natural structure of real-analytic manifold because the action is free and properly discontinuous. Finally the map $(A, \Phi) \to A \circ \Phi$ gives an analytic embedding of the quotient into $C^3(\mathbb{S}^2;\mathbb{S}^4)$.
\end{proof}

\begin{remark}
\label{jacobiintegrability}	
It follows from the previous lemma and the $C^1$-continuity of the Dirichlet energy $E$ in \eqref{tanenergy} that for any $\omega \in \mathcal{H}\mbox{arm}_1(\mathbb{S}^2;\mathbb{S}^4)$ there exists $\gamma>0$ such that $\psi \in \mathcal{H}\mbox{arm}(\mathbb{S}^2;\mathbb{S}^4)$ and $\| \psi -\omega \|_{C^3} <\gamma$ yields $\psi \in \mathcal{H}\mbox{arm}_1(\mathbb{S}^2;\mathbb{S}^4)$, therefore $C^3$-close critical points $\psi$ belong to a small neighborhood of an analytic manifold passing through $\omega$.
\end{remark}

As a consequence of the previous discussion we see that the integrability assumption in \cite[Chapter 3.14]{Simon} are satisfied and we can finally recall the celebrated {\L}ojasiewicz-Simon inequality for the Dirichlet energy functional on $C^3(\mathbb{S}^2;\mathbb{S}^4)$ with optimal exponent around any harmonic sphere of energy $4\pi$.

\begin{proposition}\label{prop:Loj-Sim}
	Let $\omega\in \mathcal{H}\mbox{arm}_1(\bbS^2;\bbS^4)$ be a harmonic map. Then there are $C > 0$ and $\gamma \in (0,1)$ such that
	\begin{equation}\label{eq:Loj-Sim}
		\abs{E(\psi) - E(\omega) } \leq C \| \calM(\psi) \|^2_{L^2(\bbS^2)}\,,
	\end{equation}
	for any $\psi \in C^3(\bbS^2;\bbS^4)$ such that $\|\psi-\omega\|_{C^3(\bbS^2)} < \gamma$, where $\calM(\psi)$ denotes the {\sl tension field} for a map $\psi \in C^3(\bbS^2;\bbS^4)$, i.e.,
\[
	\calM(\psi) := \Delta_T \psi + \abs{\nabla_T \psi}^2 \psi\,.
\]
\end{proposition}
For a proof of a weaker analogue of Proposition~\ref{prop:Loj-Sim}, in the general case of a real-analytic compact target manifold $N$, we refer the interested reader to \cite[Section~3.14]{Simon}. Under integrability assumptions as the one in Remark \ref{jacobiintegrability}, the generalization of the optimal inequality \eqref{eq:Loj-Sim} is given by \cite[Chapter 3.14, page 82, inequality (xiv)]{Simon}.

The next result gives the necessary a priori bounds of minimizers around isolated singularities to show the convergence to a unique tangent map.

\begin{proposition}\label{prop:gradient-estimates}
	Let $r>0$, let $\bar{x}$ be a point on the $x_3$-axis and let $Q_\lambda \in W^{1,2}_{\rm sym}(B_r(\bar{x});\bbS^4)$ be a minimizer of $\mathcal{E}_\lambda$ with respect to $\bbS^1$-equivariant compactly supported perturbations. Suppose in addition $\bar{x} \in \Sigma$ and $Q_\lambda \in C^\infty(B_r(\bar{x}) \setminus \{\bar{x}\})$. Then, for every $k \in \N$,
	\begin{equation}\label{eq:gradient-estimates}
		\sup_{x \in B_{r/2}(\bar{x}) \setminus \{\bar{x}\}} \abs{x-\bar{x}}^k \abs{\nabla^k Q_\lambda(x)} \leq C_k\,,
	\end{equation}
where $C_k$ is a positive constant depending only on $k$ and on $r$.
\end{proposition}

\begin{proof}
Given the sequence $r_n=r 2^{-n} \downarrow 0$, we set $Q_n(x) := Q_\lambda(\bar{x} + r_n  x)$ for $x \in B_{r/r_n}$. By Proposition~\ref{compblowup}, the sequence $\{Q_n\}$ has a (not relabeled) subsequence converging strongly in $W^{1,2}_{\rm loc}(\R^3)$ to some minimizing tangent map $Q_*$  (which is one of those maps given in Lemma~\ref{lemma:min-tang-maps}), therefore in particular we have strong convergence $Q_n \to Q_*$ in $W^{1,2}(B_{3/2} \setminus B_{1/3})$. Since $Q_*$ has only an isolated singularity at the origin and it is 0-degree homogeneous, we can find $0<\rho <1/6$ so that for every $y \in \overline{B_{1}} \setminus B_{1/2}$ we have $\frac{1}{\rho} \int_{B_{\rho}(y)} \abs{\nabla Q_*}^2 \, dx \leq \boldsymbol\eps_{\rm in}/2$, where $\boldsymbol\eps_{\rm in}$ is the critical parameter in \cite[Corollary~2.19]{DMP1}. Note that $\rho$ depends only on $\abs{ \nabla Q_*}^2$, therefore it does not depend on the chosen subsequence and on which map $Q_*$ really is among those in \eqref{stableblowups}. Thanks to strong convergence, we have $\frac{1}{\rho} \int_{B_\rho(y)} \abs{\nabla Q_n}^2\,dx \leq \boldsymbol\eps_{\rm in}$ for all $n$ large enough uniformly over  $y \in \overline{B_{1}} \setminus B_{1/2}$, therefore \cite[Corollary~2.19]{DMP1} gives $ \rho^k \| \nabla^k Q_n \|_{L^\infty(B_{\rho/8}(y))} \leq C_k$ for all sufficiently large $n$ uniformly over $y \in \overline{B_{1}} \setminus B_{1/2}$. Thus the same estimate holds for every $n$, because of the smoothness of each map $Q_n$ away from the origin, for a possibly larger constant still uniform with respect to $y \in \overline{B_{1}} \setminus B_{1/2}$. Thus, by covering $\overline{B_1} \setminus B_{1/2}$ with balls of radius $\rho/8$, we have $\sup_{y \in \overline{B_1} \setminus B_{1/2}} \abs{\nabla^k Q_n(y)} \leq C_k$, where $C_k$ does not depend on $n$. Since $B_{r/2}(\bar{x})\setminus \{\bar{x}\} = \cup_{n=1}^\infty A_n$, where $A_n = \{ x  : 2^{-(n+1)}r \leq \abs{x-\bar{x}} < 2^{-n} r\}$ are dyadic annuli around $\bar{x}$, scaling back the previous inequalities we have $2^{-kn} r^k \abs{\nabla^k Q_\lambda(x)} \leq C_k$ for every $x \in A_n$, for every $n \in \N$, and in turn we deduce $\abs{x-\bar{x}}^k \abs{\nabla^k Q_\lambda(x)} \leq C_k$ for every $x \in A_n$, for every $n \geq 1$, hence                                                                           the conclusion follows.
\end{proof}

As a corollary, a simple interpolation argument gives the following result which turns $L^2$-closeness to a tangent map into $C^3$-closeness and which will allow to let the Simon-{\L}ojasiewicz inequality \eqref{eq:Loj-Sim} come into play.

\begin{corollary}[$L^2$-closeness $ \Longrightarrow C^3$-closeness]\label{cor:interpolationC3-L2}
	Let $r>0$, let $\bar{x}$ be a point on the $x_3$-axis and let $Q_\lambda \in W^{1,2}_{\rm sym}(B_r(\bar{x});\bbS^4)$ be a minimizer of $\mathcal{E}_\lambda$ with respect to  $\bbS^1$-equivariant compactly supported perturbations such that $Q_\lambda \in C^\infty(B_r(\bar{x}) \setminus \{\bar{x}\})$. There exists $C>0$ such that for any rescaled map $\widetilde{Q}_\rho (x):= Q_\lambda(\bar{x}+\rho x)$, $0<\rho\leq r/3$, and any minimizing tangent map $Q_*$ in the sense of Proposition~\ref{compblowup} we have
	\begin{equation}\label{eq:interpolationC3-L2}
		\| \widetilde{Q}_\rho - Q_* \|_{C^3(\overline{B_{3/2}} \setminus B_{3/4})} \leq C \| \widetilde{Q}_\rho-Q_* \|_{L^2(B_{3/2} \setminus B_{3/4})}^{1/6}\,. 
	\end{equation}
	As a consequence, for any $\gamma>0$ there exists $\eta>0$ such that $\| \widetilde{Q}_\rho - Q_* \|_{C^3(\overline{B_{3/2}} \setminus B_{3/4})}<\gamma$ whenever $ \| \widetilde{Q}_\rho-Q_* \|_{L^2(B_{3/2} \setminus B_{3/4})}<\eta$.
	
\end{corollary}

\begin{proof}
 Since $B_1 \subset \R^3$ we have $W^{2,2}(B_{3/2} \setminus B_{3/4}) \subset C^{0,1/2} (\overline{B_{3/2}} \setminus B_{3/4})$ and in turn $W^{5,2}(B_{3/2} \setminus B_{3/4}) \hookrightarrow C^3(\overline{B_{3/2}} \setminus B_{3/4})$ with compact embedding, hence 
	\[
		\| \widetilde{Q}_\rho - Q_* \|_{C^3(\overline{B_{3/2}} \setminus B_{3/4})} \leq C \| \widetilde{Q}_\rho-Q_* \|_{W^{5,2}(B_{3/2} \setminus B_{3/4})}\,,
	\]
	for some constant $C > 0$ independent of $\rho$. On the other hand, classical interpolation results among $W^{k,2}$-spaces give
	\[
		\| \widetilde{Q}_\rho-Q_* \|_{W^{5,2}(B_{3/2} \setminus B_{3/4})} \leq C \| \widetilde{Q}_\rho-Q_* \|_{W^{6,2}(B_{3/2} \setminus B_{3/4})}^{5/6} \| \widetilde{Q}_\rho-Q^*\|_{L^2(B_{3/2} \setminus B_{3/4})}^{1/6}\,,
	\] 
	for some constant $C>0$ independent of $\rho$.
Clearly $C^6 \subset W^{6,2}$ with continuous embedding, 
and the derivative bounds \eqref{eq:gradient-estimates} from Proposition~\ref{prop:gradient-estimates} yield 
$$\| \widetilde{Q}_\rho-Q_* \|_{C^6(\overline{B_{3/2}} \setminus B_{3/4})} \leq C (\| \widetilde{Q}_\rho \|_{C^6(\overline{B_{3/2}} \setminus B_{3/4})} + \|Q_* \|_{C^6(\overline{B_{3/2}} \setminus B_{3/4})}) \leq C \, ,$$ 
for another constant $C>0$ independent of $\rho$. Finally, \eqref{eq:interpolationC3-L2} follows from the previous three inequalities and the final claim follows immediately.
\end{proof}

The next result is the final ingredient in proving uniqueness of tangent maps at isolated singularities. It gives the inductive step to improve $L^2$-closeness to a tangent map from each dyadic scale to the next one assuming  we start the process sufficiently close to a given tangent map.

\begin{proposition}
\label{simoninductivestep}
	Let $r>0$, let $\bar{x}$ be a point on the $x_3$-axis and let $Q_\lambda \in W^{1,2}_{\rm sym}(B_r(\bar{x});\bbS^4)$ be a minimizer of $\mathcal{E}_\lambda$ with respect to  $\bbS^1$-equivariant compactly supported perturbations such that $Q_\lambda \in C^\infty(B_r(\bar{x}) \setminus \{\bar{x}\})$.

Fix $\rho_* \leq r/3$ a small number such that $\int_{B_{\rho_*}} \left| \frac{\partial Q_\lambda }{\partial |x-\bar{x}|}\right|^2 \frac{dx}{|x|} \leq \frac12 $ 
and let $Q_*$ be a minimizing tangent map at $\bar{x}$ as in Proposition~\ref{compblowup}.

There exist $C_*>1$ and $\eta_* \in (0,1)$ depending only on $Q_*$ and $\rho_*$ with the following properties.
If for some $0<\hat{\rho} \leq \rho_*$ the scaled map $\widehat{Q}(x):= Q_\lambda(\bar{x}+ \hat{\rho} x)$ satisfies $\| \widehat{Q}-Q_*\|_{L^2(B_1 \setminus B_{1/2})}<\eta_*$, then

\begin{equation}
\label{simonstep}
\int_{B_{1/2}} \frac{1}{|x|}\abs{\frac{\partial \widehat{Q}}{\partial |x|}}^2\,dx  \leq C_* \left( \int_{B_1 \setminus B_{1/2}} \frac{1}{|x|}\abs{\frac{\partial \widehat{Q}}{\partial |x|}}^2\,dx +   \frac14\hat{\rho}^2 \right)\,.
\end{equation}

\end{proposition}

\begin{proof}
The proof follows closely the one in \cite[page 83-85, inequality (8)]{Simon}, with some modifications to handle the extra terms coming from the potential energy $W$ and to take advantage of the Simon-{\L}ojasiewicz inequality with optimal exponent.

According to Corollary~\ref{formulablowups} and to \eqref{harmstar}-\eqref{harmstarinharmone}, we have $Q_*(x)=\omega \left( \frac{x}{|x|}\right)$ for some harmonic sphere $\omega \in \mathcal{H}\mbox{arm}_*(\mathbb{S}^2;\mathbb{S}^4) \subset \mathcal{H}\mbox{arm}_1(\mathbb{S}^2;\mathbb{S}^4) $. By Proposition~\ref{prop:Loj-Sim} we can choose $\gamma>0$ such that \eqref{eq:Loj-Sim} holds whenever $\psi \in C^3(\mathbb{S}^2;\mathbb{S}^4)$ satisfies $\| \psi-\omega\|_{C^3}<\gamma$. Given $\gamma$ as above, we fix $\eta$ as in Corollary~\ref{cor:interpolationC3-L2} and we set $\eta_*:=  \left(\frac{2}{3}\right)^{3/2} \eta $. 

For $\rho =\frac32 \hat{\rho} \in \left(0,\frac32 \rho_*\right] \subset (0, r/2]$, we consider the scaled map $\widetilde{Q}(x)=\widetilde{Q}_\rho(x)$ as in Corollary \ref{cor:interpolationC3-L2}, so that $\widehat{Q}(x) = \widetilde{Q}\left(\frac{3}{2}x\right)
$ on $B_1$. Clearly the assumption of the proposition yields
\[  \| \widetilde{Q}-Q_*\|_{L^2(B_{3/2} \setminus B_{3/4})}  = \left(\frac{3}{2} \right)^{3/2} \| \widehat{Q} - Q_* \|_{L^2(B_1 \setminus B_{1/2})} < \eta \, , \] 
hence
\begin{equation}\label{eq:est-C3-dist}
	\| \widetilde{Q} - Q_* \|_{L^2(B_{3/2}\setminus B_{3/4})} < \eta \implies \| \widetilde{Q} - Q_* \|_{C^3(B_{5/4} \setminus B_{7/8})} < \gamma 
\end{equation}
because of Corollary~\ref{cor:interpolationC3-L2}, and we are in the position to apply the {\L}ojasiewicz-Simon inequality \eqref{eq:Loj-Sim} to the map $\psi=\widetilde{Q}|_{\partial B_1}$.

Rewriting \eqref{simonstep} in terms of $\widetilde{Q}$, it is clear that to finish the proof we must show that
\begin{equation}\label{simonsteptilde}
\int_{B_1} \frac{1}{|x|}\abs{\frac{\partial \widetilde{Q}}{\partial |x|}}^2\,dx  \leq C_* \left( \int_{B_{3/2} \setminus B_{3/4}} \frac{1}{|x|}\abs{\frac{\partial \widetilde{Q}}{\partial |x|}}^2\,dx + \frac{9}{16} \rho^2 \right), 
\end{equation}
for some $C_*>0$ independent of $\rho \in \left(0, \frac32 \rho_*\right]$.

Setting $\tilde{\lambda} = \rho^2 \lambda$, by the Interior Monotonicity Formula \eqref{IntMonForm} we have
\begin{equation}
\label{MonIdwithTheta}
	\mathcal{E}_{\tilde{\lambda}}(\widetilde{Q},B_1) - \boldsymbol{\Theta}(\widetilde{Q},0) = \int_{B_1} \frac{1}{|x|}\abs{\frac{\partial \widetilde{Q}}{\partial |x|}}^2 \,dx + 2 \tilde{\lambda} \int_0^1 \left(\frac{1}{t^2} \int_{B_t} W(\widetilde{Q}) \,dx \right)\,dt\,,
\end{equation}
where the density $\boldsymbol{\Theta}(\widetilde{Q},0)$ is defined as in \eqref{defdensity} above.

Arguing by approximation as in the proof of the monotonicity formula above, the first identity from step 2 in the proof of \cite[Proposition~2.4]{DMP1} yields
\[
	\frac{1}{2} \int_{\partial B_1} \left( |\nabla \widetilde{Q}|^2 - 2 \abs{\frac{\partial \widetilde{Q}}{\partial |x|}}^2 \right) \, d\mathcal{H}^2 + \tilde{\lambda} \int_{\partial B_1} W(\widetilde{Q})\, d\mathcal{H}^2 = \mathcal{E}_{\tilde{\lambda}}(\widetilde{Q}, B_1) + 2 \tilde{\lambda} \int_{B_1} W(\widetilde{Q})\,dx.
\]
Hence,
\[
	\mathcal{E}_{\tilde{\lambda}}(\widetilde{Q},B_1) \leq \frac{1}{2}\int_{\partial B_1} |\nabla_{\rm tan} \widetilde{Q}|^2 \,d\mathcal{H}^2 + \tilde{\lambda} \int_{\partial B_1} W(\widetilde{Q}) \,d\mathcal{H}^2.
\]
On the other hand,
\[
	\boldsymbol{\Theta}(\widetilde{Q},0) = \boldsymbol{\Theta}(Q_\lambda,\bar{x}) = \boldsymbol{\Theta}(Q_*, 0) = \frac{1}{2} \int_{\partial B_1} \abs{\nabla_{\rm tan} Q_*}^2\,d\mathcal{H}^2\,,
\]
so that the last inequality can be rewritten as
\[
	\mathcal{E}_{\tilde{\lambda}}(\widetilde{Q},B_1) - \boldsymbol{\Theta}(\widetilde{Q},0) \leq \frac{1}{2} \int_{\partial B_1} \left( |\nabla_{\rm tan} \widetilde{Q}|^2 - |\nabla_{\rm tan} Q_*|^2 \right)\,d\mathcal{H}^2 + \tilde{\lambda} \int_{\partial B_1} W(\widetilde{Q}) \,d\mathcal{H}^2\,,
\]
which combined with \eqref{MonIdwithTheta} in turn leads to
\[
	\int_{B_1} \frac{1}{|x|}\abs{\frac{\partial\widetilde{Q}}{\partial |x|}}^2 \,dx \leq \frac{1}{2} \int_{\partial B_1} \left( |\nabla_{\rm tan} \widetilde{Q}|^2 - \abs{\nabla_{\rm tan} Q_*}^2 \right)\,d\mathcal{H}^2 + \tilde{\lambda} \int_{\partial B_1} W(\widetilde{Q}) \,d\mathcal{H}^2\,.
\]
As already mentioned, we can apply \eqref{eq:Loj-Sim} from Proposition~\ref{prop:Loj-Sim} with $\psi = \left.\widetilde{Q} \right.\vert_{\partial B_1=\bbS^2}$ to deduce
\begin{equation}
\label{simonsteptension}
	\int_{B_1} \frac{1}{|x|}\abs{\frac{\partial\widetilde{Q}}{\partial |x|}}^2 \,dx \leq C \| \calM(\widetilde{Q})\|_{L^2(\partial B_1)}^2 + \tilde{\lambda} \int_{\partial B_1} W(\widetilde{Q}) \,d\mathcal{H}^2\,,
\end{equation}
because of the condition $\| \widetilde{Q} - Q_* \|_{L^2(B_{3/2} \setminus B_{3/4})} < \eta$ and of \eqref{eq:est-C3-dist}.

Now, the rescaled map $\widetilde{Q}$ satisfies
\[
	-\Delta \widetilde{Q} = |\nabla \widetilde{Q}|^2 \widetilde{Q} + \tilde{\lambda}\left( \widetilde{Q}^2 - \frac{1}{3}I - {\rm tr}(\widetilde{Q}^3) \widetilde{Q} \right) \,, 
\]
which in spherical coordinates rewrites as
\[
	\frac{1}{|x|^2}\frac{\partial}{\partial |x|}\left( |x|^2 \frac{\partial \widetilde{Q}}{\partial |x|} \right) + \frac{1}{|x|^2} \Delta_{T} \widetilde{Q} + \frac{1}{|x|^2} |\nabla_{\rm tan} \widetilde{Q}|^2 \widetilde{Q} + \abs{\frac{\partial \widetilde{Q}}{\partial{|x|}}}^2 \widetilde{Q} + \tilde{\lambda}\left( \widetilde{Q}^2 - \frac{1}{3}I - {\rm tr}(\widetilde{Q}^3) \widetilde{Q} \right) = 0\,.
\]
Separating the terms with angular derivatives we obtain 
\[
	\calM(\widetilde{Q}) = -\frac{1}{|x|^2}\frac{\partial}{\partial |x|}\left( |x|^2 \frac{\partial \widetilde{Q}}{\partial |x|} \right) -  \abs{\frac{\partial \widetilde{Q}}{\partial{|x|}}}^2 \widetilde{Q} - \tilde{\lambda}\left( \widetilde{Q}^2 - \frac{1}{3}I - {\rm tr}(\widetilde{Q}^3) \widetilde{Q} \right)\,,
\]
which combined with \eqref{simonsteptension} leads to (recall that $\tilde{\lambda} = \rho^2 \lambda$)
\begin{multline*}
	\int_{B_1} \frac{1}{|x|}\abs{\frac{\partial\widetilde{Q}}{\partial |x|}}^2 \leq C \left[ \left( \int_{\partial B_1} \abs{\frac{\partial}{\partial |x|}\left( |x|^2 \frac{\partial \widetilde{Q}}{\partial |x|} \right) }^2 + \abs{\frac{\partial \widetilde{Q}}{\partial{|x|}}}^4 \,d\mathcal{H}^2 + \rho^4 \right) + \rho^2 \right] \\
	\leq C  \left( \int_{\partial B_1} \abs{\frac{\partial}{\partial |x|}\left( |x|^2 \frac{\partial \widetilde{Q}}{\partial |x|} \right) }^2 +\abs{\frac{\partial \widetilde{Q}}{\partial{|x|}}}^4 \,d\mathcal{H}^2 + \rho^2 \right) \,,
\end{multline*}
with $C = C(\lambda,\rho_*, Q_*)$. Therefore, expanding the derivative on the product and applying the gradient bound \eqref{eq:gradient-estimates} on the scaled map $\widetilde{Q}$ we arrive at
\begin{equation}\label{eq:loj-est}
	\int_{B_1} \frac{1}{|x|}\abs{\frac{\partial\widetilde{Q}}{\partial |x|}}^2\,dx  \leq C \left( \int_{\partial B_1} \abs{\frac{\partial}{\partial |x|}\left( |x| \frac{\partial \widetilde{Q}}{\partial |x|} \right) }^2 +\abs{\frac{\partial \widetilde{Q}}{\partial{|x|}}}^2 \,d\mathcal{H}^2 + \rho^2 \right)\,,
\end{equation}
because of our assumption $\| \widetilde{Q} - Q_* \|_{L^2(B_{3/2} \setminus B_{3/4})} < \eta$.

In order to obtain \eqref{simonsteptilde}, we apply elliptic regularity.
For $\sigma \in (3/4,4/3)$ we define 
\[
	\widetilde{Q}_{\sigma}(x) := \widetilde{Q}(\sigma x)\, ,
\]
so that $\widetilde{Q}_\sigma \in C^\infty(B_{3/2} \setminus \overline{B_{3/4}})$ and it is a smooth solution to the rescaled system
\begin{equation}\label{eq:EL-tildeQ-sigma}
	\Delta \widetilde{Q}_{\sigma} + |\nabla \widetilde{Q}_{\sigma}|^2 \widetilde{Q} + \tilde{\lambda} \sigma^2\left( \widetilde{Q}_\sigma^2 - \frac{1}{3}I - {\rm tr}\left( \widetilde{Q}_{\sigma}^3 \right)\widetilde{Q}_{\sigma} \right) = 0\,.
\end{equation}
Since
\[
	\frac{\partial}{\partial \sigma}\left( \widetilde{Q}(\sigma x) \right) = x \cdot \nabla \widetilde{Q}(\sigma x) = \frac{1}{\sigma} x \cdot \nabla{\widetilde{Q}_\sigma(x)} = \frac{|x|}{\sigma} \frac{\partial \widetilde{Q}_\sigma}{\partial |x|}(x) \,,
\]
differentiating \eqref{eq:EL-tildeQ-sigma} with respect to $\sigma$ at $\sigma = 1$ and setting $
	v := |x| \frac{\partial \widetilde{Q}}{\partial |x|}$
yields
\begin{multline*}
	\Delta v + 2 (\nabla \widetilde{Q} : \nabla v) \widetilde{Q} + |\nabla \widetilde{Q}|^2 v \\
	+ \tilde{\lambda}\sigma^2 \left( v \widetilde{Q} + \widetilde{Q} v - {\rm tr}(\widetilde{Q}^3) v - 3 {\rm tr}(\widetilde{Q}^2v)\right) = 2 \tilde{\lambda} \left( \widetilde{Q}^2 - \frac{1}{3}I -{\rm tr}(\widetilde{Q}^3) \widetilde{Q} \right) \, .
\end{multline*}
Thus  $v(x) = x \cdot \nabla \widetilde{Q}(x) $ is a smooth solution in $B_{3/2} \setminus \overline{B_{3/4}}$ of the ellyptic system
\begin{equation}
\label{systemv}
	\calL(v) = \tilde{\lambda}f\,,
\end{equation}
where $\calL(v) = \Delta v + b \cdot \nabla v + c\cdot v$, and

\[
	\| b \|_{C^1(B_{3/2} \setminus \overline{B_{3/4}})} + \|c \|_{C^1(B_{3/2} \setminus \overline{B_{3/4}})} + \|f \|_{C^1(B_{3/2} \setminus \overline{B_{3/4}})}\leq C  
\]
because of the estimate $\|\widetilde{Q} - Q_*\|_{C^3(B_{5/4}\setminus B_{7/8})} < \gamma$.

Applying local $H^2$-regularity theory for linear elliptic system as in \cite[Theorem 4.11]{GiaqMart} in  view of the bounds on the coefficients, we have
\[ \| v\|_{H^2(B_{5/4} \setminus B_{7/8})} \leq C \left( \| v\|_{L^2(B_{3/2} \setminus B_{3/4})}+ \tilde{\lambda}  \| f\|_{L^2(B_{3/2} \setminus B_{3/4})} \right)\leq   C \left( \| v\|_{L^2(B_{3/2} \setminus B_{3/4})}+ \rho^2 \right) \, ,\]
whence the 3d-embedding $H^2 \hookrightarrow C^{0,1/2}$ yields
 \begin{equation}
 \label{holderestv}
 	\| v\|_{C^{0,1/2}(\overline{B_{5/4}} \setminus B_{7/8})} \leq   C \left( \| v\|_{L^2(B_{3/2} \setminus B_{3/4})}+ \rho^2 \right) \, ,
 \end{equation}
 On the other hand, rewriting the first order terms in \eqref{systemv} as $b \cdot \nabla v= \nabla \cdot (b v)- (\nabla \cdot b) v$ and applying Schauder regularity theory for elliptic systems in divergence form as in \cite[Theorem 5.20]{GiaqMart} we obtain

\[
 	\| \nabla v\|_{C^{0,1/2}(\overline{B_{9/8}} \setminus B_{8/9})} \leq   C \left( \| v\|_{C^{0,1/2}(\overline{B_{5/4}} \setminus B_{7/8})}+ \tilde{\lambda}  \| f\|_{C^{0,1/2}(\overline{B_{5/4}} \setminus B_{7/8})} \right) \, ,
 \]
whence \eqref{holderestv} yields

\begin{equation}
\label{C1boundv}
\| v\|_{C^1(\overline{B_{9/8}} \setminus B_{8/9})} 
	 \leq C \left( \|v\|_{L^2(B_{3/2}\setminus b_{3/4})} + \rho \right)\, ,
\end{equation}
where the constant $C$ is independent of the rescaled map $\widetilde{Q}$ and $v=x \cdot \nabla \widetilde{Q}$.

Finally, combining \eqref{C1boundv} with \eqref{eq:loj-est} we easily obtain \eqref{simonsteptilde} and the proof is complete.
\end{proof}

We are finally in the position to conclude the proof of Theorem \ref{thm:partial-regularity}. The proof here differs substantially from \cite{Simon} and \cite{LiWa2}, as it is based on the improved inequality \eqref{simoninductivestep} and an elementary iteration argument.
\begin{proof}[Proof of Theorem~\ref{thm:partial-regularity}, part 2]
Let $\bar{x} \in \Sigma$. We know from the previous subsection that $Q_\lambda$ is (at least) $C^1$-smooth in a neighborhood of the boundary and that $\Sigma$ is a finite set of interior singularities on the symmetry axis, therefore singularities are isolated and we can fix $r > 0$ so that $Q_\lambda \in C^\infty(\overline{B_r(\bar{x})} \setminus \{\bar{x}\})$. We set $\widetilde{Q}_\rho(x) := Q_{\lambda}(\bar{x}+ \rho x)$, $0<\rho \leq \rho_* \leq r/3$ and $\rho_*$ as in Proposition~\ref{simoninductivestep}, hence $\widetilde{Q}_\rho$ is well-defined for $x \in B_2 \setminus \{0\}$ and $\widetilde{Q}_\rho$ is minimizing $\mathcal{E}_{\tilde{\lambda}}$, with $\tilde{\lambda} = \lambda \rho^2$, in $B_2$ with respect to $\bbS^1$-equivariant compactly supported perturbations.

Notice that a simple application of the fundamental theorem of calculus gives
\begin{equation}
\label{L2shortoscillation}
\| \widetilde{Q}_{\sigma'} -\widetilde{Q}_\sigma \|_{L^2(B_1\ \setminus B_{1/2})} \leq  \sqrt{\int_{B_1} \frac1{|x|} \left| \frac{\partial \widetilde{Q}_\sigma}{\partial |x|} \right|^2 dx}=  \sqrt{\int_{B_\sigma(\bar{x})} \frac1{|x|} \left| \frac{\partial Q_\lambda}{\partial |x|} \right|^2 dx}\, , \qquad \sigma/2 \leq \sigma' < \sigma \leq \rho_* \, ,
\end{equation}
so the $L^2$-oscillation between comparable scales tends to zero as $\sigma \to 0$. 

We are going to improve the estimate \eqref{L2shortoscillation} to a power-type decay in terms of $\sigma$ valid for all $0<\sigma' <\sigma\leq \rho_*$ by proving a quantitative decay of the right hand side, at least for $\sigma$ small enough so that inequality \eqref{simonstep} can be applied.

For fixed $0<\rho<\rho^*$ and any $j \in \N$ we define the sequence $\{\widetilde{Q}_{\rho,j} \}_j$ as $\widetilde{Q}_{\rho,j}:=\widetilde{Q}_{2^{-j}\rho}$. In view of Lemma~\ref{lemma:min-tang-maps}, there exists a minimizing tangent map $Q_*$ such that 
\begin{itemize}
\item[i)]
 $\widetilde{Q}_{\rho, j} \to Q_*$ in $L^2(B_1 \setminus B_{1/2})$ as $j \to \infty$ along a subsequence; 
\item[ii)] for each $\eta' \in (0,\eta_*)$, $\eta_*$ as in Proposition \ref{simoninductivestep}, there exist $\bar{j}=\bar{j}(\eta')$ such that
\begin{equation}
\label{initialcondtilde}
  \sqrt{\int_{B_1} \left| \frac{\partial \widetilde{Q}_{\rho,\bar{j}} }{\partial |x|}\right|^2 \frac{dx}{|x|}} + 2^{-\bar{j}}\rho +\| \widetilde{Q}_{\rho,\bar{j}}-Q_*  \|_{L^2(B_1 \setminus B_{1/2})} < \eta' \, . 
\end{equation}

\end{itemize}

Before making explicit the choice for $\eta'$ in terms of $C_*$ and $\eta_*$ from Proposition \ref{simoninductivestep}, we simplify the notation, setting for brevity $\widehat{Q}=\widetilde{Q}_{\rho, \bar{j}}$. From now on we will work with the sequence $\{\widehat{Q}_\ell\}_{\ell \in \N}$, where $\ell \in \N$ is such that $j=\bar{j}+\ell$ and, obviously, $\widehat{Q}_\ell:= \widetilde{Q}_{\rho, \bar{j}+\ell}=\widetilde{Q}_{2^{-(\bar{j}+\ell)}\rho}$. 

We are going to describe the behavior of the whole sequence $\{\widehat{Q}_\ell\}_{\ell \in \N}$, knowing that

\begin{itemize}
\item[1)]
 $\widehat{Q}_\ell \to Q_*$ in $L^2(B_1 \setminus B_{1/2})$ as $\ell \to \infty$ along a subsequence; 
\item[2)] the map $\widehat{Q}_0$ satisfies
\begin{equation}
\label{initialcondhat}
\sqrt{\int_{B_1} \left| \frac{\partial \widehat{Q}_0}  {\partial |x|}\right|^2 \frac{dx}{|x|}}  + 2^{-\bar{j}}\rho +  \| \widehat{Q}_0-Q_*  \|_{L^2(B_1 \setminus B_{1/2})} < \eta' \,, 
\end{equation}
where $\eta'$ is the fixed constant in claim ii) above. 
\end{itemize}

Applying Proposition \ref{simoninductivestep} with $\hat{\rho}= 2^{-(\bar{j}+\ell)}\rho $, if $\widehat{Q}_\ell$ satisfies $\| \widehat{Q}_\ell-Q_*  \|_{L^2(B_1 \setminus B_{1/2})} < \eta_*$ then

\[ \int_{B_{1/2}} \frac{1}{|x|}\abs{\frac{\partial \widehat{Q}_\ell}{\partial |x|}}^2\,dx  \leq C_*  \int_{B_1 \setminus B_{1/2}} \frac{1}{|x|}\abs{\frac{\partial \widehat{Q}_\ell}{\partial |x|}}^2\,dx + \left(2^{-\ell}\right)^2   C_* \frac14 \bar{\rho}^2 \,, \]
where for brevity $\bar{\rho}=2^{-\bar{j}}\rho$.

Now we apply Widman's hole-filling trick. As $C_*>1$ we have $\frac14 \bar{\rho}^2 \leq \bar{\rho}^2 - \frac{C_*+1}{C_*}\left(\frac{\bar{\rho}}{2}\right)^2 $, hence summing to both sides $C_*$ times $  \int_{B_{1/2}} \frac{1}{|x|}\abs{\frac{\partial \widehat{Q}_\ell}{\partial |x|}}^2\,dx= \int_{B_1} \frac{1}{|x|}\abs{\frac{\partial \widehat{Q}_{\ell+1}}{\partial |x|}}^2\,dx $  and rearranging we obtain

\begin{equation}
\label{radderindstep}
\int_{B_1} \frac{1}{|x|}\abs{\frac{\partial \widehat{Q}_{\ell+1}}{\partial |x|}}^2\,dx 
  + \left( 2^{-(\ell+1)}\right)^2\bar{\rho}^2 \leq \left(\frac{C_*}{C_*+1}\right) \left(  \int_{B_1} \frac{1}{|x|}\abs{\frac{\partial \widehat{Q}_\ell}{\partial |x|}}^2\,dx  + \left( 2^{-\ell}\right)^2 \bar{\rho}^2 \right)\,,
\end{equation}
provided $\| \widehat{Q}_\ell - Q_* \|_{L^2(B_1 \setminus B_{1/2})} < \eta_*$. 

Now we set $\vartheta:=\sqrt{\frac{C_*}{C_*+1}} \in (0,1)$ and for $\ell \in \N$ we define two sequences
\begin{equation}
\label{defyellzell}
	y_\ell:=\| \widehat{Q}_\ell - Q_* \|_{L^2(B_1 \setminus B_{1/2})} \, , \qquad z_\ell:=\sqrt{\int_{B_1} \frac{1}{|x|}\abs{\frac{\partial \widehat{Q}_\ell}{\partial |x|}}^2\,dx  + \left( 2^{-\ell}\right)^2 \bar{\rho}^2 } \, .
\end{equation}
Combining \eqref{initialcondhat}, \eqref{L2shortoscillation} with $\sigma'=2^{-(\ell+1)}\bar{\rho}$ and $\sigma=2^{-\ell}\bar{\rho}$ together with triangle inequality, and the iterative estimate \eqref{radderindstep}, we obtain the following three properties valid for each $\ell \geq 0$:
\begin{itemize}
\item[a)] $y_0+z_0 <\eta'$;
\item[b)] $y_{\ell+1}\leq y_\ell+ z_\ell$;
\item[c)]	$y_\ell < \eta_* \Longrightarrow z_{\ell+1} \leq \vartheta z_\ell$.
\end{itemize}
As a consequence, choosing $\eta'=\frac{1-\vartheta}4 \eta_*$, a simple induction argument using a), b) and c) yields the following two inequalities for all $\ell \geq 0$, namely
\begin{equation}
\label{boundsyellzell}	
y_{\ell+1}\leq y_0+ \frac{1-\vartheta^{\ell+1}}{1-\vartheta} z_0 <\eta_* \, , \qquad z_{\ell+1} \leq \vartheta^{\ell+1} z_0\,.
\end{equation}
Going back to the definition of $z_\ell$ and $\widehat{Q}_\ell$, by \eqref{boundsyellzell} we obtain for all $\ell \geq 0$
\[
  \sqrt{\int_{B_{2^{-\ell} }} \frac1{|x|} \left| \frac{\partial \widetilde{Q}_{ \bar{\rho}}}{\partial |x|} \right|^2 dx} =\sqrt{\int_{B_1} \frac1{|x|} \left| \frac{\partial \widetilde{Q}_{2^{-(\bar{j}+\ell)} \rho}}{\partial |x|} \right|^2 dx} =\sqrt{\int_{B_1} \frac1{|x|} \left| \frac{\partial \widehat{Q}_\ell}{\partial |x|} \right|^2 dx} \leq z_\ell \leq \vartheta^\ell z_0 \leq \vartheta^\ell \eta' \, ,
\]
hence, if for fixed $0<\sigma \leq \bar{\rho} $ we choose $\ell$ such that $2^{-(\ell+1)} \bar{\rho} \leq \sigma <2^{-\ell} \bar{\rho} $, then  
\begin{equation}
\label{radderdecay}	
\sqrt{\int_{B_\sigma(\bar{x})} \frac1{|x-\bar{x}|} \left| \frac{\partial Q_\lambda}{\partial |x|} \right|^2 dx} =
\sqrt{\int_{B_1} \frac1{|x|} \left| \frac{\partial \widetilde{Q}_{\sigma }}{\partial |x|} \right|^2 dx}\leq
 C\sigma^{\bar{\nu}} \, , \qquad 0<\sigma \leq \bar{\rho} \, ,
\end{equation}
where $\bar{\nu} \in (0,1)$ is such that $2^{-\bar{\nu}}=\vartheta$ (note that $\vartheta \in (1/2,1)$ since $C_*>1$).

Applying the same comparison argument between dyadic and arbitrary radii and estimating the terms of a telescopic sum through \eqref{radderdecay} and \eqref{L2shortoscillation}, we easily obtain

\[
\| \widetilde{Q}_{2^{-\ell} \bar{\rho}} -\widetilde{Q}_\sigma \|_{L^2(B_1\ \setminus B_{1/2})} \leq  C\sigma^{\bar{\nu}} \, , \qquad  2^{-\ell} \bar{\rho} < \sigma \leq \bar{\rho} \, ,
\]
hence, for fixed $\sigma$, taking the limit $\ell \to \infty$ along the same subsequence chosen above finally gives the $L^2$-decay estimate

\begin{equation}
\label{L2distancedecay}
\| Q_* -\widetilde{Q}_\sigma \|_{L^2(B_1\ \setminus B_{1/2})} \leq  C\sigma^{\bar{\nu}} \, , \qquad  0 < \sigma \leq \bar{\rho} \, .
\end{equation}

Setting $\nu=\bar{\nu}/6$ and combining \eqref{L2distancedecay} with \eqref{eq:interpolationC3-L2} in Corollary \ref{cor:interpolationC3-L2} the conclusion follows.
\end{proof}
\section{Topology of minimizing equivariant configurations}
\label{sec:topology}
In this final section we prove the other results announced in the Introduction, namely the existence of torus solutions given in Theorem~\ref{thm:examples-tori}, the existence of split minimizers in Theorem \ref{thm:examples-split}, and the topological properties for smooth and singular equivariant minimizers of the functional \eqref{LDGenergytilde} presented respectively in Theorem~\ref{topology-sym-torus} and Theorem~\ref{topology-sym-split}. 

In Theorem~\ref{thm:examples-tori} we exhibit the first examples of smooth solution to \eqref{MasterEq} which are minimizers of $\mathcal{E}_\lambda$ in the class of $\bbS^1$-equivariant maps on the unit ball subject to appropriate positive uniaxial smooth boundary conditions. Since the boundary data are topologically equivalent to the radial anchoring, the corresponding minimizers inherit a nontrivial topological structure in the interior. As discussed in the Introduction, the structure of such minimizers resembles that of the torus solutions found in many numerical simulations \cite{SKH,KVZ,KV,GaMk,DR,HQZ}. In particular, they possess a negative uniaxial ring inside the ball (an embedded copy of $\bbS^1$) which is surrounded by biaxial tori and which is mutually linked to the region of positive uniaxiality made up by the boundary of the ball and the vertical axis. 

The structure properties of the solutions in Theorem~\ref{thm:examples-tori} are actually quite robust. With Theorem~\ref{topology-sym} below we show that these are general features of smooth $\bbS^1$-equivariant maps under hypotheses ($HP_0$)-($HP_3$). Thus, they pertain to general smooth equivariant critical points of $\mathcal{E}_\lambda$ and Theorem~\ref{topology-sym-torus} follows as a special case of this more general result. On this basis, we propose a definition of torus solution (see Definition~\ref{def:torus-sol}) that seems natural and consistent with the phenomenological picture emerging from the numerical simulations mentioned above. 

The topological structure of singular minimizers is instead really different. Theorem~\ref{topology-sym-split} deals with singular minimizers in the class of equivariant maps $\mathcal{A}_{Q_{\rm b}}^{\rm sym}(\Omega)$ still assuming smoothness of the domain and of the boundary data and the validity of conditions ($HP_1$)--($HP_3$). Due to Theorem~\ref{thm:partial-regularity}, we know that the singular set of such minimizers is a finite subset of the symmetry axis. In Proposition~\ref{prop:dipoles} we show that, thanks to ($HP_1$)--($HP_3$) and $\bbS^1$-equivariance, it has a more particular structure, namely, it consists of finitely many \emph{dipoles} (see Remark~\ref{rmk:axis} for this terminology). The proof of Theorem~\ref{topology-sym-split} exploits the behavior of tangent maps at isolated singularities and shows in particular that for each regular value of the biaxiality parameter each dipole belongs to spherical connected components of the corresponding biaxiality surface. 

Finally, examples of split minimizers under suitably chosen topologically nontrivial boundary data in the unit ball $B_1$ are provided by Theorem \ref{thm:examples-split}. A very important point here is that singularities appear because they are energetically convenient although not necessary for trivial topological reasons. 
\subsection{Existence and topology of smooth  minimizers}

\label{smoothtopology}
 In this subsection, we explore the topology of biaxial sets of smooth $\bbS^1$-equivariant maps under assumptions ($HP_0$)--($HP_3$). To such maps, all the results in \cite[Section~5]{DMP1} apply. However, $\bbS^1$-equivariance will allow us to give more direct arguments and to obtain more refined information. In particular, in view of the $\mathbb{S}^1$-symmetry we are able to improve \cite[Theorem~5.7, claim 1)]{DMP1}, controlling the genus of the biaxial surfaces, at least for regular values of the signed biaxiality below the critical value $\bar{\beta}$ in ($HP_1$), which must be therefore finite unions of tori (see below for the precise statement). 
 Since this analysis mainly relies on the smoothness property of the configurations, the same qualitative properties will hold for arbitrary equivariant critical points of \eqref{LDGenergytilde}.

The first result of this subsection provides the key step to reveal in any smooth $\bbS^1$-equivariant configuration (assuming ($HP_0$)--($HP_3$) are in force) the phenomenological picture of torus solutions discussed in the Introduction.

\begin{proposition}\label{prop:semidisk}
Let $\Omega \subset \mathbb{R}^3$ be an axisymmetric bounded open set with $C^1$-smooth boundary and let $Q : \overline{\Omega} \to \bbS^4$ be an $\bbS^1$-equivariant map. Suppose that $\Omega$ and $Q$ satisfy assumptions ($HP_0$)-($HP_3$). Then the biaxiality set $\{\beta = -1\}$ of $Q$ contains an invariant circle $\bbS^1$ which is mutually linked to $\partial \mathcal{D}^+_\Omega$, where $\partial\mathcal{D}^+_\Omega$ is the boundary of the simply connected domain $\mathcal{D}^+_\Omega = \Omega \cap \{ x_2 = 0, x_1 > 0\}$ as already defined in Corollary \ref{domains-structure}.
\end{proposition}

\begin{proof} 
By symmetry of $\Omega$ and assumptions ($HP_0$)-($HP_1$), the maximal eigenvalue $\lambda_{\rm max} \equiv \lambda_3$  varies continuously and it is always simple on $\partial \Omega$ and hence on $\partial \mathcal{D}^+_{\Omega}$ (because $Q(x)\equiv \eo$, hence $\lambda_3=2/\sqrt{6}$ and it is simple on the symmetry axis). Note that, in view of assumption ($HP_2$) and Corollary \ref{domains-structure}, the section $\mathcal{D}^+_\Omega$ is simply connected and with piecewise smooth boundary. Notice that the eigenspace map $V_{\rm max} : \partial \Omega \to \R P^2$ is well-defined and smooth because of ($HP_2$), moreover it is equivariant, because $\lambda_{\rm max}(\cdot)$ is invariant and $Q(\cdot)$ is equivariant. As in  \eqref{nontrivialloop-j},
we define $\gamma : \partial \mathcal{D}^+_\Omega \to \R P^2$ as the restriction of the map $V_{\rm max}$ to $\partial \mathcal{D}^+_\Omega$ (extended to be the vertical direction $\eo \in \mathbb{R}P^2$ for every $x \in I=\Omega \cap \{ x_3 \text{ axis} \}$) and we claim that $\gamma$ performs the nontrivial path in $\bbR P^2$. Indeed, suppose the converse: then $\gamma$ would have a continuous extension to $\overline{\mathcal{D}^+_\Omega}$ and in turn a continuous equivariant extension $\widetilde{V} \in C(\overline{\Omega};\mathbb{R}P^2)$ because of \eqref{reconstruction}. Due to $(HP_2)$ the map $\widetilde{V}$ would have a continuous (and equivariant) lifting $\tilde{v} \in C(\overline{\Omega};\mathbb{S}^2)$, in particular at the boundary, where in turn $\deg (\tilde{v},\partial \Omega)=0$. On the other hand, in view of assumption ($HP_3$) any lifting of $\widetilde{V}\vert_{\partial\Omega} = V_{\rm max}$ at the boundary must have odd degree, which gives a contradiction and proves the previous claim.

 Now we claim that there exists in $\mathcal{D}^+_\Omega$ a point $x_0$ at which $\lambda_{2}(x_0) = \lambda_{\rm max}(x_0)$, so that $x_0 \in \mathcal{D}^+_\Omega \cap \{\beta=-1\}$. Note that this fact could be deduced using \cite[Lemma 5.2]{DMP1} which is valid also in the nonsymmetric context, but we prefer to give here a more transparent and elementary argument. Indeed, suppose this is not the case: then $\lambda_{\rm max}$ would be always simple on $\overline{\mathcal{D}^+_\Omega}$. Arguing as above, the eigenspace map $V_{\rm max}$ would be well-defined and smooth on the whole $\overline{\Omega}$, therefore the map $\gamma : \partial \mathcal{D}^+_\Omega \to \R P^2$ defined above, setting $\gamma(x) = V_{\rm max}(x)$, could be extended to a map $\gamma \in C^1(\overline{\mathcal{D}^+_\Omega}; \R P^2)$, hence it would be homotopically equivalent to a constant again because $\mathcal{D}^+_\Omega$ is simply connected. Since $\gamma$ on the boundary is the nontrivial loop in $\pi_1(\R P^2)$, then we have a contradiction and such $x_0 \in \mathcal{D}^+_\Omega \cap \{\beta=-1\}$ must exist. Since $\widetilde\beta\circ Q(x_0) = -1$ and $\widetilde\beta \circ Q$ is an invariant function under $\bbS^1$-action, then we have $\widetilde\beta\circ Q(R x_0) = -1$ for all $R \in \bbS^1$, that is, on the whole orbit of $x_0$ which is an embedded copy of $\bbS^1$. Thus, the negative biaxial set $\{\beta = -1 \}$ contains an embedded copy of $\bbS^1$ and it is clearly mutually linked to $\partial \mathcal{D}^+_\Omega $. 
\end{proof}
 
Proposition~\ref{prop:semidisk} is crucial in the proof of Theorem~\ref{thm:examples-tori} below together with the following auxiliary result. 
\begin{lemma}
\label{toplbenergy}
Let $\mathbb{D}\subset \R^2$ and $\R P^2 \subset \mathbb{S}^4 \subset \mathcal{S}_0$. Let $\mathcal{U}_{2\eta}:= \{\rmdist(\, \cdot \, , \R P^2)<2\eta \} \subset \mathcal{S}_0$, $\eta>0$ small, a tubular neighborhood of $\R P^2$ such that the nearest point projection $\Pi: \mathcal{U}_{2\eta} \to \R P^2$ is well-defined and smooth. There exists $\delta>0$ depending only on $\eta$ such that the following holds. For any $\bar{Q} \in C(\overline{\mathbb{D}};\mathcal{S}_0)\cap W^{1,2}(\mathbb{D};\mathcal{S}_0 )$ such that
 
\begin{itemize}
	\item[1)]  $\int_\mathbb{D} |\nabla \bar{Q}|^2 \, dx< \delta$,
	\item[2)] $\bar{Q} (\partial \mathbb{D}) \subset \mathcal{U}_\eta$,   	
\end{itemize}
 the normalized map $\bar{\gamma} \in C(\partial \mathbb{D}; \R P^2) $ given by $\bar{\gamma}=\Pi \circ \bar{Q}$ satisfies $[ \bar{\gamma}]=0$ in $\pi_1(\R P^2)$.
\end{lemma}
\begin{proof}
We argue by contradiction and assume there exists a sequence $\{\bar{Q}_j\} \subset C(\overline{\mathbb{D}};\mathcal{S}_0)\cap W^{1,2}(\mathbb{D};\mathcal{S}_0 )$ such that $\int_\mathbb{D} |\nabla \bar{Q}_j|^2 \, dx \to 0$ as $j \to \infty$,  $\bar{Q}_j (\partial \mathbb{D}) \subset \mathcal{U}_\eta$ for each $j$ and for the corresponding sequence of ``normalized'' boundary traces  $\{ \bar{\gamma}_j \} \in C(\partial \mathbb{D}; \R P^2) $ given by $\bar{\gamma}_j=\Pi \circ \bar{Q}_j$ we have $[ \bar{\gamma}_j]\neq 0$ in $\pi_1(\R P^2)$ for all $j$.
We replace each $\bar{Q}_j$ with the harmonic extension $\hat{Q}_j$ with values into $\mathcal{S}_0$ of its boundary trace $\bar{\gamma}_j$; by energy minimality, regularity up to the boundary and maximum principle for $\mathcal{S}_0$-valued harmonic functions, we see that $\{\hat{Q}_j\} \subset C(\overline{\mathbb{D}};\mathcal{S}_0)\cap W^{1,2}(\mathbb{D};\mathcal{S}_0 )$ and they satisfy 
\begin{itemize}
	\item[a)]  $\int_\mathbb{D} |\nabla \hat{Q}_j|^2 \, dx \to 0$ as $j \to \infty$,
	\item[b)] $\hat{Q}_j (\partial \mathbb{D}) \subset \mathcal{U}_\eta$ and $| \hat{Q}_j | \leq 1+\eta$ on $\overline{\mathbb{D}}$ for all $j$\,.   	
\end{itemize}
Now we claim that $\hat{Q}_j( \overline{\bbD}) \subset \mathcal{U}_{2\eta}$ for $j$ large enough. Before proving the claim, we show that it yields the desired contradiction. Indeed, assuming the claim for a moment, then the ``normalized'' maps $\Gamma_j =\Pi\circ \hat{Q}_j$ would be well-defined and $\Gamma_j \in C(\overline{\mathbb{D}}; \R P^2)$ for $j$ large enough, whence for such $j$ the maps $ \bar{\gamma}_j={\Gamma_j}|_{\partial \mathbb{D}}$ would be homotopic through $\Gamma_j$ to a constant map $\Gamma_j(0)$ in $C(\mathbb{S}^1; \R P^2)$, a contradiction. Thus the conclusion of the lemma is true up to proving that $\hat{Q}_j( \overline{\bbD}) \subset \mathcal{U}_{2\eta}$ for $j$ large enough.

In order to prove the last claim we argue by contradiction and we suppose that, up to a subsequence, for each $j$ there exists a point $z_j \in \mathbb{D}$ such that $\rmdist(\hat{Q}_j(z_j), \R P^2) \geq 2 \eta$. We rescale each map $\hat{Q}_j$ by composing with the M\"obius trasformation $\Phi_j(z)
=\frac{z+z_j}{1+\overline{z_j} z}$. Since each $\Phi_j$ is a conformal self-diffeomorphism of $\overline{\mathbb{D}}$ with $\Phi_j(0)=z_j$, then the compositions $U_j=\hat{Q}_j \circ \Phi_j$ are still harmonic functions such that $\{U_j\}\subset C(\overline{\mathbb{D}};\overline{B_{1+\eta}})\cap W^{1,2}(\mathbb{D};\mathcal{S}_0 )$. Moreover, by conformal invariance and normalization,
 \begin{itemize}
	\item[i)]  $\int_{\mathbb{D}} |\nabla U_j|^2 \, dx =\int_\mathbb{D} |\nabla \hat{Q}_j|^2 \, dx \to 0$ as $j \to \infty$,
	\item[ii)] $\rmdist(U_j(\cdot), \R P^2)\leq \eta$ on $\partial \mathbb{D}$ and $\rmdist(U_j(0), \R P^2) \geq 2 \eta$  for all $j$. 
		\end{itemize}
Since $\{U_j(0)\} \subset \overline{B_{1+2\eta}}$, passing to a further subsequence if necessary we obtain
a constant map $U_*=\lim_{j \to \infty} U_j(0)$ such that $U_j \rightharpoonup U_*$ weakly in $W^{1,2}(\mathbb{D};\mathcal{S}_0)$ and locally uniformly in $\mathbb{D}$ (even smoothly, by elliptic regularity). On the other hand by weak convergence of traces still to the constant map $U_*$ and the compactness of the embedding $W^{1/2,2}(\partial \mathbb{D};\mathcal{S}_0) \hookrightarrow L^2(\partial \mathbb{D}; \mathcal{S}_0)$, we also obtain, up to subsequences, $U_j \to U_*$ a.e. on $\partial \mathbb{D}$. Passing to the limit in the inequalities ii) we have a contradiction and the proof is complete.
\end{proof}

\begin{proof}[Proof of Theorem~\ref{thm:examples-tori}]
We divide the proof into three steps to make the argument easier to follow.

{\em Step 1:} we construct comparison maps $\{\Theta_j\} \subset \rmLip_{\rm sym}(\Omega; \mathbb{S}^4)$ such that $\sup_j \mathcal{E}_\lambda(\Theta_j) \leq C$ for some constant $C > 0$ and the corresponding traces $Q_{\rm b}^j:={\rm tr}\,\Theta_j$ give a bounded sequence in $W^{1/2,2}(\bbS^2;\R P^2)$ converging weakly to $\eo$ and correspond as in \eqref{Qblift} to a sequence $\{ v_j\} \subset C_{\rm sym}^\infty(\mathbb{S}^2;\mathbb{S}^2)$ equivariantly homotopic to the outer normal to $\mathbb{S}^2$.

We first consider maps $v \in C^\infty_{\rm sym}(\mathbb{S}^2;\mathbb{S}^2)$ described in terms of spherical coordinates $(\phi,\theta)$ using an angle function (see \cite{HKL,HLP}) $h\in C^\infty([0,\pi])$, with $0 \leq h(\theta) \leq \pi$, so that
\begin{equation}
\label{anglefunction}
v(\phi,\theta)=(\cos \phi \sin h(\theta), \sin \phi \sin h(\theta), \cos h(\theta)) \, , \quad 0 \leq \phi \leq 2\pi \,, \, \, \, 0\leq \theta \leq \pi \, .  
\end{equation}
We assume for simplicity the extra symmetry $h(-\theta+ \pi/2)+h(\theta+\pi/2)=\pi$, $0\leq \theta \leq \pi/2$, 
so that the corresponding map $v$ commutes with the reflection with respect to the plane $\{x_3 = 0 \}$. The basic example is the function $h(\theta)\equiv \theta$, for which the corresponding map $v$ is the outer normal
 $\overset{\to}{n}$ (i.e., the identity map).  
We fix an increasing smooth function $\bar{h}$ with the symmetry above and such that $\bar{h} \equiv 0$ for $0\leq \theta \leq \pi/6$ and $\bar{h}\equiv \pi$ for $ 5\pi/6 \leq \theta \leq \pi$, denoting with $\bar{v} \in C^\infty_{\rm sym}(\mathbb{S}^2;\mathbb{S}^2)$ the corresponding map. Consequently, we denote with $Q_{\rm b} \in C^\infty_{\rm sym}(\mathbb{S}^2;\R P^2)$ the map obtained from $\bar{v}$ using \eqref{Qblift} and we observe that by construction of $\bar{h}$ we have $Q_{\rm b} \equiv \eo$ in $\mathbb{S}^2 \cap \{ x_1^2+x_2^2< 1/4\}$. Notice that $\bar{v}$ is (equivariantly) homotopic to the identity map simply through the affine homotopy of their angle functions $H(\theta, t)=t \theta +(1-t) \bar{h}(\theta)$, $0 \leq t \leq 1$, so that in particular $\deg(\bar{v}, \partial \Omega)=\deg({\rm Id}, \partial \Omega)=1$. 

In order to extend $Q_{\rm b}$ to $\Omega=B_1$, we find it is more convenient to work on the vertical slice $\mathcal{D}^+=\Omega \cap \{ x_2=0 \, , \, x_1>0\}$, defining the map $\Theta : \partial \mathcal{D}^+ \to \R P^2$ by restriction of $Q_{\rm b}$ on the curved part of the boundary extended with the constant value $\eo$ on the vertical segment $I=\Omega \cap \{x_3\text{-axis}\} \subset \partial \mathcal{D}^+$. Notice that the previous map $\Theta : \partial \mathcal{D}^+\simeq \mathbb{S}^1 \to \R P^2 $ is continuous and $[ \Theta ] \neq 0$ in $\pi_1(\R P^2)$ because $\bar{v}$ is equivariantly homotopic to the identity. Now we extend $\Theta$ to $\mathcal{D}^+$ as the constant $\eo$ on $\mathcal{D}^+ \cap \{x_1 \leq \frac 12 \}$ and then on $\mathcal{T}^+:=\mathcal{D}^+ \cap \{x_1 > \frac 12\}$ as any fixed Lipschitz extension with values into $\mathbb{S}^4$ (as the latter is simply connected there is no obstruction for such an extension).
 To  summarize, there exists $\Theta \in \rmLip (\overline{\mathcal{D}^+}; \mathbb{S}^4)$ such that $\Theta \equiv \eo$ on $\overline{\mathcal{D}^+}\cap \{ 0 \leq x_1 \leq \frac 12\}$ and $\Theta \in C(\partial \mathcal{T}^+; \R P^2)$ but not homotopic to a constant.  Then, due to \eqref{reconstruction}, we can extend $\Theta$ equivariantly to the whole $\overline{\Omega}$ to have a map $\Theta \in \rmLip_{\rm sym}(\overline{\Omega};\mathbb{S}^4)$ such that $\Theta |_{\partial \Omega}=Q_{\rm b}$ and $\Theta \equiv \eo$ on $\overline{\Omega}\cap \{x_1^2+x_2^2 < \frac 14\}$. 

Finally, we construct the sequence $\{\Theta_j\}$ and the corresponding boundary traces $\{Q_{\rm b}^j\}$ by deforming the map $\Theta$ as follows. First we extend $\Theta$ from $\overline{\mathcal{D}^+}$ to a map $\bar{\Theta}$ on the whole $\overline{\mathcal{D}}=\overline{\Omega}\cap \{x_2=0\}$ with the constant value $\eo$ for $x_1<0$. Then for a deformation parameter $\rho \geq 1$ we consider M\"obius maps $\Phi_\rho : \overline{\mathcal{D}} \to \overline{\mathcal{D}}$ defined as $\Phi_\rho(z)=\frac{z-1+1/\rho}{1-(1-1/\rho)z}$. Note that as $\rho$ increases the conformal diffeomorphisms $\Phi_\rho$ ``squeeze the interior of $\mathcal{D}$ towards the point $(-1,0)$''. In addition, $\Phi_\rho^{-1}(\overline{\mathcal{D}^+})\subset \overline{\mathcal{D}^+}$,     $\Phi_\rho^{-1}(\overline{\mathcal{T}^+})\subset \overline{\mathcal{T}^+}$ and $\Phi_\rho^{-1}(\overline{\mathcal{T}^+}) \downarrow \{(1,0)\}$ as $\rho \to \infty$. Thus, if we set $\Theta_\rho=\bar{\Theta} \circ \Phi_\rho$ then $\Theta_\rho \in \rmLip( \overline{\mathcal{D}^+};\mathbb{S}^4)$ and $\Theta_\rho$ is obtained by a continuous deformation of $\bar{\Theta}$, ``concentrating $\bar{\Theta}$ near the point $(1,0)$''. 

Taking $\rho=j$ and extending each $\Theta_j$ equivariantly to $\overline{\Omega}$ we have the following: 
\begin{itemize}
\item[1)] for each $j \geq 1$ we have $\Theta_j \in \rmLip(\partial \mathcal{D}^+;\R P^2)$ and $[\Theta_j]\neq 0$ in $\pi_1(\R P^2)$ because the same property holds for $\bar{\Theta}$ by construction (just use $\rho\in [1, j]$ as a homotopy parameter);

	\item[2)] $\{\Theta_j \} \subset \rmLip_{\rm sym}(\overline{\Omega};\mathbb{S}^4)$ and $\Theta_j \equiv \eo$ out of $
	\mathbb{S}^1 \cdot \Phi_j^{-1}(\overline{\mathcal{T}^+}) \downarrow \mathbb{S}^1 \cdot\{(1,0)\}=\mathcal{C}$ as $j \to \infty$;
	thus $\Theta_j \to \eo$ locally uniformly on $\overline{\Omega} \setminus \mathcal{C}$ as $j \to \infty$ because of the properties of the M\"obius transformations combined with those of $\bar{\Theta}$;
	
	\item[3)] for every $j \geq 1$, by equivariance of $\Theta_j$ and conformal invariance in 2d we have 
	 \begin{multline}
	 \label{Thetabound}
 \int_{\Omega} |\nabla \Theta_j|^2 \, dx =\int_{\Omega \cap \{x_1^2+x_2^2> \frac 14\}} |\nabla \Theta_j|^2 \, dx 
 \leq C \left(1+\int_{\mathcal{T}^+} |\nabla_{x_1,x_3} \Theta_j |^2 d\mathcal{H}^2\right) \\ \leq C \left(1+ \int_{\mathcal{D}^+} |\nabla_{x_1,x_3} \Theta |^2 d\mathcal{H}^2 \right)  \, ; \end{multline}
	\item[4)] $\sup_j \mathcal{E}_\lambda(\Theta_j) \leq C$ for some $C>0$ because of the equiboundedness of the potential energy $W$ on $\mathbb{S}^4$ and the a priori bound \eqref{Thetabound} in 3);
	
	\item[5)] the traces $Q_{\rm b}^j:={\rm tr}\,\Theta_j$ are bounded in $W^{1/2,2}(\partial \Omega; \R P^2)$ by the pointwise properties of $\Theta_j$ at the boundary and 2)+3); in addition, the weak limit of $Q_{\rm b}^j$ is the constant map $\eo$ since $\Theta_j \rightharpoonup \eo$ in $W^{1,2}(\Omega)$ as $j \to \infty$ because of 2)+3); thus, claim 1) of Theorem \ref{thm:examples-tori} hold;
	\item[6)] the sequence $\{Q_{\rm b}^j\}$ corresponds as in \eqref{Qblift} to a sequence $\{ v_j\} \subset C_{\rm sym}^\infty(\mathbb{S}^2;\mathbb{S}^2)$ equivariantly homotopic to the outer normal to $\mathbb{S}^2$ given through \eqref{anglefunction} by angle functions $h_j=\bar{h} \circ   \Phi_j$, $h_j \in C^\infty([0, \pi])$ increasing (here we extend $\bar{h}$ to the whole $(-\pi/2, 3\pi /2) \simeq \partial \mathcal{D} \setminus \{(-1,0)\}$ as a constant $0$ in $(-\pi/2,0)$ and $\pi$ in $(\pi, 3\pi/2)$ respectively, hence they are still smooth). 
		\end{itemize}
Notice that in 6) the equivariant homotopy for each $j$ fixed is given through a homotopy of angle functions $h_j$ with $\bar{h}$ varying the parameter $\rho \in [1,j]$ of the M\"obius maps $\Phi_\rho$. Finally, since $\bar{h}(\pi/2)=\pi/2$ because of the symmetry above and $\bar{h}$ being increasing, the same hold for corresponding angle functions $\{h_j\}$, therefore for the corresponding maps $\{v_j\}$ equation \eqref{anglefunction} yields $v_j \cdot \overset{\to}{n}>0$ on $\partial \Omega$ for each $j \geq 1$.

{\em Step 2:} we prove that the corresponding minimizers $Q^j$ of $\mathcal{E}_\lambda$ in $\mathcal{A}^{\rm sym}_{Q_{\rm b}^{j}}(\Omega)$ for $j$ sufficiently large are smooth up to the boundary, converge locally smoothly in $\overline{\Omega}$ to the constant map $Q_*\equiv \eo$ away from the circle $\mathcal{C}$ and moreover that claim 2) and 3) of Theorem \ref{thm:examples-tori} hold.

Indeed, first notice that by Proposition \ref{prop:symmetric-criticality} each minimizer $Q^j$ is a weak solution to \eqref{MasterEq} and, due to Proposition \ref{corbdmonotform}, the monotonicity formulas \eqref{IntMonForm}-\eqref{BdMonForm} are satisfied.

Using $\Theta_j$ as comparison maps for every $j \geq 1$, in view of Step 1, claim 4), the minimizers $Q^j$ satisfy
\begin{equation}
\label{Qjenergybound}
	\sup_{j} \mathcal{E}_\lambda(Q^j,\Omega) \leq \sup_{j} \mathcal{E}_\lambda(\Theta_j,\Omega) \leq C \, .
\end{equation}
 Since the energies $\mathcal{E}_\lambda(Q^j,\Omega)$ are equibounded and $\Omega = B_1$, we can apply Theorem~\ref{compintthm} to the sequence $\{Q^j\}$ (with $\lambda_j \equiv \lambda$) to deduce there is a (not relabeled) subsequence and a limiting map $Q_* \in W^{1,2}_{\rm sym}(\Omega; \bbS^4)$ minimizing the energy in $\Omega$ with respect to its boundary trace so that $Q^j \rightharpoonup Q_*$ in $W^{1,2}(\Omega)$ as $j \to \infty$ and $Q^j \to Q_*$ strongly in $W^{1,2}_{\rm loc}(\Omega)$ as $j \to \infty$.
  From the fact that $Q_{\rm b}^j \rightharpoonup \eo$ weakly in $W^{1/2,2}(\bbS^2;\R P^2)$ as $j \to \infty$ and the commutativity of the trace operator with weak limits, we infer that the trace of $Q_*$ on $\partial \Omega =\bbS^2$ is the constant map $\eo$, hence $Q_* = \eo$ because such constant map is clearly the unique minimizer of $\mathcal{E}_\lambda$ with constant trace $\eo$ on the boundary (note that uniqueness also implies the convergence of the whole sequence $\{Q^j\}$ to $\eo$). Thus claim 2) is proved. 
 
 In order to prove claim 3), first we recall that according to Step 1), claim 2), the boundary data $\Theta_j$ are constant away from neighborhoods of the equatorial circle $\mathcal{C} = \partial \Omega \cap \{ x_3 = 0\}$, and such neighborhoods shrink onto $\mathcal{C}$ as $j \to \infty$. Therefore, recalling the equiboundedness of $\{Q^j\}$, using Theorem~\ref{compintthmbdry} (and the already obtained locally strong convergence in the interior) we deduce that $Q^j \to \eo$ strongly in $W^{1,2}_{\rm loc}(\overline{\Omega} \setminus \mathcal{C})$. Now observe that in view of \eqref{Qjenergybound} and the pointwise equiboundedness of the potential $W(Q^j)$, the energy measures $\mu_j= \abs{\nabla Q^j}^2 dx $ have equibounded mass. Since each $Q^j$ is equivariant then each $\mu_j$ is invariant, hence the local strong convergence to a constant map just mentioned above yields, up to subsequences, 
\begin{equation}
\label{measureconvergence}
 \abs{\nabla Q^j}^2 dx \rightharpoonup c \mathcal{H}^1 \res \mathcal{C}
\end{equation}
as measures on $\overline{\Omega}$ as $j\to \infty$, for some $c \geq 0$. 

To conclude the proof of claim 3) it remains to show that $c>0$. Before doing this we first observe that, even if $c \geq 0$, we have the constancy of the boundary data in uniform neighborhoods of the poles (due to Step 1, claim 2)) and in view of \eqref{measureconvergence} also smallness of the scaled energy centered at the poles for uniform neighborhoods $B_r \cap \Omega$. Actually, using again \eqref{measureconvergence} we see that for any $0<r<1/2$ there exists $j_0$ such that $\frac{1}{r}\mathcal{E}_\lambda(Q^j, B_r(\bar{x})\cap \Omega) < \min \{\boldsymbol{\eps}_{\rm in}, \boldsymbol{\eps}_{\rm bd} \}$ for all $j \geq j_0$ and for all $\bar{x}\in \overline{\Omega} \cap \{x_3 \text{ -axis}\}$. Neglecting finitely many maps with $j<j_0$ and applying the $\varepsilon$-regularity property from \cite[Corollary~2.19 and 2.20]{DMP1}, we see that all the minimizers $Q^j$ must be smooth on the whole $\overline{\Omega}$ because each boundary datum $Q_{\rm b}^j$ is $C^\infty$-smooth and no interior singularities on the symmetry axis are allowed.
 We can push the same argument further and indeed in the truncated domain $\overline{\Omega}\cap \{ x_1^2+x_2^2 \leq \frac 14\}$, possibly neglecting finitely many maps if necessary, combining again \eqref{measureconvergence} with \cite[Corollary 2.19 and 2.20]{DMP1} we see that all the minimizers $Q^j$ must be equi-Lipschitz continuous, therefore $Q^j \to \eo$ uniformly on $\overline{\Omega}\cap \{ x_1^2+x_2^2 \leq \frac 14\}$. 
 
 We are finally in the position to prove that the constant $c\geq 0$ in \eqref{measureconvergence} is indeed strictly positive. We argue by contradiction and suppose that $c=0$, aiming to derive a contradiction for the sequence of minimizers $\{Q^j\}$ restricted to the vertical slice $\overline{\mathcal{T}^+}$ with the help of Lemma \ref{toplbenergy}. First notice that for $j$ large enough we have $Q^j \in \mathcal{U}_{\eta}$ on the whole $\overline{\Omega}\cap \{ x_1^2+x_2^2 \leq \frac 14\}$ because of the uniform convergence to $\eo$ just established above. Since $Q^j|_{\partial \mathcal{D}^+} =\Theta_j$ we see that, with the notations of Lemma \ref{toplbenergy}, the maps $\Theta \in C(\partial \mathcal{D}^+; \R P^2)$ and $Q^j \in C(\partial \mathcal{T}^+;\mathcal{U}_{\eta})$ satisfy $[\Theta_j]=[ \Pi\circ Q^j] \neq 0$ in $\pi_1(\R P^2)$ for $j$ large enough (because of Step 1, claim 1) and using $Q^j$ itself as homotopy to deform the restriction of $\Theta_j$ to $\partial \mathcal{D}^+$ into the one to $\partial \mathcal{T}^+$). Now observe that if $c=0$ then using equivariance and arguing as in Step 1, claim 3), as $j \to \infty$ we have 
 \[ \int_{\mathcal{T}^+} \abs{\nabla_{x_1,x_3} Q^j}^2 \, dx_1dx_3 \leq C  \int_{\Omega \cap \{ x_1^2 +x_2^2 > \frac 14\}} \abs{\nabla Q^j}^2 \, dx \to 0 \, .
    \]
Then we infer a contradiction from Lemma \ref{toplbenergy} up to pulling back the maps onto the unit disc $\mathbb{D}$ by composition with a biLipschitz homeomorphism $\Phi:\overline{\mathbb{D}} \to \overline{ \mathcal{T}^+}$. Indeed, setting $\bar{Q}_j:=Q^j \circ \Phi$, all the assumptions in Lemma \ref{toplbenergy} are trivially satisfied (assumption 1) for $j$ large enough) but $[\Pi \circ \bar{Q}_j]=[\Pi \circ Q^j]\neq 0$ in $\pi_1(\R P^2)$ for all $j$, a contradiction.

{\em Step 3:} we show that for $j$ large enough the uniaxial sets and the biaxial regions corresponding to the smooth minimizers $Q^j$ possess all the announced qualitative properties.

Indeed, first notice that each minimizer is smooth up to the boundary of $\Omega$ and analytic in the interior (see \cite[Corollary 2.19]{DMP1}), hence assumption ($HP_0$) holds. On the other hand, since $\Omega=B_1$, assumptions ($HP_1$)--($HP_3$) are clearly satisfied by the domain and the boundary data, because of Step 1, claim 6). Since the whole set of assumptions ($HP_0$)--($HP_3$) is verified by each minimizer $Q^j$, applying Proposition~\ref{prop:semidisk} we infer that, for every $Q^j$, the set $\{\beta_j = -1 \}$ contains an invariant circle $\Gamma_j$ mutually linked to $\partial \mathcal{D}^+$. Moreover, since $Q_{\rm b}^j$ takes values into $\R P^2$ for every $j$, we have $\partial B_1 \subset \{ \beta_j = 1\}$ (actually, we have also $\partial B_1 \cup I \subset \{ \beta_j = 1\}$, where $I$ denotes the vertical diameter because smoothness yields $Q^j\equiv \eo$ on the symmetry axis). From \cite[Lemma~5.2]{DMP1}, it then follows that all levels of biaxiality are nonempty. Moreover, the set of singular values for $\beta_j$ in the whole range $[-1,1]$ is at most countable and can accumulate only at $t=1$. Now, if $-1 < t < 1$ is a regular value for $\beta_j$, we know from \cite[Theorem~5.7, claim 1)]{DMP1} that $\{ \beta = t \}$ is a smooth surface with a connected component of positive genus. On the other hand, because of equivariance of each map the set $\{\beta = t \}$ is an $\bbS^1$-invariant set, hence we can study it by looking to its section in the vertical plane, i.e, by looking at $\{\beta = t\} \cap \{ x_2 = 0 \}$. This fact allows to completely characterize $\{\beta = t \}$. Indeed, because of $\bbS^1$-equivariance and the regularity of $Q^j$, its planar section can only look as the union of finitely many smooth closed simple curves, so that $\{\beta_j = t\}$ is the union of finitely many axially symmetric tori. 

It remains to prove that the biaxial regions $\{\beta_j \leq t\}$, where $t \in (-1,1)$, are pushed towards the circle $\mathcal{C}$ as $j \to \infty$. This property is a straightforward consequence of the fact that $Q^j \to \eo$ uniformly on $\overline{\Omega} \cap \{x_1^2 +x_2^2 \leq r^2\}$ for any $0<r <1$ by an argument entirely similar to the one used in Step 2 above for the case $r=1/2$. We leave the simple adaptation to the reader. Thus, for each $0 < \rho < 1$ and for each distance neighborhood $\mathcal{O}_\rho$ from $\mathcal{C}$ we have $\beta_j(x)=\widetilde{\beta}(Q^j(x)) \to 1$ uniformly in $\overline{B_1}\setminus \mathcal{O}_\rho$ as $j \to \infty$, 
  whence for each number $t\in (-1,1)$ we have $\emptyset \neq \{\beta_j \leq t\} \subset \mathcal{O}_\rho$ for any $j$ large enough and the proof is complete.
   
\end{proof}

\begin{remark}
\label{H12bubbling}	($W^{1/2,2}$-bubbling of nontrivial loops)
As shown during the proof of Theorem \ref{thm:examples-tori}, the sequence of minimizers $\{Q^j\}$ concentrates energy on the horizontal circle $\mathcal{C}$ in the sense described in \eqref{measureconvergence}. Considering their restrictions to the vertical slice $\overline{\mathcal{D}^+}$, one has the uniform convergence $Q^j \to \eo$ on $ \overline{\mathcal{D}^+} \cap \{x_1 \leq \frac 12\}$ (actually on   $ \overline{\mathcal{D}^+} \cap \{x_1 \leq r \}$ for any fixed $0<r<1$). On the other hand, slicing \eqref{measureconvergence} one has the 2d-energy convergence
\[ \abs{\nabla_{x_1,x_3} Q^j}^2 dx_1 dx_3 \rightharpoonup \bar{c} \delta_{(1,0)} \, ,\]
for some $\bar{c}>0$ in the sense of measures on $\overline{\mathcal{T}^+}$, where $\mathcal{T}^+=\mathcal{D}^+\cap\{x_1 > \frac 12\}$, whence the restrictions also satisfy the condition $Q^j \rightharpoonup \eo$ weakly in $W^{1,2}(\mathcal{T}^+)$ as $j \to \infty$. As observed during the proof of the theorem,  the sequence of traces $\gamma^j:=Q^j|_{\partial \mathcal{T}^+}$ inherits the following two properties (compare Lemma~\ref{toplbenergy} for the relevant definitions):
\begin{itemize}
\item[1)] for a given neighborhood $\mathcal{U}_\eta$ of $\R P^2$ we have $\{\gamma^j\} \subset C(\partial \mathcal{T}^+;\mathcal{U}_\eta)$ and for the normalized maps $\bar{\gamma}^j:=\Pi \circ \gamma^j$ we have $[\bar{\gamma}^j]\neq 0$ in $\pi_1(\R P^2)$ for all $j$ large enough;
\item[2)] $\{\gamma^j\} \subset W^{1/2,2}(\partial \mathcal{T}^+;\mathbb{S}^4)$ is bounded and $\gamma^j \rightharpoonup \eo$ weakly in $W^{1/2,2}$ as $j \to \infty$.
  \end{itemize}
  As the maps $Q^j$ are $\mathbb{S}^1$-equivariant, the properties above amount to say that, for the 1d-restriction of $Q^j$ to the simple loop $\partial \mathcal{T}^+$ as well as to each of its congruent copies under rotation, the ``normalizations'' $\bar{\gamma}^j:=\Pi \circ \gamma^j$ are homotopically nontrivial, bounded in $W^{1/2,2}(\mathcal{T}^+;\R P^2)$ and weakly convergent to the constant map $\eo$ as $j\to \infty$ ({\em bubbling-off} of a topologically nontrivial loop under weak $W^{1/2,2}$-convergence).
 \end{remark}

In the final part of this subsection we study more generally the topology of smooth $\bbS^1$-equivariant maps satisfying ($HP_0$)--($HP_3$). The following result is the counterpart of \cite[Theorem 5.7]{DMP1} in the axially symmetric setting and obviously implies Theorem~\ref{topology-sym-torus} as a particular case.

\begin{theorem}
\label{topology-sym}
Let $\Omega \subset \mathbb{R}^3$ be an axisymmetric bounded open set with $C^1$-smooth boundary and let $Q: \overline{\Omega} \to \bbS^4$ be $\bbS^1$-equivariant. Assume that $\Omega$ and $Q$ are such that ($HP_0$)-($HP_3$) hold. Then the biaxiality regions associated with $Q$ are nonempty $\bbS^1$-invariant closed subsets of $\overline{\Omega}$ and satisfy:

\begin{itemize}
\item[1)] the set of singular values of $\beta$ in $[-1,\bar{\beta}]$ is at most countable and can accumulate only at $\bar{\beta}$; moreover, for any regular value $-1 < t <{\bar{\beta}}$ the set $\{\beta = t\}$ is the disjoint union of finitely many (at least one) revolution tori well contained in $\Omega$ while for any regular value $\bar{\beta} \leq t < 1$ the set $\{\beta = t\}$ is the disjoint union of, possibly, finitely many revolution tori, finitely many $\mathbb{S}^1$-invariant strips touching the boundary and finitely many circles lying on the boundary. If in addition $\Omega$ is real-analytic and $Q \in C^\omega(\overline{\Omega})$, then the set of singular values of $\beta$ in $[-1,\bar{\beta}]$ is finite. 
\item[2)] For any $-1 \leq t_1  < \bar{\beta} \leq t_2 < 1$, the set $\{\beta \leq t_1 \}$ contains an invariant  circle $\Gamma \subset \{ \beta = -1 \}$ and the set $\{ \beta \geq t_2 \}$ contains $\partial \mathcal{D}^+$. $\Gamma$ and $\partial\mathcal{D}^+$ are mutually linked. As a consequence, for any $-1 \leq t_1  < t_2 < 1$ the sets $\{\beta \leq t_1 \}$ and $\{\beta \geq t_2 \}$ are nonempty, compact and mutually linked. In particular, the set $\{\beta = 1 \} \cap \Omega$ is nonempty. If in addition $\bar{\beta} = 1$, then $\{\beta = 1 \} \subset \overline{\Omega}$ is not simply connected.
\end{itemize}
\end{theorem}

\begin{proof}
	As a preliminary remark we observe that, due to ($HP_1$), ($HP_2$) and $\bbS^1$-equivariance, $\beta_0 := \max_{\partial\Omega} \beta = 1$, hence it follows from \cite[Lemma~5.2]{DMP1} that all levels of biaxiality are non-empty. Now, the first part of Statement 1) follows as in Theorem 5.7 in \cite{DMP1}, using the analytic Morse-Sard theorem from \cite{SoSo}. The more detailed information we are claiming about the genus of the biaxiality surfaces comes as follows: we first recall that for a regular value $-1 < t < 1$ the biaxiality set $\{ \beta = t \}$ is a finite union of smooth connected orientable surfaces $\Sigma_i$ which are analytic in the interior (and also  boundaryless when $t < \bar{\beta}$). Notice that such surfaces do not touch the $x_3$-axis, where $\beta\equiv 1$ because of ($HP_1$) and the equivariance of $Q$ (see Remark \ref{rmk:invariantconf}). Then we observe that $\beta$ is invariant under rotations around the $x_3$-axis, so that each $\Sigma_i$ must be a smooth revolution surface and, as a consequence, it is enough to discuss its behavior looking at its cross-section with the planar domain $\overline{\mathcal{D}^+}$. From this and the implicit function theorem, the section $\Sigma_i \cap \overline{\mathcal{D}^+}$ appears either as a simple closed curve, or as a smooth curve connecting two or more boundary points or even as a single point on the boundary. This in turn implies that each $\Sigma_i$ is either a revolution torus, a cylinder-type surface touching the boundary (which we call a \emph{strip} for brevity) or a circle lying on the boundary. In particular, when $t < \bar{\beta}$ we can only have a finite number of tori (at least one) well inside $\Omega$, which concludes the proof of 1). 
	
	The first claim of Statement 2) follows from Proposition \ref{prop:semidisk} and definition of $\bar{\beta}$. Now we show that $K_1 := \{\beta \leq t_1\}$ and $K_2 := \{ \beta \geq t_2\}$ are always mutually linked (even if there are critical values for $\beta$ between $t_1$ and $t_2$). Indeed, we claim that no one of $K_1$, $K_2$ is contractible in the complement of the other. To see this, denote $\Gamma_1 = \Gamma$ and $\Gamma_2 = \partial \mathcal{D}^+$ with continuous parametrizations $\gamma_i: \mathbb{S}^1 \to \Gamma_i$, $i=1,2$. Suppose, for a contradiction, that $K_i$ is contractible in $K_j^c := \overline{\Omega} \setminus K_j$, with $i \neq j$. This means there exists a homotopy $G_j : [0,1] \times K_j^c \to K_j^c$ so that $G(0, \cdot)\vert_{K_i} = \rm{Id}$ and $G(1, \cdot)\vert_{K_i} = \mbox{constant}$\,. Thus, we can find a homotopy $H_j :[0,1] \times \bbS^1 \to \Gamma_j^c$ with $H_j(0,s) = \gamma_i(s)$ and $H_j(1,s) = \mbox{const.}$ (in fact, $H_j(t,s) = G_j(t, \gamma_i(s))$). This means $\gamma_i$ is homotopic to a constant in $\overline{\Omega} \setminus \Gamma_j$, i.e., $\Gamma_i$ is contractible in $\overline{\Omega} \setminus \Gamma_j$; however this is false, because $\Gamma$ and $\partial\mathcal{D}^+$ are mutually linked.
	
	To conclude the proof, notice that, as already observed above, $\beta \equiv 1$ on $\Omega \cap \{\,x_3 \text{-axis}\,\}$, hence on the vertical part of $\partial\mathcal{D}^+$. In addition, when $\bar{\beta} = 1$, then $\{ \beta = 1\}$ on the whole $\partial\mathcal{D}^+$. Then it is clear that $\{ \beta = 1 \}$ cannot be simply connected because we have just seen that it contains non contractible loops, and we are done. 
\end{proof}

As a particular case of the above theorem, we have the following corollary, generalizing Theorem~\ref{topology-sym-torus} to the case of critical points (note that for $C^1$ solutions to \eqref{MasterEq} linear higher regularity theory yields $C^\infty$-regularity and in turn $C^\omega$-regularity because of \cite{Morrey}; hence assumption ($HP_0$) automatically holds in the corollary).

\begin{corollary}\label{cor:topology-smooth-crit-points}
	Let $\Omega \subset \mathbb{R}^3$ be an axisymmetric bounded open set with $C^1$-smooth boundary and let $Q_{\rm b} \in C^{1}(\partial\Omega,\bbS^4)$ be $\bbS^1$-equivariant. Suppose that $Q_\lambda \in \mathcal{A}^{\rm sym}_{Q_{\rm b}}(\Omega) \cap C^1(\overline{\Omega})$ is a smooth critical point of $\mathcal{E}_\lambda$ over $\mathcal{A}^{\rm sym}_{Q_{\rm b}}(\Omega)$ and assume that $\Omega$ and $Q_\lambda$ satisfy ($HP_1$)--($HP_3$). Then the biaxiality regions of $Q_\lambda$ satisfy the conclusions of Theorem~\ref{topology-sym}. 
\end{corollary}

Thus, the topology of smooth $\bbS^1$-equivariant critical points of $\mathcal{E}_\lambda$ looks very similar to the picture emerging from numerical simulations. For this reason, the following definition seems appropriate.  
\begin{definition}[Torus solution with Lyuksyutov constraint]\label{def:torus-sol}
Let $\Omega \subset \R^3$ be an axysimmetric bounded open set with $C^1$-smooth boundary and let $Q_{\rm b} \in C^1(\partial\Omega,\bbS^4)$ be $\bbS^1$-equivariant. Suppose $Q_\lambda \in \mathcal{A}^{\rm sym}_{Q_{\rm b}}(\Omega)$ is a critical point of $\mathcal{E}_\lambda$ over $\mathcal{A}^{\rm sym}_{Q_{\rm b}}(\Omega)$. If $\Omega$ and $Q_\lambda$ satisfy $(HP_0)-(HP_3)$  --- so that the conclusion of Theorem~\ref{topology-sym} holds --- we call $Q_\lambda$ a {\bf torus solution} of the Euler-Lagrange equations~\eqref{MasterEq} in $\mathcal{A}^{\rm sym}_{Q_{\rm b}}(\Omega)$.
\end{definition}
Note that this definition entails the existence of a negative uniaxial ring mutually linked to a region of strictly larger signed biaxiality, namely $\partial \Omega \cup I$ where $I=\Omega \cap \{\, x_3 \text{-axis} \, \}$. In particular, when the domain is a ball and the boundary condition is the radial anchoring, we have a negative uniaxial ring mutually linked to a positive uniaxial region, as for the minimizers constructed in Theorem \ref{thm:examples-tori}. 

\begin{remark}
There are two main reasons to encode smoothness (i.e., requiring $(HP_0)$) inside the very definition of torus solution. The first is that numerical simulations, even the ones in which the norm constraint is imposed, hint that biaxial torus solutions are smooth (see, e.g., \cite{GaMk,DR,HQZ}, see also \cite{KV} and \cite{KVZ} for the constrained case). The second lies in the fact that, in principle, also in split minimizers (see next subsection) there can be mutually linked biaxial regions and tori. This could \emph{a priori} happen also for the other singular solutions once we drop the hypothesis $Q_{\rm b}\equiv \mathbf{e}_0$ on $\mathcal B$. To rule out these apparently spurious cases, we restrict the notion of torus solution to smooth solutions.
\end{remark}

\subsection{Existence and topology of split minimizers}
In this subsection, we study the topology of singular minimizers $Q_\lambda$ of $\mathcal{E}_\lambda$ in $\mathcal{A}^{\rm sym}_{Q_{\rm b}}(\Omega)$ assuming that the domain has smooth boundary, the Dirichlet datum is smooth enough and that assumptions ($HP_1$)--($HP_3$) hold. The key ingredients to this purpose, as essentially explained in the remark below and then formalized in Proposition~\ref{prop:dipoles}, are a special structure of the singular set of such minimizers and the detailed knowledge of the asymptotic behavior of minimizers at the singular points in terms of the tangent maps described in the previous section.

\begin{remark}\label{rmk:axis}
Let $\Omega$ be an axially symmetric domain with $C^1$-smooth boundary and suppose we have an axially symmetric configuration $Q$ belonging to $C^1(\overline{\Omega} \setminus \Sigma)$, where $\Sigma =\rmsing(Q) \subset \Omega \cap \{\, x_3 \text{-axis} \,\}$ is a finite set, so that in view of Remark \ref{rmk:invariantconf} we have $Q(x)=\pm {\bf e_0} $ at any regular point. We assume $Q= {\bf e_0}$ on $\mathcal{B}$, where $\mathcal{B}$ is the set of boundary points defined in~\eqref{eq:def-B}, and that when passing through a singular point along the symmetry axis $Q$ jumps from ${\bf e_0}$ to $-{\bf e_0}$ or vice-versa. 

Then, it is clear that singularities come out in pairs. That is, if $\rmsing(Q) = \{a_1, a_2, \dots, a_M \}$ are ordered increasingly on the $x_3$-axis, then $M = 2N$ for some $N \in \N$ and, moving monotonically along the symmetry axis, if $Q$ around $a_k$ jumps from $\pm{\bf e_0}$ to $\mp{\bf e_0}$, then it jumps from $\mp{\bf e_0}$ to $\pm{\bf e_0}$ at $a_{k+1}$. In this sense, the singularities not only are even in number but they are naturally grouped as pairs of endpoints for the segments on the symmetry axis where $Q\equiv- \eo$. In addition, each pair is contained in some segment $\ell_k$ as defined in \eqref{eq:def-ellk}. We term these pairs {\em dipoles} (extending to this case the terminology of \cite{BCL}; see also \cite{HKL,HLP} for related statements for the case of axially symmetric maps from $B_1$ to $\bbS^2$ described by an angle function).
\end{remark}

\begin{proposition}\label{prop:dipoles}
	Suppose $\Omega$ is an axisymmetric bounded open set with $C^3$-smooth boundary, let $Q_{\rm b}\in C^{1,1}(\partial\Omega;\bbS^4)$ be $\bbS^1$-equivariant and assume that ($HP_1$)--($HP_3$) are satisfied. Let $Q_\lambda \in \mathcal{A}_{Q_{\rm b}}^{\rm sym}(\Omega)$ be a minimizer of $\mathcal{E}_\lambda$ over $\mathcal{A}_{Q_{\rm b}}^{\rm sym}(\Omega)$. Then, $\rmsing(Q_\lambda)$, the singular set of $Q_\lambda$, either is empty or consists of a finite number of dipoles located on the $x_3$-axis. In fact, more precisely, on each segment $\ell_k$ as in \eqref{eq:def-ellk} there are either no singularities of $Q_\lambda$ or a finite number of dipoles.
\end{proposition}

\begin{proof}
The claim is a straightforward consequence of Theorem~\ref{thm:partial-regularity} and  Remark~\ref{rmk:axis}. Indeed, all the structural assumptions in Remark~\ref{rmk:axis} follow from the partial regularity result from the previous section. In particular, since the tangent map at each singularity of $Q_\lambda$ is one of those given by equation~\eqref{stableblowups}, when passing through a singular point along the symmetry axis $Q_\lambda$ can only jump from ${\bf e_0}$ to $-{\bf e_0}$ or vice-versa. In addition, $Q_{\rm b}\equiv \mathbf{e}_0$ on $\mathcal{B}$ because of ($HP_1$) and $\bbS^1$-equivariance (compare Remark \ref{rmk:invariantconf}), hence along each connected component $\ell_k$ of $\Omega \cap \{\, x_3 \text{-axis} \, \}$ there must be either no jumps (i.e., no singularities) or an even number of jumps, i.e., a finite number of dipoles.
\end{proof}

\begin{remark}
\label{oddsingularities}
	Symmetrically, the same conclusion in Proposition \ref{prop:dipoles} holds if $Q_{\rm b}\equiv - \mathbf{e}_0$ on $\mathcal{B}$. If instead we allow $Q_{\rm b}$ to be both $\mathbf{e}_0$ and $-\mathbf{e}_0$ on $\mathcal{B}$ and in particular if there is an $\ell_k=[b_{2k-1},b_{2k}]$    such that $Q_{\rm b}(b_{2k-1})=-Q(b_{2k})$, then there is an odd number of interior singularities of $Q_\lambda$ in $\ell_k$. As a model example, which will be of use below, we may take $\Omega = B_1$ and $Q_{\rm b}= Q^{(0)}$, with $Q^{(0)}$ as in formula~\eqref{stableblowups} for $\alpha=0$. Then any minimizing solution of the Euler-Lagrange equations in $\mathcal{A}_{Q_{\rm b}}^{\rm sym}(B_1)$ must have an odd number of singularities.
\end{remark}

In view of the peculiar structure of their singular set, it is natural to give a special name to the singular minimizers $Q_\lambda$ in Proposition~\ref{prop:dipoles}.

\begin{definition}[Split minimizer with Lyuksyutov constraint]\label{def:split-minimizer}
Let $\Omega$ be an axisymmetric bounded open set with $C^3$-smooth boundary, $Q_{\rm b}\in C^{1,1}(\partial\Omega;\bbS^4)$ be $\bbS^1$-equivariant and assume that ($HP_1$)--($HP_3$) are satisfied. If $Q_\lambda$ is a minimizer of $\mathcal{E}_\lambda$ over $\mathcal{A}^{\rm sym}_{Q_{\rm b}}(\Omega)$ and $\rmsing(Q_\lambda) \neq \emptyset$, we call $Q_\lambda$ a {\bf split minimizer} of $\mathcal{E}_\lambda$ over $\mathcal{A}^{\rm sym}_{Q_{\rm b}}(\Omega)$. 
\end{definition}

For exposition purposes, it will be convenient to associate a {\em sign} with each singularity. This sign will play a role similar to the degree of tangent maps in the case of harmonic maps from $B_1$ to $\bbS^2$; it can be defined as follows.
\begin{definition}\label{def:sign}
Let $Q_\lambda \in \mathcal{A}_{Q_{\rm b}}^{\rm sym}(\Omega)$ be a singular minimizer of $\mathcal{E}_\lambda$ over $\mathcal{A}^{\rm sym}_{Q_{\rm b}}(\Omega)$ and suppose $a_1, a_2, \dots, a_N \in \rmsing(Q_\lambda)$ are the singularities of $Q_\lambda$. We say that $a_n$ is {\em positive} if, for $x_3$ increasing, approaching $a_n$ from below along the $x_3$-axis we have $Q_\lambda = -\mathbf{e}_0$ and approaching $a_n$ from above along the $x_3$-axis we have $Q_\lambda =  \mathbf{e}_0$. We say that $a_n$ is {\em negative} if it is not positive.
\end{definition}

After the previous preliminary discussion we are finally in the position to prove Theorem \ref{thm:examples-split}, on the existence of split minimizers in a nematic droplet. The proof is a combination of compactness and partial regularity results from Sec.~\ref{sec:comp} and Sec.~\ref{sec:6} together with the classification of harmonic spheres from Sec.~\ref{sec:axisymm}.

\begin{proof}[Proof of Theorem \ref{thm:examples-split}]
	We divide the proof into three steps to make the argument easier to follow. Some technical points are essentially the same as in the proof of Theorem~\ref{thm:examples-tori}, hence in these instances we shall not repeat full details here.
	
	{\it Step 1. Construction of boundary data and their weak convergence.} We shall construct $\bbS^1$-equivariant boundary data $Q_{\rm b}^j \in C^{1,1}_{\rm sym}(\bbS^2;\bbS^4)$, for $j = 1,2, \dots$, such that  $Q_{\rm b}^j \rightharpoonup Q^{(0)}$ in $W^{1,2}(\bbS^2;\bbS^4)$ as $j \to \infty$; here $Q^{(0)}$ is given by \eqref{stableblowups} for $\alpha=0$, i.e., it is the $Q$-tensor field corresponding to the equator map ${\boldsymbol{\omega}}^{(1)}_{\rm eq}$. 	To this end, we shall exploit the classification results in Sec.~\ref{sec:axisymm} and the correspondence between harmonic spheres and $Q$-tensor fields given by Lemma~\ref{prop:Li-action}.
	
	More precisely, from formula~\eqref{eq:classification} and making use of Lemma~\ref{prop:Li-action}, we immediately obtain a family of $\bbS^1$-equivariant $Q$-tensor fields $\{Q_{\mu_1,\mu_2} \}\subset C^{\omega}_{\rm sym}(\partial B_1;\bbS^4)$ indexed by the two complex parameters $\mu_1$, $\mu_2$ appearing in~\eqref{eq:classification}.
	
 Notice that as the two parameters $\mu_1$ and $\mu_2$ vary we have: 
	\begin{itemize}
		\item[i)] The constant norm hedgehog $\overline{H}$ corresponds to the case $\mu_1 = \mu_2 = \sqrt{3}$.
		\item[ii)] The limiting boundary datum $Q^{(0)}$ corresponds to the case $\mu_1 = 1$, $\mu_2 = 0$.
		\item[iii)] Every map $Q_{\mu_1,\mu_2}$ with $\mu_1 \neq 0$ and $\mu_2 \neq 0$ (i.e., corresponding to a linearly full harmonic sphere in \eqref{eq:classification}) attains $\eo$ at both poles and has Dirichlet energy $12\pi$. 
		
	\item[iv)] Each map $Q_{\mu_1,\mu_2}$ with $\mu_1$, $\mu_2$ real, $\mu_1 > 0$ and $\mu_2 > 0$, satisfies $(HP_1)$. 
		\item[v)] Each map $Q_{\mu_1,\mu_2}$ with $\mu_1$, $\mu_2$ real, $\mu_1 > 0$ and $\mu_2 > 0$, satisfies $(HP_3)$. 
	\end{itemize}
Clearly, i)-iii) are obvious consequences of the results in Sec.~\ref{sec:axisymm} (see Remark~\ref{hedgehog}, Proposition~\ref{degeneratespheres} and Proposition~\ref{prop:explfullspheres} respectively).

 To see iv), first notice that for any $Q \in \bbS^4$ we have $\widetilde{\beta}(Q) = -1$ if and only if $Q$ belongs to $-\R P^2 \subset \bbS^4$ (the mirror image of $\R P^2$, embedded in $\bbS^4$ as in \eqref{vacuum}, through the antipodal map of $\bbS^4$). Now, let $\omega_{\mu_1,\mu_2}$ be a harmonic sphere given by~\eqref{eq:classification}, with $\mu_1 > 0$ and $\mu_2 > 0$, and let $Q_{\mu_1,\mu_2}$ be the corresponding $Q$-tensor field as in the above. By equivariance, if at $x \in \bbS^2$ we have $\widetilde{\beta}(Q_{\mu_1,\mu_2})(x) = -1$, then the same is true on the whole orbit of $x$ under the action of $\bbS^1$. We are going to rule out the two cases: either $Q_{\mu_1,\mu_2}(x) = -\eo$ or $Q_{\mu_1,\mu_2}(x) \in -\R P^2 \setminus \{-\eo\}$. In the first case $x$ is not a pole by property iii) above . Then $\omega_{\mu_1,\mu_2}$ takes the constant value $-\eo$ on a circle, hence it must be constant on the enclosed spherical cap by Lemaire's Theorem \cite{Le} and in turn $\omega_{\mu_1,\mu_2} \equiv -\eo$ by unique continuation, contradicting the fact that it is linearly full. Next, we show that $Q_{\mu_1,\mu_2}(x) \in -\R P^2 \setminus \{-\eo\}$ is also impossible, where $x$ cannot be a pole again by iii) above.  Indeed, if this were the case, from~\eqref{vacuum}, Lemma~\ref{prop:Li-action} and equivariance, we would infer that the function $\omega_2$, relative to $\omega = \omega_{\mu_1,\mu_2}$ as in~\eqref{eq:omega}, is strictly negative on the orbit of $x$, whereas~\eqref{eq:classification} tells us it is strictly positive. Hence, if $Q_{\mu_1,\mu_2}$ is built upon the harmonic spheres in~\eqref{eq:classification} and $\mu_1, \mu_2$ are real and positive, then $\min_{\partial\Omega}\widetilde\beta(Q_{\mu_1,\mu_2}) > -1$, i.e., assumption $(HP_1)$ holds.  
	
	To prove v), first observe it holds true for $\overline{H}$. Let $Y = \{ (\mu_1,\mu_2) \in \R^2 : \mu_1 > 0,\,\mu_2 > 0 \}$. We notice that $Y$ is path connected and that, by the discussion above, the leading eigenvalue remains simple for every $x \in \partial\Omega$ moving the parameters along any path in $Y$. In particular, every pair $(\mu_1, \mu_2) \in Y$ can be connected to $(\sqrt{3},\sqrt{3})$ by a path in $Y$ with parameter $s \in [0,1]$ and the corresponding path of maps $Q_{\mu_1(s),\mu_2(s)}$, $0\leq s \leq 1$, gives an equivariant homotopy between  $Q_{\mu_1,\mu_2}$ and $\overline{H}$. In addition, for every fixed value of the parameter $s$ of the path the maximal eigenvalue $\lambda^{(s)}_{\rm max}(x)$ of the corresponding map $Q_{\mu_1(s),\mu_2(s)}$ is simple for every $x \in \partial\Omega$ because of property iv). Thus, the corresponding eigenspace $V_{\rm max}^{(s)} : \partial\Omega = \bbS^2 \to \R P^2$ is well-defined for every $s\in [0,1]$ again by property iv), it is of class $C^1$ and it depends continuously on $s$. Hence $V_{\rm max}^{(0)}$ and $V_{\rm max}^{(1)}$ are homotopic. Moreover, for each $s\in [0,1]$, $V^{(s)}_{\rm max}$ lifts to $v^{(s)}_{\rm max} \in C^1_{\rm sym}(\bbS^2;\bbS^2)$ and indeed the whole homotopy lifts \cite[Proposition~1.30]{Hatcher}. As a consequence, $v^0_{\rm max}$ and $v^1_{\rm max}$ are also homotopic, so that $\deg(v^0_{\rm max},\partial\Omega) = \deg(v^1_{\rm max}, \partial\Omega) = \pm 1$ and assumption $(HP_3)$ holds. 
	
	Finally,  for any sequence $\mu_{2,j} \downarrow 0$ we define $Q_{\rm b}^j := Q_{1,\mu_{2,j}}$.
Note for each $Q_{\rm b}^j $ the maximal eigenvalue $\lambda^j_{\rm max}$ is simple on $\bbS^2$ because of property iv) above; in addition, the corresponding eigenspace map $V^j_{\rm max}$ is homotopic to $\bar{H}$ as shown in the proof of property v) above, hence claim (1) holds. Now a direct inspection of~\eqref{eq:classification} reveals that $ Q_{\rm b}^j(x) \to Q^{(0)}(x)$ a.e. on $\bbS^2$ as $j\to \infty$. Moreover, in view of (iii) the sequence  $\{ Q_{\rm b}^j \} \subset W^{1,2}(\bbS^2;\bbS^4)$ is bounded, which in turn easily yields claim (2) of the theorem.

	{\it Step 2. Strong and locally smooth convergence to a singular minimizer.} Now, for each boundary datum $Q_{\rm b}^j$ on $\bbS^2 = \partial B_1$, we take an $\mathcal{E}_\lambda$-minimizer $Q^j$ (which exists by the direct method). Notice that for every $j\geq 1$ the degree-zero homogeneous extension $\widehat{Q}^j$ to $\Omega = B_1$ of $Q_{\rm b}^j$ is an admissible competitor for testing the minimality of $Q^j$ and, since the potential energies of the maps $\widehat{Q}^j$ are equibounded, we have $\sup_j \mathcal{E}_\lambda(\widehat{Q}^j, \Omega) < +\infty$ and in turn $\{Q^j\}$ is bounded in $W^{1,2}(\Omega;\bbS^4)$. Passing to a subsequence, we have  $Q^j \rightharpoonup Q^*$ in $W^{1,2}$ to some map $Q^* \in W^{1,2}_{\rm sym}(B_1;\bbS^4)$ and, moreover,  $Q^* \vert_{\bbS^2} = Q^{(0)}$ by the weak continuity of the trace operator. 
	 	 We now observe we can apply Theorem~\ref{compintthm} (with $\lambda_j \equiv \lambda$) to deduce $Q^j \to Q^*$ in $W^{1,2}_{\rm loc}(B_1;\bbS^4)$ and that $Q^*$ is minimizing $\mathcal{E}_\lambda$ w.r.to $Q^{(0)}$ as boundary datum. Actually, thanks to the properties of the boundary data and the minimizers proved in Step 1, we may also apply Theorem~\ref{compintthmbdry} to improve local strong convergence to strong convergence $Q^j \to Q^*$ in $W^{1,2}(\Omega; \bbS^4)$. Moreover, by the uniform gradient estimates in \cite[Corollary 2.19]{DMP1}, the  strong convergence $Q^j \to Q^*$ improves to local smooth convergence in $\Omega \setminus \Sigma_*$, where $\Sigma_*\subset \Omega$ is the finite set of possible interior singularities of $Q^*$ according to Theorem~\ref{thm:partial-regularity}. We now observe that $Q^{(0)}$ attains $\eo$ at the north pole and $-\eo$ at the south pole of $B_1$, hence any $\mathcal{E}_\lambda$-minimizer $Q^*$ over $\mathcal{A}^{\rm sym}_{Q^{(0)}}(\Omega)$ must have an odd number of singularities in view of Remark~\ref{oddsingularities}.

	{\it Step 3. Persistence of singularities and existence of split solutions.} According to claim (1) of the theorem and to Remark~\ref{rmk:axis}, each minimizer $Q^j$ has either no singularities or an even number of singularities coming in dipoles. We now claim that, for $j$ large enough, any such minimizer is not smooth, hence it has an even number of singularities. Notice that, as shown in Step 2, the limiting map $Q^*$ has an odd number of singularities. 
	 In order to prove the claim, let $x_* \in \Sigma_*$ and, just to fix ideas, suppose $x_*$ is a singularity with positive sign (in the sense of Definition~\ref{def:sign}). Letting $r_* > 0$ so small that there are no singularities of $Q^*$ in $B_{r_*}(x_*) \setminus \{x_*\}$ then, for every point $z_+ \in B_{r_*}(x_*)$ on the $x_3$-axis which is above $x_*$, we have $Q^*(z_+) = \eo$ and, for every $z_- \in B_{r_*}(x_*)$ on the $x_3$-axis below $x_*$, we have $Q^*(z_-) = -\eo$. From the strong $W^{1,2}$-convergence and $\veps$-regularity we have $C^1_{\rm loc}$-convergence $Q^j \to Q^*$ in $\Omega \setminus \Sigma_*$. As a consequence, up to discarding finitely many values of $j$ it follows that for every map $Q^j$ there are points in $B_{r_*}(x_*) \cap \{ x_3\mbox{-axis} \}$ above $x_*$ at which $Q^j$ is $\eo$ and points below $x_*$ at which $Q^j$ is $-\eo$. 
	 This means that $Q^j$ has at least one singularity in $B_{r_*}(x_*) \cap \{x_3\mbox{-axis}\}$ and the proof is complete.
\end{proof} 
\begin{remark}
	Although the boundary data $Q_{\rm b}^j$ in Theorem~\ref{thm:examples-split} are real-analytic, analyticity does not play a role in the argument. Thus, singularities can occur because of energy efficiency even when the regularity of the boundary data is maximal.
\end{remark} 
\begin{remark}
	Keeping in mind the examples of torus solutions in the unit ball constructed in Theorem \ref{thm:examples-tori}, the theorem above also suggests how delicate is the interplay between the geometry of the domain and the properties of boundary conditions in determining the (regularity, hence topological) properties of the corresponding minimizers. Further results in this sense, including the coexistence of smooth and singular minimizers, will be discussed in our forthcoming paper \cite{DMP2}.
\end{remark}
The previous theorem has two remarkable consequences even for the minimization of the Landau-de Gennes energies among nonsymmetric maps. Indeed, for $\{Q_{\rm b}^j\}$ as in Theorem~\ref{thm:examples-split} we know from \cite[Theorem~1.2]{DMP1} that any minimizer of $\mathcal{E}_\lambda$ over each norm-constrained nonsymmetric classes $\mathcal{A}_{Q_{\rm b}^j}(\Omega)$ is smooth up to the boundary. In view of  Theorem~\ref{thm:examples-split} such minimizers cannot be axially symmetric for $j$ large enough because the minimizers in the subclass  $\mathcal{A}^{\rm sym}_{Q_{\rm b}^j}(\Omega)$
must be singular. As a consequence, symmetry breaking and nonuniqueness of minimizers in the class $\mathcal{A}_{Q_{\rm b}^j}(\Omega)$ occur as summarized in the following corollary.


\begin{corollary}[Symmetry breaking]\label{cor:symmetry-breaking-I}
	Let $\Omega = B_1$, let $\{Q_{\rm b}^j\} \subset C^{1,1}(\bbS^2;\bbS^4)$ be the $\bbS^1$-equivariant boundary data constructed in Theorem~\ref{thm:examples-split} and let $\bar{Q}^j$ be a sequence of minimizer of $\mathcal{E}_\lambda$ over $\mathcal{A}_{Q_{\rm b}^j}(\Omega)$. Then, for every $j$ large enough $\bar{Q}^j$ is not $\bbS^1$-equivariant (i.e., $\bar{Q}^j$ belongs to $\mathcal{A}_{Q_{\rm b}^j}(\Omega) \setminus \mathcal{A}_{Q_{\rm b}^j}^{\rm sym}(\Omega)$). In particular $R \cdot \bar{Q}^j (R^t \cdot )$ is a minimizer for each $R \in \mathbb{S}^1$, hence nonuniqueness of minimizers occurs.
\end{corollary}


The symmetry breaking phenomenon extends to the {\it Lyuksyutov regime} introduced in \cite{DMP1}, as made precise in the corollary below. Let us recall that without the norm constraint \eqref{Lyuk} the relevant energy functional is
\[
	\mathcal{F}_{\lambda,\mu}(Q) = \int_\Omega \frac{1}{2}\abs{\nabla Q}^2  + \lambda W(Q) + \frac{\mu}{4}\left( 1 - \abs{Q}^2 \right)^2 \,dx,
\]
where $\lambda$ is defined after \eqref{LDGenergytilde}, $\mu := \frac{a^2}{L} > 0$, and 
\[
	W(Q) = \frac{1}{4\sqrt{6}} \abs{Q}^4 - \frac{1}{3}{\rm tr}(Q^3) + \frac{1}{12\sqrt{6}}\,.
\]
The functional $\mathcal{F}_{\lambda,\mu}$ is defined over $W^{1,2}(\Omega;\mathcal{S}_0)$ and it reduces to $\mathcal{E}_\lambda$ under the Lyuksyutov constraint \eqref{Lyuk} in view of the definition of $\widetilde\beta$ in \eqref{signedbiaxiality} (see also discussion in the introduction of \cite{DMP1}). The Lyuksyutov regime considered in \cite{DMP1} corresponds to
\[
	\lambda = \mbox{const.}, \qquad \mu \to +\infty\,,
\]
where the energy minimization was performed under a fixed
boundary datum $Q_{\rm b} \in C^{1,1}(\partial \Omega; \mathbb{S}^4)$. According to \cite[Theorem~1.3]{DMP1}, in the Lyuksyutov regime the minimizers of $\mathcal{F}_{\lambda,\mu}$ over $W^{1,2}_{Q_{\rm b}}(\Omega;\mathcal{S}_0)$ strongly converge in $W^{1,2}$ to minimizers of $\mathcal{E}_\lambda$ over $\mathcal{A}_{Q_{\rm b}}(\Omega)$ which are smooth up to the boundary, as already recalled above. Since $\mathbb{S}^1$-equivariance would persist under strong convergence, we deduce that it eventually fails when it fails for the limiting map. Thus, Corollary \ref{cor:symmetry-breaking-I} immediately yields the following result.
 
\begin{corollary}[Symmetry breaking in the Lyuksyutov regime]
\label{cor:symmetry-breaking-II}
	Let $\Omega = B_1$ and let $\{Q_{\rm b}^j\} \subset C^{1,1}(\bbS^2;\bbS^4)$ be the sequence of $\bbS^1$-equivariant boundary data constructed in Theorem~\ref{thm:examples-split} and let us fix $j$ large enough such that Corollary \ref{cor:symmetry-breaking-I} holds. Then, for each $\mu>0$ large enough any minimizer $\bar{Q}^j_\mu$ of $\mathcal{F}_{\lambda,\mu}$ over $W^{1,2}_{Q_{\rm b}^j}(\Omega;\mathcal{S}_0)$ is not $\bbS^1$-equivariant. 
\end{corollary}

Finally, we discuss general properties of singular minimizers under norm and symmetry constraints as those constructed in Theorem~\ref{thm:examples-split}. We are now going to prove Theorem~\ref{topology-sym-split} but we still need some preparation. Indeed, another important ingredient in the proof of Theorem~\ref{topology-sym-split} is the following result which explains how biaxiality behaves near the singular points. The statement is an immediate consequence of the asymptotic analysis at isolated singularities from the previous section and the structure of tangent maps in \eqref{stableblowups}.
\begin{proposition}\label{prop:conic-structure}
 Let $Q_\lambda \in \mathcal{A}^{\rm sym}_{Q_{\rm b}}(\Omega)$ be a minimizer of $\mathcal{E}_\lambda$ over $\mathcal{A}^{\rm sym}_{Q_{\rm b}}(\Omega)$ and let us fix $x_0 \in \rmsing (Q_\lambda)$ and $r_0>0$ such that   $Q_\lambda \in C^\infty(B_{r_0}(x_0) \setminus \{x_0\})$. Let $Q^{(\alpha)}$ denote the corresponding tangent map at $x_0$ and $Q_\lambda^{x_0,r}(x)=Q_\lambda(x_0+r x)$ the rescaled maps, so that by Theorem \ref{thm:partial-regularity} we have $Q_{\lambda}^{x_0,r} \to Q^{(\alpha)}$ in $C^2(\overline{B_2}\setminus B_1 )$ as $r \to 0$. 
	Then for the biaxiality functions of the rescaled maps we have that $\widetilde\beta \circ Q_\lambda^{x_0,r_j}$ smoothly converges to $\widetilde \beta \circ Q^{(\alpha)}$ as $j \to \infty$ in $\overline{B_2} \setminus B_1$ along any sequence $r_j \to 0$. As a consequence, for any $-1 < t < 1$, there is a cylinder $\mathcal C_t$ coaxial with the $x_3$-axis with the property that $\{\widetilde{\beta} \circ Q_\lambda = t\} \cap \mathcal{C}_t$ is a smooth disc-type surface with a conical point at $x_0$. 
\end{proposition}

\begin{remark}
\label{independence}
	Notice that, due to the explicit formula \eqref{stableblowups} for any minimizing tangent map, the function $\beta=\widetilde{\beta} \circ Q^{(\alpha)}$ is invariant under rotations around the $x_3$-axis, it is actually independent of $\alpha$ and in spherical coordinates depends just on the colatitude $\theta \in [0,\pi]$. Indeed, an elementary calculation gives $\beta(\theta)=-\frac12 \cos^3 \theta+\frac32 \cos \theta$, therefore all the values $t \in (-1,1)$ are regular.  As a consequence, Proposition \ref{prop:conic-structure} would hold even without uniqueness of the tangent map proved in the previous section.
\end{remark}
\begin{proof}
The proof is elementary, so we only sketch the main idea. First we observe that in view of \eqref{stableblowups} and Remark~\ref{independence} the level sets $\{ \widetilde \beta \circ Q^{(\alpha)} =t\}$, $t\in (-1,1)$ fixed, are $\mathbb{S}^1$-invariant round cones with tip at $x_0$ and opening angle $\vartheta \in (0,\pi)$ depending on $t\in (-1,1)$. As for the limit function all the values are regular (away from the origin), the same property holds for the biaxiality functions $\widetilde\beta \circ Q_\lambda^{x_0,r_j}$ for $j\geq j_0$, $j_0$ large enough, and indeed smooth convergence of the functions implies smooth convergence of the surfaces seen as graphs over a fixed annular region of the limiting cone. Choosing dyadic radii $r_j=r_0 2^{-j}$ and undoing the scaling, we see that $\{\widetilde{\beta} \circ Q_\lambda = t\}$ is a smooth disc-type surface (it is indeed the union of the rescaled annular graphs) with a conical point at $x_0$ and having $\{ \widetilde \beta \circ Q^{(\alpha)} =t\}$ as tangent cone at it. Thus the conclusion follows in a small cylinder $\mathcal{C}_t$ of radius $r=r_0 2^{-j_0}$.      
\end{proof}

\begin{proof}[Proof of Theorem~\ref{topology-sym-split}]
Since $Q_\lambda$ is equivariant, the function $\beta$ is invariant under rotations around the $x_3$-axis, therefore it is enough to work in the planar domain $\overline{\mathcal{D}^+_\Omega}$, where $\mathcal{D}^+_\Omega$ is as in \eqref{eq:decomp-D}. The function $\beta$ is real-analytic where $Q_\lambda$ is; thus in particular $\beta$ is real-analytic in $\mathcal{D}^+_\Omega$. Recall that $\rmsing(Q_\lambda)$ is a finite set of isolated points on the $x_3$-axis. Then we infer from Sard's theorem for analytic functions \cite{SoSo} that the set of singular values of $\beta$ is finite on each compact set $K \subset \mathcal{D}^+_\Omega$ (note that because of $\mathbb{S}^1$-invariance the critical values of $\beta$ in $\Omega$ are exactly those of the restriction to $\mathcal{D}^+_\Omega$), hence all but at most countably many $t \in [-1, \bar{\beta}]$ are regular values of $\beta$ (note that $\beta=\pm 1$ on the symmetry axis). 

Suppose there is a sequence of distinct singular values $\{ \beta_n\} \subset [-1, \bar{\beta})$ accumulating at some $-1<\beta_*<\bar{\beta}$ with corresponding points $\{x_n\} \subset \overline{\mathcal{D}^+_\Omega} \setminus \rmsing(Q_\lambda)$. 
Passing to a subsequence we may assume $x_n \to x_*$ and $\nabla \beta(x_*)=0$. Since $\beta(x_*)=\beta_*< \bar{\beta}$ we conclude that $x_* \in \Omega$. 
Note that $x_* \not \in \Omega \setminus \rmsing(Q_\lambda)$, because otherwise $\beta$ would have countably many distinct singular values in some closed ball $\overline{B_r(x_*)} \subset \Omega \setminus \rmsing(Q_\lambda) $, which is impossible by Sard's Theorem. 
We are going to show that the last possible option, i.e., $x_* \in \rmsing(Q_\lambda)$, is also impossible, therefore the critical values can accumulate only at $\bar{\beta}$ or at $-1$ as claimed. 
Indeed, assume the converse. Then there would be $x_* \in \rmsing(Q_\lambda)$ and a (not relabeled) sequence $x_n \to x_*$ such that $\nabla \beta(x_n)=0$ for all $n$. 
Applying Proposition \ref{prop:conic-structure} with $x_0= x_*$, $r_n=|x_n-x_*|$ and tangent map $Q^{(\alpha)}$, passing to a (not relabeled) subsequence we would have $y_n=(x_n-x_*)/r_n \to y_* \in \partial B_1$ which is not on the vertical axis, because clearly $\widetilde{\beta}\circ Q^{(\alpha)}(y_*)=\beta_* \neq \pm 1$, in contrast with Remark~\ref{rmk:invariantconf}. On the other hand, the $C^1$-convergence in Proposition \ref{prop:conic-structure} yields  $\nabla (\widetilde{\beta}\circ Q^{(\alpha)}) (y_*)= \lim_n \nabla ( \widetilde{\beta} \circ Q_\lambda^{x_0,r_n})(y_n)=0$ which is clearly false, since  $t=\beta_*$ (indeed every $t \in (-1,1)$) is a regular value of $\widetilde{\beta}\circ Q^{(\alpha)} $ away from the origin in view of Remark \ref{independence}.

Let $t \in (-1,\bar{\beta})$ be a regular value of $\beta$ and $a \in \rmsing(Q_\lambda)$. Then the biaxiality set $\{ \beta = t \}\cap \mathcal{D}^+_\Omega$ comes out from $a$  tangent to a straight line by Proposition \ref{prop:conic-structure}, it is contained in $\mathcal{D}^+_\Omega$ by the definition of $\bar\beta$ and it is a finite union of analytic connected arcs and, possibly, of finitely many disjoint analytic closed simple curves. Let $\mathcal{C}$ be a maximal arc in $\{ \beta = t \}\cap \mathcal D^+_\Omega$ originating from $a$. We want to show that it ends at another singular point $a_1 \neq a$ of $Q_\lambda$, with opposite sign w.r.to $a$. Indeed, we observe that $\mathcal{C}$ cannot end on $\partial\mathcal{D}^+_\Omega \setminus \rmsing(Q_\lambda)$, since either $\beta > t$ or $\beta = -1$ there. On the other hand, $\mathcal{C}$ cannot end at a point $x_1$ in the interior of $\mathcal{D}^+_\Omega$, since here $\beta$ is analytic and, with the aid of the implicit function theorem, we could continue $\mathcal{C}$ a bit as a smooth arc across $x_1$, contradicting maximality. Thus, $\mathcal{C}$ can only end on a singular point $a_1$, necessarily with opposite sign w.r.to $a$ since $t$ is a regular value of $\beta$ (the sign of the normal derivative of $\beta$ along the arc with respect to the outer normal of the enclosed region is constant, so the sign of the two endpoint singularities must be opposite). Rotating around the $x_3$-axis, we have topological axisymmetric spheres and, possibly, tori. Note that the topological axisymmetric spheres so determined have corners for any $t \neq 0$ because of the asymptotically conical behavior in Proposition \ref{prop:conic-structure}. In the special case $t = 0$ is a regular value, then the set $\{ \beta = 0 \}$ contains $N$ smooth axisymmetric topological spheres.

Now, let $\bar{\beta} \leq t < 1$ be a regular value of $\beta$. Arguing as in the proof of 1) in Theorem \ref{topology-sym} and in the above, we see that $\{\beta = t \} \cap \overline{\mathcal D^+_\Omega}$ looks like the disjoint union of finitely many of the followings: (a) arcs connecting singularities with opposite sign (as in the above); (b) arcs connecting boundary points; (c) arcs connecting singularities and boundary points; (d) points on the boundary; (e) simple closed curves. Therefore, rotating the planar section around the $x_3$-axis gives both 1) and 2).

Finally we now prove 3). In view of the information previously obtained, going down along the symmetry axis we have a first singularity $a^+ \in \Omega$, which is clearly positive because of ($HP_1$), and which is the north pole of a sphere $\frS$ contained inside the biaxial set $\overline{\{ \beta= t_2\}}$, whose south pole is a negative singularity $a^-$. Notice that this pair could be a dipole or not, if there are other singularities in between. In both cases there is a regular point $\tilde{a}$ such that $Q_\lambda(\tilde{a})=-\eo$ and which is on the symmetry axis in between the two singularities (this is trivial if the two forms a dipole but also obvious if there is an extra singularity in between, choosing $\tilde{a}$ sufficiently close and ``on the negative side'' of it) and therefore contained in the interior of the biaxial sphere $\frS$.

Clearly, $ \frS \subset \overline{\{\beta \geq t_2\}}$ and $\tilde{a} \in \{ \beta \leq t_1\}$, hence $\frS \subset \overline{ \{\beta \geq t_2\}} \subset \overline{\Omega} \setminus \{ \beta \leq t_1\} \subset \mathbb{R}^3 \setminus \{\tilde{a}\}$, therefore if $\{\beta \geq t_2\}$ is contractible inside  $\overline{\Omega} \setminus \{ \beta \leq t_1\}$ then $\frS$ is contractible inside $\mathbb{R}^3\setminus \{\tilde{a}\}$. However, the latter fact is clearly impossible by the following classical argument from degree theory, which gives the desired conclusion.  Indeed, since $\frS$ is topologically a sphere we have a bi-Lipschitz embedding $\varphi : \bbS^2 \to \frS \subset \overline{\{\beta \geq t_2 \}}$ and of course $\tilde{a} \not \in \frS$ by construction. If $\mathcal{O} \subset \Omega$ is the region enclosed by $\frS$, then for $r>0$ small enough we have $\tilde{a} \in \overline{B_r(\tilde{a})}\subset \mathcal{O}$, hence the map $\Phi : \bbS^2 \to \bbS^2$ given by
\[
	\Phi(y) = \Pi_{\tilde{a}} \circ \varphi(y)=\frac{\varphi(y) - \tilde{a}}{\abs{\varphi(y) - \tilde{a}}}
\] 
is well-defined and by Stokes theorem $\deg(\Phi,\bbS^2)=\deg(\Pi_{\tilde{a}},\frS)=\deg(\Pi, \partial B_r(\tilde{a}))=1$. Now, for the sake of a contradiction, assume that $\overline{\{\beta \geq t_2 \}}$ is contractible in $\overline{\Omega} \setminus \{\beta \leq t_1 \}$. Thus, there exists a homotopy $H : (\overline{\Omega} \setminus \{\beta \leq t_1\} )\times [0,1] \to \overline{\Omega} \setminus \{ \beta \leq t_1\}$ so that $H(\cdot, 0)\vert_{\{\beta \geq t_2\}} = \rm{Id}$, $H(\cdot, 1)\vert_{\{\beta \geq t_2\}} = \mbox{const}$. Considering the induced continuous map $G : \bbS^2 \times[0,1] \to \bbS^2$ defined by
\[
	G(y,s) = \frac{H(\Phi(y),s) - \tilde{a}}{\abs{H(\Phi(y),s) -\tilde{a}}}\,,
\] 
we have $\deg G(\cdot,1) = \deg  \mbox{const} =0$ while $\deg G(\cdot, 0) = \deg \Phi=1$, which contradicts the invariance of the degree under homotopy.
\end{proof}

\begin{remark}
\label{splitcore}
	Split minimizers are somewhat analogues of the so-called \emph{split core solutions} found numerically in \cite{GaMk}. Also split core solutions contain pairs of special points but these are \emph{isotropic points} (i.e., points at which $Q = 0$) rather than singularities, and because of axial symmetry they still bound a segment of negative uniaxiality. In fact, singularities are impossible in the framework of \cite{GaMk} since there is no norm-constraint there. Conversely, isotropic points are impossibile here, because of the norm constraint. It is natural to conjecture that the norm constraint turns isotropic points into singularities. Arguments in \cite{DMP2} will provide more evidence for this. For all these reasons, we preferred to give these singular solutions a still evocative but slightly different name.
	\end{remark}

\begin{remark}
\label{extratori}
	As stated in the previous theorem, the presence of tori in the biaxial sets of split minimizers is not excluded. The difference with the smooth case is that, according to Corollary~\ref{cor:topology-smooth-crit-points}, tori \emph{must} be contained in the biaxiality surfaces of smooth solutions, at least for levels of biaxiality lower than $\bar{\beta}$.
	For singular minimizers this is not yet clear, because we are not able at present to describe the topological structure of a singular configuration near each dipole and consequently to understand its relation with the topological properties of the boundary data. 
\end{remark}

\begin{remark}
\label{Yu}
	Further, numerics from \cite{KV} and \cite{KVZ} suggest these singular solutions are \emph{not minimizing} deep in the nematic phase, at least in the ball under homeotropic boundary condition, where energy minimizing configurations should be smooth with torus-like structure. However, as commented in the final subsection of the paper, the coexistence  of smooth  and singular minimizers in such model case has been proved in the recent remarkable paper \cite{Yu} in the smaller class of $\On(2) \times \mathbb{Z}_2$-equivariant configurations. As we will discuss in \cite{DMP2} in the context of $\mathbb{S}^1$-equivariant minimizers under radial anchoring, this coexistence can occur but it depends in a subtle way on the geometry of the domain $\Omega$ and it may be lost for suitable deformations of the ball for which energy minimizing configurations turn out to be necessarily minimizing torus solution in the sense of Definition \ref{def:torus-sol} or minimizing split solutions in the sense of Definition \ref{def:split-minimizer}. 
\end{remark}

\subsection{Concluding remarks}\label{sec:comparison}
Many results of this paper, in particular Theorem~\ref{thm:partial-regularity}, Theorem~\ref{topology-sym-torus}, Theorem~\ref{thm:examples-tori} and Theorem~\ref{classification}, are much improved versions of results proven for the first time in the Ph.D. thesis of the first author \cite{DipThesis}, where the concepts of torus solution and split minimizer have been firstly formalized, at the least to the best of our knowledge. Theorem~\ref{topology-sym-torus} and Theorem~\ref{topology-sym-split} are the first results in literature describing the topology of $\bbS^1$-equivariant (LdG) minimizers in such detail and, at the same time, in such generality. 

However, as already remarked in the Introduction, apparently similar and somewhat related results very recently appeared in the interesting paper \cite{Yu}, where the minimization problem of the energy functional \eqref{LDGenergytilde} in a symmetric class of competitors is considered, and here we want to comment a bit on the differences between our work and \cite{Yu}. 

First of all, as we already mentioned in the Introduction, the analysis in \cite{Yu} is restricted to the case when the domain is the unit ball $\Omega=B_1$ and the boundary condition is always the constant norm hedgehog $\overline{H}$ given by \eqref{Hbar}. In addition, the class of $Q$-tensor fields considered in \cite{Yu} is strictly smaller than $\mathcal{A}_{\overline{H}}^{\rm sym}(B_1)$, because instead of considering arbitrary $\mathbb{S}^1$-equivariant configurations the author restricts to the smaller class of $\On(2)\times \mathbb{Z}_2$-equivariant configurations. More explicitly, in \cite{Yu} the author considers maps $Q$ that in cylindrical coordinates $(r,\phi, z) \in B_1$ can be equivalently written as 
\begin{equation}
\label{yurepresentation} 
	Q(r,\phi,z)=(f_0(r,z)), f_1(r,z)e^{i \phi}, f_2(r,z) e^{i 2\phi}) \, , \qquad (r\cos \phi, r \sin \phi, z) \in B_1 \, ,
\end{equation}
using the identification of $\mathcal{S}_0$ with $\R \oplus \C \oplus \C$ instead of a reference moving frame adopted there. It is assumed that $f_j \in \mathbb{R}$ for $j=0, 1, 2$, to entail $\On(2)$-equivariance, so that the admissible space of tensor at each point is only a two dimensional sphere, i.e., $f(r,z) \in \mathbb{S}^2$ and $Q(\cdot, \cdot, \phi) \in \mathbb{S}^2_{\phi} \subset \mathbb{S}^4$. In addition each $f_j$ is also assumed to be even/odd symmetric in the $x_3$ variable, more precisely $f_j(r, -z)=(-1)^j f_j(r,z)$ for each $j=0,1,2$, which amounts to the $\mathbb{Z}_2$-equivariance with respect to the reflection across the plane $\{x_3=0\}$. In this restricted class the author performs a clever parametrized constrained minimization which yields in the limit coexistence of ``torus'' and ``split'' minimizers of the unconstrained minimization problem in the $\On(2)\times \mathbb{Z}_2$-equivariant class having the same energy.  

In essence, both the minimizers in \cite{Yu} are smooth near the origin, and the different behavior depends essentially on the possible values $Q_\lambda(0)=\pm \eo$. Since the extra $\mathbb{Z}_2$-symmetry forces the singular set to be symmetric as well, in agreement with Remark \ref{rmk:axis}, the number of singularities in the upper half space is shown to be even (possibly zero) if $Q_\lambda(0)=\eo$ or odd in the opposite case $Q_\lambda(0)=-\eo$.
In the former case, $\mathbb{Z}_2$-symmetry is cleverly used to deduce the existence of a negative uniaxial ring on the symmetry plane $\{x_3=0\}$, surrounded by a coaxial thin solid torus of biaxial tensors. In the latter case instead, a vertical segment of negative uniaxiality containing the origin with a pair of singularities at the endpoint is shown to exist, surrounded by a thin neighborhood of biaxial phase. These conclusions, as well as further interesting results concerning the behavior of the eigenvalues and the eigenframes in these neighborhoods, depend in a crucial way on the $\On(2)$-equivariance. Indeed, compared to our setting, $\On(2)$-equivariance allows to conclude that the vector field  $e_\phi(x)= ( -\sin \phi , \cos \phi,0)$ is an  eigenvector of any $\On(2)$-equivariant $Q$-tensor at any point. In turn this allows to write the corresponding eigenvalue $Q e_\phi \cdot e_\phi$ as a linear combination of the entries of $Q$, to give manageable formulas for the remaining ones and to deduce ordering properties between the eigenvalues which are crucial when discussing the behavior of the eigenframe mentioned above. It would be very interesting to extend these conclusions to the full $\mathbb{S}^1$-equivariant context, in which a better understanding of the behavior of the eigenframe near the uniaxial sets would be highly desirable. 

Also the explicit boundary datum plays a major role in \cite{Yu}, in combination with the $\On(2)$-equivariance and at least in two respects. Firstly, it allows to deduce by the maximum principle that $f_2>0$ in the interior and $f_1 \gtrless 0$ for $x_3 \gtrless 0$ from the same properties on the boundary, useful sign properties which are crucial to prove that the uniaxial sets mentioned above are surrounded by the biaxial phase. The second fundamental aspect in which it enters concerns the discussion of the regularity theory in \cite{Yu}, where through an a priori energy upper bound on the minimizers it allows to exclude a priori (and not by any stability analysis, as done here in Sec.~\ref{sec:stability}) the presence of linearly full harmonic spheres as tangent map, therefore no classification as the one in Sec.~\ref{sec:axisymm} is needed. As a result, the possible tangent maps at isolated singularities are still of the form \eqref{stableblowups}, with the further restriction $\alpha = 0$ in the upper half space and $\alpha=\pi$ otherwise, due to $\On(2)$-equivariance and the sign condition on $f_1$ mentioned above.
 
  Concerning qualitative properties of the minimizers in \cite{Yu}, these seems to be somehow weaker counterparts of the ones in Theorem \ref{topology-sym-torus} and \ref{topology-sym-split} above. Clearly the minimizers in the restricted class described above are critical point of the energy functional \eqref{LDGenergytilde} as those in the class $\mathcal{A}^{{\rm sym}}_{Q_{\rm b}}(B_1)$ by a symmetric criticality principle as the one in Sec.~\ref{sec:6}, although their energy minimality in the class $\mathcal{A}_{\overline{H}}^{\rm sym}(B_1)$ remains unclear. On the other hand, the fundamental difference between our definition of torus solution (and, in turn, of split minimizer) and the corresponding one in \cite{Yu} is the full regularity assumption we make in Definition \ref{def:torus-sol} (and in turn its violation in Definition~\ref{def:split-minimizer}) and which is absent in \cite{Yu}. Indeed, as already recalled above, torus solutions as constructed in \cite{Yu} may have singularities (although when they are smooth then they satisfy our Definition \ref{def:torus-sol} in view of Corollary \ref{cor:topology-smooth-crit-points}); more precisely, they may have a finite number of dipoles (in our sense) in each half-space. As a consequence of regularity, in our case the linking property of the negative uniaxial ring with the positive uniaxial region made up by the boundary and the $x_3$-axis holds and is actually encoded in Definition \ref{def:torus-sol}. On the other hand, without assuming regularity it unclear whether the linking property or some weaker counterpart of it still holds when a negative uniaxial ring is present. This is quite undesirable, since the linking property seems to be the most striking feature of biaxial torus solutions according to numerical simulations, leading to a foliation of the domain in tori corresponding to the level sets of the (signed) biaxiality parameter. 
  
  Without further regularity information, the most intriguing conclusion of \cite{Yu}, the coexistence of a torus and a split minimizer, should be reinterpreted in the weaker sense of ``coexistence of two minimizers among $\On(2)\times \mathbb{Z}_2$-equivariant configurations with a different number of singularities''. In the full $\bbS^1$-equivariant class, but for well-chosen domains and boundary conditions (allowing also for the stronger conclusion that one of the two must be a torus solution), this fact was already proven by the first author in \cite[Theorem~8.9]{DipThesis}. Anyway, even if weakened and even if valid only in a restricted class, the conclusion of \cite{Yu} remains very interesting because no rigorous result was previously known for the ball with the constant norm hedgehog on the boundary, which is by far the most interesting and representative case considered in literature. Inevitably, comparison with \cite{Yu} puts in even more evidence the interest of Theorem~\ref{thm:examples-tori} and Theorem~\ref{thm:examples-split}, since we not only proved that torus solutions and split solutions in the ball do exist but also that there are fairly natural boundary conditions with respect to which  they are the \emph{only} minimizers in the \emph{full} $\bbS^1$-equivariant class. Furthermore, in analogy with the results announced in Remark~\ref{Yu}, in our companion paper \cite{DMP2} we will elaborate more on this theme and we will show how singular and smooth solutions for the Euler-Lagrange equations do appear simultaneously as minimizers among $\bbS^1$-equivariant maps in a nematic droplet for suitable deformations of the radial anchoring as boundary data. 




\begin{thebibliography}{100}

\bibitem{ABL} \textsc{S. Alama, L. Bronsard, X. Lamy} : \textit{Spherical particle in a nematic liquid crystal under an external field: the Saturn ring regime.} J. Nonlinear Sci. \textbf{28} (2018), 1443-1465.

\bibitem{ABGL} \textsc{S. Alama, L. Bronsard, D. Golovaty, X. Lamy} : \textit{Saturn ring defect around a spherical particle immersed in nematic liquid crystal.} arXiv:2004.04973  [math.AP] (2020).


\bibitem{AlLi} \textsc{F. J. Almgren, E. H. Lieb} : \textit{Singularities of energy minimizing maps from the ball to the sphere: examples, counterexamples, and bounds.} Ann. of Math. (2) \textbf{128} (1988),  483-530.

\bibitem{BairdWood} \textsc{P. Baird, J. C. Wood} : \textit{Harmonic morphisms between Riemannian Manifolds.} (London Mathematical Society Monographs: New Series 29) Oxford University Press, 2003.



\bibitem{BeFr} \textsc{R. Benedetti, R. Frigerio, R. Ghiloni} : \textit{The topology of Helmholtz domains.} Expo. Math. \textbf{30} (2012), 319-375. 


\bibitem{BoWo1} \textsc{J. Bolton, L. M. Woodward} :  \textit{Higher singularities and the twistor fibration $\pi : \mathbb{C}P^3 \to \mathbb{S}^4$.} Geom. Dedicata \textbf{80} (2000), 231-246.

\bibitem{BoWo2} \textsc{J. Bolton, L. M. Woodward} : \textit{Linearly full harmonic 2-spheres in $\mathbb{S}^4$ of area $20\pi$.} Internat. J. Math. (5) \textbf{12} (2001), 535-553.

\bibitem{BoWo3} \textsc{J. Bolton, L. M. Woodward} : \textit{The space of harmonic two-spheres in the unit four-sphere.} Tohoku Math. J. (2) \textbf{58} (2006), 231-236.

\bibitem{Br} \textsc{R. Bryant} : \textit{Conformal and minimal immersions of compact surfaces into the 4-sphere.}  J. Differential Geom. \textbf{17} (1982), 455-473. 

\bibitem{BCL} \textsc{H. Brezis, J. M. Coron, E. H. Lieb} : \textit{Harmonic maps with defects.} Comm. Math. Phys. \textbf{107} (1986), 649-705.

\bibitem{Ca} \textsc{E. Calabi} : \textit{Minimal immersions of surfaces in Euclidean spheres.} J. Differential Geom. \textbf{1} (1967), 111-125.


\bibitem{Can} \textsc{G. Canevari} : \textit{Line defects in the small elastic constant limit of a three-dimensional Landau-de Gennes model}, Arch. Rational Mech. Anal. \textbf{223} (2017), 591--676. 

\bibitem{CanOrl} \textsc{G. Canevari, G. Orlandi} : \textit{Topological singular set of vector-valued maps, II: $\Gamma$-convergence for Ginzburg-Landau type functionals}, preprint arXiv:2003.01354 (2020).  





\bibitem{DR} \textsc{G. De Luca, A. D. Rey} : \textit{Ringlike cores of cylindrically confined defects.} J. Chem. Phys. \textbf{126}(9) (2007), 104902. 

\bibitem{DipThesis} \textsc{F. Dipasquale} : \textit{Variational methods in the Landau-de Gennes theory of liquid crystals.} Ph.D. thesis, Sapienza -- Universit\`a di Roma, http://hdl.handle.net/11573/1234866 (Feb. 2019). 


\bibitem{DMP1} \textsc{F. Dipasquale, V. Millot, A. Pisante}: \textit{Torus-like solutions for the Landau-de Gennes model. Part I: Lyuksyutov regime.} To appear in Arch. Rational Mech. Anal. DOI: 10.1007/s00205-020-01582-8. 
 arXiv:1912.12160 (math.ap).

\bibitem{DMP2} \textsc{F. Dipasquale, V. Millot, A. Pisante} : \textit{Torus-like solutions for the Landau-de Gennes model. Part III: torus solutions vs split solutions.} In preparation.


\bibitem{Fawley} \textsc{H. L. Fawley} : \textit{Twistor theory of immersions of surfaces in four-dimensional spheres and hyperbolic spaces}. Thesis, University of Duhram, U.K., 1997.
 
\bibitem{GaMk} \textsc{E. C. Gartland Jr., S. Mkaddem} : \textit{Fine structure of defects in radial nematic droplets.}  Phys. Rev. E \textbf{62} (2000),  6694-6705.

\bibitem{Gastel} \textsc{A. Gastel} : \textit{Regularity theory for minimizing equivariant ($p$-)harmonic mappings.} Calc. Var. Partial Differential Equations \textbf{6} (1998), 329-367. 

\bibitem{GiaqMart} \textsc{M. Giaquinta, L. Martinazzi} : {\it An introduction to the regularity theory for elliptic systems harmonic maps and minimal graphs.} Lecture Notes. Scuola Normale Superiore di Pisa (New Series) {\bf 11}, Edizioni della Normale, Pisa (2012). 

\bibitem{GilbTrud} \textsc{D. Gilbarg, N.S. Trudinger} : {\it Elliptic Partial Differential Equations of Second Order.} Classics in Mathematics, Springer-Verlag, Berlin (2001).


\bibitem{HL} \textsc{R. Hardt, F.~H. Lin} : \textit{A remark on $H^1$-mappings.}  Manuscripta Math. \textbf{56} (1986), 1-10.

\bibitem{HKL} \textsc{R. Hardt, D. Kinderlehrer, F.~H. Lin} : \emph{The Variety of Configurations of Static Liquid Crystals}. In: Variational Methods: Proceedings of a Conference Paris, June 1988. Ed. by H. Berestycki, J. M. Coron, and I. Ekeland. Birkh\"{a}user Boston, 1990, pp. 115-131. 



\bibitem{HLP} \textsc{R.~Hardt, F.H.~Lin, C.C.~Poon} : 
\textit{Axially symmetric harmonic maps minimizing a relaxed energy.} Comm. Pure Appl. Math. \textbf{45} (1992), 417-459.

\bibitem{Hatcher} \textsc{A. Hatcher} : \textit{Algebraic Topology.} Cambridge University Press, 2002.


\bibitem{HeleinBook} \textsc{F. H\'{e}lein} :  \textit{Constant Mean Curvature Surfaces, Harmonic Maps and Integrable Systems.}  Lectures in Mathematics: ETH Z\"{u}rich. Birk\"{a}user Basel, 2001.


\bibitem{HorMos} \textsc{P. Hornung, R. Moser} : \textit{Existence of equivariant biharmonic maps.} Internat. Math. Res. Notices (8) \textbf{2016} (2016), 2397-2422.

\bibitem{HQZ} \textsc{Y. Hu, T. Qu, P. Zhang} : \textit{On the Disclination Lines of Nematic Liquid Crystals.} Commun. Comp. Phys. \textbf{19}(2) (2016), 354-379. 
 
\bibitem{INSZ1} \textsc{R. Ignat, L. Nguyen, V. Slastikov, A. Zarnescu} : \textit{Stability of the melting hedgehog in the Landau-de Gennes theory of nematic liquid crystals.} Arch. Ration. Mech. Anal. \textbf{215} (2015), 633-673. 

\bibitem{INSZ2} \textsc{R. Ignat, L. Nguyen, V. Slastikov, A. Zarnescu} : \textit{On the uniqueness of minimizers of Ginzburg-Landau energy functionals.} Ann. Sci. \'{Ec}. Norm. Sup\'{e}r. (4) \textbf{53} (2020), 589-613.

\bibitem{INSZ3} \textsc{R. Ignat, L. Nguyen, V. Slastikov, A. Zarnescu} : \textit{Instability of point defects in a two-dimensional nematic liquid crystal model.} Ann. Inst. H. Poincar\'{e} Anal. Nonlin\'{e}are \textbf{33} (2016), 1131-1152. 

\bibitem{INSZ4} \textsc{R. Ignat, L. Nguyen, V. Slastikov, A. Zarnescu} : \textit{Stability of point defects of degree $\pm \frac12$ in a two-dimensional nematic liquid crystal model.} Calc. Var. Partial Differential Equations \textbf{55} (2016), 55-119.

\bibitem{INSZ5} \textsc{R. Ignat, L. Nguyen, V. Slastikov, A. Zarnescu} : \textit{Symmetry and multiplicity of solutions in a two-dimensional Landau-de Gennes model for liquid crystals.} Arch. Ration. Mech. Anal. \textbf{237} (2020), 1421-1473. 

\bibitem{Kleman} \textsc{M. Kleman} : \textit{Defects in liquid crystals.} Rep.  Progr. Phys. \textbf{52} (1989), 555-654. 

\bibitem{KV} \textsc{S. Kralj, E. G. Virga} : \textit{Universal fine structure of nematic hedgehogs.} J. Phys. A \textbf{34} (2001), 829-838.

\bibitem{KVZ} \textsc{S. Kralj, E. G. Virga, S. \v{Z}umer} : \textit{Biaxial torus around nematic point defects.} Phys. Rev. E \textbf{60} (1999), 1858-1866.

\bibitem{LPZZ} \textsc{O. Lavrentovich, P. Pasini, C. Zannoni, S. Zumer (Eds.)} : \textit{Defects in Liquid Crystals: Computer Simulations, Theory and Experiments.} Nato Science Series II. Proceedings of the Nato advanced research workshop, Erice 19-23 september 2000. Springer, 2012.

\bibitem{La} \textsc{H. B. Lawson} : \textit{Surfaces minimales et la construction de Calabi-Penrose.}
Seminar Bourbaki, Vol. 1983/84.  Ast\'{e}risque No. \textbf{121-122} (1985), 197-211. 

\bibitem{Le} \textsc{L. Lemaire} : \textit{Applications harmoniques de surfaces Riemanniennes.} J. Differential Geom. \textbf{13} (1978), 51-78.

\bibitem{LemaireWood} \textsc{L. Lemaire, J. C. Wood} : \textit{Jacobi fields along harmonic 2-spheres in 3- and 4-spheres are not all integrable.} Tohoku Math. J. \textbf{61} (2009), 165-204.


\bibitem{LiWa1} \textsc{F. H. Lin, C. Y. Wang} : \textit{Stable Stationary Harmonic Maps to Spheres.}  Acta Math. Sin. (Engl. Ser.) \textbf{22} (2006), 319-330.

\bibitem{LiWa2} \textsc{F. H. Lin, C. Y. Wang} : \textit{The analysis of harmonic maps and their heat flows.} World Scientific, 2008.



\bibitem{Lu} \textsc{S. Luckhaus} : \textit{Partial Holder Continuity for Minima of Certain Energies among Maps into a Riemannian Manifold.} Indiana Univ. Math. J. \textbf{37} (1988), 349-367.

\bibitem{Ly} \textsc{I. F. Lyuksyutov} : \textit{Topological instability of singularities at small distances in nematics.} Zh. Eksp. Teor. Fiz \textbf{75} (1978), 358-360.





\bibitem{Morrey} \textsc{C. B. Morrey, Jr.} : \textit{Multiple integrals in the calculus of variations.}  Springer Science \& Business Media, 1966.


\bibitem{Palais} \textsc{R. S. Palais} : \textit{The principle of symmetric criticality.} Comm. Math. Phys. \textbf{69} (1979), 19-30.



\bibitem{PeTr} \textsc{E. Penzenstadler, H.-R. Trebin} : \textit{Fine structure of point defects and soliton decay in nematic liquid crystals.} J. Phys. France \textbf{50}, (1989), 1027-1040.






\bibitem{SU1} \textsc{R. Schoen, K. Uhlenbeck} : \textit{A regularity theory for harmonic maps.} J. Differential Geom. \textbf{17} (1982), 307-335.

\bibitem{SU2} \textsc{R. Schoen, K. Uhlenbeck} : \textit{Boundary regularity and the Dirichlet problem for harmonic maps.} J. Differential Geom. \textbf{18} (1983), 253-268.

\bibitem{SU3} \textsc{R. Schoen, K. Uhlenbeck} :  \textit{Regularity of minimizing harmonic maps into the sphere.} Invent. Math. \textbf{78} (1984), 89-100.

\bibitem{ScSl} \textsc{N. Schopohl, T. Slucking} : \textit{Defect core structure in nematic liquid crystals.} Phys. Rev. Lett. \textbf{59}(22) (1987), 2582-2585.

\bibitem{Sim83} \textsc{L. Simon} : \textit{Asymptotics for a class of non-linear evolution equations, with applications to geometric problems.} Ann. of Math. \textbf{118} (1983), 525-572.

\bibitem{Sim84} \textsc{L. Simon} : \textit{Isolated singularities of extrema of geometric variational problems. Harmonic mappings and minimal immersions.} (Montecatini, 1984), 
Lecture Notes in Math., \textbf{1161}, 206-277. Springer, Berlin, 1985.

\bibitem{Simon} \textsc{L. Simon} : \textit{Theorems on Regularity and Singularity of Energy Minimizing Maps.} Lectures in Mathematics: ETH Z\"{u}rich. Birkh\"{a}user Basel, 2012.

\bibitem{SKH} \textsc{A. Sonnet, A. Killian, S. Hess} : \textit{Alignment tensor vs director: Description of defects in nematic liquid crystals.} Phys. Rev. E \textbf{52}(1) (1995), 718-722.

\bibitem{SoSo} \textsc{J. Sou\v{c}ek, V. Sou\v{c}ek} : \textit{Morse-Sard theorem for real-analytic functions.} Comment. Math. Univ. Carolin., \textbf{13} (1972), 45-51.

\bibitem{Ura} \textsc{H. Urakawa} : \textit{Calculus of Variations and Harmonic Maps.} Translation of Mathematical Monographs \textbf{132}. AMS, Providence, Rodhe Island, 1993.

\bibitem{Verdier} \textsc{J. L. Verdier} : \textit{Applications harmonique de $\bbS^2$ dans $\bbS^4$.}
 Geometry today (Rome, 1984), 267--282, Progr. Math., \textbf{60}, Birkh \"auser Boston, Boston, MA, 1985.


\bibitem{Yu} \textsc{Y. Yu} : \textit{Disclinations in limiting Landau-de Gennes theory.} Arch. Ration. Mech. Anal. \textbf{237} (2020), 147-200.


\end{thebibliography}
\end{document}